\documentclass[preprint,10pt]{elsarticle}

\usepackage{hyperref, amsthm, amsmath, amssymb, amsfonts, mathrsfs, graphicx, xcolor, multicol, geometry, comment, tikz}
\usetikzlibrary{patterns}
\usepackage{pgfplots}
\usepackage{tikz-3dplot}

\newtheorem{theorem}{Theorem}[section] 
\newtheorem{remark}[theorem]{Remark}
\newtheorem{proposition}[theorem]{Proposition}
\newtheorem{lemma}[theorem]{Lemma}
\newtheorem{corollary}[theorem]{Corollary}
\newtheorem{problem}[theorem]{Problem} 
\newtheorem{assumption}[theorem]{Assumption}

\newcommand{\vect}[1]{#1}
\newcommand{\conj}[1]{\overline{#1}}
\newcommand{\HH}{H}
\newcommand{\LL}{L}
\newcommand{\Gtop}{\varGamma_{\textsf{u}}}
\newcommand{\Gwall}{\varGamma_{\textsf{w}}}
\newcommand{\Gbottom}{\varGamma_{\textsf{b}}}
\newcommand{\tumort}{\varOmega_{0,t}}
\newcommand{\tumor}{\varOmega_{0}}
\newcommand{\bartumor}{\overline{\varOmega}_{0}}
\newcommand{\bartumort}{\overline{\varOmega}_{0,t}}
\newcommand{\healthy}{\varOmega_{1}}
\newcommand{\healthyt}{\varOmega_{1,t}}
\newcommand{\Vomega}{V(\varOmega)}
\newcommand{\Womega}{W(\varOmega)}
\newcommand{\Vomegat}{V(\varOmega_{t})}
\newcommand{\holdall}{U}

\newcommand{\sfTheta}{\Theta}

\newcommand{\VV}{\theta}
\newcommand{\Vn}{\theta_n}
\newcommand{\nn}{n}      
\newcommand{\dn}[1]{\partial_{\nn}{#1}} 
\newcommand{\ddn}[1]{\dfrac{\partial{#1}}{\partial \nn}}  
    
\newcommand{\intO}[1]{\int_{\varOmega}{#1}{\, {dx}}}

\newcommand{\intGtop}[1]{\int_{\Gtop}{#1}{\, {ds}}} 
\newcommand{\intGtopt}[1]{\int_{\varGamma_{\textsf{u},t}}{#1}{\, {ds}}}  
  
\newcommand{\abs}[1]{\vert{#1}\vert}
\newcommand{\bigabs}[1]{\left\vert{#1}\right\vert}
\newcommand{\norm}[1]{\|{#1}\|} 
\newcommand{\cnorm}[1]{\left\|{#1}\right\|} 
\newcommand{\op}[1]{\operatorname{#1}} 
\newcommand{\dive}{\op{div}} 
 
\newcommand{\jump}[1]{\left[{#1}\right]}

\newcommand{\idmat}{I} 
\newcommand{\Lag}{F}  
\newcommand{\imagpart}[1]{\Im\left\{{#1}\right\}} 
\newcommand{\imaginary}[1]{{#1}^{\textsf{i}}}   
\newcommand{\real}[1]{{#1}^{\textsf{r}}}   
\newcommand{\reu}{u^{\textsf{r}}}
\newcommand{\imu}{u^{\textsf{i}}}
\newcommand{\reuout}{u^{\textsf{r}}_{1}}
\newcommand{\imuout}{u^{\textsf{i}}_{1}}
\newcommand{\reuin}{u^{\textsf{r}}_{0}}
\newcommand{\imuin}{u^{\textsf{i}}_{0}}
\newcommand{\dett}{I_{t}}
\newcommand{\At}{A_{t}}
\newcommand{\bt}{b_{t}}
\newcommand{\Mt}{M_{t}}
\newcommand{\wt}{w^{t}}
\newcommand{\ut}{u^{t}}
\newcommand{\pt}{p^{t}}
\newcommand{\zt}{z^{t}}

\newcommand{\coeffk}{k}
\newcommand{\mincoeffk}{\underline{\coeffk}}
\newcommand{\maxcoeffk}{\overline{\coeffk}}
\newcommand{\minsigma}{\underline{\sigma}}
\newcommand{\maxsigma}{\overline{\sigma}}
\newcommand{\kt}{{\coeffk}^{t}}
\newcommand{\Qt}{{Q}^{t}}
\newcommand{\dzero}{{d}_{\circ}}
 
\newcommand{\bigdO}{\partial \varOmega}

\newcommand{\intOt}[1]{\int_{\varOmega_{t}}{#1}{\, {d} x_{t}}}

\newcommand{\meshh}{\textsf{h}}
\newcommand{\meshgamma}{\textsf{h}_{\gamma}}
\newcommand{\imaginaryh}[1]{{#1}^{\textsf{i}}_{\meshh}}   
\newcommand{\realh}[1]{{#1}^{\textsf{r}}_{\meshh}}   
\newcommand{\imuh}{u_{\meshh}^{\textsf{i}}}

\newcommand{\aK}{\textsf{a}_{K}}
\newcommand{\hK}{\textsf{h}_{K}}
\newcommand{\RK}{R_{K}}
\newcommand{\bK}{b_{K}}
\newcommand{\Jgamma}{\textsf{J}_{\gamma}}
\newcommand{\RKa}{R_{K}^{a}}

\newcommand{\Jgammaa}{\textsf{J}_{\gamma}^{a}}
\newcommand{\ui}{u_{i}}
\newcommand{\dotui}{\dot{u}_{i}}
\newcommand{\psii}{\psi_{i}}
\newcommand{\II}{I}  
\newcommand{\xx}{\vect{x}}
\newcommand{\yy}{\vect{y}}
\newcommand{\edgemat}{{E}_{K}} 
\newcommand{\dotEK}{\dot{E}_{K}}

\newcommand{\etaK}{{\eta}_{K}}
\newcommand{\muK}{{\mu}_{K}}

\newcommand{\wh}{{w}_{\meshh}}
\newcommand{\uh}{{u}_{\meshh}}
\newcommand{\ph}{{p}_{\meshh}}
\newcommand{\erru}{{e}}
\newcommand{\errp}{{e}^{a}}
\newcommand{\vh}{{v}_{\meshh}}
\newcommand{\dotuh}{\dot{u}_{\meshh}}
\newcommand{\dotvh}{\dot{v}_{\meshh}}
\newcommand{\tildedotuh}{\dot{\tilde{u}}_{\meshh}}
\newcommand{\spaceVh}{{V}_{\meshh}(\varOmega)} 
\newcommand{\spaceWh}{{W}_{\meshh}(\varOmega)}
\newcommand{\varOmegah}{\varOmega_{\meshh}}
\newcommand{\tumorh}{\varOmega_{0,\meshh}} 
\newcommand{\healthyh}{\varOmega_{1,\meshh}} 
\newcommand{\Vomegah}{V(\varOmegah)}
\newcommand{\Jh}{{J}_{\meshh}} 
\newcommand{\holdallh}{U_{\meshh}}
\newcommand{\spacePone}{{P}^{1}_{\meshh}} 
\newcommand{\spaceVhOh}{{S}^{1}_{\meshh}} 
\newcommand{\spaceVhOhd}{{S}^{1}_{\meshh}(\varOmegah)^{d}} 
\newcommand{\VVh}{{\VV}_{\meshh}} 
\newcommand{\dotXh}{\dot{X}_{\meshh}}
\newcommand{\tu}{\tilde{{u}}}
\newcommand{\tv}{\tilde{{v}}}
\newcommand{\tuh}{\tilde{{u}}_{\meshh}}  
\newcommand{\tvh}{\tilde{{v}}_{\meshh}}  
\newcommand{\dottuh}{\dot{\tilde{{u}}}_{\meshh}}  
\newcommand{\hk}{{\meshh}_{K}}
\newcommand{\ak}{\aK}
\newcommand{\dotxxi}{\dot{\vect{x}}_{i}}
\newcommand{\xxi}{\vect{x}_{i}}
\newcommand{\xxo}{\vect{x}_{0}}
\newcommand{\xxd}{\vect{x}_{d}} 
\newcommand{\Th}{\mathcal{T}_{\meshh}}
\newcommand{\Tht}{\mathcal{T}_{\meshh}(t)}
\newcommand{\Kt}{K(t)}
\newcommand{\Ko}{K(0)}
\newcommand{\FK}{F_{K(t)}}

\newcommand{\dotFK}{\dot{F}_{K(t)}}
\newcommand{\detJac}{\mathbb{J}}
\newcommand{\Jac}{{J}_{F}}

\newcommand{\xok}{\vect{x}_{0}^{K}}
\newcommand{\tpsi}{\tilde{{\psi}}}

\newcommand{\tnabla}{\tilde{{\nabla}}}
\newcommand{\intgamma}[1]{\int_{\gamma}{#1}{\, ds}} 
\newcommand{\intOh}[1]{\int_{\varOmegah}{#1}{\, d\xx}} 
 
\newcommand{\intGtopy}[1]{\int_{\Gtop(t)}{#1}{\, {d\yy}}} 

\newcommand{\intK}[1]{ \int_{K} {#1} \,d\xx }

\newcommand{\sumint}[1]{\sum_{K}\int_{K}{\left( #1 \right)}{\, d\xx}}
\newcommand{\suminto}[1]{\sum_{K(0)}\int_{\Ko}{\left( #1 \right)}{\, d\xx}}
\newcommand{\sumintt}[1]{\sum_{K(t)}\int_{\Kt}{\left( #1 \right)}{\, d\yy}}
\newcommand{\Gtoph}{\varGamma_{\textsf{u}, \meshh}}
\makeatletter
\newcommand{\oset}[3][0ex]{%
  \mathrel{\mathop{#3}\limits^{
    \vbox to#1{\kern-1\ex@
    \hbox{$\scriptstyle#2$}\vss}}}}
\makeatother

\newcommand{\alert}[1]{{\color{black}{#1}}}
\newcommand{\harbrecht}[1]{{\color{black}{#1}}}
\newcommand{\paruch}[1]{{\color{black}{#1}}}

\journal{a journal}

\begin{document}

\begin{frontmatter}



\title{Localization of tumor through a non-conventional numerical shape optimization technique} 


\author[J. F. T. Rabago]{Julius Fergy Tiongson Rabago} 
\ead{jfrabago@gmail.com \sep jftrabago@gmail.com}

\affiliation{organization={Faculty of Mathematics and Physics, Institute of Science and Engineering, Kanazawa University},
            addressline={Kakumamachi}, 
            city={Kanazawa},
            postcode={920-1192}, 
            state={Ishikawa},
            country={Japan}}

\begin{abstract}
This paper presents a method for estimating the shape and location of an embedded tumor using shape optimization techniques, specifically through the coupled complex boundary method.  
The inverse problem---characterized by a measured temperature profile and corresponding heat flux (e.g., from infrared thermography)---is reformulated as a complex boundary value problem with a complex Robin boundary condition, thereby simplifying its over-specified nature.  
The geometry of the tumor is identified by optimizing an objective functional that depends on the imaginary part of the solution throughout the domain.
The shape derivative of the functional is derived through shape sensitivity analysis.
An iterative algorithm is developed to numerically recover the tumor shape, based on the Riesz representative of the gradient and implemented using the finite element method.
In addition, the mesh sensitivity of the finite element solution to the state problem is analyzed, and bounds are established for its variation with respect to mesh deformation and its gradient.
Numerical examples are presented to validate the theoretical results and to demonstrate the accuracy and effectiveness of the proposed method.
\end{abstract}




\end{frontmatter}



\section{Introduction}\label{sec:Introduction}%
Body temperature is widely recognized as a key health indicator. 
Generally, the surface temperature of the skin is influenced by underlying blood circulation, local metabolic activity, and heat exchange between the skin and its environment \cite{Chato1985,Bowman1985}. 
Variations in any of these factors can alter skin surface temperature and heat flux, reflecting the body's physiological state. 
Tumors, due to their unique structure and angiogenesis processes, often exhibit abnormal behavior. 
Features such as inflammation, increased metabolic rates, interstitial hypertension, abnormal blood vessel morphology, and a lack of homeostatic regulation contribute to tumors generating and dissipating heat differently compared to normal tissue. 
For instance, skin temperature above a tumor, such as a malignant melanoma or breast tumor, is often significantly higher than that of surrounding tissues \cite{Knappetal2022,Lawson1956,LawsonChugtai1963,MiyakawaBolomey1995,SantaCruzetal2009}.

Abnormal skin surface temperature can indicate tumor location and size, as well as monitor changes after treatment.
Infrared thermography, a non-invasive and contact-free technique, detects subtle temperature variations by measuring skin-emitted radiation.
While early adoption was limited by cost and sensitivity issues, recent advances have improved its utility in assessing superficial tissues and related physiological processes.
For instance, Santa Cruz et al. \cite{SantaCruzetal2009} used thermography to identify acute skin reactions after boron neutron capture therapy (BNCT), linking high-temperature regions to dose distributions and undetectable melanoma nodules on CT scans.

\paruch{Accurate, non-invasive diagnosis is essential for effective treatment planning, particularly in detecting tumors beneath the skin. 
Prior studies have used thermography and inverse modeling to estimate internal temperature distributions and guide hyperthermia-based therapies. 
For example, \cite{Paruch2020} and \cite{MajchrzakParuch2011} studied tumor ablation via electromagnetic heating, while \cite{VaronOrlandeElicabe2015} and \cite{NgJamil2014} focused on thermal estimation and parameter sensitivity. 
These works highlight the diagnostic potential of thermal imaging. 
Shape optimization builds on this by providing a mathematical framework to estimate the size and location of abnormalities from surface temperature data (see, e.g., \cite{MajchrzakParuch2011,AgnelliPadraTurner2011,PadraSalva2013}). 
This approach can improve geometry-aware diagnostics and support treatments like radiofrequency ablation, especially when surgery is not feasible.}

Building on these advancements, we focus on a specific diagnostic scenario involving malignant melanoma. 
In particular, this study addresses the inverse problem of estimating the location and size of a malignant melanoma based on abnormal skin surface temperature. 
The tissue domain $\varOmega$ is modeled using Pennes' bioheat equation \cite{Pennes1948}, which accounts for thermal differences between healthy and tumor regions \cite{Lawson1956,LawsonChugtai1963,MiyakawaBolomey1995,SantaCruzetal2009}. 
Assuming piecewise constant properties, we aim to identify the tumor region $\bartumor \subset \varOmega$ using boundary temperature data on $\partial\varOmega$. 
A shape optimization approach \cite{DelfourZolesio2011,HenrotPierre2018,SokolowskiZolesio1992,MuratSimon1976} is used, where surface measurements (e.g., from thermography) are incorporated as complex Robin boundary conditions. 
The tumor shape is recovered by minimizing a functional involving the imaginary part of the solution, with optimization guided by the shape gradient (\textit{Eulerian} derivative).

In shape calculus, the Eulerian derivative \cite[p.~471]{DelfourZolesio2011} of a shape functional can be expressed in either distributed \cite{LaurainSturm2016} or boundary form \cite{DelfourZolesio2011,HenrotPierre2018,SokolowskiZolesio1992}. 
The boundary form, derived from the Hadamard-Zol\'{e}sio structure theorem \cite[Thm. 3.6, p.~479]{DelfourZolesio2011}, is notable in numerical shape optimization due to its concise representation \cite{MohammadiPironneau2001,SchulzSiebenbornWelker2015}. 
However, finite element approximations of this form are not suitable for irregular boundaries. 
The distributed Eulerian derivative, being more broadly applicable \cite{LaurainSturm2016}, warrants greater consideration in numerical computations.

This work employs finite element methods (FEMs) for numerical computations. 
The motivation for this approach is discussed in subsection \ref{subsec:review_of_literature_and_motivation}. 
FEMs are widely used in shape optimization to discretize and solve PDEs on arbitrary domains \cite{MohammadiPironneau2001}. 
Accurate numerical shape gradients are crucial for optimization algorithms \cite{DelfourZolesio2011}, particularly in geometric inverse problems. 
Distributed shape gradient algorithms are preferable to surface-based ones, as they improve finite element mesh quality and enhance efficiency, as shown in \cite{LaurainSturm2016,SchulzSiebenbornWelker2015}, where numerical comparisons confirm the robustness and efficiency of the distributed approach.
In \cite{HiptmairPaganiniSargheini2015}, the distributed Eulerian derivative demonstrates faster convergence and higher accuracy for elliptic problems. 
In this work, shape gradient information will be obtained using Delfour-Zol\'{e}sio's minimax formulation approach \cite{DelfourZolesio1988a}.
This approach, in contrast to the chain rule method, avoids the need to compute material or shape derivatives of the states. 
It is a well-established method, frequently used in optimal shape design and control theory \cite{Azegami2020}. 
However, its application in shape sensitivity analysis is not straightforward due to the pseudo-time dependence of the function spaces in the minimax formulation. 
To address this, two methods are available \cite{DelfourZolesio2011,DelfourZolesio1988a}: the function space parameterization and the function embedding technique. 
In this study, we will apply the former.


\textit{Paper organization.}  Following the medical context of skin cancer and its thermal characteristics, the remainder of this paper is organized as follows:  
Section~\ref{sec:mathematical_model} describes the mathematical model for heat transfer in the human body.  
Section~\ref{sec:inverse_problem} formulates the inverse problem.  
Section~\ref{sec:CCBM} introduces the coupled complex boundary method (CCBM) and the resulting shape optimization problem, where we clarify the motivation and contribution of this study.  
Section~\ref{sec:sensitivity_analysis} presents the computation of the shape derivative of the corresponding shape functional in the CCBM formulation.  
Section~\ref{sec:Mesh_deformation_and_transformation_of_integrals} introduces some tools needed to examine mesh deformations.
Section~\ref{sec:mesh_sensitivity_analysis} conducts a mesh sensitivity analysis for the finite element solution of the state problem.  
Section~\ref{sec:discretization_of_the_cost_function} briefly discusses the discretization of the objective function and highlights key considerations for using the volume expression in shape gradient calculations.  
Section~\ref{sec:numerical_algorithm_and_examples} details the numerical implementation of our approach and presents simulations, including cases with and without random noise in both two and three dimensions, to evaluate the algorithm’s performance.  
Finally, Section~\ref{sec:conclusion} concludes with final remarks and future directions.
Several appendices are also provided to present key results used in the study, including an \textit{a posteriori} estimate for the proposed formulation.

%
\section{The Model Equation}\label{sec:mathematical_model}
\subsection{\harbrecht{The bioheat transfer problem}}
Following \cite{AgnelliPadraTurner2011}, we consider the steady-state form of the Pennes equation \cite{Pennes1948}:

\[
    -\nabla \cdot (\sigma(x) \nabla{u}(x)) + \coeffk(x)(u(x) - T_b) = q(x), \quad x \in \varOmega \subset \mathbb{R}^{d}, \quad d \in \{ 2, 3 \},
\]
where $\sigma$ is the thermal conductivity, $k = G_b \rho_b c_b$ is the perfusion coefficient ($G_b$ is the perfusion rate, $\rho_b$ is the blood density, and $c_b$ is the specific heat of blood), $q$ is the metabolic heat source, and $T_b$ is the blood temperature.

This study primarily considers a single tumor region, denoted by $\tumor$, surrounded by healthy tissue, denoted by $\healthy$, but the results also apply in the case of multiple tumors.
The coefficients are defined as piecewise functions, where (cf. \cite{AgnelliPadraTurner2011}):
\[
    \sigma(x) =
    \begin{cases}
        \sigma_1, & \text{if } x \in \healthy, \\
        \sigma_0, & \text{if } x \in \tumor,
    \end{cases}
    \quad
    \coeffk(x) =
    \begin{cases}
        \coeffk_1, & \text{if } x \in \healthy, \\
        \coeffk_0, & \text{if } x \in \tumor,
    \end{cases}
    \quad
    Q(x) =
    \begin{cases}
        Q_1, & \text{if } x \in \healthy, \\
        Q_0, & \text{if } x \in \tumor.
    \end{cases}
\]

Transmission conditions are applied on $\partial \tumor$, with boundary conditions on $\partial \varOmega$ consisting of $\Gtop$, $\Gwall$, and $\Gbottom$: the bottom boundary $\Gbottom$ has a constant temperature $T_b > 0$, $\Gwall$ is adiabatic, and $\Gtop$ represents convective heat exchange.
Defining $u_1 = u|_{\healthy}$ and $u_0 = u|_{\tumor}$, the transmission problem is:
\begin{equation}\label{eq:main_equation}
    \begin{cases}
        -\sigma_{1} \Delta u_{1} + \coeffk_{1} u_{1} = Q_{1}, & \text{in } \healthy, \\[0.5em]
        -\sigma_{0} \Delta u_{0} + \coeffk_{0} u_{0} = Q_{0}, & \text{in } \tumor, \\[0.5em]
        u_{1} = u_{0}, & \text{on } \partial \tumor, \\[0.5em]
        -\sigma_{1} \ddn{u_{1}} = -\sigma_{0} \ddn{u_{0}}, & \text{on } \partial \tumor, \\[0.5em]
        -\sigma_{1} \ddn{u_{1}} = \alpha(u_{1} - T_{a}), & \text{on } \Gtop, \\[0.5em]
        -\sigma_{1} \ddn{u_{1}} = 0, & \text{on } \Gwall, \\[0.5em]
        u_{1} = T_{b}, & \text{on } \Gbottom,
    \end{cases}
\end{equation}
where $\alpha$ is the heat transfer coefficient, $T_a$ is the ambient temperature, and $\nn$ is the outward-pointing unit normal.
Figure~\ref{fig:domain} provides an illustration of the domain shape in two dimensions (2D).
\begin{figure}[h!]
    \centering
    \begin{tikzpicture}[scale=0.9]
        \fill[pattern=north west lines] (0,0) ellipse (1 and 0.5);
        \draw[thick] (-3, -1) rectangle (3, 1);
        \draw[thick] (0,0) ellipse (1 and 0.5);
        \node at (0,0) {\Large $\tumor$};
        \node at (1.4,-0.3) {\large $\partial\tumor$};
        \node at (0,1.3) {\Large $\Gtop$};
        \node at (0,-1.3) {\Large $\Gbottom$};
        \node at (-3.4,0) {\Large $\Gwall$};
        \node at (3.4,0) {\Large $\Gwall$};
        \node at (-2.0,0.5) {\Large $\healthy$};
    \end{tikzpicture}
    \caption{Two-dimensional illustration}
    \label{fig:domain}
\end{figure}

Evaporation on the skin surface can also be considered. 
In this case, a time-dependent heat equation must be used, as the humidity coefficient varies with time. 
\subsection{Notations}\label{subsec:notations}
In this study, we mainly work on complex versions of Sobolev and Lebesgue spaces, as the proposed method requires complex-valued functions. 
We will not, however, distinguish between spaces of real and complex functions, as the context will clarify their use.

Let $d \in \{2,3\}$ denote the spatial dimension. 
We define $\partial_{i} := \partial/\partial x_{i}$ and $\nabla := (\partial_{1}, \ldots, \partial_{d})^{\top}$. 
Also, $\partial_{\nn} := \partial/\partial \nn = \sum_{i=1}^{d} n_{i} \partial_{i}$, where $\nn \in \mathbb{R}^{d}$ is the outward unit normal vector to $\varOmega$. 
The inner product of $\vect{a}, \vect{b} \in \mathbb{R}^{d}$ is $\vect{a} \cdot \vect{b} := \vect{a}^{\top} \vect{b} \equiv \langle \vect{a}, \vect{b} \rangle$, and we use $\varphi_{n}$ for $\varphi \cdot \nn$ when convenient.

\harbrecht{For a vector-valued function $\varphi := (\varphi_1, \ldots, \varphi_d)^{\top} : \varOmega \to \mathbb{R}^d$, we define the matrix $\nabla \varphi \in \mathbb{R}^{d \times d}$ by $(\nabla \varphi)_{ij} := \partial \varphi_j / \partial x_i$ for $i, j = 1, \ldots, d$. The Jacobian of $\varphi$ is then given by $(D\varphi)_{ij} := \partial \varphi_i / \partial x_j$, so that $D\varphi = (\nabla \varphi)^{\top}$.}
Accordingly, we write $\partial_{\nn} {\varphi} := (D{\varphi})\nn$.

Let $m \in \mathbb{N} \cup \{0\}$. 
The spaces $L^{2}(\varOmega)$ and $H^{m}(\varOmega) = W^{m,2}(\varOmega)$ denote real Lebesgue and Sobolev spaces with natural norms. 
The space $H^{m}(\varOmega)^{d}$ consists of vector-valued functions ${\varphi}: \varOmega \to \mathbb{R}^{d}$ with $\harbrecht{\varphi}_{i} \in H^{m}(\varOmega)$ for $i = 1, \ldots, d$, and its norm is $\norm{{\varphi}}_{H^{m}(\varOmega)^{d}} := \left( \sum_{i=1}^{d} \norm{\harbrecht{\varphi}_{i}}^{2}_{H^{m}(\varOmega)} \right)^{1/2}$.

For the complex version, we define the inner product $(\!( {\varphi}, {\psi} )\!)_{m,\varOmega,d} = \sum_{j=1}^{d}({\varphi}, \conj{\psi})_{m,\varOmega}$ and the norm $\cnorm{{\psi}}_{m,\varOmega,d} = \sqrt{ (\!( {\psi},\conj{\psi} )\!)_{m,\varOmega, d}}$. 
For real-valued functions, $\cnorm{\cdot}_{m,\varOmega,d} = \norm{\cdot}_{W^{m,p}(\varOmega)}$. 
Here, $\conj{\varphi}$ denotes the complex conjugate of $\varphi$, and `$:$' denotes the Frobenius inner product: $\nabla {\varphi} : \nabla {\psi} = \sum_{i,j=1}^{d} {\partial_{i} {\varphi}_{j}} {\partial_{i} {\overline{\psi}}_{j}}$.
\alert{The spaces in discrete setting, as understood in the present study, are also understood to be complex-valued.}

For later use, we also denote by $\nn$ the outward unit normal to $\partial\tumor$ pointing into $\healthy$. 
We define $\dn{u}_{0}$ (respectively, $\dn{u}_{1}$) as the normal derivative from inside $\tumor$ (respectively, $\healthy$) at $\partial\healthy$, and $\jump{\cdot}$ represents the jump across this interface.
We write $\intO{\cdot}$ instead of $\int_{\healthy}{\cdot , dx} + \int_{\tumor}{\cdot , dx}$ when it causes no confusion.
\harbrecht{Finally, throughout the paper, $\Re$ and $\Im$ denote the real and imaginary parts of a complex-valued expression and $\abs{\cdot}$ denotes the modulus of a complex number.}

\subsection{Weak formulation of the problem}
In this investigation, we generally assume that 
\begin{assumption}\label{assume:weak_assumptions}
	\begin{itemize}
		\item[(A1)] $\varOmega \subset \mathbb{R}^{d}$, $d\in\{2,3\}$, is a bounded domain of class ${{C}}^{k,1}$, $k \in \mathbb{N}$;
		\item[(A2)] $\varOmega$ contains a ${{C}}^{k,1}$ subdomain $\omega \Subset \varOmega$ characterized by a finite jump in the coefficients of the PDE such that $\varOmega \setminus \overline{\omega}$ is connected;
		\item[(A3)] $\sigma_{0}, \sigma_{1} \in W^{1,\infty}(\varOmega)$ and $\sigma_{0}, \sigma_{1} > 0$; 
		\item[(A4)] $\coeffk_{0}, \coeffk_{1} \in W^{1,\infty}(\varOmega)$ and $\coeffk_{0}, \coeffk_{1} > 0$; 
		\item[(A5)] $Q_{0}, Q_{1} \in {H}^{1}({\varOmega})$.
	\end{itemize}
\end{assumption}
For later reference and technical purposes, especially in Section~\ref{sec:mesh_sensitivity_analysis}, we assume the existence of constants $\minsigma, \maxsigma, \mincoeffk, \maxcoeffk > 0$ such that for almost every $x \in \varOmega$, the following inequalities hold:
\begin{equation}\label{eq:bounds_for_sigma_and_k}
    \minsigma \leqslant \sigma(x) \leqslant \maxsigma
    \quad \text{and} \quad
    \mincoeffk \leqslant \coeffk(x) \leqslant \maxcoeffk.
\end{equation}

The regularity assumed in (A3) and (A4) of Assumption~\ref{assume:weak_assumptions} is stronger than necessary for most of the results, such as Proposition~\ref{prop:well_posedness_of_state}.
To prove this proposition, it suffices to assume that $\alpha \in L_{+}^{\infty}(\Gtop)$, $T_{a} \in H^{-1/2}(\Gtop)$, $T_{b} \in H^{1/2}(\Gbottom)$, and $Q \in L^{2}(\varOmega)$, where $L_{+}^{\infty}$ denotes the space of essentially bounded functions with positive essential infima. 
For simplicity, we often treat $\alpha$, $T_{a}$, and $T_{b}$ as constants, and $\sigma$, $\coeffk$, and $Q$ as piecewise constants.

Let $u_{b} \in H^1(\varOmega)$ be such that $u_{b} = T_{b}$ on $\Gbottom$, and define the following:
\begin{align}
\Vomega &:= \{ v \in H^1(\varOmega) \mid v = 0 \text{ on } \Gbottom \},\label{eq:space_Vomega}\\
a^{\dagger}(u, v) &:= \intO{\sigma \nabla{u} \cdot \nabla{v}} + \intO{kuv} + \intGtop{ \alpha uv }, \quad u, v \in H^{1}(\varOmega), \nonumber\\
l^{\dagger}(v) &:= \intO{Qv} + \intGtop{ \alpha T_{a} v }, \quad v \in H^{1}(\varOmega).\nonumber
\end{align}
With these, the weak formulation of \eqref{eq:main_equation} can be stated as follows:
\begin{problem}\label{prob:original_weak_form}
Find $u \in H^1(\varOmega)$ such that $u = u_{b} + w$, $w \in \Vomega$, and $a^{\dagger}(w, v) = l^{\dagger}(v) - a^{\dagger}(u_{b}, v)$, for all $v \in \Vomega$.
\end{problem}

If $u_{b} \equiv T_{b}$ in $\varOmega$, then the problem simplifies: 
Find $u \in \Vomega$ such that $a^{\dagger}(u, v) = l^{\dagger}(v) - a^{\dagger}(T_{b}, v)$, for all $v \in \Vomega$.
Unless stated otherwise, we assume $T_{b} \equiv 0$ to simplify the exposition.
For later use, we introduce the linear manifold $\Womega := \{ w \in H^1(\varOmega) \mid w|_{\Gbottom} = T_{b} \}$.

\section{The Inverse Problem}\label{sec:inverse_problem}
According to \cite{Lawson1956,LawsonChugtai1963,SantaCruzetal2009}, highly vascularized tumors can increase local blood flow and heat production, raising the temperature of the skin surface.
This abnormal temperature can help predict the location and size of the tumor.
To achieve this, the following inverse problem is solved:
\begin{problem}\label{prob:inverse_problem}
	Given $\varOmega$, constants $T_{a}$, $T_{b}$, $\alpha$, and a measured temperature profile $h$ on the upper boundary $\Gtop$, \alert{find a subdomain $\tumor^{\ast} \subset \varOmega$ such that the solution $u=u(\varOmega)=u(\healthy \cup \overline{\tumor^{\ast}})$ of the boundary value problem \eqref{eq:main_equation} matches $h$ on $\Gtop$; i.e., $u|_{\Gtop} = h$ where $u$ solves \eqref{eq:main_equation} with $\tumor = \tumor^{\ast}$.}
\end{problem}
For clarity, we emphasize here that all coefficients in \eqref{eq:main_equation}, the data $Q$, $T_{a}$, $T_{b}$, and $\partial\varOmega$ are known parameters.

The straightforward way to find the \textit{exact} inclusion $\tumor^{\ast}$ is by minimizing the least-squares misfit function:
\begin{equation}\label{eq:least_squares}
	J_{LS}(\tumor,u(\tumor)) = \frac{1}{2} \intGtop{ (u(\tumor) - h)^{2}},
\end{equation}
where $J_{LS}$ reaches its minimum when $\tumor = \tumor^{\ast}$ (see \cite{AgnelliPadraTurner2011}). 
To find the optimal shape solution of this functional, one applies shape sensitivity analysis \cite{DelfourZolesio2011,HenrotPierre2018,MuratSimon1976,SokolowskiZolesio1992}, which examines the variation of the functional with respect to the unknown shape.

Since we will use tools from shape calculus, we must precisely define the admissible set of geometries for the inclusion $\tumor$.
Let \harbrecht{$\dzero > 0$ be a fixed (small) real number} and define $\varOmega_{{\circ}} \Subset \varOmega$ with a ${{C}}^{\infty}$ boundary as:
\[
\varOmega_{{\circ}} := \{ x \in \varOmega \mid d(x,\bigdO) > \dzero \}.
\]
Let $k \in \mathbb{N}$. We define the set of all admissible geometries, or simply the \textit{hold-all} set $\holdall^{k}$, as follows:
\begin{equation}\label{eq:admissible_set_of_subdomains}
	\holdall^{k} := \{\tumor \Subset \varOmega_{{\circ}} \mid \harbrecht{d(x,\bigdO) > \dzero}, \forall x \in \tumor,\ \healthy \text{ is connected, and } \tumor \text{ is of class } {{C}}^{k,1} \}.
\end{equation}
We say that $\varOmega$ is \textit{admissible} if it contains a subdomain $\tumor \in \holdall^{k}$. 
Setting $k = 1$ suffices for our analysis.
Hereinafter, without further notice, we assume that $\varOmega$ is admissible.
\harbrecht{In \eqref{eq:admissible_set_of_subdomains}, the first condition means that $\tumor$ is compactly contained in $\varOmega$, i.e., no part of $\tumor$ touches or comes arbitrarily close to the boundary $\partial \varOmega$.
Equivalently, the tubular neighborhood of radius $\dzero > 0$ around $\tumor$ lies strictly inside $\varOmega$, which ensures a uniform positive distance $\dzero$ between $\tumor$ and the boundary $\partial \varOmega$.}

We assume the following key assumption:
\begin{assumption}\label{eq:key_assumption}
	There exists $\tumor^{\ast} \in \holdall^{k}$ such that $u^{\ast} = u((\healthy \setminus \bartumor^{\ast}) \cup \tumor^{\ast})$ solves Problem~\ref{prob:inverse_problem}.
\end{assumption} 
To emphasize the dependency of $u$ on the inclusion $\tumor$, we will use the notation $u(\tumor)$ instead of $u((\healthy \setminus \bartumor) \cup \tumor)$ for simplicity. 
Note that the mapping $\varOmega(\tumor) \mapsto u(\varOmega(\tumor))$ is well-defined since Problem~\ref{prob:original_weak_form} is well-posed. 
Thus, we adopt the simplified notation $J(\tumor) = J(\tumor, u(\tumor))$ for the functional $J$.
Nonetheless, either notation is used when convenient.
\section{The Coupled Complex Boundary Method}\label{sec:CCBM}
In this study, we propose the Coupled Complex Boundary Method (CCBM) for shape optimization to numerically approximate the solution to the inverse problem.
In the standard approach, the data needed for fitting, as described in \eqref{eq:least_squares}, is collected from part or all of the accessible boundary or surface of the medium or object.
However, this approach often leads to numerical instability, especially when tracking Neumann data using the least-squares method.
In contrast, CCBM transfers the fitting data from the boundary to the interior of the domain, similar to the Kohn-Vogelius method \cite{KohnVogelius1987}.
For a Cauchy pair on the boundary subjected to measurement, CCBM achieves this by coupling Neumann and Dirichlet data to derive a Robin condition. 
Neumann and Dirichlet data represent the real and imaginary parts of the Robin boundary condition, respectively. 
For Problem~\ref{prob:inverse_problem}, the coupling is done using a Robin boundary condition and a Dirichlet data derived from measurements.
This approach has a regularizing effect similar to that observed with the Kohn-Vogelius method \cite{Rabago2023a}.
Note, however, that this new boundary value problem is more complex and increases the dimension of the discrete system. 
Despite this, CCBM is more robust than the conventional misfit functional approaches, which rely on boundary integrals. 
Effective numerical methods can be developed within this framework.

Because measurements are only available on $\Gtop \subsetneq \partial\varOmega$, the Kohn-Vogelius method does not yield a simple form for the objective function, which defines the difference between two auxiliary problems. 
These problems involve the solutions of two PDEs: one with Dirichlet data and the other with corresponding Neumann data; see, e.g., \cite{AfraitesRabago2025,EpplerHarbrecht2012a}. 
\subsection{The CCBM and contribution to the literature}\label{subsec:review_of_literature_and_motivation}
Originally proposed in \cite{Chengetal2014} for an inverse source problem, the CCBM approach has since been applied to a variety of inverse problems. 
It was employed for an inverse conductivity problem in \cite{Gongetal2017}, and for parameter identification in elliptic problems in \cite{ZhengChengGong2020}. 
Extensions to shape inverse problems and geometric inverse source problems were presented in \cite{Afraites2022} and \cite{AfraitesMasnaouiNachaoui2022}, respectively. 
Further applications include the exterior Bernoulli problem in \cite{Rabago2023b}, free surface problems in \cite{RabagoNotsu2024}, an inverse Cauchy problem for the Stokes system in \cite{Ouiassaetal2022}, and the detection of an immersed obstacle in Stokes fluid flow in \cite{RabagoAfraitesNotsu2025}.

In this study, we use CCBM to locate a tumor in healthy tissue via thermal imaging based on Pennes' equation. 
Our approach is novel in formulating an objective function defined by a domain integral with a single pair of Cauchy data on the accessible boundary.
In free boundary \cite{Rabago2023b} and free surface problems \cite{RabagoNotsu2024}, CCBM has shown greater stability and accuracy than the classical least-squares approach, with computational cost similar to the Kohn-Vogelius method.

This study makes several contributions to the literature:
\begin{itemize}
    \item We introduce a new objective function \eqref{eq:ccbm_functional} for addressing Problem~\ref{prob:inverse_problem}, derived from the CCBM formulation; see subsection~\ref{subsec:CCBM_formulation}. 
    This functional has not been previously \harbrecht{studied} in the context of the current problem.
    \item We derive the first-order shape derivative of the objective function using Delfour-Zol\'{e}sio's minimax formulation approach \cite{DelfourZolesio1988a}. 
    To the best of our knowledge, this is the first work to consider a Lagrangian formulation of an optimal control problem with a complex PDE constraint. 
    We emphasize that the construction of the Lagrangian functional must be done carefully to obtain the correct expression.
    The method naturally applies to broader shape optimization problems constrained by complex-valued PDEs.
\end{itemize}
To compare our study with previous work, we begin by discussing the implementation of the approach \eqref{eq:least_squares} in \cite{AgnelliPadraTurner2011}. 
In that study, the authors employed a second-order finite difference scheme to solve \eqref{eq:main_equation} for both the primal and adjoint systems.
Focusing on the localization of melanoma nodules, they assumed the tumor to be circular, requiring only the center coordinates and radius for the inversion procedure. 
This simplification reduced the dimensionality of the design space for optimizing $J_{LS}$, as the design variables were limited to the tumor's center and radius.
In contrast, the earlier work \cite{ParuchMajchrzak2007} identified the tumor region using an evolutionary algorithm combined with the multiple reciprocity boundary element method.

\paruch{In this work, we apply a numerical approximation method based on finite elements, similar to that in \cite{PadraSalva2013}, while introducing five key innovations:
\begin{itemize}
	\item In contrast to \cite{AgnelliPadraTurner2011}, we employ FEMs---a numerical method similar to that used in \cite{PadraSalva2013}---to solve the state and adjoint state equations numerically. 
	The design variables in our approach are the nodal points on the boundary interface. 
	This method supports not only symmetrical tumors but also asymmetrical shapes and tumors with concave regions, thereby accommodating complex geometries.
	We emphasize that the consideration of non-elliptical tumor shapes reveals limitations in prior approaches \cite{AgnelliPadraTurner2011} (see also \cite{ParuchMajchrzak2007}), which control only the radii and center, restricting their ability to capture such complex geometries.
	\item In the computation of the deformation fields, we use the volume integral representation of the shape gradient instead of the boundary integral form (see Hadamard-Zol\'{e}sio structure theorem \cite[Thm 3.6, p.~479]{DelfourZolesio2011}).
	For a comparison of the convergence order of volume and boundary formulations in the finite element framework, we refer the reader to \cite{HiptmairPaganiniSargheini2015}.
	\item Whereas \cite{AgnelliPadraTurner2011,PadraSalva2013} test their numerical approaches only in two dimensions, our work extends the methodology to three dimensions, thereby demonstrating its applicability to higher-dimensional problems.
	\item To handle noisy measurements, especially in identifying deep small tumor, we apply $\rho$-weighted volume penalization together with the \textit{balancing principle} \cite{ClasonJinKunisch2010b}.
	Originally developed for parameter identification, this principle is, to the best of our knowledge, applied here for the first time to a shape identification problem; see subsections~\ref{subsec:test_2} and \ref{subsec:test_3d}.
	The motivation and advantages of this approach, along with the Sobolev gradient method, are discussed in subsections~\ref{subsec:test_2} and \ref{subsec:test_3d}. Additionally, we note that this approach avoids the need to calculate the curvature expression, which arises when using perimeter or surface measure penalization, as their shape derivatives involve curvature. While curvature is easy to compute in 2D, it is more complex in 3D.
	\item Although our numerical framework is similar to \cite{PadraSalva2013}, we extend the method to reconstruct multiple inclusions of varying sizes and incorporate noise handling---features not addressed in the aforementioned work.
\end{itemize}
}
\alert{Additionaly, a mesh sensitivity analysis is carried out for the FE solution of the \textit{material derivative} of the state problem; see Section~\ref{sec:mesh_sensitivity_analysis}. 
The goal is to demonstrate how the FE solution continuously depends on the mesh. 
To the best of the author's knowledge, this study is the first to conduct such an analysis, particularly in the context of numerical shape optimization problems solved using the Lagrangian method. 
The analysis builds on the ideas developed in \cite{HeHuang2021} on mesh sensitivity analysis for FE solutions of linear elliptic PDEs.}

Some remarks on the choice of shape gradient structure are necessary.
\begin{remark}\label{rem:volume_versus_boundary_expression_for_the_shape_gradient}
As already mentioned, shape gradients of PDE-constrained functionals can be expressed in two equivalent ways.
Both typically involve solving two boundary value problems (BVPs), but one integrates their traces on the domain boundary, while the other evaluates volume integrals.
When these BVPs are solved approximately using, for example, FEMs, the equivalence of the two formulas no longer holds.
\cite{HiptmairPaganiniSargheini2015} analyzed volume and boundary formulations, demonstrating through convergence analysis and numerical experiments that the volume-based expression generally provides greater accuracy in a finite element setting.
Following the aforementioned work, numerous studies have explored this direction. 
We provide a non-exhaustive list of research examining various objective functionals constrained by PDEs, including the elliptic PDEs, the Stokes equation, the Navier-Stokes equation, etc: \cite{EtlingHerzogLoayzaWachsmuth2020,GongLiRao2024,GongLiZhu2022,GongZhu2021,LaurainSturm2016,Laurain2020,LiZhu2019,LiZhu2022,LiZhu2023,LiZhuShen2022,Paganini2014,PaganiniHiptmair2016,Zhu2018}.
\end{remark}
%
%
\subsection{The CCBM formulation}\label{subsec:CCBM_formulation}
The CCBM formulation of Problem~\ref{prob:inverse_problem} is similar to \eqref{eq:main_equation}, except that on the measurement region $\Gtop$, the boundary condition is:
\[
-\sigma_{1} \ddn{u_{1}} - i u_{1} = \alpha(u_{1} - T_{a}) - i h, \quad \text{on } \Gtop,
\]
where $i = \sqrt{-1}$ and $h$ is the measured data on $\Gtop$. 
In the rest of the paper, we refer to \eqref{eq:main_equation} with this boundary condition as the \textit{CCBM equation}.

Let $u = \reu + i \imu : \varOmega \to \mathbb{C}$ be the solution of the CCBM equation.
Then, the real-valued functions ${\reu} : \varOmega \to \mathbb{R}$ and ${\imu} : \varOmega \to \mathbb{R}$ satisfy the following systems of real PDEs:
\begin{multicols}{2}
\noindent
\begin{equation}\label{eq:ccbm_real}
    \begin{cases}
        -\sigma_{1} \Delta \reuout + \coeffk_{1} \reuout = Q_{1}, & \text{in } \healthy, \\[0.5em]
        -\sigma_{0} \Delta \reuin + \coeffk_{0} \reuin = Q_{0}, & \text{in } \tumor, \\[0.5em]
        \reuout = \reuin, & \text{on } \partial \tumor, \\[0.5em]
        -\sigma_{1} \ddn{\reuout} = -\sigma_{0} \ddn{\reuin}, & \text{on } \partial \tumor, \\[0.5em]
        -\sigma_{1} \ddn{\reuout} = \alpha(\reuout - T_{a}) - \imuout, & \text{on } \Gtop, \\[0.5em]
        -\sigma_{1} \ddn{\reuout} = 0, & \text{on } \Gwall, \\[0.5em]
        \reuout = T_{b}, & \text{on } \Gbottom,
    \end{cases}
\end{equation}
\begin{equation}\label{eq:ccbm_imaginary}
    \begin{cases}
        -\sigma_{1} \Delta \imuout + \coeffk_{1} \imuout = 0, & \text{in } \healthy, \\[0.5em]
        -\sigma_{0} \Delta \imuin + \coeffk_{0} \imuin = 0, & \text{in } \tumor, \\[0.5em]
        \imuout = \imuin, & \text{on } \partial \tumor, \\[0.5em]
        -\sigma_{1} \ddn{\imuout} = -\sigma_{0} \ddn{\imuin}, & \text{on } \partial \tumor, \\[0.5em]
        -\sigma_{1} \ddn{\imuout} = \alpha \imuout - h + \reuout, & \text{on } \Gtop, \\[0.5em]
        -\sigma_{1} \ddn{\imuout} = 0, & \text{on } \Gwall, \\[0.5em]
        \imuout = 0, & \text{on } \Gbottom.
    \end{cases}
\end{equation}
\end{multicols}
\begin{remark}\label{rem:equivalent}
If \harbrecht{${\imu} = 0$ in $\varOmega$}, then $\jump{\imu} = \jump{\sigma \dn{\imu}} = 0$ on $\partial \tumor$.  
On $\Gbottom$, $\dn{\imuout} = 0$.  
Similarly, on $\Gwall$, $\imuout = 0$ since $\sigma_{1} > 0$.  
Thus, $\imuout = 0$ on $\Gbottom \cup \Gwall \cup \partial \tumor$.  

Considering the variational equation:
\begin{equation}\label{eq:var_on_top}
\int_{\Gtop} \sigma_{1} \dn{\imuout} \varphi \, ds 
= \int_{\partial \healthy} \sigma_{1} \dn{\imuout} \varphi \, ds 
= \int_{\healthy} \left( \sigma_{1} \nabla \imuout \cdot \nabla \varphi + \coeffk_{1} \imuout \varphi \right) dx = 0, \quad \forall \varphi \in \Vomega.
\end{equation}
By variational arguments, $\dn{\imuout} = 0$ on $\Gtop$, so $\imuout = 0$ on $\Gtop$.  
Note that, if we suppose that $\dn{\imuout} \neq 0$ on $\Gtop$, then $\imuout \neq 0$ on $\Gtop$. 
Taking $\varphi = \imuout \in \Vomega$ as the test function in \eqref{eq:var_on_top} leads to a contradiction.
Hence, $h - \reuout = \sigma_{1} \ddn{\imuout} + \alpha \imuout = 0$ on $\Gtop$, or equivalently, $\reuout = h$ on $\Gtop$.  
Therefore, from the PDE systems \eqref{eq:ccbm_real} and \eqref{eq:ccbm_imaginary}, $(\tumor, \reu)$ solves Problem~\ref{prob:inverse_problem}.  
Conversely, if $(\tumor, u)$ solves Problem~\ref{prob:inverse_problem}, then it satisfies the CCBM equation.
\end{remark}
Remark~\ref{rem:equivalent} implies that Problem~\ref{prob:inverse_problem} can equivalently be formulated as follows:
\begin{problem}{\label{prob:optimization_problem}}
	Find $\tumor \in \holdall^{k}$ such that ${\imu} \equiv {0}$ in $\varOmega$, where $u$ satisfy the CCBM equation.
\end{problem}
\subsection{Well-posedness of the CCBM equation}\label{subsec:wellposedness_of_CCBM_equation}
For functions $\varphi, \psi : \varOmega \to \mathbb{C}$ on the complex Sobolev space $\Vomega := \{ \varphi = \varphi^{\textsf{r}} + i \varphi^{\textsf{i}} \in H^{1}(\varOmega) \mid \varphi^{\textsf{r}} = \varphi^{\textsf{i}} = 0 \text{ on $\Gbottom$}\}$, we define the following sesquilinear and linear forms:
\begin{equation}\label{eq:forms}
\begin{aligned}
	a(\varphi,\psi) 
	&:= \intO{ \left( \sigma \nabla{\varphi} \cdot \nabla{\conj{\psi}} + \coeffk {\varphi}\conj{\psi} \right) } 
		+ \intGtop{ ( \alpha + i ) {\varphi}\conj{\psi} },\\
	l(\psi) &:= \intO{Q{\conj{\psi}}} + \alpha \intGtop{ T_{a} {\conj{\psi}} } + i \intGtop{ h {\conj{\psi}} }.
\end{aligned}
\end{equation} 
The weak form of the CCBM equation can be stated as follows:
\begin{problem}\label{prob:CCBM_weak_form}
Find $u \in \Vomega$ such that $a(\varphi,\psi) = l(\psi)$, for all $\psi \in \Vomega$.
\end{problem}
The solution to Problem~\ref{prob:CCBM_weak_form} will henceforth be called the \textit{state solution}. 
The following result establishes the well-posedness of the weak formulation.
\begin{proposition}\label{prop:well_posedness_of_state}
	Let $\varOmega$ be a Lipschitz bounded domain, $\alpha \in \mathbb{R}_{+}$, $Q \in H^{-1}(\varOmega)$, $T_{a} \in H^{-1/2}(\Gtop)$, $h \in H^{1/2}(\Gtop)$, and the known parameters be given such that Assumption~\ref{assume:weak_assumptions} holds.
	Then, Problem~\ref{prob:CCBM_weak_form} admits a unique weak solution $u \in \Vomega$ which continuously depends on the data $Q$, $T_{a}$, and $h$.
	Moreover, 
	\begin{equation}\label{eq:continuous_dependence}
		\cnorm{u}_{1,\varOmega} \leqslant c_{0} \left( \norm{Q}_{H^{-1}(\varOmega)} + \norm{T_{a}}_{H^{-1/2}(\Gtop)} + \norm{h}_{H^{1/2}(\Gtop)}\right),
	\end{equation}
	where $c_{0} > 0$ depends on $\alpha$, $\minsigma$, and $\mincoeffk$.
\end{proposition}
\begin{proof}
    The assertion follows from standard arguments (see, e.g., the proof of \cite[Prop.~2.2]{Chengetal2014}). 
    Specifically, the conclusion is inferred from the complex version of the Lax-Milgram lemma \cite[p.~376]{DautrayLionsv21998} by demonstrating that the sesquilinear form $a(\cdot, \cdot)$ is continuous and coercive on $\Vomega$, and the linear form $l(\cdot)$ is continuous on $\Vomega$ (\cite[Def. 4, p.~368]{DautrayLionsv21998}). 
    That is, there exists a constant $c_{1} > 0$ such that
    \begin{equation}\label{eq:V_coercivity_of_a}
    	\Re{\{a(u,u)\}} \geqslant c_{1} \cnorm{u}_{1,\varOmega}^{2}, \quad \forall u \in \Vomega. 
    \end{equation}
    The proof of coercivity uses bounds on the coefficients $\sigma$ and $\coeffk$; i.e., $c_{1}=c_{1}(\minsigma,\mincoeffk)$. 
    Uniqueness also follows from standard arguments; the proofs are omitted.
\end{proof}
In Proposition~\ref{prop:well_posedness_of_state}, we see that it is sufficient to assume lower regularity for $\varOmega$, $Q$, $T_{a}$, and $h$ than what is specified in Assumption~\ref{assume:weak_assumptions}. 
Additionally, $T_{a}$ is considered non-constant. 
Moving forward, we will adopt the regularity outlined in Assumption~\ref{assume:weak_assumptions} since a higher regularity of the state solution is necessary.
Moreover, unless otherwise specified, $T_{a} \in \mathbb{R}_{+}$.
\subsection{The shape optimization formulation}\label{subsec:shape_optimization_formulation}
From Proposition~\ref{prop:well_posedness_of_state}, we know that the state problem is uniquely solvable in $\Vomega$. 
Consequently, for any $\tumor \in \holdall^{1}$, we can define the map $\varOmega(\tumor) \mapsto u(\varOmega(\tumor))$. 
This allows us to consider, in view of Remark~\ref{rem:equivalent}, minimizing the following shape functional to solve Problem~\ref{prob:optimization_problem}:
\begin{equation}
\label{eq:ccbm_functional}
	J(\tumor) := \frac{1}{2}\intO{({\imu})^{2}}.
\end{equation}
We then pose the minimization problem:
\begin{problem}\label{prob:minimization_problem}
    Let $\varOmega$ be admissible. 
    Find $\tumor^{\ast} \in \holdall^{1}$ such that $\tumor^{\ast} = \operatorname{argmin}_{\tumor \in \holdall^{1}} J(\tumor)$.
\end{problem}
\begin{remark}
The new method allows us to define the objective function $J$ in the entire domain $\varOmega$, which brings advantages in terms of robustness in the reconstruction, such as the Kohn-Vogelius cost functional \cite{KohnVogelius1987}, compared to the least-squares misfit functional $J_{LS}$, which is defined only on the subboundary $\Gtop$; see \eqref{eq:least_squares}.
\end{remark}
\section{Shape sensitivity analysis}
\label{sec:sensitivity_analysis} 
\subsection{Some concepts from shape calculus}
We consider specific deformations of $\tumor$ through a set of admissible deformation fields, which we define as follows:
\begin{equation}\label{eq:admissible_deformation_fields}
{{\sfTheta}^{k}} := \{\VV \in \alert{{{C}}^{k,1}(\overline{\varOmega}; \mathbb{R}^{d})} \mid \operatorname{supp}{\VV} \subset \overline{\varOmega}_{{\circ}} \},
\end{equation}
where $k \in \mathbb{N} \cup \{0\}$.
For $\VV \in {{\sfTheta}^{k}}$, we denote $\Vn := \langle \VV, \nn \rangle$.

We define the transformation $T_{t} : \overline{\varOmega} \to \overline{\varOmega}_{t}$ as $T_{t} = \idmat + t \VV$, where $\VV \in {{\sfTheta}^{k}}$. 
Accordingly, we define $\tumort := T_{t}(\tumor)$, $\partial{\tumort}:= T_{t}(\partial{\tumor})$, $\healthyt := T_{t}(\healthy)$, and $\partial\healthyt := T_{t}(\partial\healthy)$.
In addition, $\varOmega_{t}=T_{t}(\varOmega) = \healthyt \cup \bartumort$ and $\partial\varOmega_{t} = \partial{\varOmega}$ since $\VV\big|_{\partial{\varOmega}} = {0}$.
\harbrecht{Let $\varphi_{t} : \varOmega_t \to \mathbb{R}$ be a function defined on $\varOmega_{t}$.  
Given the transformation $T_{t} : \healthy \cup \bartumor \to \healthyt \cup \bartumort$ from the reference domain $\varOmega = \healthy \cup \bartumor$,  
the pullback to $\varOmega$ is denoted by $\varphi^{t} := \varphi_{t} \circ T_{t} : \varOmega \to \mathbb{R}$.}

We let $t_{0} > 0$ be such that, for all $t \in \textsf{I} := [0, t_{0})$, $T_{t}$ becomes a ${{C}}^{k,1}$ diffeomorphism from $\varOmega$ onto its image (cf. \cite[Thm. 7]{BacaniPeichl2013} for $k=1$).
\harbrecht{Other important properties of $T_t$ are summarized in Appendix~\ref{appx:properties_of_POI}.}
Hereinafter, we assume that $t \in \textsf{I}$ unless stated otherwise.

The set of all admissible perturbations of $\tumor$, denoted here by ${\Upsilon}^{k}$, $k \in \mathbb{N} \cup \{0\}$, is defined as:
\begin{equation}\label{eq:admissible_domains}
	{\Upsilon}^{k} = \left\{ T_{t}(\VV)(\bartumor) \subset \varOmega \mid \tumor \in \holdall^{k}, t \in \textsf{I}, \VV \in {{\sfTheta}^{k}} \right\}.
\end{equation}

The fixed boundary $\bigdO$ does not need ${{C}}^{k,1}$ regularity; Lipschitz continuity suffices.
Nevertheless, higher regularity is assumed here for simplicity.

The functional $J : {\Upsilon}^{k} \to \mathbb{R}$ has a directional first-order Eulerian derivative at $\varOmega$ in the direction $\VV \in {{\sfTheta}^{k}}$ if the limit
\begin{equation}\label{eq:limit_shape_derivative}
	\lim_{t \searrow 0} \frac{J(\tumort) - J(\tumor)}{t} =: dJ(\tumor)[\VV]
\end{equation}
exists \cite[Sec. 4.3.2, Eq.~(3.6), p.~172]{DelfourZolesio2011}. 
We say $J$ is shape differentiable at $\tumor$ if the \alert{limit exists for all $\VV \in {{\sfTheta}^{k}}$, and} the mapping $\VV \mapsto dJ(\tumor)[\VV]$ is linear and continuous in ${{C}}^{k,1}(\overline{\varOmega})^{d}$, for some $k \in \mathbb{N} \cup \{0\}$. 
The resulting map is called the \textit{shape gradient} of $J$.
%
\subsection{Shape derivative of the cost}
Let us introduce the sesquilinear form $a_{\textsf{adj}}(\cdot, \cdot)$:
\begin{equation}\label{eq:adjoint_bilinear_form}
a_{\textsf{adj}}(\varphi,\psi)
	:=\intO{(\sigma \nabla {\varphi} \cdot \nabla \conj{\psi} + \coeffk {\varphi} \conj{\psi})}
	+ \alpha \intGtop{{\varphi} \conj{\psi}}
	- i \intGtop{{\varphi} \conj{\psi}}, \qquad {\varphi}, {\psi} \in \Vomega.
\end{equation}
The main result of this section is the shape derivative of the objective function $J$ given in the following theorem.
\alert{For simplicity, we assume $\sigma$, $\coeffk$, and $Q$ are piecewise constants.}
\begin{theorem}\label{thm:distributed_shape_gradient}
Let Assumption~\ref{assume:weak_assumptions} holds.
Assume $\varOmega$ is an admissible domain and $\VV$ an admissible deformation field. 
Then, $J$ is shape differentiable, and its distributed shape derivative $dJ(\tumor)[\VV]$ is given by
\begin{equation}\label{eq:shape_gradient}
\begin{aligned}
    dJ(\tumor)[\VV] 
    &= \frac{1}{2} \intO{\dive{\VV} {\imu}^{2}} 
	- \intO{ \dive{\VV} \sigma \sum_{j=1}^{d} \left( \partial_{j} \imaginary{u} \partial_{j}\real{p} - \partial_{j}\real{u} \partial_{j}\imaginary{p}\right) }	\\
    	&\qquad + \intO{\sigma \sum_{m=1}^{d} \sum_{j=1}^{d} \partial_{j}{\VV_{m}} \left( \partial_{j} \imaginary{u} \partial_{m}\real{p} - \partial_{j}\real{u} \partial_{m}\imaginary{p}\right) }	\\
    	&\qquad + \intO{\sigma \sum_{m=1}^{d} \sum_{j=1}^{d} \partial_{j}{\VV_{m}} \left( \partial_{m} \imaginary{u} \partial_{j}\real{p} - \partial_{m}\real{u} \partial_{j}\imaginary{p}\right) }	\\
	&\qquad + \intO{\dive{\VV}k (\imaginary{u} \real{p} - \real{u} \imaginary{p}) }
	+ \intO{\dive{\VV} {Q} \imaginary{p}},
\end{aligned}
\end{equation}
where $u$ is the state solution, and $p \in \Vomega$ is the adjoint variable solving:
\begin{equation}\label{eq:adjoint_equation}
    a_{\textsf{adj}}(p, v)
    = \intO{\imu \conj{v}}, \quad \forall v \in \Vomega.
\end{equation} 
\end{theorem}
\begin{remark}[Construction of the Lagrangian functional]
In the minimax formulation, constructing the Lagrangian functional is crucial. In optimal control problems with real PDE constraints, the application is straightforward \cite{Sturm2016}, as it involves the sum of a utility function and equality constraints. 
However, in this work, careful construction of the Lagrangian is necessary due to the complex nature of the PDE constraints.

To illustrate how the Lagrangian functional should be constructed, we provide the following formal computation of the \textit{directional} derivative of $J$ in the direction of an arbitrary function $v \in \Vomega$:
\begin{equation}\label{eq:directional_derivative}
\begin{aligned}
	J^{\prime}(u(\varOmega))[v(\varOmega)]
	= \left. \frac{d}{d \varepsilon} \left\{ \frac{1}{2} \intO{\left( \imu + \varepsilon \Im\{v\} \right)^2}\right\} \right|_{\varepsilon = 0}
	= \intO{\imu \Im\{v\}}
	= \Im\left\{ \intO{\imu v}\right\}.
\end{aligned}
\end{equation}

Suppose $\Lag(t, \varphi, \psi)$ denotes the Lagrangian functional associated with the CCBM formulation. 
Assume $\Lag$ is sufficiently differentiable with respect to $t$, $\phi$, and $\psi$, and the strong material derivative $\dot{u}$ exists in $\Vomega$. 
We can then compute:

\begin{equation}\label{eq:derivative_of_Lagrangian}
dJ(\tumor)[\theta] = \left. \frac{d}{dt} \Lag(t, u_t, p) \right|_{t=0}
= \underbrace{\left. \frac{\partial \Lag}{\partial t}(t, u, p) \right|_{t=0}}_{\text{shape derivative}}
+ \underbrace{\frac{\partial \Lag}{\partial \phi}(0, u, p)(\dot{u})}_{\text{adjoint equation}}.
\end{equation}
Since $\dot{u} \in \Vomega$, it follows that:
\[
dJ(\tumor)[\theta] = \left. \frac{\partial \Lag}{\partial t}(t, u, p) \right|_{t=0}.
\]

Now, taking the complex conjugate of both sides of \eqref{eq:adjoint_equation}, we obtain:
\begin{equation}\label{eq:adjoint_relation_with_cost_function}
    \overline{a_{\textsf{adj}}(p, v)}
    = \intO{(\sigma \nabla \conj{p} \cdot \nabla{v} + \coeffk \conj{p} v)}
    + \intGtop{(\alpha + i) \conj{p} v}
    = \intO{\imu v}.
\end{equation}
Thus, to ensure consistency between \eqref{eq:directional_derivative}, \eqref{eq:derivative_of_Lagrangian}, and \eqref{eq:adjoint_relation_with_cost_function}, the imaginary part of the equality constraint, corresponding to the adjoint equation, must be taken in the intermediate computation of the shape derivative using the minimax formulation.

\end{remark}
We derive the shape derivative of $J$ (i.e., the proof of Theorem~\ref{thm:distributed_shape_gradient}) via minimax formulation in the spirit of \cite{DelfourZolesio1988a}.
This method enables us to find the derivative of the shape functional without computing the material or shape derivative of the state. 
An alternative is the rearrangement method developed in \cite{IKP2008, IKP2006}, which avoids such computation.
To proceed, we introduce the Lagrangian functional 
	\[
		L({\tumor}, \varphi, \psi) := J({\tumor},\varphi) + \imagpart{l(\psi) - a(\varphi,\psi)}, \qquad \varphi, \psi \in {\Vomega},
	\]
where $a$ and $l$ are given in \eqref{eq:forms}.
Observe that
	\[	
		J({\tumor}, u({\tumor})) = \min_{\varphi \in {\Vomega}} \sup_{\psi \in {\Vomega}} L({\tumor}, \varphi, \psi), 
	\]
because
	\[
		\sup_{\psi \in {\Vomega}} L({\tumor}, \varphi, \psi)  =
			\begin{cases}
				J({\tumor}, u({\tumor})) & \text{if $\varphi = u \in {\Vomega}$},\\
				+\infty & \text{otherwise}.
			\end{cases}
	\]
It can be verified that $L$ is convex and continuous with respect to $\varphi$ and concave and continuous with respect to $\psi$ (cf. \cite{Rabago2023a}). 
Thus, by Ekeland and Temam \cite{EkelandTemam1976}, $L$ has a saddle point $(u,p) \in {\Vomega}^{2}$ if and only if $(u, p) \in {\Vomega}^{2}$ satisfies the system:
\begin{equation}\label{eq:saddle_point}
\left\{
\begin{aligned}
	\left. \dfrac{\partial}{\partial{\varphi}}L({\tumor}, \varphi, \psi)[v]  \right|_{(\varphi, \psi) = (u,p)}&= 0, \qquad  v \in {\Vomega},\\
	\left. \dfrac{\partial}{\partial{\psi}}L({\tumor}, \varphi, \psi)[v] \right|_{(\varphi, \psi) = (u,p)}&= 0, \qquad  v \in {\Vomega}.
\end{aligned}
\right.
\end{equation}
The first equation corresponds to the variational equation of the state problem, while the second characterizes the adjoint solution.

An analogous Lagrangian functional can be defined over the perturbed domain, and the previous analysis still applies. 
Specifically, we have the following expression: 
\begin{equation}\label{eq:minimax_form}
    J(\tumort, u(\tumort)) = \min_{\varphi \in \Vomegat} \sup_{\psi \in \Vomegat} L(\tumort, \varphi, \psi),
    \end{equation}
with the corresponding saddle point characterized by the following system of PDEs:
\begin{equation}\label{eq:saddle_point_perturbed_case}
    \left\{
    \begin{aligned}
    \left. \dfrac{\partial}{\partial{\varphi}}L(\tumort, \varphi, \psi)[v]  \right|_{(\varphi, \psi) = (u(\tumort), p(\tumort))} &= 0, \qquad  v \in \Vomegat,\\
    \left. \dfrac{\partial}{\partial{\psi}}L(\tumort, \varphi, \psi)[v] \right|_{(\varphi, \psi) = (u(\tumort), p(\tumort))} &= 0, \qquad  v \in \Vomegat.
    \end{aligned}
    \right.
\end{equation}

The main steps for the remainder of the proof are as follows:
\begin{itemize}
	\item we differentiate $L(\tumort, \varphi, \psi)$ with respect to $t \geqslant 0$; i.e., we evaluate \eqref{eq:limit_shape_derivative} by using \eqref{eq:minimax_form};
	\item we verify the assumptions of Correa-Seeger Theorem \cite{CorreaSeeger1985}.
\end{itemize}
%
To derive $L(\tumort, \varphi, \psi)$ with respect to $t \geqslant 0$, we apply Theorem \ref{thm:Correa_Seeger} (\cite{CorreaSeeger1985}). 
To achieve this, the domain $\tumort$ must be mapped back to the reference domain ${\tumor}$ using the transformation $T_{t}$. 
However, composing by $T_{t}$ inside the integrals results in terms ${u} \circ T_{t}$ and ${p} \circ T_{t}$, which could be non-differentiable. To address this, we parameterize $H^{1}(\varOmega)$ by composing its elements with $T^{-1}_{t}$. Thus, we write
\[
	\Lag(t, \varphi, \psi) := L(\tumort, \varphi \circ T^{-1}_{t}, \psi \circ T^{-1}_{t}).
\]
After applying a change of variables, we obtain
\[
\begin{aligned}
	\Lag(t, \varphi, \psi) 
	&= \frac{1}{2}\intO{\dett {\imaginary{\varphi}}^{2}}
		 - \Im\left\{ \intO{( {{\sigma}^{t}} \At \nabla{\varphi} \cdot \nabla{\conj{\psi}} + \dett {\coeffk}^{t} {\varphi}{\conj{\psi}})} 
		+ \intGtop{ (\alpha + i)  \bt {\varphi}{\conj{\psi}}} \right.\\
	&\qquad  \qquad  \qquad \qquad \qquad \left. - \intO{\dett {\Qt}{\conj{\psi}}}
		- \alpha \intGtop{ \bt T_{a}^{t} {\conj{\psi}} } 
		- i \intGtop{ \bt h^{t} {\conj{\psi}} } \right\}.
\end{aligned}
\]
On $\partial{\varOmega}$, we note that $\bt = 1$ for all $t>0$ since $\VV\big|_{\partial{\varOmega}} = \vect{0}$. 
In view of \eqref{eq:saddle_point_perturbed_case}, $\Lag$ has a unique saddle point $(\ut, \pt) \in {\Vomega}^{2}$, for all $t \in \textsf{I}$, satisfying the system of PDEs
\begin{equation}
\label{eq:saddle_point_systems}
\left\{
\begin{aligned}
	\intO{( {{\sigma}^{t}} \At \nabla{\ut} \cdot \nabla{v} + \dett {\coeffk}^{t} {\ut}{v})} 
				+ \intGtop{ (\alpha + i) \bt{\ut}\conj{v}}\\
	\qquad - \intO{\dett {\Qt}{\conj{v}}}
		- \alpha \intGtop{ \bt T_{a}^{t} \conj{v} } 
		- i \intGtop{ \bt h^{t} \conj{v} } &= 0, \quad \forall{v} \in {\Vomega},\\
	\intO{( {{\sigma}^{t}} \At \nabla{\pt} \cdot \nabla{\conj{v}} + \dett {\coeffk}^{t} {\pt} \conj{v})} 
				- i \intGtop{ \bt{\pt} \conj{v}}
				+ \intO{\dett {\imuout} \conj{v}}
	&= 0, \quad \forall{v} \in {\Vomega}.
\end{aligned}
\right.
\end{equation}
\harbrecht{The validity of the assumptions of Theorem~\ref{thm:Correa_Seeger} (\cite{CorreaSeeger1985}) can be established without difficulty, as the arguments closely follow those presented in \cite[pp.~521 and 551]{DelfourZolesio2011}. 
Accordingly, the details are omitted for brevity.} 
Consequently, by invoking Theorem~\ref{thm:Correa_Seeger}, we obtain
\[
\begin{aligned}
	{d}J({{\tumor}})[\VV]
	&= \partial_{t} \Lag(t, u, p)\big|_{t=0}\\
	&= \frac{1}{2} \intO{\dive{\VV} {\imu}^{2}} 
    	- \Im \left\{ \intO{(\sigma A \nabla{u} \cdot \nabla \conj{p} + \dive{\VV} \coeffk u \conj{p} - \dive{\VV} Q \conj{p})} \right\} \\
   	&\qquad - \Im \left\{ \intO{( \nabla{\sigma} \cdot \VV (\nabla{u} \cdot \nabla \conj{p}) + \nabla{\coeffk} \cdot \VV {u} \conj{p} - \VV \cdot \nabla Q) \conj{p}} \right\},
\end{aligned}
\]	 
which, after taking the imaginary part of the expression inside the brackets, leads to \eqref{eq:shape_gradient}.
\alert{The spatial derivatives of $\sigma$, $\coeffk$, and $Q$ vanish, given that they are piecewise constants.}

The following result can be easily drawn from \eqref{eq:shape_gradient} and \eqref{eq:adjoint_equation}.
\begin{corollary}[Necessary condition]\label{cor:necessary_condition}
	Let the subdomain ${\tumor^{\ast}}$ be such that $u=u(\varOmega({\tumor^{\ast}}))$ satisfies \eqref{eq:main_equation}, i.e., there holds $\imu \equiv 0$ on $\varOmega({\tumor^{\ast}})$.
	Then, ${\tumor^{\ast}}$ is a stationary solution of Problem~\ref{prob:optimization_problem}.
	That is, it fulfills the necessary optimality condition
	\begin{equation}
	\label{eq:optimality_condition}
		{d}J({\tumor^{\ast}})[\VV] = 0, \quad \forall\VV \in \sfTheta^{1}.
	\end{equation}
\end{corollary}
\begin{proof}
    Assuming that $\imu \equiv 0$ in $\varOmega({\tumor^{\ast}})$, it follows that $p=0$ in $\varOmega({\tumor^{\ast}})$.  
    Therefore, we have ${d}J({\tumor^{\ast}}) \equiv 0$ in $\varOmega({\tumor^{\ast}})$, which implies that ${d}J({\tumor^{\ast}})[\VV] = 0$ for all $\VV \in \sfTheta^{1}$.
\end{proof}
%
\section{Mesh deformation}
\label{sec:Mesh_deformation_and_transformation_of_integrals}
Before presenting the numerical results, we first analyze how the FE solution depends on mesh variations. 
Theorem~\ref{thm:mesh_sensitivity_main_result} provides a bound on the solution variation based on mesh deformation and its gradient, with a similar bound for nonsmooth mesh velocity fields given in equation~\eqref{eq:mesh_sensitivity_non_smooth_case}. 
These results demonstrate that the FE solution continuously depends on the mesh, regardless of dimension, unstructured simplicial meshes, or the finite element approximation used. Numerical examples are provided in subsection~\ref{subsec:mesh_sensitivity_numerical_examples}. 

\harbrecht{This section largely follows the presentation in~\cite{HeHuang2021}, and adopts the notations therein, as well as those in~\cite{ErnGuermond2004}.}
We consider the FE solution of Problem~\ref{prob:CCBM_weak_form} under the assumption that the state $u$ possesses $H^2$ regularity.
Let $\Th$ be a simplicial mesh triangulation of $\varOmega$, and let $K$ be a generic element of $\Th$. 
We denote by $\hk$ the diameter (the length of the longest side) of $K$, and by $\ak$ its minimum height, defined as the distance from a vertex to its opposite facet. 
For reference, we label the vertices of $K$ as $\xxo^{K}, \ldots, \xxd^{K}$. 
The mesh $\Th$ is called shape-regular if there exists a constant $\bar{\meshh} > 0$ such that
\[
	\frac{\hk}{\ak} \leqslant \bar{\meshh}, \quad \forall K\in \Th.
\]
In the literature, $\Th$ is often defined in terms of the mesh size $\meshh$, given by $\meshh = \max_{K \in \Th} \hk \equiv \max_{K \in \Th} \max_{0 \leqslant i,j \leqslant d} \abs{\xxi - \xx_{j}}$.

While shape-regularity is a standard assumption in finite element error analysis, we do not assume it in our mesh sensitivity analysis. 
Instead, we only require that each element be \textit{non-degenerate} (see, e.g., \cite[Def. 4.4.13, Eq.~(4.4.16), p.~108]{BrennerScott2008}) or \textit{non-inverted}, which amounts to requiring $\ak > 0$, for all $K \in \Th$.
%
%
%


In the latter part of this section, we discuss key tools in our analysis related to the properties of the mesh element $K$.  
For notational simplicity, we occasionally omit the superscript $K$ when no confusion arises.  
Using its vertices $\{\xxo\}_{i=0}^{d}$, we define the edge matrix (of size $d \times d$) of $K$ (see, e.g., \cite{LiHuang2017}) as follows:  
\[
	E = [E_{1}, \ldots, E_{d}] =  [\xx_{1} - \xxo, \ldots, \xxd - \xxo].
\]  
Here, the $i$th column vector, given by $E_{i} = \xxi - \xxo$ for $i=1,\ldots,d$, represents an edge of the triangle or element $K$.  
Since $K$ is non-degenerate, the matrix $E$ is non-singular.  

Consider a set of linearly independent points $\{\xxi\}_{i=0}^{d}$ in $\mathbb{R}^d$ and their convex hull, called a \textit{simplex}. 
For $0 \leqslant i \leqslant d$, let $F_i$ be the face (an edge in 2D) opposite to vertex $\xxi$, with outward normal $\nn_i$. 

The barycentric coordinates $\{\lambda_i\}$ for a simplex $K \subset \mathbb{R}^d$ are defined by
\[
\lambda_i(\xx) = 1 - \frac{(\xx - \xxi) \cdot \nn_i}{(\xx_j - \xxi) \cdot \nn_i} = \frac{(\xx_j - \xx) \cdot \nn_i}{(\xx_j - \xxi) \cdot \nn_i},
\]
where $\xx_j$ is any vertex of the face $F_i$. They satisfy $0 \leqslant \lambda_i \leqslant 1$, $\lambda_i(\xx_j) = \delta_{ij}$, and vanish on $F_i$. For a $2$-simplex, $\lambda_i$ corresponds to the area ratio of the triangle formed by $\xx$ and the edge opposite $\xxi$. The functions $\lambda_i$ are affine, with $\sum_{i=0}^d \lambda_i = 1$ and $\sum_{i=0}^d \lambda_i \xxi = \xx$. Differentiating the latter gives $\sum_{i=0}^d (\xxi - \xxo)(\nabla \lambda_i)^\top = \II$, or equivalently, $E[\nabla \lambda_1, \ldots, \nabla \lambda_d]^\top = \II$, where $\II$ is the $d \times d$ identity matrix, hence $E^{-\top} = [\nabla \lambda_1, \ldots, \nabla \lambda_d]^\top$.

The edge matrix $\edgemat$ of $K$ depends on the vertex ordering but many geometric properties of $K$, independent of the ordering, can still be computed using $\edgemat$. 
For instance, the area (in 2D) or volume (in 3D) of $K$ is given by $\abs{\op{det}(\edgemat)}/d!$. 
The $i$th height of $K$ is $1/\abs{\nabla \lambda_{i}^K}$ \harbrecht{(see, e.g., \cite[Chap.~7]{ErnGuermond2021})}, and by our definition, $\aK = \min_i \abs{\nabla \lambda_{i}^{K}}^{-1}$. 
This quantity will be useful for estimating matrix norms involving $\edgemat$. 

Now, let us consider the FE space associated with $\Th$ as
\[
    \spaceVh = \{v \in \Vomega \cap C(\overline{\varOmega}) \mid \ v|_K \in P_{k}(K), \ \forall K \in \Th \},
\]
where $\Vomega$ is given by \eqref{eq:space_Vomega} and $P_{k}(K)$, $k \in \mathbb{N} \cup \{0\}$, is the set of polynomials of degree no more than $k$ defined on $K$. 
Any function $\vh$ in $\spaceVh$ can be expressed as $ \vh = \sum_i v_i \psii(\xx)$, where $\{\psi_1, \psi_2, \ldots\}$ is a basis for $\spaceVh$. 
We distinguish FE basis functions $\{\psi_1, \psi_2, \ldots\}$ from the linear Lagrange basis functions $\{\phi_i, i = 1,2, \ldots\}$ (with $\phi_i$ being associated with $\xxi$) and emphasize that they can be different (even when $k = 1$). 

The FE solution of the BVP \eqref{prob:CCBM_weak_form} is obtained by solving the following problem:
\begin{problem}\label{prob:CCBM_weak_form_fem}
Find $\uh \in \spaceVh$ such that
\begin{equation}\label{eq:weak_form_fem}
	a(\uh, \psi) = l(\psi), \quad \forall \psi \in \spaceVh,
\end{equation}
where $a$ and $l$ are the same forms as those given in \eqref{eq:forms}.
\end{problem}

The following error estimate can be proven similarly to the case of general real-valued elliptic PDEs found in most FEM textbooks (e.g., see \cite[Eq.~(6.1.6), p.~156]{BrennerScott2008}).
\begin{proposition}Assume that $u \in \HH^2(\varOmega)$ and the mesh $\Th$ is regular. 
Then,
\begin{equation}\label{eq:standard_error_estimate}
    \cnorm{\nabla (\uh - u)}_{L^2(\varOmega)} \leqslant c \meshh \cnorm{\Delta{u}}_{L^2(\varOmega)},
\end{equation}
where $\meshh = \max_{K \in \Th} \hk$ and $c > 0$ is a constant independent of $u$, $\uh$, and $\Th$.
\end{proposition}
Assumption~\ref{assume:weak_assumptions} is sufficient to achieve $H^{2}$ regularity for state $u$.

The error estimate \eqref{eq:standard_error_estimate} confirms the stable dependence of the FE solution on the mesh. 
It ensures that the FE solution remains close to the exact solution for all regular meshes with a maximum element diameter $\meshh$, regardless of their variations. 
However, the estimate does not indicate whether or how the FE solution continuously depends on the mesh.
This is what we aim to explain in this work in the spirit of \cite{HeHuang2021}.
\section{Mesh Sensitivity Analysis} \label{sec:mesh_sensitivity_analysis}
To analyze variations in mesh qualities and functions under mesh deformation, we adopt an approach akin to gradient methods in optimal control (e.g., \cite{BrysonHo1975}) and sensitivity analysis in shape optimization (e.g., \cite{DelfourZolesio2011,HaslingerMakinen2003,SokolowskiZolesio1992}). 
This involves introducing a small mesh deformation, deriving a finite element formulation, and establishing bounds on the resulting variation in the FE solution. 
Several supporting lemmas from the literature (e.g., \cite{Ciarlet2002,ErnGuermond2004,HeHuang2021}) are employed to achieve this. 
The mesh sensitivity of the finite element solution is presented in the next section.

We assume that a smooth vector field $\dot{X} = \dot{X}(\xx)$ is given on $\varOmega$ and satisfies
\begin{equation}\label{eq:assumption_on_velocity_field}
	\norm{\dot{X}}_{W^{1,\infty}(\varOmega)^{d}} < \infty.
\end{equation}
Hereinafter, we say that a vector field is smooth if it satisfies the conditions stated above.
\alert{
\begin{remark}
In shape optimization, the vector field $\dot{X}$ corresponds to the \textit{flow} speed generated by the computed shape gradient, such as in \eqref{eq:shape_gradient}.  
More precisely, we can associate the vector field $\dot{X}$ with the derivative of the perturbation of the identity operator $T_t$ with respect to the fictitious time parameter $t$, i.e.,  
\[
	\dfrac{d}{dt} T_{t} \big|_{t=0} = \VV = \dot{X},
\]
provided that $\VV$ inherits sufficient regularity from the state and adjoint state variables.  
To obtain a sufficiently regular deformation field $\VV$ associated with the shape gradient ${d}J({{\tumor}})$, one approach is to compute its Riesz representative, for instance, in the $H^{3}(\varOmega)$ Sobolev space. 
This can be achieved using the well-known Sobolev gradient method \cite{Neuberger1997,Doganetal2007}.  
In the numerical solution of Problem~\ref{prob:minimization_problem}, however, the $H^{1}$ Riesz representative of ${d}J({{\tumor}})$ (cf. \eqref{eq:extension_regularization}) is sufficient, as demonstrated in the latter part of Section~\ref{sec:numerical_algorithm_and_examples}.
\end{remark}
}

We analyze mesh deformation of $\Th$ with fixed connectivity, where vertices move according to  
\begin{equation}\label{eq:mesh_perturbation}  
\xxi(t) = \xxi(0) + t \dotxxi, \qquad 0 \leqslant t < t_{1}, \quad i = 1,2, \ldots,  
\end{equation}  
with constant-in-time nodal velocities $\dotxxi = \dot{X}(\xxi(0))$ and small $t_1 > 0$.  
The time-dependent mesh is denoted by $\Tht$. 
In our numerical setting, boundary vertices are fixed, i.e., $\dot{X} = 0$ on $\partial\varOmega$, though the analysis also covers moving boundaries.  
Since any smooth mesh deformation can be linearized in the form of \eqref{eq:mesh_perturbation}, this model is sufficiently general.

To analyze mesh deformation, we consider the bijective affine map $\FK$ from $\Ko = K \in \Th$ to $\Kt \in \Tht$ (see \cite[p.~83]{Ciarlet2002}), using coordinates $\xx$ and $\yy$ on $\Ko$ and $\Kt$, respectively.  
Quantities in $\yy$ are marked with a tilde (e.g., $\tnabla := \nabla_{\yy}$ versus $\nabla := \nabla_{\xx}$), and we examine their deformation by pulling them back to $\xx$ and differentiating in time, denoted by a dot, analogous to material differentiation.

We begin by considering the time derivatives of the Jacobian matrix and its determinant. 
Let the Jacobian matrix of $\FK$, also known as the \textit{deformation gradient}, be denoted by $\Jac = \partial \FK(\xx)/\partial \xx$, and let the Jacobian determinant be denoted by $\detJac = \op{det}(\Jac)$ such that $\detJac > 0$. 
Consequently, we have the invertibility of $\Jac$ and the bijectivity of $\FK$.
We can express $\Jac$ in terms of the edge matrices of $\Ko$ and $\Kt$. 
As $\yy = \FK(\xx)$ is affine, it can be written as (cf. \cite[Eq.~(1.91), p.~58]{ErnGuermond2004}) 
\begin{equation}\label{eq:affine_transform} 
	\yy = \FK(\xx) = \xok(t) + \Jac(\xx - \xok(0)), \quad \forall\xx\in \Ko. 
\end{equation}
Accordingly, we have
\[
    \Jac^{ij} = D \FK(\xx)_{ij} = \frac{\partial y_{i}}{\partial x_{j}}, 
    \qquad (\Jac^{ij})^{-1} = D \FK^{-1}(\yy)_{ij} = \frac{\partial x_{i}}{\partial y_{j}},
    \qquad i,j = 1, \ldots, d,
\]
where $\xx = (x_{1}, \ldots, x_{d})$ and $\yy = (y_{1}, \ldots, y_{d})$.
Now, by taking $\xx=\xxi^{K}(0)$ and $\yy=\xxi^{K}(t)$, for $i=1,\ldots,d$, sequentially, we get
\[	
	[\xx_1^{K}(t) - \xok(t), \ldots, \xx_d^{K}(t) - \xok(t)] = \Jac [\xx_1^{K}(0) - \xok(0), \ldots, \xx_d^{K}(0) - \xok(0)],
\]
which gives (cf. \cite[Eq.~(3.2)]{HuangKamenski2018})
\begin{equation}\label{eq:Jacobian_{b}y_E}
	\Jac = E_{\Kt}E_{\Ko}^{-1},
	\qquad
	\Jac^{-1} = E_{\Ko} E_{\Kt}^{-1},
	\qquad\text{and}\quad
	\Jac\big|_{t=0} = \II,
\end{equation}
where $\II$ stands for the $d \times d$ identity matrix.

Differentiating \eqref{eq:affine_transform} with respect to $t$ gives
\[
	\dotFK(\xx) = \dot{\xx}_{0}^{K}(t) + \dot{\Jac}(\xx - \xok(0))
		= \dot{\xx}_{0}^{K}(t) + \dot{E}_{\Kt}E_{\Ko}^{-1}(\xx - \xok(0)),
\]
where $\dot{E}_{\Kt} = [\dot{\xx}_1^{K}(t) - \dot{\xx}_{0}^{K}(t), \ldots, \dot{\xx}_d^{K}(t) - \dot{\xx}_{0}^{K}(t)]$.

The map $\FK(\xx)$ admits a linear basis representation $\FK(\xx) = \sum_{i=0}^d \xxi^{K}(t) \lambda_{i}^{K}(\xx)$, whose time derivative is given by  $\dotFK(\xx) = \sum_{i=0}^d \dot{\xx}_{i}^{K}(t) \lambda_{i}^{K}(\xx)$.

We now present key lemmas supporting Theorem~\ref{thm:mesh_sensitivity_main_result} and estimate~\eqref{eq:mesh_sensitivity_non_smooth_case}, with proofs available in \cite{Ciarlet2002, ErnGuermond2004, ErnGuermond2021, HeHuang2021}.
\begin{lemma}\label{lem:identities}
	Let $\dotXh$ be defined as a piecewise linear velocity field, i.e., $\dotXh = \sum_{i=0}^{d} \dotxxi \lambda_i(\xx)$, and assume that the mesh velocity field is smooth.
	We have the following identities:
	\begin{align}
		\dotFK &= \dotXh \big|_{K},\qquad
		\nabla \cdot ( \dotFK  ) 
			= \op{tr}( \dot{\Jac} )	
			= \op{tr}( \dot{E}_{\Kt}E_{\Ko}^{-1} )
			= \nabla \cdot ( \dotXh \big|_{K} ),\label{eq:identity1}\\[0.5em]
		\dot{\Jac}\big|_{t=0} &= \dot{E}_{\Ko}E_{\Ko}^{-1}, \qquad
		\dot{( \Jac^{-1} )}\big|_{t=0} = -\dot{E}_{\Ko}E_{\Ko}^{-1}, \qquad
		\dot{\detJac}\big|_{t=0} = \nabla \cdot ( \dotXh \big|_{K} ).\nonumber
	\end{align}
	In addition, the following inequalities hold:
	\begin{equation}\label{eq:velocity_inequalities}
		\norm{ \dotXh }_{L^{\infty}(\varOmega)^{d}} \leqslant \norm{ \dot{X} }_{L^{\infty}(\varOmega)^{d}},\qquad
		\norm{ \nabla \cdot \dotXh }_{L^{\infty}(\varOmega)^{d \times d}} \leqslant d \norm{\nabla \dot{X}}_{L^{\infty}(\varOmega)^{d \times d}} \max_{K} \frac{\hK}{\aK}.
	\end{equation}
\end{lemma}
Let $\norm{\cdot}_{2}$ and $\norm{\cdot}_{F}$ denote the usual 2-norm and Frobenius norm (or Hilbert-Schmidt norm) for matrices.
Then, we have the sharp inequality $\norm{M}_{2} \leqslant \norm{M}_{F}$.
We thus have the following lemma.
\begin{lemma}\label{lem:estimates_for_edge_matrices}
	For smooth mesh velocity field, there hold the estimates
	\begin{equation}\label{eq:edge_matrices_estimates}
		\norm{\edgemat^{-1}}_{2} \leqslant \frac{\sqrt{d}}{\ak}
		\qquad
		\text{and} 
		\qquad
		\norm{\dotEK}_{2} \leqslant \sqrt{d} \hk \norm{\nabla \dot{X}}_{L^{\infty}(\varOmega)^{d \times d}}.
	\end{equation}
\end{lemma}
\begin{remark}\label{rem:non-smooth_velocity}
	In Lemma \ref{lem:identities} and Lemma \ref{lem:estimates_for_edge_matrices}, we assumed that the mesh velocity is smooth which allows one to utilize 
	the identity 
	\[
		\dot{\xx}_{i}^{K} - \dot{\xx}_{0}^{K} 
		= \dot{X}(\xxi^{K}) - \dot{X}(\xx_{0}^{K})
		=  \int_{0}^{1} \nabla \dot{X}(\xx_{0}^{K} + t ( \dot{\xx}_{i}^{K} - \dot{\xx}_{0}^{K} )) \cdot (\dot{\xx}_{i}^{K} - \dot{\xx}_{0}^{K} ) \ dt,
	\]
	and get the estimates \eqref{eq:velocity_inequalities}$_{2}$ and \eqref{eq:edge_matrices_estimates}$_{2}$.
	If the mesh velocity field is not smooth, then we have
	\begin{equation}\label{eq:divdot_estimate}
	\begin{aligned}
		\abs{\nabla \cdot \dotXh}
		\leqslant \frac{1}{\min_{K} \ak } \left| \sum_{i=0}^d \dot{\xx}_{i} \right|
		\stackrel{\triangle\text{-ineq.}}{\leqslant} \frac{1}{\min_{K} \ak } \sum_{i=0}^d \left| \dot{X}(\xxi(0)) \right|.
	\end{aligned}
	\end{equation}
	In the first inequality above, the following identity was used:	
	\[
		\nabla \cdot \dotXh 
		= \sum_{i=0}^d \dot{\xx}_{i}^{K}\cdot \nabla \lambda_{i}^{K}(\xx)
		= \sum_{i=1}^d ( \dot{\xx}_{i}^{K} - \dot{\xx}_{0}^{K} )\cdot \nabla \lambda_{i}^{K}(\xx).
	\]

	Taking the supremum of the left- and right-most sides of inequality \eqref{eq:divdot_estimate} gives us
	\[
		\norm{\nabla \cdot \dotXh}_{L^\infty(\varOmega)} \leqslant \frac{d+1}{\min_{K} \ak } \norm{\dot{X}}_{L^\infty(\varOmega)^{d}}.
	\] 
	Similarly, we have 
	$
	\norm{\dot{E}_{K(0)}^{-1}}_{2}^{2}
		\leqslant \sum_{i=1}^{d} | \dot{\xx}_{i} - \dot{\xx}_{0}|^{2}
		\leqslant 2d \norm{\dot{X}}_{L^{\infty}(\varOmega)^{d}}^2$,
	from which we get
	\[
	\|\dot{E}_{K(0)}^{-1}\|_{2} \leqslant \sqrt{2d} \norm{\dot{X}}_{L^{\infty}(\varOmega)^{d}}
	\]	
	when the mesh velocity field is not smooth.
\end{remark}

We next review the (pseudo) time derivative of functions on $K$ and $K(t)$, starting with the relation between basis functions:
\[
\tilde{\psi}(\yy, t) = \psi(\FK^{-1}(\yy))
\qquad \stackrel{\eqref{eq:affine_transform}}{\Longleftrightarrow}\qquad
\tilde{\psi}(\FK(\xx), t) = \psi(\xx). 
\]
This implies $\tnabla \tilde{\psi}(\FK(\xx),t) = \Jac^{-\top}\nabla \psi(\xx)$, where $\Jac^{-\top} = (\Jac^{-1})^{\top}$, and component-wise this reads:
\[
	\frac{\partial \tv(\yy)}{\partial y_{i}} 
		= \sum_{i=1}^{d} \frac{\partial v}{\partial x_{j}}(\xx) \frac{\partial x_{j}}{\partial y_{i}}
		\qquad\text{and}\qquad
	\frac{\partial v(\xx)}{\partial x_{i}} 
		= \sum_{i=1}^{d} \frac{\partial \tv}{\partial y_{j}} (\yy) \frac{\partial y_{j}}{\partial x_{i}},
\]
where the index $i$ denotes the $i$th row of a matrix. 
Furthermore, we have
	\begin{equation}\label{eq:material_derivative_of_a_basis_function}
		\dot{\tilde{\psi}}(\cdot, 0) = 0
		\qquad \text{and} \qquad
		\oset{\vect{.}}{({\tnabla\tilde{\psi}})}\!\!(\cdot, 0) = -E_{\Ko}^{-\top} \dot{E}_{\Ko}^{\top} \nabla \psi.
	\end{equation}
For a function $f = f(\xx)$ on $K$ and $\tilde{f}(\yy,t) = f(\FK(\xx))$, the time derivative at $t = 0$ is
\begin{equation}\label{eq:material_derivative_of_fixed_function}
\dot{\tilde{f}}(\cdot, 0) = \nabla f \cdot \dotXh\big|_{K}.
\end{equation}

In the next lemma, we compute the time derivative at zero of the transport of a finite element approximation function and its corresponding gradient.
\begin{lemma}\label{lem:material_derivative_fe_approximation}
	Consider a finite element approximation $\vh = \sum_{i=1}^{d} v_{i} \psi_{i}(\xx)$, where $\{ \psi_{1}, \ldots, \psi_{d} \}$ is a basis for $\spaceVh$, and its corresponding transport $\tvh(\yy, t) = \sum_{i=1}^{d} v_{i}(t) \tilde{\psi}(\yy,t)$, where $\{ \tilde{\psi}_{1}, \ldots, \tilde{\psi}_{d}\}$ is a basis for the space $\spaceVh(t)$.
	Then, we have
	\begin{equation}\label{eq:material_derivative_fe_approximation}
		\dot{\tv}_{h}(\cdot, 0)\big|_{K} = \dotvh\big|_{K}
		\quad\text{and}\quad
		\oset{\vect{.}}{({\tnabla\tvh})}\!\!(\cdot, 0) \big|_{K} 
		= -\edgemat^{-\top} \dotEK^{\top} \nabla \vh \big|_{K} + \nabla \dotvh\big|_{K},
	\end{equation}
	for all $K \in \Th$, where $\dotvh = \sum_{i=1}^{d} \dot{v}_{i}(0) \psi_{i}(\xx)$.
\end{lemma}
\subsection{Mesh sensitivity analysis for the finite element solution} \label{subsec:mesh_sensitivity_analysis_for_the_finite_element_solution}
In this section we analyze the mesh sensitivity for the FE solution $\uh = \sum_{i} \ui \psii(\xx)$ satisfying \eqref{eq:weak_form_fem}. 
On the deformed mesh  $\Th$, the perturbed FE solution can be expressed as $\tu = \sum_{i} \ui(t) \tilde{\psi}_{i}(\yy,t)$. 

Assuming that $\tu$ is differentiable, we obtain from Lemma~\ref{lem:material_derivative_fe_approximation} that the material derivative of $\tu$ at $t=0$ is given by
\begin{equation}
    \tildedotuh |_{t=0} = \sum_{i} \dot{u}_i(0) \psii.
\end{equation}
We denote this by $\dotuh$, i.e., $\dotuh = \sum_{i} \dot{u}_i(0) \psii$. This derivative quantifies the change in $\uh$ due to mesh deformation. 

First, we derive the FE formulation for $\dotuh$ and then establish a bound for $\norm{\nabla \dotuh}_{L^2(\varOmega)}$.
\begin{theorem}\label{thm:material_derivative_equation_fem}
    The material derivative $\dotuh \equiv \sum_{i} \dotui(0) \psii \in \spaceVh$ uniquely satisfies
\begin{equation}\label{eq:material_derivative_fem}
\begin{aligned} 
a(\dotuh,{\psi})
    &= \sumint{ {\sigma} \nabla{\uh} \cdot (\dotEK \edgemat^{-1} + \edgemat^{-\top} \dotEK^{\top}) \nabla{\conj{\psi}}}\\
    &\quad - \sum_{j=0}^{1}\int_{\varOmega_{j}}{ \left[ 
    		( \nabla{{\sigma}_{j}} \cdot {\dotXh}) (\nabla{\uh} \cdot \nabla{\conj{\psi}}) 
    			+ (\nabla{\coeffk_{j}} \cdot {\dotXh}) {\uh} {\conj{\psi}}
			- (\nabla{{Q}_{j}} \cdot {\dotXh}) {\conj{\psi}}
    	\right] }{\, dx}\\
   &\quad -\sum_{j=0}^{1}\int_{\varOmega_{j}}{ \left(
     		{\sigma_{j}} \nabla{\uh} \cdot \nabla{\conj{\psi}}
		+ {\coeffk_{j}}  {\uh} {\conj{\psi}}
		- {{Q}_{j}} {\conj{\psi}}
    	\right) (\nabla \cdot {\dotXh}) }{\, dx},
	\qquad \forall {\psi} \in \spaceVh.
\end{aligned}
\end{equation} 
\end{theorem}
\begin{proof}
    The proof proceeds as follows.
    We start by rewriting the variational formulation of the state's FE solution on the deformed mesh $\Tht$. 
    Then, we transform all integrals into the reference mesh and differentiating both sides with respect to time while keeping $\xx$ fixed, we obtain the weak formulation of the material derivative. 
    Afterward, we apply Lemma~\ref{lem:identities}, equations~\eqref{eq:material_derivative_of_a_basis_function}--\eqref{eq:material_derivative_of_fixed_function}, and Lemma~\ref{lem:material_derivative_fe_approximation}, and taking $t=0$, we derive the variational equation~\ref{eq:material_derivative_fem}.

We rewrite \eqref{eq:weak_form_fem} into
\begin{equation}\label{eq:discretized_weak_form}
\begin{aligned}
    &\sumint{  \sigma \nabla{\uh} \cdot \nabla{\conj{\psi}} + \coeffk {\uh}\conj{\psi} } 
    + \intGtop{ ( \alpha + i ) {\uh}\conj{\psi} }\\
    &\qquad = \sumint{ Q{\conj{\psi}}} + \alpha \intGtop{ T_{a} {\conj{\psi}} } + i \intGtop{ h {\conj{\psi}} }, \quad \forall \psi \in \spaceVh.
\end{aligned}
\end{equation}
On $\Tht$, the perturbed FE solution can be expressed as $\tuh = \sum_{i} \ui(t) \psii(\yy,t)$. 
Then, equation \eqref{eq:discretized_weak_form} becomes
\[
\begin{aligned}
    &\sumintt{  \sigma \tnabla{\tuh} \cdot \tnabla{\conj{\tpsi}} + \coeffk {\tuh}\conj{\psi} } 
    + \intGtopy{ ( \alpha + i ) {\tuh}\conj{\tpsi} }\\
    &\qquad = \sumintt{ {Q}(\yy) {\conj{\tpsi}} } + \alpha \intGtopy{ T_{a} {\conj{\tpsi}} } + i \intGtopy{ h {\conj{\tpsi}} }, \quad \forall \tpsi \in \spaceVh(t),
\end{aligned}
\]
where the test space $\spaceVh(t) = \op{span}\{\tpsi_{1}, \tpsi_{2}, \ldots\}$. 
Note that in above variational equation, we actually have $\Gtop(t) = \Gtop(0) \equiv \Gtop$ because $\dot{X} = \dotXh = 0$ on $\partial\varOmega$.

Transforming all integrals onto $K(0)$ and noting that $\detJac = 1$ and $\abs{\Jac^{-\top}\nn} = 1$ on $\Gtop$ for all $t > 0$, we obtain
\begin{align*}
    &\suminto{ ( \sigma \tnabla{\tuh} \cdot \tnabla{\conj{\tpsi}} + \coeffk {\tuh}\conj{\tpsi} ) \detJac } 
    + \intGtop{ ( \alpha + i ) {\tuh}\conj{\tpsi} }\\
    &\qquad = \suminto{ {Q}(\FK(\xx)) {\conj{\tpsi}} \detJac } 
    			+ \alpha \intGtop{ T_{a} {\conj{\tpsi}} } 
    			+ i \intGtop{ h {\conj{\tpsi}} }, \quad \forall \tpsi \in \spaceVh.
\end{align*}
Differentiating the above equation with respect to $t$ while keeping $\xx$ fixed, we obtain
\[
\begin{aligned} 
    &\suminto{ ( \sigma \dot{(\tnabla{\tuh})} \ \cdot \tnabla{\conj{\tpsi}} + \coeffk {\dottuh}\conj{\tpsi} ) \detJac } 
    + \intGtop{ ( \alpha + i ) {\dottuh}\conj{\tpsi} }\\
    &\qquad = - \suminto{ ( \sigma \tnabla{\tuh} \cdot \dot{(\tnabla{\conj{\tpsi}})} \
    			+ \coeffk {\tuh} \dot{\conj{\tpsi}} ) \detJac } 
    			- \intGtop{ ( \alpha + i ) {\tuh} \dot{\conj{\tpsi}} }\\
    &\qquad \quad - \suminto{ ( \dot{\sigma} \tnabla{\tuh} \, \cdot \tnabla{\conj{\tpsi}} 
    			+ \dot{\coeffk} {\tuh}\conj{\tpsi} ) \detJac }
    - \suminto{ ( \sigma \tnabla{\tuh} \, \cdot \tnabla{\conj{\tpsi}} 
    			+ \coeffk {\tuh}\conj{\tpsi} ) \dot{\detJac} }  \\   
    &\qquad \quad + \suminto{ \dot{Q} {\conj{\tpsi}} \detJac } 
    			+ \alpha \intGtop{ \dot{T_{a}} {\conj{\tpsi}} } 
    			+ i \intGtop{ \dot{h} {\conj{\tpsi}} }\\
    &\qquad \quad + \suminto{ {Q} \dot{\conj{\tpsi}} \detJac } 
    			+ \alpha \intGtop{ T_{a} \dot{\conj{\tpsi}} } 
    			+ i \intGtop{ h \dot{\conj{\tpsi}} }\\
    &\qquad \quad + \suminto{ {Q} \conj{\tpsi} \dot{\detJac} },
    \quad \forall \tpsi \in \spaceVh.
\end{aligned}
\]
Now, applying Lemma~\ref{lem:identities}, equations~\eqref{eq:material_derivative_of_a_basis_function}--\eqref{eq:material_derivative_of_fixed_function}, and Lemma~\ref{lem:material_derivative_fe_approximation}, while taking $t=0$ and noting that $\detJac = 1$ and $\yy = \xx$ at $t = 0$, we obtain
\[
\begin{aligned} 
    &\suminto{ ( \sigma (\nabla{\dotuh} - \edgemat^{-\top} \dotEK^{\top} \nabla{\uh})  \cdot \nabla{\conj{\psi}} 
    			+ \coeffk {\dotuh}\conj{\psi} ) } 
    			+ \intGtop{ ( \alpha + i ) {\dotuh}\conj{\psi} }\\
    &\qquad = - \suminto{ ( \sigma \nabla{\uh} \cdot (- \edgemat^{-\top} \dotEK^{\top}) \nabla{\conj{\psi}} } \\
    &\qquad \quad - \suminto{ ( \nabla{\sigma} \cdot \dotXh) \nabla{\uh} \, \cdot \nabla{\conj{\psi}} 
    			+ ( \nabla{\coeffk} \cdot \dotXh ) {\uh}\conj{\psi} }\\
    &\qquad \quad - \suminto{ ( \sigma \nabla{\uh} \, \cdot \nabla{\conj{\psi}} 
    			+ \coeffk {\uh}\conj{\psi} ) ( \nabla \cdot \dotXh ) }  \\   
    &\qquad \quad + \suminto{ (\nabla{Q} \cdot \dotXh) {\conj{\psi}} } 
    		+ \suminto{ {Q} \conj{\psi} ( \nabla \cdot \dotXh ) },
   		\quad \forall \psi \in \spaceVh.
\end{aligned}
\]
Since $K(0) = K$, the following equation, obtained after some rearrangement, is equivalent to the previous one
\[
\begin{aligned} 
    a(\dotuh, \psi)
    &= \suminto{ \sigma \nabla{\uh} \cdot (\dotEK \edgemat^{-1} + \edgemat^{-\top} \dotEK^{\top}) \nabla{\conj{\psi}} } \\
    &\quad - \suminto{ ( \nabla{\sigma} \cdot \dotXh) \nabla{\uh} \, \cdot \nabla{\conj{\psi}} 
    			+ ( \nabla{\coeffk} \cdot \dotXh ) {\uh}\conj{\psi} 
			- (\nabla{Q} \cdot \dotXh) {\conj{\psi}}}\\
    &\quad - \suminto{ ( \sigma \nabla{\uh} \, \cdot \nabla{\conj{\psi}} 
    			+ \coeffk {\uh}\conj{\psi} - {Q} \conj{\psi} ) ( \nabla \cdot \dotXh ) },
   \quad \forall \psi \in \spaceVh.
\end{aligned}
\]
Noting that $\varOmega = \healthy \cup \bartumor$, we obtain \eqref{eq:material_derivative_fem} as desired.

In addition, by the regularity assumptions on $\sigma$, $\coeffk$, and ${Q}$, together with the bounds in \eqref{eq:bounds_for_sigma_and_k}, it can be easily verified that $\dotuh$ uniquely solves \eqref{eq:material_derivative_fem} in $\spaceVh$.
\end{proof}
Using the previously established theorem, we can readily derive a stability estimate for the material derivative of the state, as stated in the following theorem.
Its proof requires the following Poincar\'{e} inequality over $\Vomega$: there exists a constant $c_{P} > 0$ such that  
\begin{equation}\label{eq:Poincare_inequality}
	\cnorm{u}_{L^{2}(\varOmega)} \leqslant c_{P} \cnorm{\nabla{u}}_{L^{2}(\varOmega)^{d}}, \quad \forall u \in \Vomega.
\end{equation}
This inequality can be verified by modifying one of the standard proofs of the Poincar\'{e} inequality, namely the proof by contradiction (see, e.g., \cite[Proof of Prop.~2, p.~127]{DautrayLionsv21998}).
\begin{theorem} \label{thm:mesh_sensitivity_main_result}
Let $\varOmega \subset \mathbb{R}^{d}$, $d \in \{2,3\}$, be a polygonal/polyhedral domain. Let $\alpha \in \mathbb{R}_{+}$, $Q \in H^{1}(\varOmega)$, $T_{a}, h \in H^{1/2}(\Gtop)$, and assume the parameters satisfy Assumption~\ref{assume:weak_assumptions}. 
Let $\dot{X}$ satisfy \eqref{eq:assumption_on_velocity_field}, and let $\Th$ be a simplicial mesh with $\min_{K} \aK > 0$. 
Then, the material derivative $\dotuh \in \spaceVh$ of the FE solution to Problem~\ref{prob:CCBM_weak_form} due to mesh deformation satisfies
\begin{equation}\label{eq:mesh_sensitivity_main_result}
\begin{aligned}
    &c_{1} \cnorm{\nabla \dotuh}_{L^{2}(\varOmega)^{d}} \\
    &\qquad \leqslant \left( c_{2} \norm{\nabla{\sigma}}_{L^{\infty}(\varOmega)^{d}} 
				+ c_{2} c_{P}^{2} \norm{\nabla{\coeffk}}_{L^{\infty}(\varOmega)^{d}}  
				+ c_{P} \norm{\nabla{Q}}_{L^{2}(\varOmega)^{d}}
			\right)  \norm{\dot{X}}_{L^{\infty}(\varOmega)^{d}}\\
    &\qquad  \quad + d \left( c_{P} \norm{{Q}}_{L^{2}(\varOmega)} 
			+ 3 c_{2} \norm{{\sigma}}_{L^{\infty}(\varOmega)} 
			+ c_{2} c_{P}^{2} \norm{{\coeffk}}_{L^{\infty}(\varOmega)}  
	\right) \norm{\nabla \dot{X}}_{L^{\infty}(\varOmega)^{d}} \max_{K} \frac{\hK}{\aK},
\end{aligned}
\end{equation}
where 
\begin{equation}\label{eq:constant_for_continuous_dependence}
	c_{2}=c_{2}({Q},T_{a},h) := c_{0} \left( \norm{Q}_{H^{1}(\varOmega)} + \norm{T_{a}}_{H^{1/2}(\Gtop)} + \norm{h}_{H^{1/2}(\Gtop)} \right).
\end{equation}
\end{theorem}
\begin{proof}
From Theorem~\ref{thm:material_derivative_equation_fem}, we know that $\dotuh \equiv \sum_{i} \dotui(0)\psii \in \spaceVh$.
Choosing $\psi = \uh$ in \eqref{eq:material_derivative_fem} and then taking the real part on both sides, we obtain the following estimates through the Cauchy--Schwarz inequality (applied twice), the Poincar\'{e} inequality \eqref{eq:Poincare_inequality} (also applied twice), inequality \eqref{eq:continuous_dependence} with higher regularity of the data ${Q}$, $T_{a}$, and ${h}$, and the coercivity estimate \eqref{eq:V_coercivity_of_a}:
\begin{align*}
    c_{1} \cnorm{\nabla \dotuh}_{L^{2}(\varOmega)^{d}}
    	&\leqslant c_{2} \norm{{\sigma}}_{L^{\infty}(\varOmega)} \left( 2 \max_{K} \norm{\dotEK \edgemat^{-1}}_{2} 
		+ \norm{\nabla \cdot \dotXh}_{L^{\infty}(\varOmega)}\right)
    		+ c_{2} \norm{\nabla{\sigma}}_{L^{\infty}(\varOmega)^{d}} \norm{\dotXh}_{L^{\infty}(\varOmega)^{d}}  \\
	&\quad + c_{2} c_{P}^{2}(\norm{\nabla{\coeffk}}_{L^{\infty}(\varOmega)^{d}} \norm{\dotXh}_{L^{\infty}(\varOmega)^{d}} 
			+ \norm{{\coeffk}}_{L^{\infty}(\varOmega)} \norm{\nabla \cdot \dotXh}_{L^{\infty}(\varOmega)}) \\
    	&\quad + c_{P}(\norm{\nabla{Q}}_{L^{2}(\varOmega)^{d}} \norm{\dotXh}_{L^{\infty}(\varOmega)^{d}} 
			+ \norm{{Q}}_{L^{2}(\varOmega)} \norm{\nabla \cdot \dotXh}_{L^{\infty}(\varOmega)}),
\end{align*} 
where $c_{2} > 0$ is given in \eqref{eq:constant_for_continuous_dependence}.
By employing Lemma~\ref{lem:identities} and Lemma~\ref{lem:estimates_for_edge_matrices}, and after some rearrangement, we obtain the desired estimate as follows:
\begin{align*}
    c_{1} \cnorm{\nabla \dotuh}_{L^{2}(\varOmega)^{d}}
	&\leqslant \left( c_{2} \norm{\nabla{\sigma}}_{L^{\infty}(\varOmega)^{d}} 
				+ c_{2} c_{P}^{2} \norm{\nabla{\coeffk}}_{L^{\infty}(\varOmega)^{d}}  
				+ c_{P} \norm{\nabla{Q}}_{L^{2}(\varOmega)^{d}}
			\right)  \norm{\dot{X}}_{L^{\infty}(\varOmega)^{d}}\\
    	&\quad + \left( c_{P} \norm{{Q}}_{L^{2}(\varOmega)} 
			+ c_{2} \norm{{\sigma}}_{L^{\infty}(\varOmega)} 
			+ c_{2} c_{P}^{2} \norm{{\coeffk}}_{L^{\infty}(\varOmega)}  
	\right) \norm{\nabla \cdot \dot{X}}_{L^{\infty}(\varOmega)}\\
	&\quad + 2 c_{2} \norm{{\sigma}}_{L^{\infty}(\varOmega)} \max_{K} \norm{\dotEK \edgemat^{-1}}_{2}    \\
	&\leqslant \left( c_{2} \norm{\nabla{\sigma}}_{L^{\infty}(\varOmega)^{d}} 
				+ c_{2} c_{P}^{2} \norm{\nabla{\coeffk}}_{L^{\infty}(\varOmega)^{d}}  
				+ c_{P} \norm{\nabla{Q}}_{L^{2}(\varOmega)^{d}}
			\right)  \norm{\dot{X}}_{L^{\infty}(\varOmega)^{d}}\\
    	&\quad + d \left( c_{P} \norm{{Q}}_{L^{2}(\varOmega)} 
			+ 3 c_{2} \norm{{\sigma}}_{L^{\infty}(\varOmega)} 
			+ c_{2} c_{P}^{2} \norm{{\coeffk}}_{L^{\infty}(\varOmega)}  
	\right) \norm{\nabla \dot{X}}_{L^{\infty}(\varOmega)^{d \times d}} \max_{K} \frac{\hK}{\aK}.
\end{align*}
\end{proof}
\harbrecht{Some remarks on the estimate~\eqref{eq:mesh_sensitivity_main_result} are in order.}
This bound is derived without assuming mesh regularity; the mesh can be isotropic, anisotropic, uniform, or nonuniform, as long as it is simplicial and each element has a positive minimum height (i.e., $\min_{K} \aK > 0$).

For meshes with a large aspect ratio, the term $\max_{K} \hK/\aK$ increases, making the bound more sensitive to $\norm{\nabla \dot{X}}_{L^{2}(\varOmega)^{d \times d}}$. The estimate shows that both the size and the gradient of the mesh velocity field influence the FE solution. 
While the influence of the size is independent of mesh shape, the influence of the gradient depends on the maximum element aspect ratio $\max_{K} \hK/\aK$.

The estimate~\eqref{eq:mesh_sensitivity_main_result} is also independent of any (fictitious) time derivatives. 
Thus, $\dot{X}$ can be interpreted as a mesh displacement field rather than a velocity. 
In shape optimization, mesh displacement corresponds to the perturbation of identity method, while mesh velocity relates to the speed method. 
The homogeneity with respect to time derivatives depends on whether the deformation field is autonomous or not (see \cite{BacaniRabago2015,DelfourZolesio2011}).

Overall, the bound confirms that the FE solution changes only slightly when both $\dot{X}$ and $\nabla \dot{X}$ are small, reflecting its continuous dependence on the mesh.

\begin{remark}
	If the mesh velocity field is not smooth, from Remark~\ref{rem:non-smooth_velocity}, we can replaced \eqref{eq:mesh_sensitivity_main_result} as follows:
	\begin{equation}\label{eq:mesh_sensitivity_non_smooth_case} 
	\begin{aligned}
    c_{1} \cnorm{\nabla \dotuh}_{L^{2}(\varOmega)^{d}}
	&\leqslant \left( c_{2} \norm{\nabla{\sigma}}_{L^{\infty}(\varOmega)^{d}} 
				+ c_{2} c_{P}^{2} \norm{\nabla{\coeffk}}_{L^{\infty}(\varOmega)^{d}}  
				+ c_{P} \norm{\nabla{Q}}_{L^{2}(\varOmega)^{d}}
			\right)  \norm{\dot{X}}_{L^{\infty}(\varOmega)^{d}}\\
    	&\quad + \left( c_{P} \norm{{Q}}_{L^{2}(\varOmega)} 
			+ c_{2} \norm{{\sigma}}_{L^{\infty}(\varOmega)} 
			+ c_{2} c_{P}^{2} \norm{{\coeffk}}_{L^{\infty}(\varOmega)}  
	\right) \dfrac{d+1}{\min_{K} \aK} \norm{\dot{X}}_{L^{\infty}(\varOmega)^{d}}\\
	&\quad + 2 c_{2} \norm{{\sigma}}_{L^{\infty}(\varOmega)} \dfrac{d \sqrt{2} }{\min_{K}  \aK} \norm{\dot{X}}_{L^{\infty}(\varOmega)^{d}}.
\end{aligned}
\end{equation}
\end{remark}
For a given mesh, $\min \aK > 0$ is a fixed value. 
Therefore, inequality \eqref{eq:mesh_sensitivity_non_smooth_case} shows that the FE solution remains continuously dependent on the mesh, even when the mesh velocity field lacks smoothness.
\section{Discretization of the objective function and its gradient} 
\label{sec:discretization_of_the_cost_function}
To numerically solve Problem~\eqref{prob:minimization_problem}, a suitable discretization is required. 
A standard approach involves discretizing the PDE using a finite element space defined on a computational mesh, denoted by $\varOmegah$, where the nodal positions represent the discrete unknown domain. 
A typical choice is to approximate $\Vomega$ by the finite element space of globally continuous, piecewise linear functions, defined as follows:
\begin{equation}\label{eq:space_Vh_Omegah}
    \spaceVhOh = \Vomegah \cap \spacePone, \qquad
    \text{where } \spacePone = \{v \in C(\overline{\varOmega}) \mid \ v|_K \in P_{1}(K), \ \forall K \in \varOmegah \},
\end{equation}
defined over an approximation $\varOmegah = \Th(\varOmega)$ of $\varOmega$.
Consequently, the ``discrete version'' of Problem~\ref{prob:minimization_problem} is given as follows:
\begin{problem}\label{prob:minimization_problem_fem}
Let $\varOmegah$ be an approximation of $\varOmega$ (i.e., it consists of geometrically conforming simplicial elements $K$) and
consider the admissible set of sub-domains $\holdallh$ defined as follows:
	\[
	\holdallh := \left\{ \tumorh \Subset \varOmega_{{\circ}} \ \Big| \begin{array}{l} 
			d(x,\partial\varOmega_{\meshh}) > \alert{\dzero > 0}, \forall x \in \tumorh, \ \healthyh \text{ is connected, and } \\
			\tumorh \text{ is a Lipschitz polygonal/polyhedral domain}
			\end{array} \right\}.
	\]
    Find $\tumor^{\ast} \in \holdallh$ such that
    \[ 
    \tumorh^{\ast} = \operatorname{argmin}_{\tumorh \in \holdallh} \Jh(\tumorh)
    	:= \operatorname{argmin}_{\tumorh \in \holdallh} \frac{1}{2}\intOh{({\imuh})^{2}},
    \]
    where $\imuh = \Im\{\uh\}$, $\uh$ uniquely solves \eqref{eq:weak_form_fem} in $\spaceVhOh$.
\end{problem}
Above, the minimization is understood to be taken over the nodal points in $\varOmegah$.

Notice the distinction between Problem~\ref{prob:minimization_problem_fem} and the continuous case, Problem~\ref{prob:minimization_problem}. 
A key difference is that the admissible domains now have lower regularity. 
As a result, the shape derivative of $\Jh$ cannot be expressed in boundary integral form--that is, in accordance with the Hadamard-Zol\'{e}sio structure theorem \cite[Thm 3.6, p.~479]{DelfourZolesio2011}. 
This occurs because $\varOmegah \notin C^{1,1}$, and thus, even if the data have high regularity, $\uh \notin H^{2}(\varOmegah)$. 
The lack of $H^{2}$ regularity generally prevents transforming the domain integral into a boundary integral via integration by parts; see the last paragraph in \cite[Chap.~10, sec.~5.6, p.~562]{DelfourZolesio2011}. 
For further discussion on this topic, we refer the reader to \cite{EtlingHerzogLoayzaWachsmuth2020}.

Let us derive the discrete version of \eqref{eq:shape_gradient}.
For this purpose, we introduce the discrete adjoint equation corresponding to \eqref{eq:adjoint_equation}:
\begin{problem}\label{prob:adjoint_equation_fem}
Find $\ph \in \spaceVhOh$ such that  
\begin{equation}\label{eq:adjoint_equation_discrete}
    a_{\textsf{adj}}(\ph, \vh)
    = \intOh{\imuh \conj{\vh}}, \quad \forall \vh \in \spaceVhOh.
\end{equation}  
\end{problem}
In Problem~\ref{prob:adjoint_equation_fem}, the sesquilinear form $a_{\textsf{adj}}$ is essentially given by \eqref{eq:adjoint_bilinear_form}, except that the limits of integration are taken over $\varOmegah$ and $\Gtoph$.

Now, directly substituting the state $u$ and adjoint state $p$ with their finite element counterparts $\uh$ and $\ph$ in \eqref{eq:shape_gradient} produces the correct formula for the shape derivative ${d}\Jh (\varOmegah)[\VVh]$ of the discrete objective $\Jh$, provided that the perturbation field $\VVh$ belongs to the admissble space of (discrete) deformation fields
\begin{equation}\label{eq:space_Vh_Omegah}
    \spaceVhOhd = \{\VVh \in H^{1}_{0}(\varOmegah)^{d} \cap C(\overline{\varOmega}_{\meshh})^{d} \mid \ \VVh|_K \in P_{1}(K)^{d}, \ \forall K \in \varOmegah \}.
\end{equation}
\begin{theorem}\label{thm:distributed_shape_gradient_fem}
	Let $\uh \in \spaceVhOh$ and $\ph \in \spaceVhOh$ be the unique solution to Problem~\ref{prob:CCBM_weak_form_fem} and Problem~\ref{prob:adjoint_equation_fem}, respectively.
	Moreover, let $\VVh \in \spaceVhOhd$.
	Then, 
        \begin{equation}\label{eq:shape_gradient_fem}
        \begin{aligned}
            {d}\Jh(\tumor)[\VVh] 
            &= \frac{1}{2} \intOh{\dive{\VVh} {\imu}^{2}} 
        	- \intOh{ \dive{\VVh} \sigma \sum_{j=1}^{d} \left( \partial_{j} \imaginaryh{u} \partial_{j}\realh{p} - \partial_{j}\realh{u} \partial_{j}\imaginaryh{p}\right) }	\\
            	&\qquad + \intOh{\sigma \sum_{m=1}^{d} \sum_{j=1}^{d} \partial_{j}{\VV_{\meshh, m}} \left( \partial_{j} \imaginaryh{u} \partial_{m}\realh{p} - \partial_{j}\realh{u} \partial_{m}\imaginaryh{p}\right) }	\\
            	&\qquad + \intOh{\sigma \sum_{m=1}^{d} \sum_{j=1}^{d} \partial_{j}{\VV_{\meshh, m}} \left( \partial_{m} \imaginaryh{u} \partial_{j}\realh{p} - \partial_{m}\realh{u} \partial_{j}\imaginaryh{p}\right) }	\\
        	&\qquad + \intOh{\dive{\VVh}k (\imaginaryh{u} \realh{p} - \realh{u} \imaginaryh{p}) }
        	+ \alert{\intOh{\dive{\VVh} {Q} \imaginaryh{p}}}.
        \end{aligned}
        \end{equation}
\end{theorem}
The proof of this theorem is similar to the continuous case, so we omit it.
\begin{remark}
Theorem~\ref{thm:distributed_shape_gradient} and Theorem~\ref{thm:distributed_shape_gradient_fem} are equivalent in terms of obtaining the shape derivative. 
Additionally, Theorem~\ref{thm:distributed_shape_gradient_fem} remains true when higher-order Lagrangian finite elements on simplices are used instead of $\spaceVhOhd$. 
However, it is important that $\VVh$ remains piecewise linear so that piecewise polynomials are transformed into piecewise polynomials of the same order. 
Note that this restriction implicitly implies the need to regularize $\VVh$ once the property of being piecewise linear is lost.
\end{remark}
\section{Numerical algorithm and examples} 
\label{sec:numerical_algorithm_and_examples}
This section outlines the numerical implementation of our approach and presents simulations using the programming software \textsc{FreeFem++} \cite{Hecht2012} to demonstrate the algorithm's performance. 
We first address the forward problem and the choice of regularization in the inversion procedure. 
\alert{Based on the study by \cite{Byrdetal2020}, we consider tumors with a circular/elliptical shape (in the 2D case) or an ellipsoidal shape (in the 3D case).}
\subsection{Forward problem}\label{subsec:forward_problem}
The computational setup is as follows:  
In the forward problem, all known parameters, including the exact geometry of the tumor, are specified.  
In all cases, we define $\varOmega$ as a rectangle of dimensions $0.09 \, \text{(m)} \times 0.03 \, \text{(m)}$, and we assume the following thermal physiological parameters based on \cite{AgnelliPadraTurner2011,DengLiu2004,ParuchMajchrzak2007}:  
\begin{align*}
\sigma_1 &= 0.5 \, (\text{W/m} \, {}^{\circ}{\text{C}}),  &&\sigma_0 = 0.75 \, (\text{W/m} \,{}^{\circ}{\text{C}}), &&T_{b} = 37{}^{\circ}{\text{C}}, \\
k_1 &= 1998.1 \, (\text{W/m}^3 \, {}^{\circ}{\text{C}}), && k_0 = 7992.4 \, (\text{W/m}^3 \, {}^{\circ}{\text{C}}), && T_a = 25{}^{\circ}{\text{C}}, \\
Q_1 &= 4200 \, (\text{W/m}^3), && Q_0 = 42000 \, (\text{W/m}^3), && \alpha = 10 \, (\text{W/m}^2 \, {}^{\circ}{\text{C}}). 
\end{align*}

The single observed measurement $h$ is generated synthetically by solving the direct problem \eqref{eq:main_equation}.  
To prevent \textit{inverse crimes} \cite[p.~179]{ColtonKress2019}, the forward problem is solved with a fine mesh and $P_2$ basis functions, while inversion uses a coarser mesh and $P_1$ basis functions. 
Gaussian noise with zero mean and standard deviation ${\harbrecht{\delta}} \norm{h}_\infty$, where ${\harbrecht{\delta}}$ is a free parameter, is added to $h$ to simulate noise.

\subsection{Regularization strategy and choice of regularization parameter value}\label{subsec:balancing_principle}
To deal with noisy measurement in our experiment, we will consider the volume functional $\op{Vol}(\tumor) = \rho \int_{\tumor}{1}\, dx$, where $\rho$ is a small positive constant. 
In this case, selecting the regularization parameter $\rho$ is crucial in the reconstruction process. 
This parameter can be determined using the discrepancy principle, which requires accurate knowledge of the noise level. 
However, reliable noise-level information is often unavailable, and errors in noise estimation can reduce reconstruction accuracy using the discrepancy principle.
To address this, we propose a heuristic rule for selecting $\rho$ that does not rely on noise-level information. 
This rule is based on the balancing principle \cite{ClasonJinKunisch2010b}, originally developed for parameter identification. 
We apply this technique for the first time in shape optimization. 
The idea is simple: fix $\beta > 1$ and compute $\rho > 0$ such that
\begin{equation}\label{eq:balancing_principle}
    (\beta - 1) J(\tumor) - \op{Vol}(\tumor) 
    := (\beta - 1) \frac{1}{2}\intO{({\imu})^{2}} - \rho \int_{\tumor}{1}\, dx = 0.
\end{equation}
This balances the data-fitting term $J(\tumor)$ with the penalty term $\op{Vol}(\tumor)$, where $\beta > 1$ controls the trade-off. 
It eliminates the need for noise-level knowledge and has been successfully applied to linear and nonlinear inverse problems \cite{ClasonJinKunisch2010, ClasonJinKunisch2010b, Clason2012, ClasonJin2012, ItoJinTakeuchi2011, Meftahi2021}.
\harbrecht{In this work, we initially set $\beta = 2$ if the radius of the tumor is accurately estimated before localization. Otherwise, we simply fix $\rho > 0$ to a sufficiently small value.}

The motivation for using the abovementioned regularization strategy is to address the difficulty of recovering small and deep tumors.  
Incorporating the shape derivative of the weighted volume functional,  $d\op{Vol}(\tumor)[\VV] = \rho \int_{\tumor} \dive{\VV} \, dx$, in computing the deformation field helps prevent excessive overshooting of the approximate shape.  
Its effectiveness is demonstrated in subsections~\ref{subsec:test_2} and \ref{subsec:test_3d}.
%
%
%
\subsection{Numerical algorithm}
\label{subsec:Numerical_Algorithm}
The main steps of our numerical algorithm follows a standard approximation procedure (see, e.g., \cite{RabagoAzegami2018}), the important details of which we provide as follows.

The tumor shape is approximated by using a domain variation technique implemented with the finite element method, similar to shape optimization methods \cite{Azegami1994,Azegami2020,Doganetal2007}. 
To prevent unwanted oscillations at the unknown interface, we use an $H^{1}$ Riesz' representative of the distributed shape gradient, a common approach in shape optimization \cite{Azegami1994,Azegami2020}, which prevents instability in the approximation process.
This also allows us to obtain the velocity of each nodes in the finite element mesh to realized the domain variation. 

To compute an $H^{1}$ Riesz' representative of the distributed shape gradient, we compute $\VV \in H_{0}^1(\varOmega)^{d}$ by solving the following variational equation 
\begin{equation}\label{eq:extension_regularization}
	\vect{b}( \VV , \vect{\varphi} ) = -dJ(\tumor)[\vect{\varphi}], \quad \forall \vect{\varphi} \in H_{0}^1(\varOmega)^{d},
\end{equation}
where $\vect{b}:H_{0}^1(\varOmega)^{d} \times H_{0}^1(\varOmega)^{d} \to \mathbb{R}$ denotes the following bounded and $H_{0}^{1}(\varOmega)$-coercive bilinear form
\begin{equation}\label{eq:bilinear_form_for_regularization}
	\vect{b}( \VV , \vect{\varphi} ) := c_{b} \intO{ \left( \nabla{\VV} : \nabla{\vect{\varphi}} + \VV \cdot \vect{\varphi} \right) }
		+ (1-c_{b}) \int_{\partial\tumor}{ \nabla_{\tau} {\VV} : \nabla_{\tau} {\vect{\varphi}} }\, ds,
\end{equation}
where $c_{b} \in (0,1]$ is a free parameter \cite{DapognyFreyOmnesPrivat2018}.

\harbrecht{In subsections~\ref{subsec:test_1} and \ref{subsec:test_2}, the parameter $c_{b}$ is fixed at 0.5, while larger values are used when testing tumors with non-elliptical shapes.
We point out that smaller values of $c_{b}$ strongly restrict changes in the volume of the initial guess. In contrast, values of $c_{b}$ close to one allow the shape of the initial guess to deform more freely.}

\alert{Equation~\eqref{eq:extension_regularization}} provides a Sobolev gradient \cite{Neuberger1997} representation of $-dJ(\tumor)$.
In \eqref{eq:bilinear_form_for_regularization}, the boundary integral enforces additional regularity of the deformation field on the boundary interface $\partial\tumor$; see, e.g., \cite{DapognyFreyOmnesPrivat2018,Doganetal2007,RabagoAfraitesNotsu2025}.
For more discussion on discrete gradient flows in shape optimization, we refer readers to \cite{Doganetal2007}.
%
%
%

To compute the $k$th boundary interface $\partial\tumor^{k}$, we carry out the following procedures:
\begin{description}
\setlength{\itemsep}{2pt}
	\item[1. \it{Initilization}] Fix the maximum number of iterations $K$, $\rho \in (0,1)$ (or $\beta > 1$) and choose an initial guess $\partial\omega^{0}$.
	\item[2. \it{Iteration}] For $k = 0, 1, \ldots, K$ , do the following:
		\begin{enumerate}
			\item[2.1] Solve the state's and adjoint's variational equations on the current domain $\varOmega^{k}$.
			\item[2.2] Choose $t>0$, and compute the deformation vector $\VV^{k}$ using \eqref{eq:extension_regularization} in $\varOmega^{k}$.
			\item[2.3] Update the current domain by setting $\varOmega^{k+1} := \{ x + t^{k} \VV^{k}(x) \in \mathbb{R}^{d} \mid  x \in \varOmega^{k}\}$.		
		\end{enumerate}
	\item[3. \it{Stop Test}] Repeat \textit{Iteration} until convergence.
\end{description}
In Step 2.2, $t^{k}$ is computed using a backtracking line search inspired by \cite[p.~281]{RabagoAzegami2020}, with the formula $t^{k} = s J({\omega}^{k})/\sqrt{\vect{b}( \VV^{k}, \VV^{k} )}$ at each iteration step $k$, where $s > 0$ is a scaling factor. 
While the calculation of the step size can be refined, in our experience, the above simple approach already yields effective results. 
To prevent inverted triangles in the mesh after the update, the step size $t^{k}$ is further reduced. 
The algorithm also terminates when $t^{k} < 10^{-8}$.

\alert{We provide some remarks on the current algorithm.
In numerical shape optimization, \textit{adaptive} FEMs are commonly used for space discretization, aiding in mean curvature computation and local mesh refinement \cite{Doganetal2007}.  
Curvature expressions naturally arise in perimeter or surface measure penalization, as their shape derivatives involve mean curvature.  
To prevent mesh degeneration and excessive stretching, we remesh instead of refining adaptively--every ten steps for 2D problems and every five for 3D.  

The bilinear form $b$ defined by \eqref{eq:bilinear_form_for_regularization} is sufficient for recovering the unknown inclusion (i.e., the tumor shape $\tumor$). 
Alternatively, it can be chosen as the elasticity operator (see, e.g., \cite{Azegami1994,Azegami1996,DapognyFreyOmnesPrivat2018,EtlingHerzogLoayzaWachsmuth2020}).  
This approach, well known in optimal shape design \cite{Azegami1994,Azegami1996,Azegami2020}, helps maintain mesh quality after deformation. 
It is based on the intuition that elastic displacements minimize compression (i.e., local volume changes) \cite{DapognyFreyOmnesPrivat2018}.  

The variational equation \eqref{eq:extension_regularization} can also be analyzed in a discrete setting. 
Although important, we omit this analysis--along with the comparison of volume and boundary shape gradients (see Remark~\ref{rem:volume_versus_boundary_expression_for_the_shape_gradient})--to avoid making the study too extensive.
For further discussion, we refer interested readers to \cite{Doganetal2007,EtlingHerzogLoayzaWachsmuth2020}, which examine discrete gradient flows in shape optimization.  

Additionally, a restricted mesh deformation procedure was developed in \cite{EtlingHerzogLoayzaWachsmuth2020} to prevent mesh deterioration or invalidation due to interior nodes penetrating neighboring cells. We refer readers to that work for more details.  
}
\harbrecht{
\begin{remark}[On adaptive refinement]
The algorithm used in this work to approximate the target shape can be modified to incorporate an adaptive strategy that reduces the error in the objective function using \textit{a posteriori} error estimates, as in \cite{PadraSalva2013}; see Remark~\ref{rem:criterion_for_adaptive_strategy} for a criterion underlying a suitable adaptive approach.
However, we find that remeshing after a fixed number of time steps already yields a reasonable identification of the tumor.  
Nonetheless, we include in Appendix~\ref{appx:a_posteriori_error_estimations} a discussion on a posteriori error estimation based on \cite{PadraSalva2013} for the current formulation.
\end{remark}
}
%
%
%
%
%
%
\subsection{Mesh sensitivity: numerical examples}\label{subsec:mesh_sensitivity_numerical_examples}
\alert{Here, we present numerical experiments to support the results established in subsection~\ref{subsec:mesh_sensitivity_analysis_for_the_finite_element_solution}. 
The main observation is that for a smooth velocity field, stability is independent of the mesh, while for a non-smooth velocity field, it becomes mesh-dependent. 
}

The results of the numerical investigation on mesh sensitivity are summarized in Figure~\ref{fig:mesh_sensitivity_analysis_2D} and Figure~\ref{fig:mesh_sensitivity_analysis_3D} for the 2D and 3D cases, respectively. 
These figures correspond to the numerical experiments conducted in subsections~\ref{subsec:test_1}--\ref{subsec:test_3d}.
The figures illustrate the error convergence behavior for exact and noisy measurements. 
When exact measurements are used, the deformation remains smooth, leading to a convergence rate that is independent of the mesh size. 
In contrast, under noisy measurements, the deformation becomes non-smooth, and the convergence behavior strongly depends on the mesh resolution.
These corroborate the results established in Theorem~\ref{thm:mesh_sensitivity_main_result} and in equation~\eqref{eq:mesh_sensitivity_non_smooth_case}.
Additionally, the presence of noise introduces instability, as seen in the increased values in the error plots. 
This highlights the essential role of measurement accuracy in maintaining robust and consistent numerical convergence. 
If the measurements are noisy and cannot be improved, we can only resort to using regularization methods to achieve at least a stable approximation. 
These methods help reduce the impact of noise and provide more reliable numerical results.
\begin{figure}[htp!]
\centering
\hfill
\resizebox{0.24\textwidth}{!}{\includegraphics{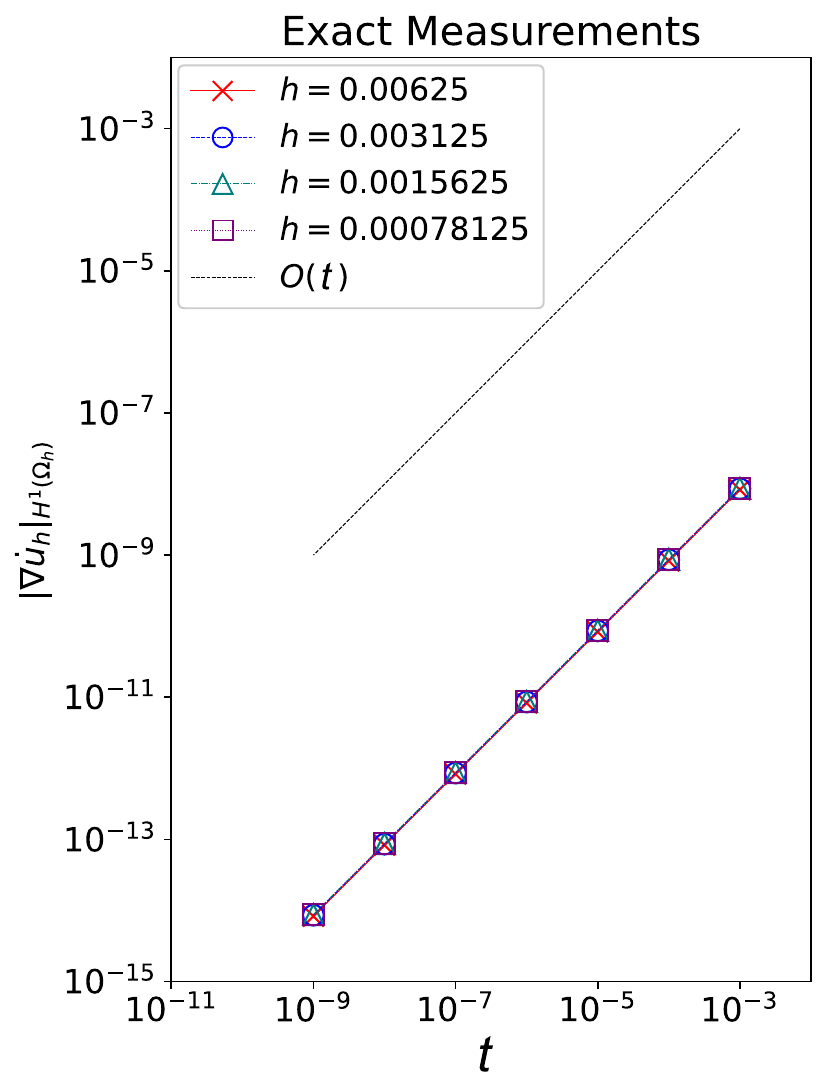}} \hfill 
\resizebox{0.24\textwidth}{!}{\includegraphics{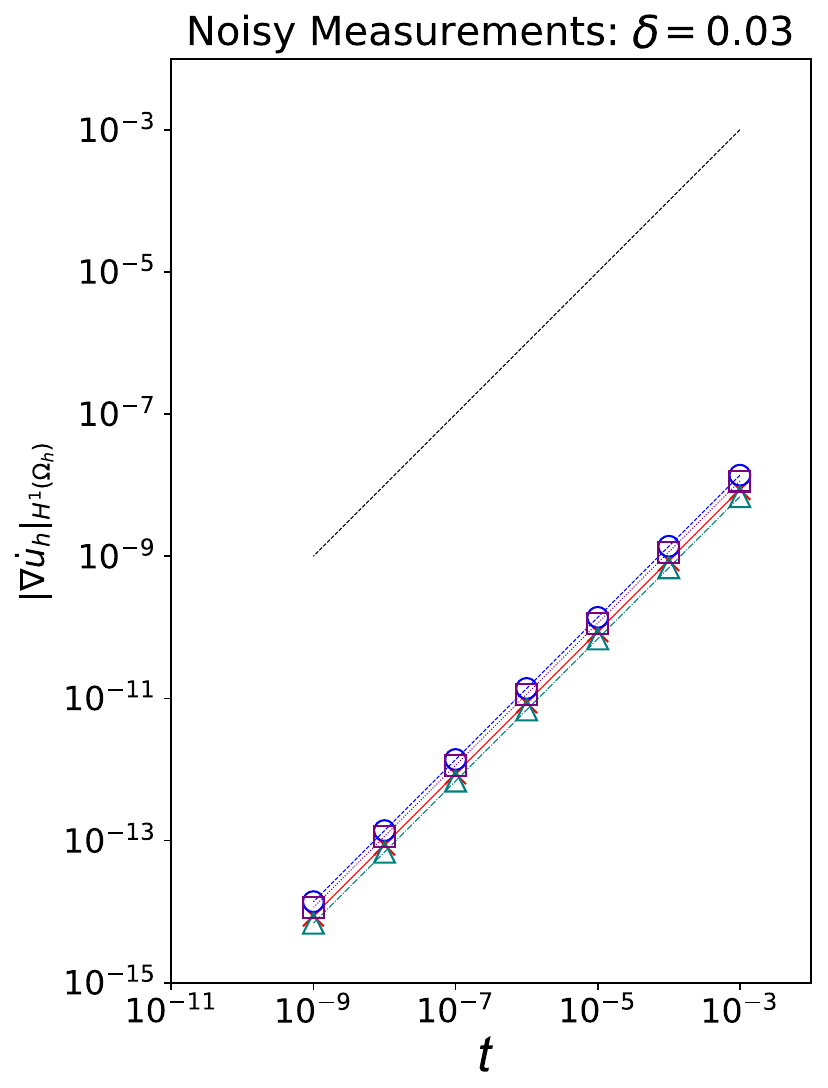}} \hfill 
\resizebox{0.24\textwidth}{!}{\includegraphics{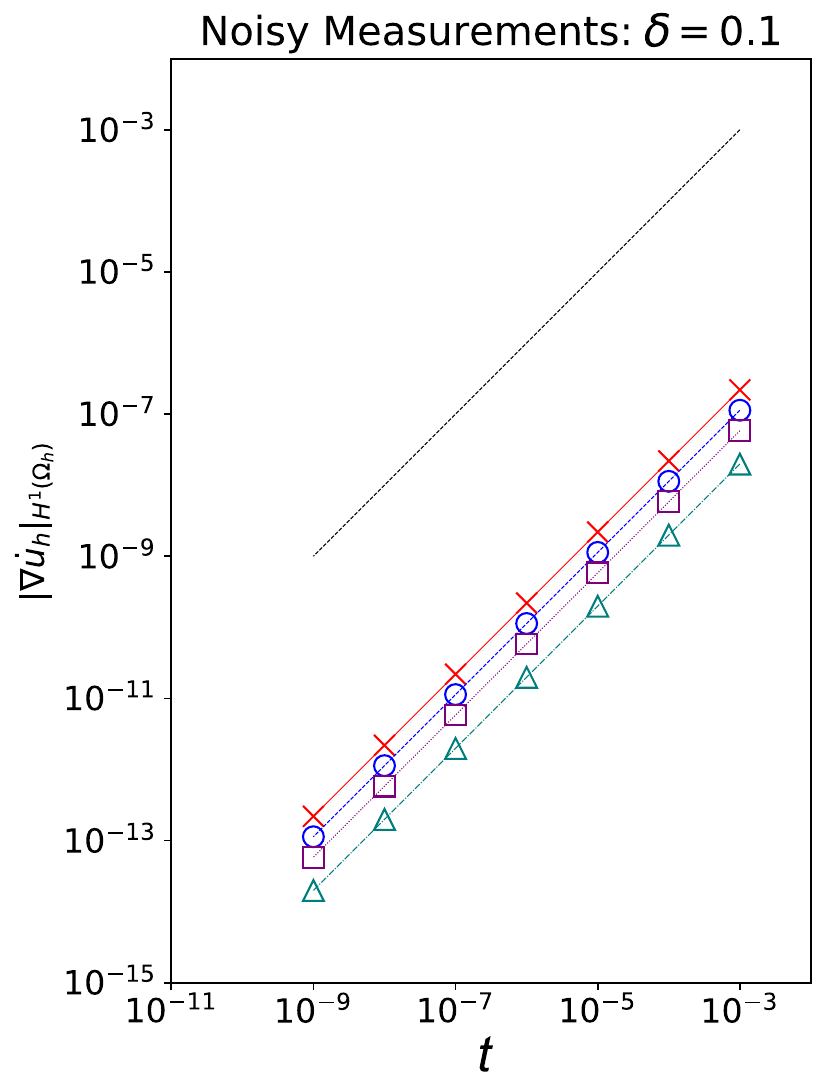}} \hfill 
\resizebox{0.24\textwidth}{!}{\includegraphics{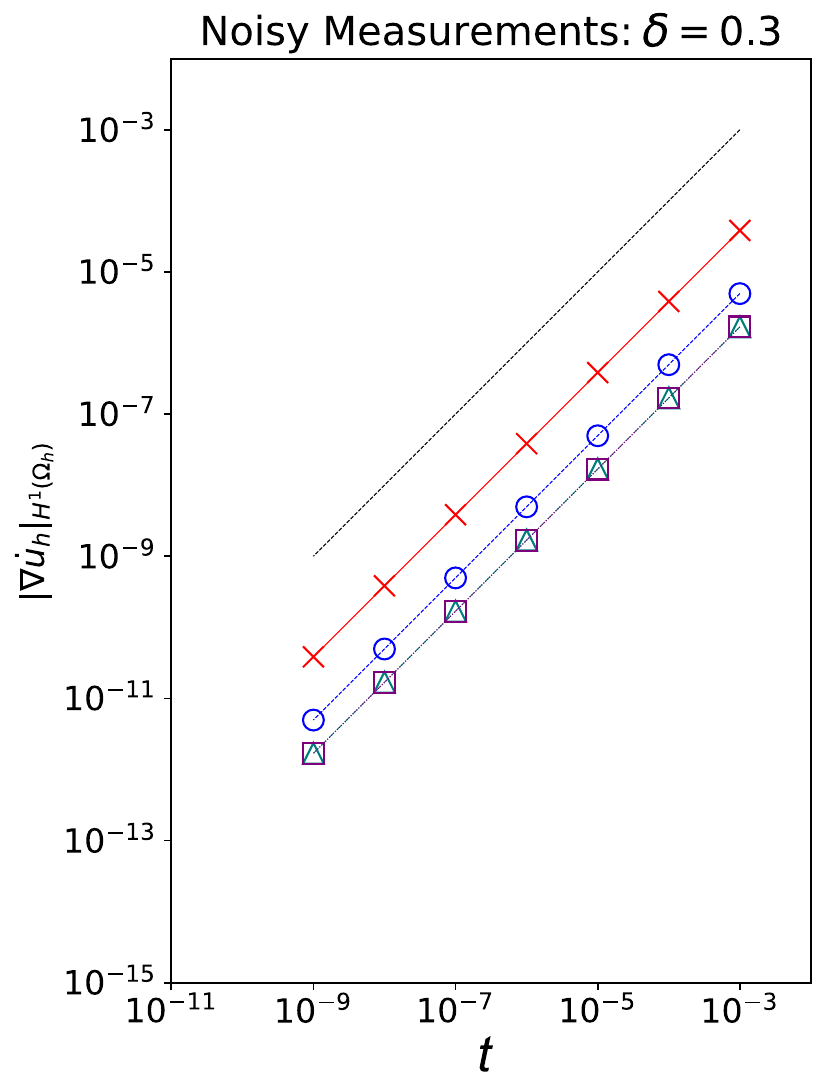}} \hfill 
\caption{Error-of-convergences with exact and noisy measurements for 2D test case}
\label{fig:mesh_sensitivity_analysis_2D}
\end{figure}
\begin{figure}[htp!]
\centering
\hfill
\resizebox{0.24\textwidth}{!}{\includegraphics{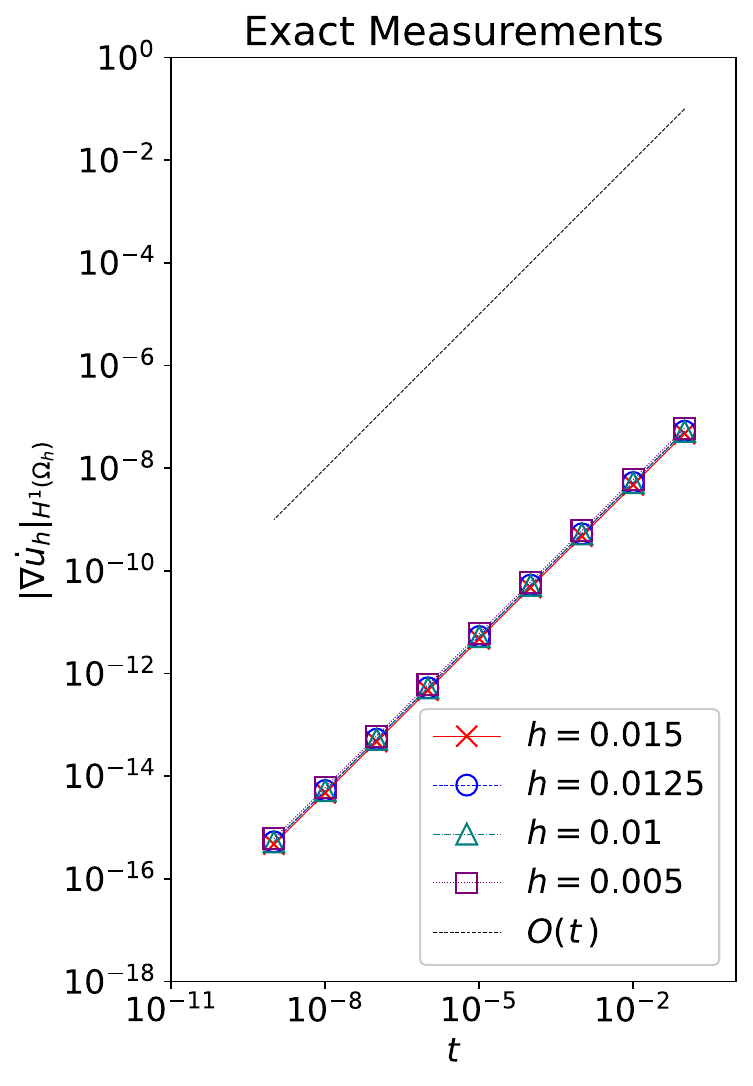}} \hfill 
\resizebox{0.24\textwidth}{!}{\includegraphics{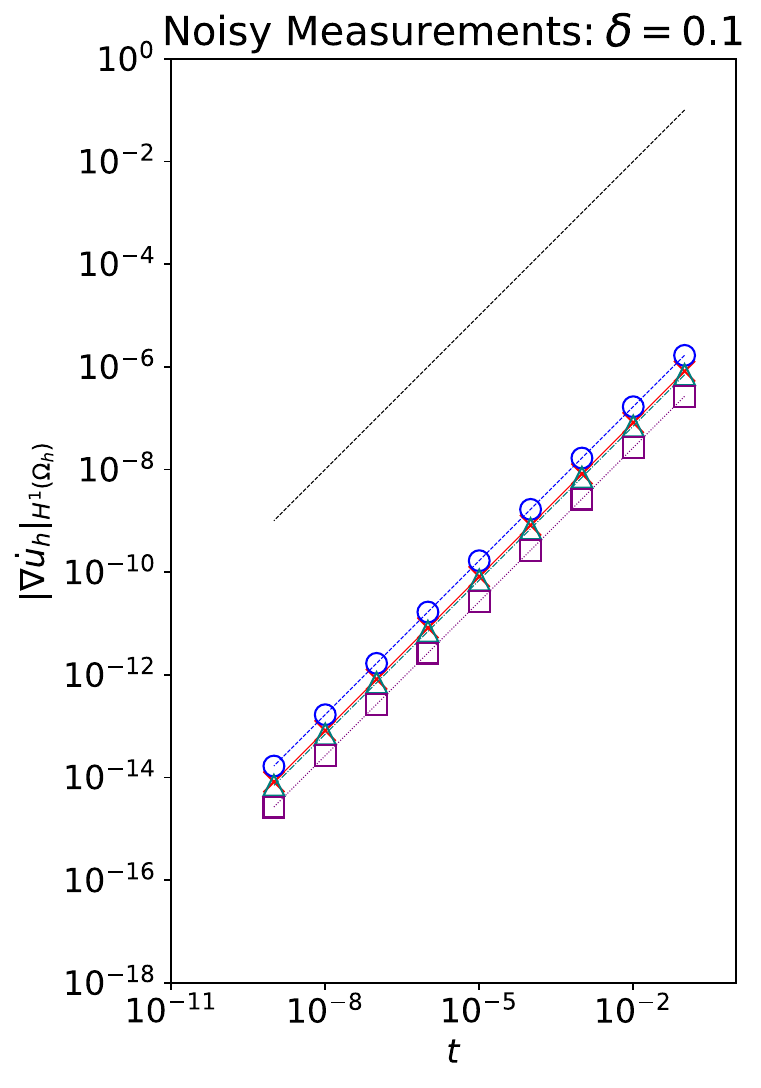}} \hfill 
\resizebox{0.24\textwidth}{!}{\includegraphics{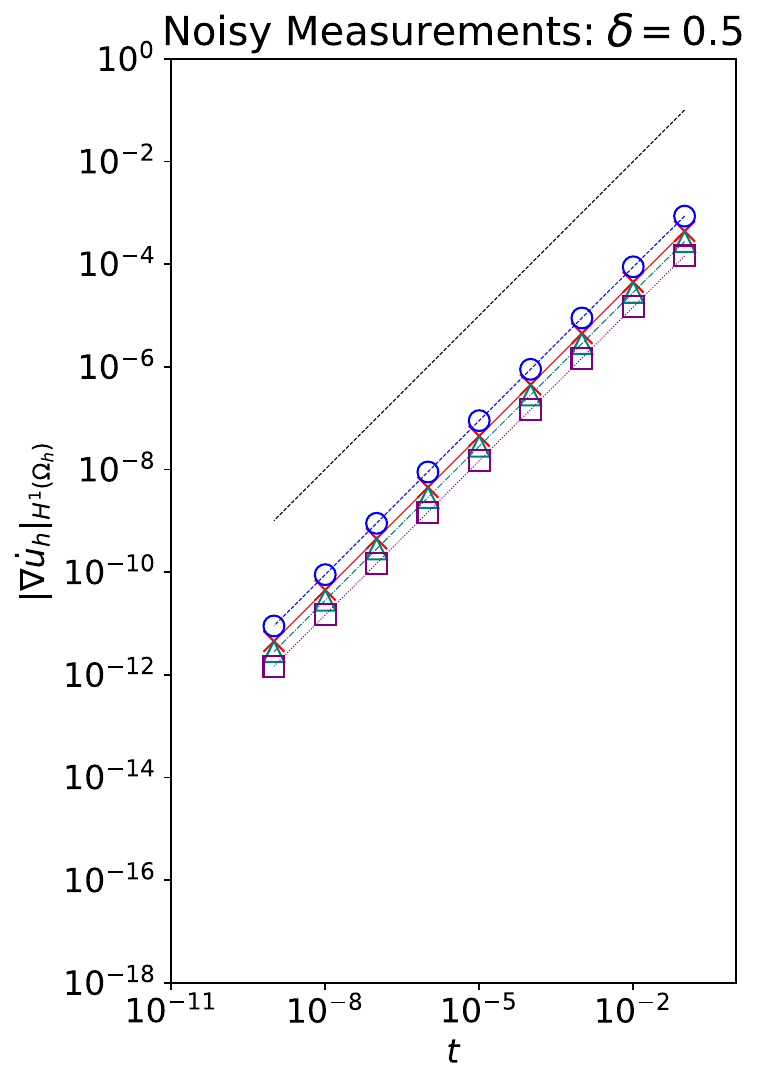}} \hfill 
\resizebox{0.24\textwidth}{!}{\includegraphics{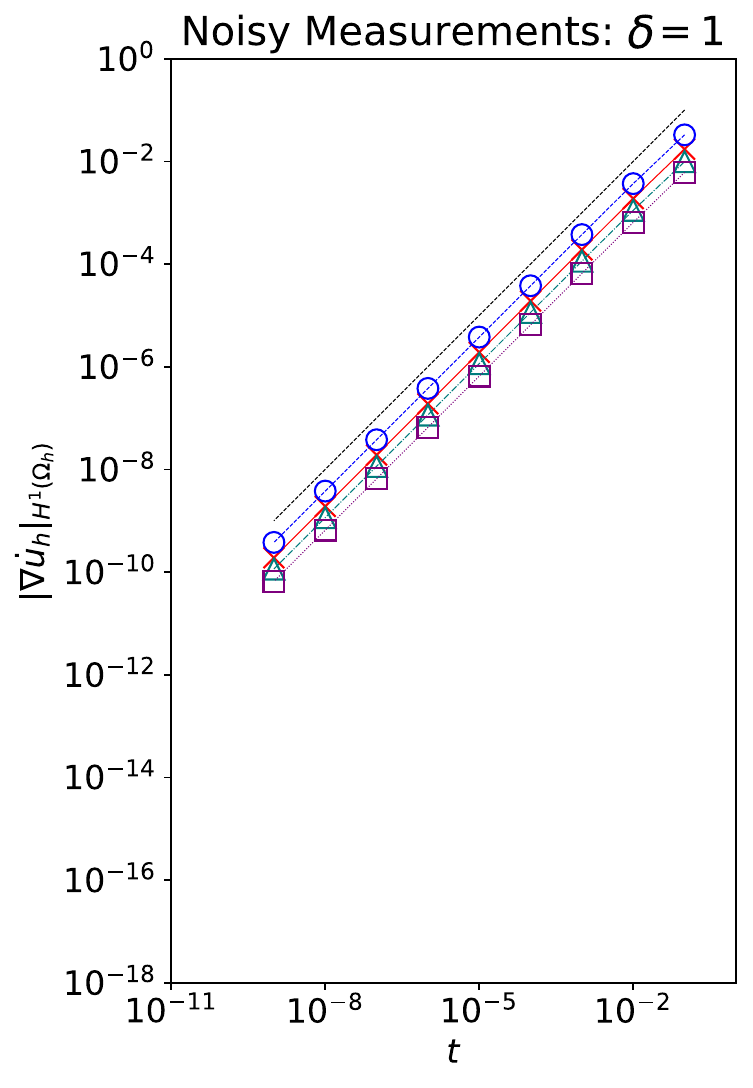}} \hfill 
\caption{Error-of-convergences with exact and noisy measurements for 3D test case}
\label{fig:mesh_sensitivity_analysis_3D}
\end{figure}

We also examine the effect of the free parameter $c_{b}$, which appears in \eqref{eq:bilinear_form_for_regularization}, on the stability of the material derivative with respect to mesh deformation.  
Figure~\ref{fig:mesh_sensitivity_analysis_2D_with_respect_to_regularization} illustrates how different values of this parameter influence error convergence in computing the Riesz representative of the shape gradient in the 2D case.  
The plots show that with exact measurements, all values of $c_{b}$ exhibit similar convergence trends. 
However, when noise is introduced, larger values of $c_{b}$ lead to smoother and more stable error decay, indicating that the computed deformation fields remain more regular across the mesh. 
In contrast, smaller values of $c_{b}$ make the method more sensitive to noise, resulting in erratic error curves and less stable deformations, as expected. 
While a higher $c_{b}$ enhances robustness and prevents excessive local deformations, an excessively large value may over-regularize the deformation, potentially obscuring finer details.  
Thus, selecting an appropriate $c_{b}$ is essential for balancing smoothness and sensitivity, ensuring stability under noise while preserving accuracy in shape gradient representation.
In conclusion, for low noise levels ($<10\%$), the scheme remains generally stable across all tested mesh sizes for $c_{b} > 10^{-5}$. 
Specifically, step sizes of $t \leqslant \varepsilon_{1}$ ($< 10^{-3}$) result in $\norm{\nabla{\dot{u}}}_{\LL^{2}(\varOmega)} < \varepsilon_{1}$.

Although not shown, we emphasize that similar convergence behavior is observed in the 3D case.
We fix $c_{b} = 0.5$ in subsections~\ref{subsec:test_1}--\ref{subsec:test_3d}.

\begin{figure}[htp!]
\centering
\hfill 
\resizebox{0.24\textwidth}{!}{\includegraphics{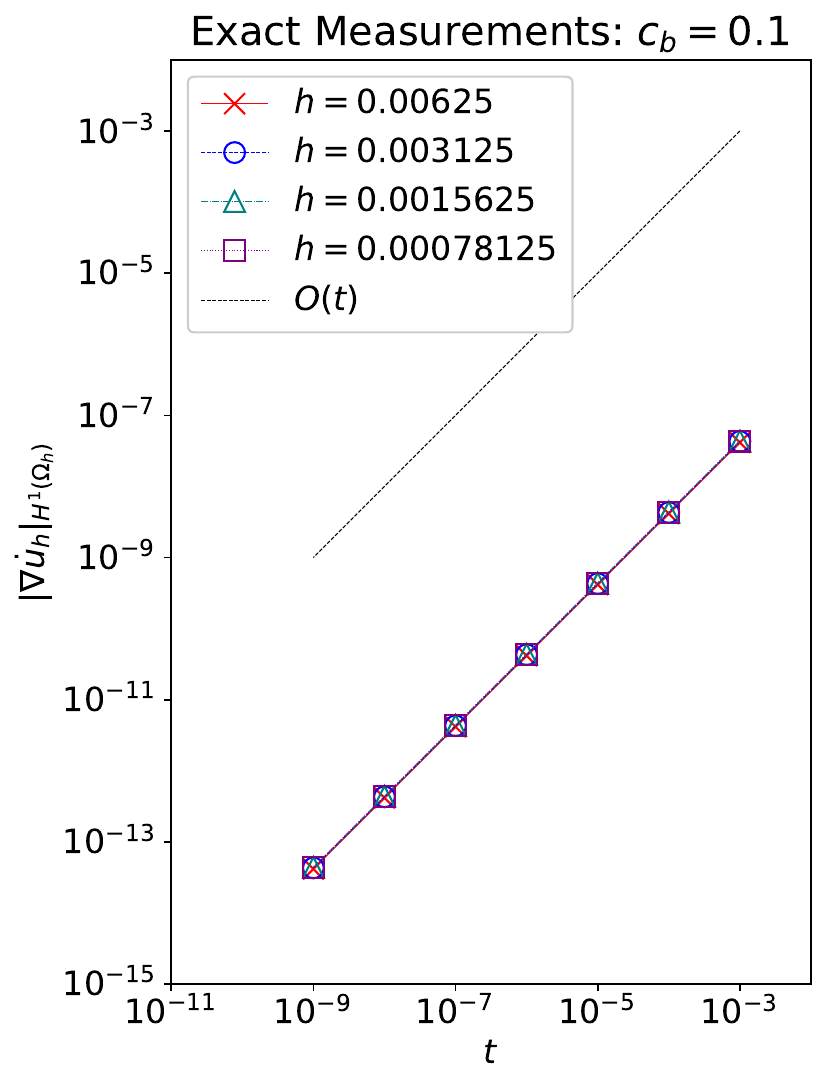}} \hfill
\resizebox{0.24\textwidth}{!}{\includegraphics{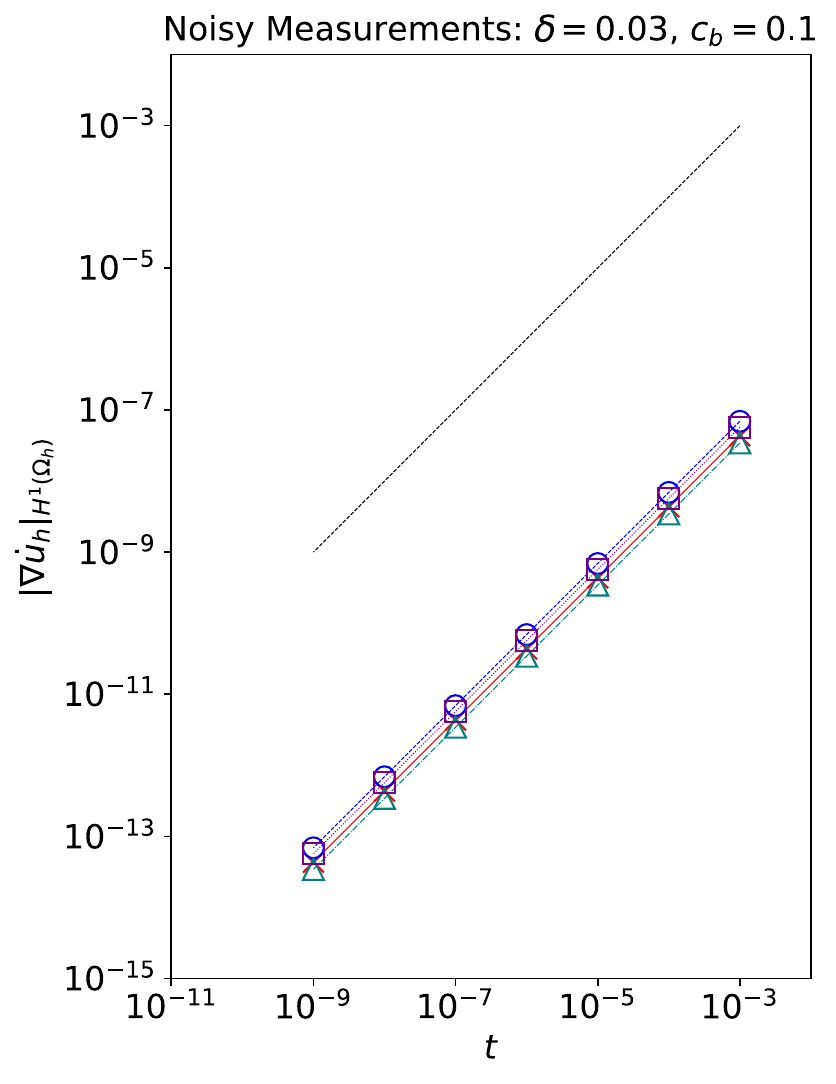}} \hfill
\resizebox{0.24\textwidth}{!}{\includegraphics{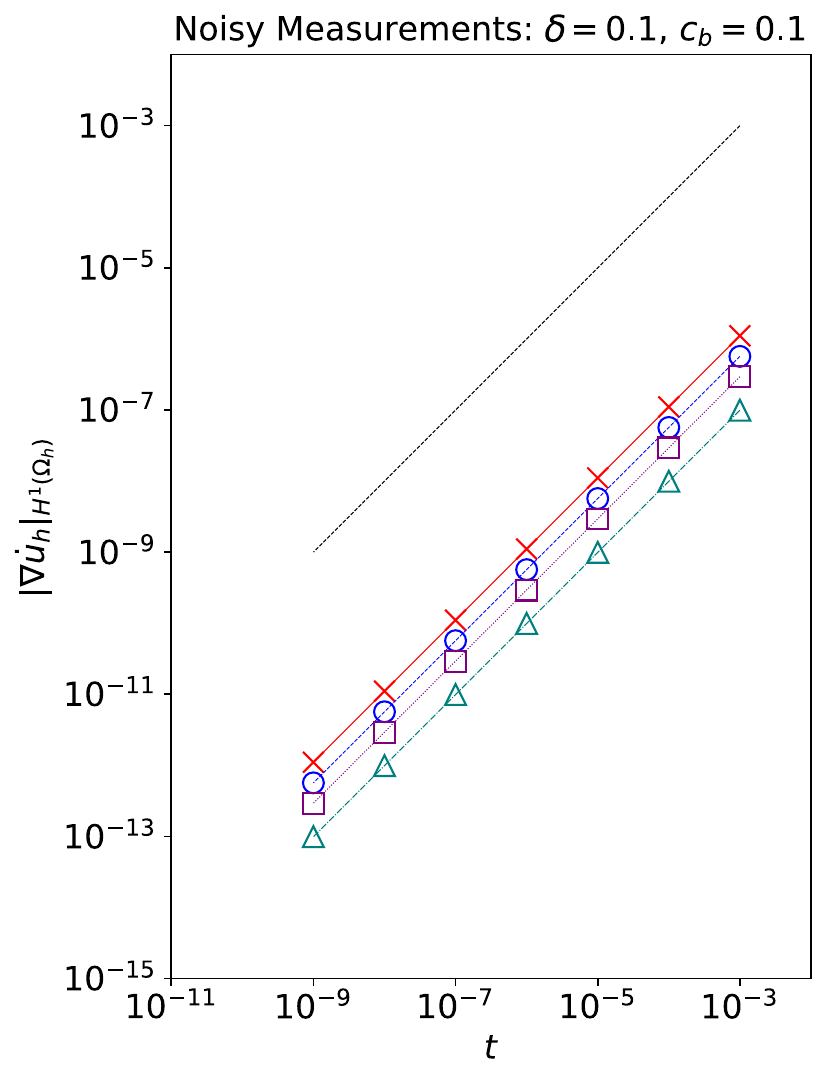}} \hfill
\resizebox{0.24\textwidth}{!}{\includegraphics{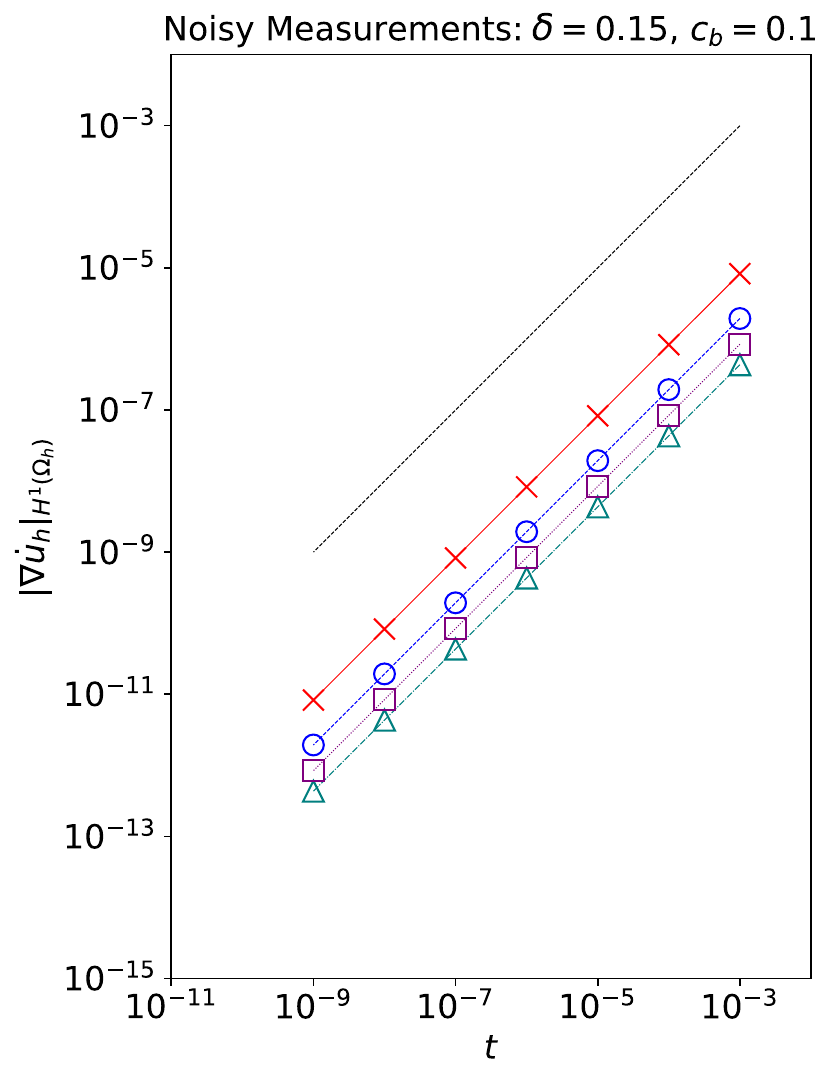}} \\[0.5em] 
\resizebox{0.24\textwidth}{!}{\includegraphics{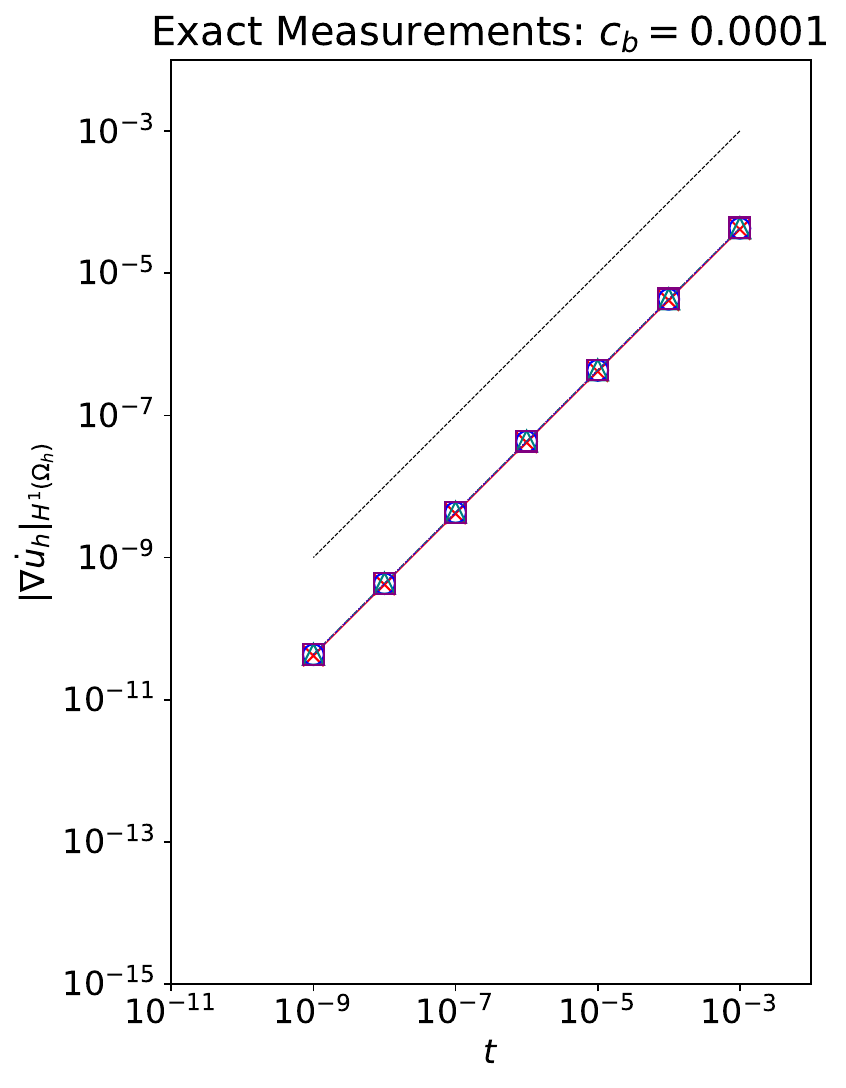}} \hfill
\resizebox{0.24\textwidth}{!}{\includegraphics{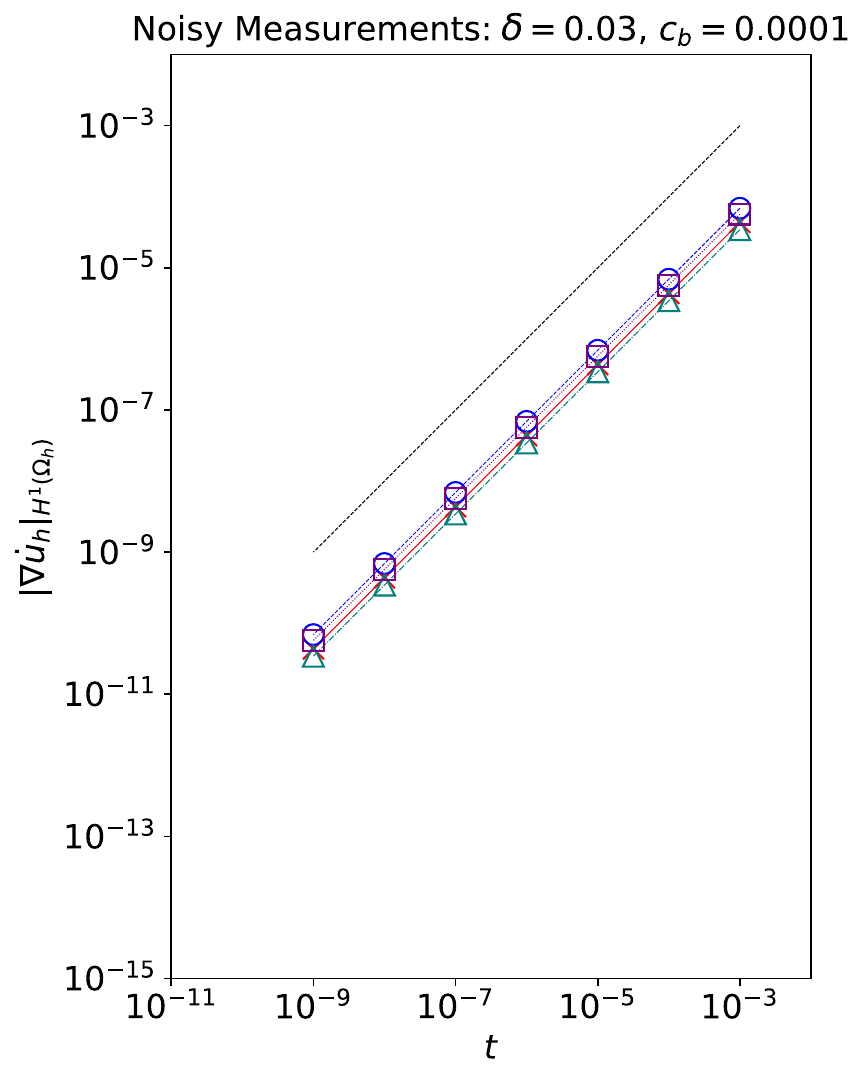}} \hfill
\resizebox{0.24\textwidth}{!}{\includegraphics{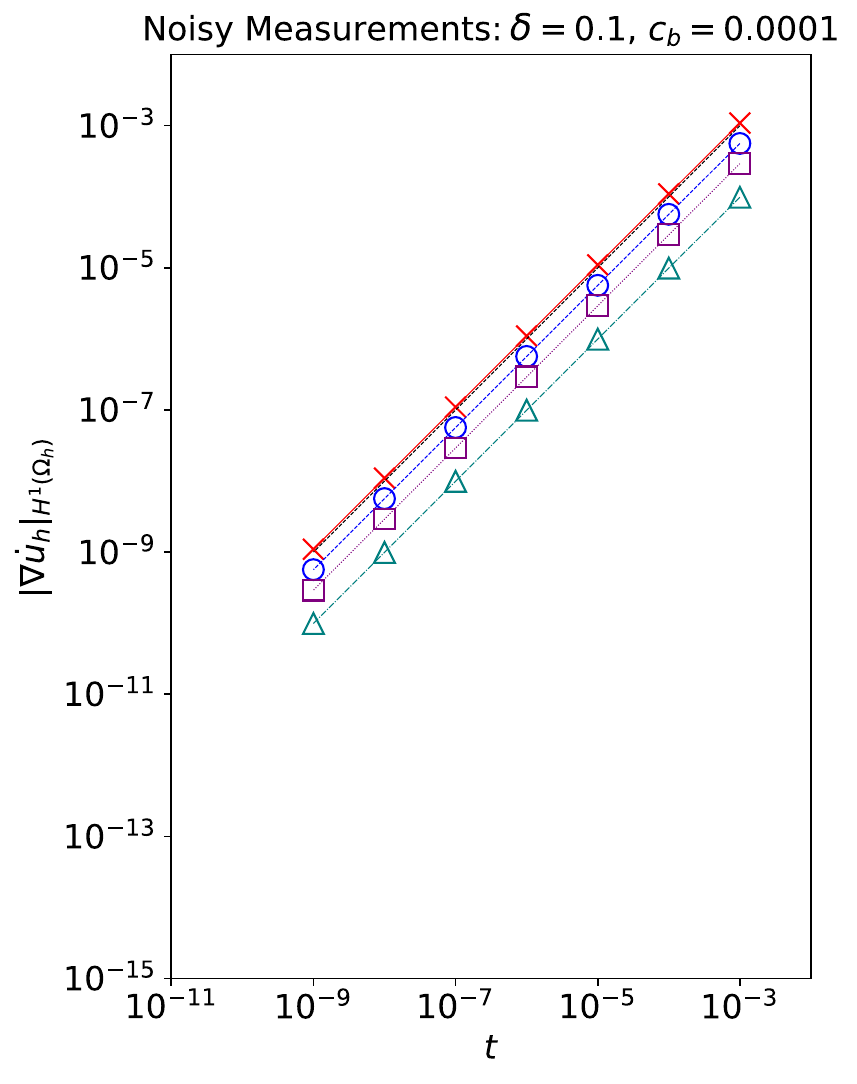}} \hfill
\resizebox{0.24\textwidth}{!}{\includegraphics{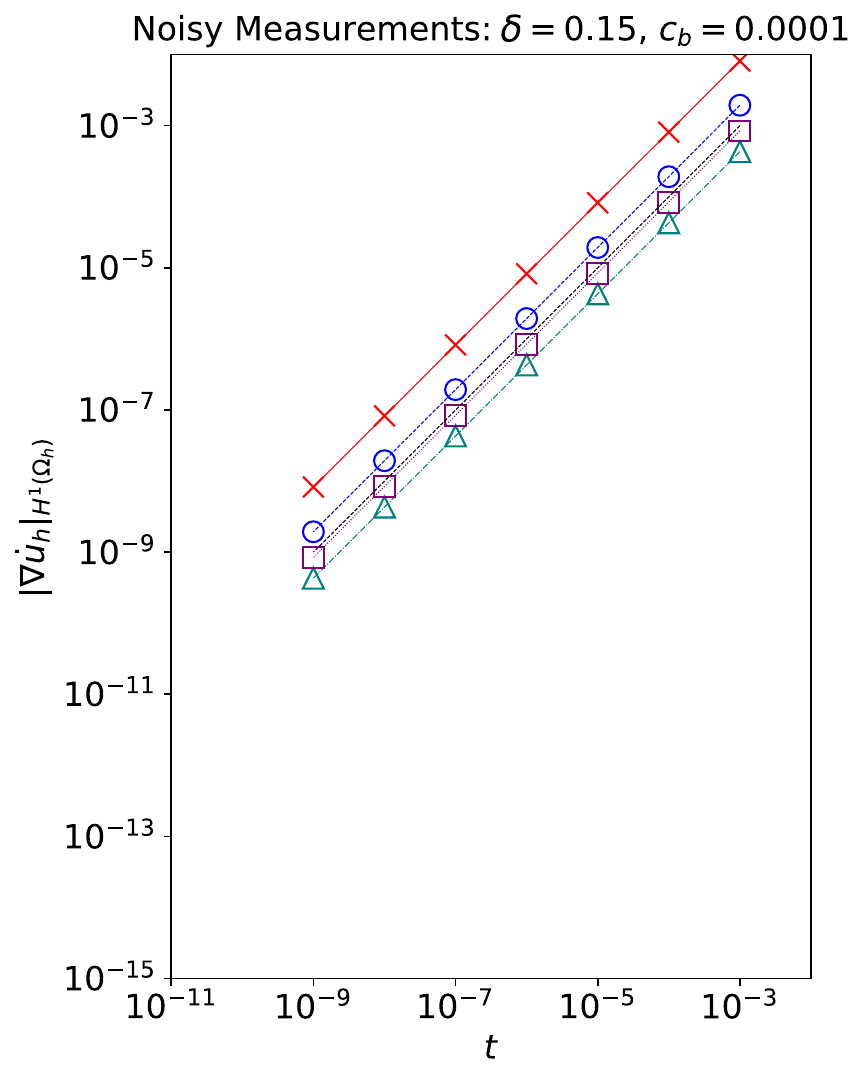}} \\[0.5em] 
\caption{Error-of-convergences with exact and noisy measurements for 2D test case testing the effect of $c_{b}$.}
\label{fig:mesh_sensitivity_analysis_2D_with_respect_to_regularization}
\end{figure}
%
%
%
\subsection{Preliminary tests}
\label{subsec:Preliminary_Tests}
Before examining specific numerical examples of the inverse problem, we first conduct some preliminary tests on the solution of the forward problem. 
These tests focus on how the size and location of the tumor affect the skin surface temperature, as well as the impact of noise on the temperature profile. 
The results will help us establish a good initial guess for the reconstruction.

\begin{figure}[htp!]
\centering 
\resizebox{0.24\textwidth}{!}{\includegraphics{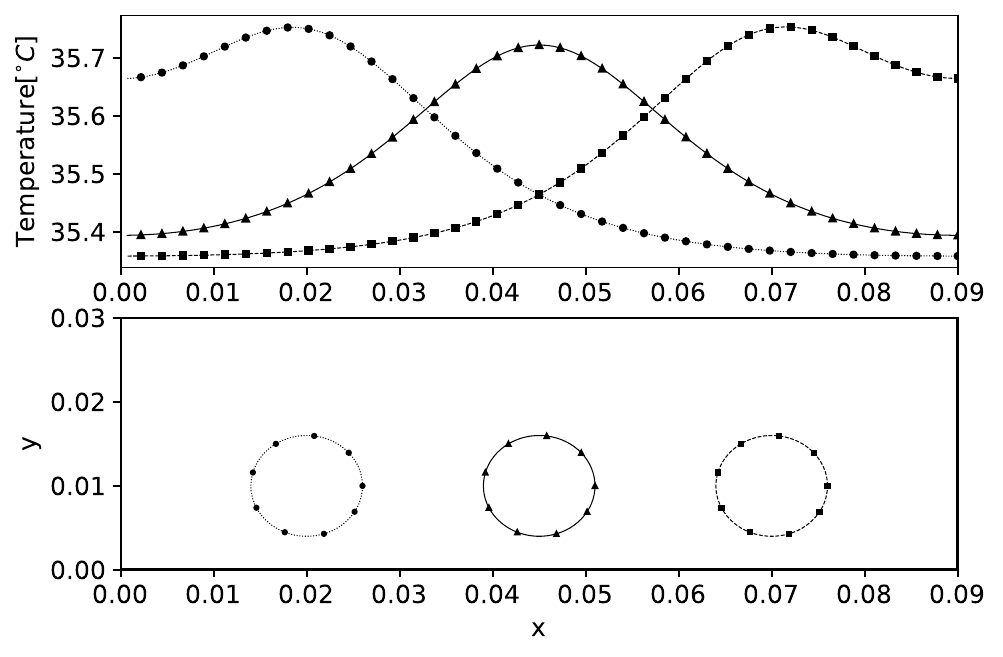}}
\resizebox{0.24\textwidth}{!}{\includegraphics{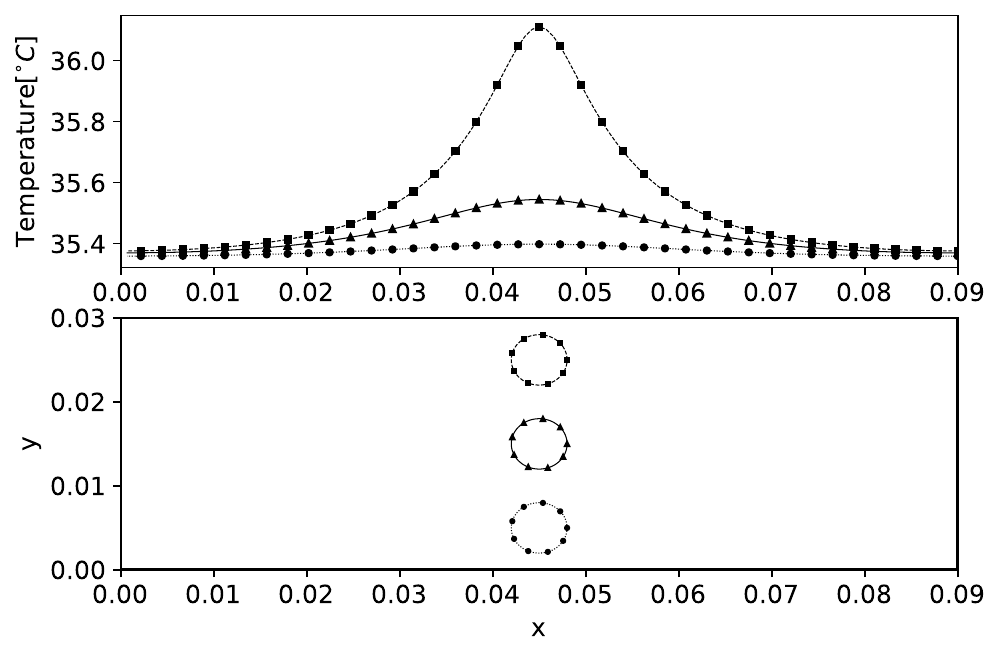}}
\resizebox{0.24\textwidth}{!}{\includegraphics{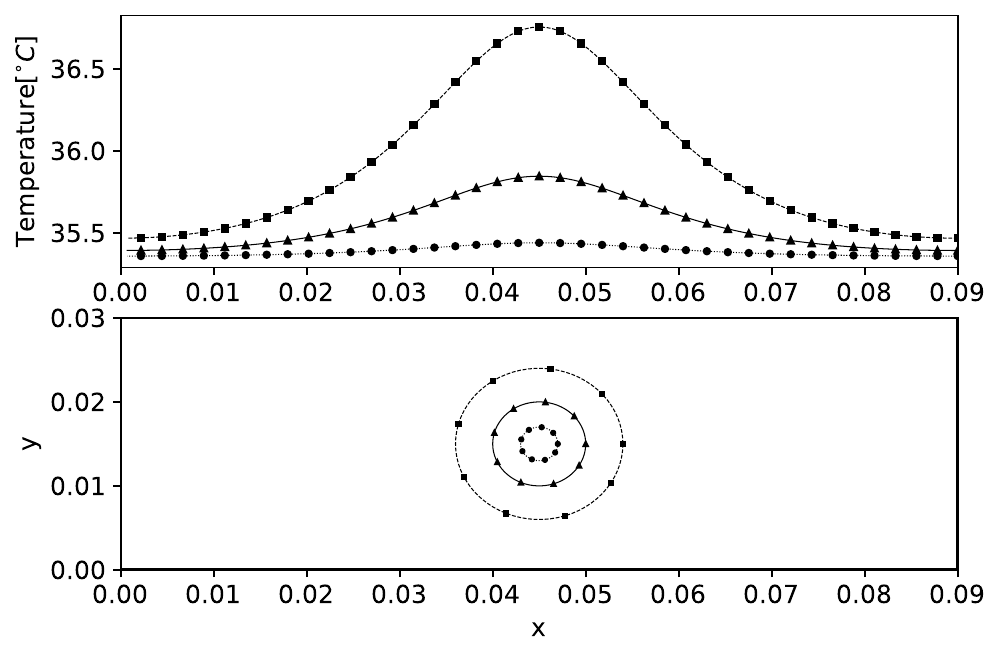}}
\resizebox{0.24\textwidth}{!}{\includegraphics{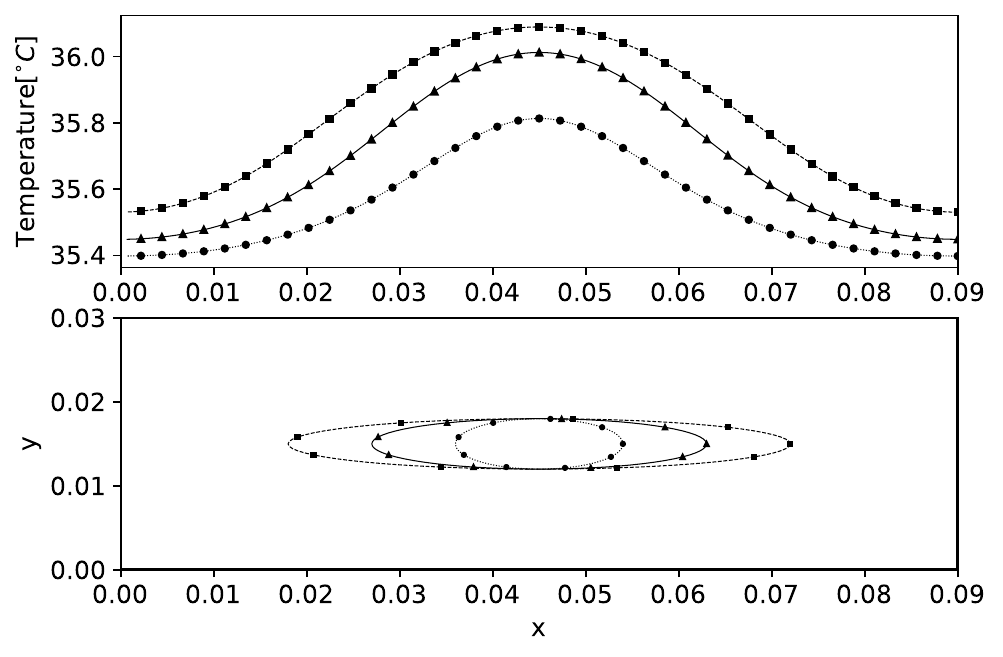}}
\caption{The effect of size and location of the tumor on the skin surface temperature profile}
\label{fig:size_and_locations_tests}
\end{figure}

The numerical results in Figure~\ref{fig:size_and_locations_tests} illustrate the relationship between tumor size and location (lower plots) and the corresponding skin surface temperature profile (upper plots). 
Tumors closer to the skin or with larger dimensions produce higher and sharper peaks in the temperature profile, while smaller or deeper tumors result in subtler variations. 
The spatial arrangement of tumors is evident in the temperature distribution, as peaks correspond to their positions. 
Furthermore, irregular tumor shapes, such as elliptical forms, generate broader or asymmetric temperature patterns. 
These observations highlight the sensitivity of skin surface temperature to underlying tumor characteristics, a crucial factor for non-invasive diagnostic methods. 
Notably, the results align with experimental findings reported in the literature. 
We will use these numerical results to guide the initialization of tumor shapes in the numerical approximation for solving the subsequent inverse problems.
It is worth noting that a similar idea was used in \cite{RabagoAzegami2018}, but in the context of cavity detection.

\begin{figure}[htp!]
\centering
\hfill
\resizebox{0.3\textwidth}{!}{\includegraphics{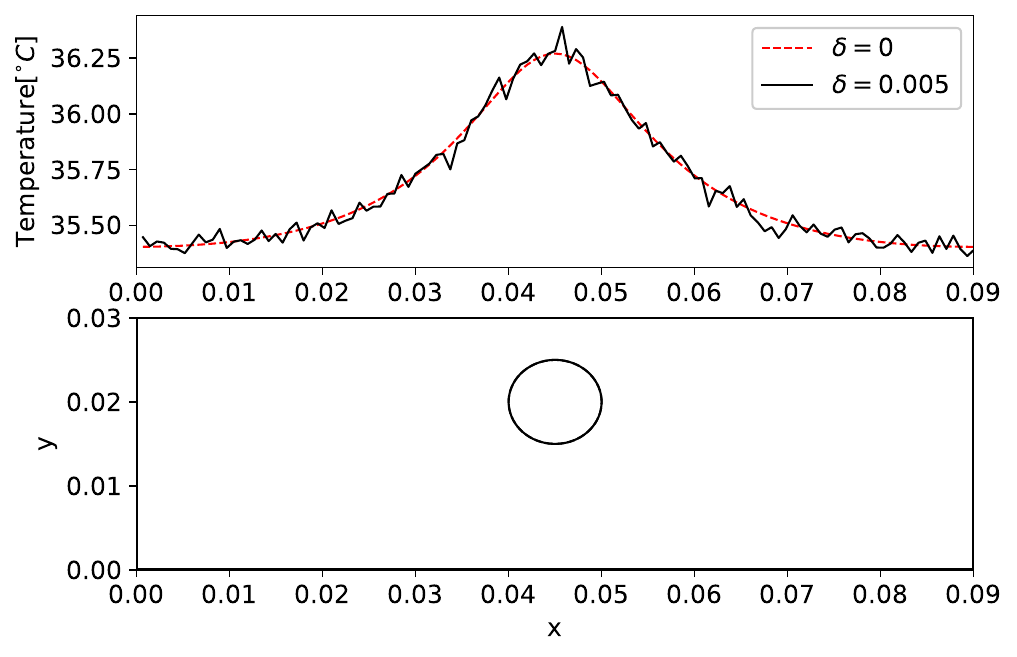}} \hfill
\resizebox{0.3\textwidth}{!}{\includegraphics{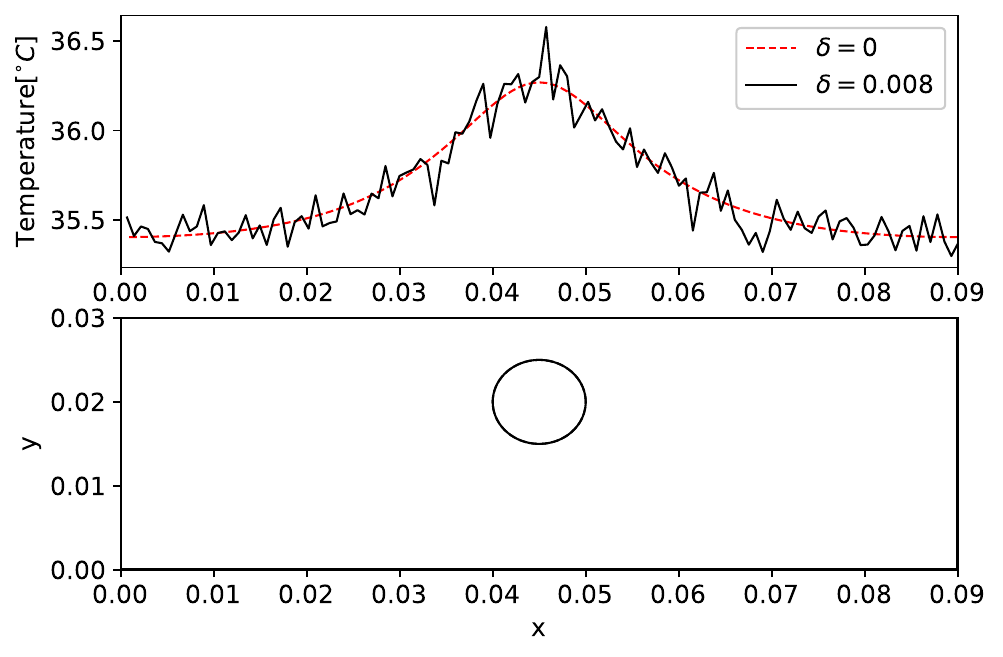}} \hfill
\resizebox{0.3\textwidth}{!}{\includegraphics{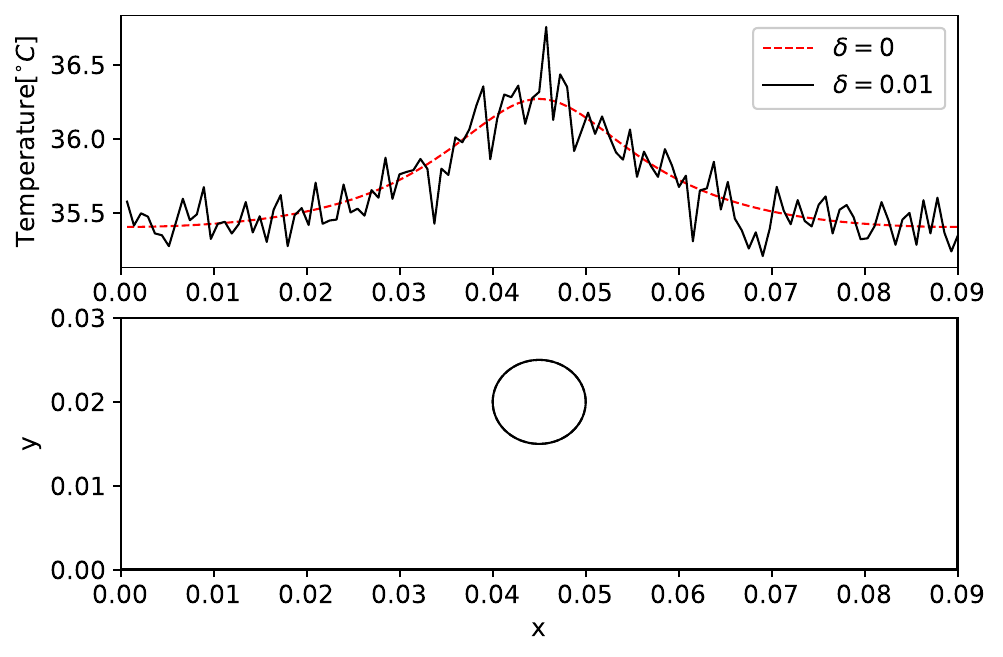}} \hfill 
\caption{The effect of noise level on the skin surface temperature profile}
\label{fig:effect_of_noise_on_skin_temperature}
\end{figure}
The top row of Figure~\ref{fig:effect_of_noise_on_skin_temperature} shows that increasing noise levels (${\harbrecht{\delta}} = 0.005, 0.008, 0.01$) introduce fluctuations in the skin surface temperature profile, making it less smooth and obscuring finer details. 
However, even in the presence of noise, the location of the tumor (indicated in the bottom row) can still be inferred from the localized temperature rise, though higher noise levels make this more challenging.
%
%
%
\subsection{\harbrecht{Identification of a shallow circular tumor}}\label{subsec:test_1}
We now consider the identification of a circular tumor centered at $(0.045, 0.020)$ with a radius of $0.005$ (m).  
The measured data is assumed to include $1\%$ noise, i.e., $\harbrecht{\delta} = 1\%$.  
Figure~\ref{fig:example_1_results} presents the numerical results for tumor shape reconstruction based on temperature measurements taken on the skin surface.

The upper left panel shows the initial setup. 
The temperature profile is fitted with an 11th-order polynomial to identify its peak, which is then used to determine the initial guess for the tumor location.  
The bottom plot displays the exact tumor location (solid line) and the initial guess (dashed line), centered at $(x_0, y_0) = (0.046, 0.01)$.

The upper right panel shows how the objective function and gradient norm evolve over iterations for initial radii $r_0 = 0.004$, $0.005$, and $0.006$. 
Among these, $r_0 = 0.005$ gives the lowest final cost, indicating the most accurate tumor approximation.
The method seems sensitive to the initial guess, especially when it is far from the actual tumor location.
We will examine this issue further in the next subsection.

The first two plots in the lower panel show the evolution of the real and imaginary parts of the state solution $u$ measured on the accessible boundary $\Gtop$.  
The third plot displays the evolution of the tumor boundary during optimization, with color indicating iteration progress (from purple for the initial guess to red for the final shape). 
The exact tumor shape is shown as a solid magenta line with `$\circ$' markers.  
Clearly, the final boundary closely approximates the true tumor shape.
The last plot displays the evolution of $J$ over iterations.
\begin{figure}[htp!]
\centering 
\resizebox{0.4\textwidth}{!}{\includegraphics{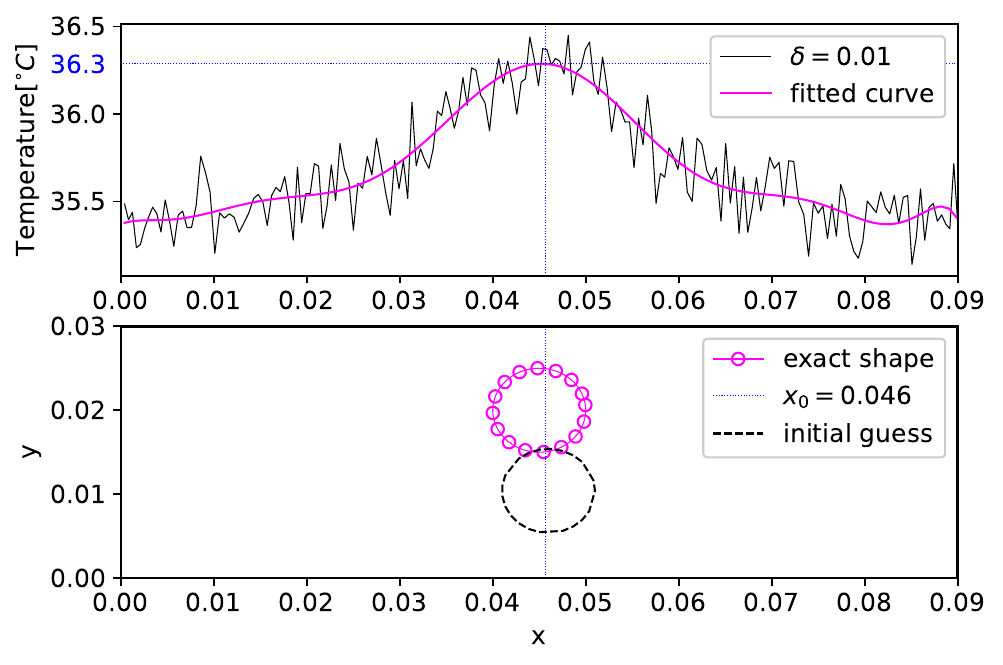}}\quad
\resizebox{0.48\textwidth}{!}{\includegraphics{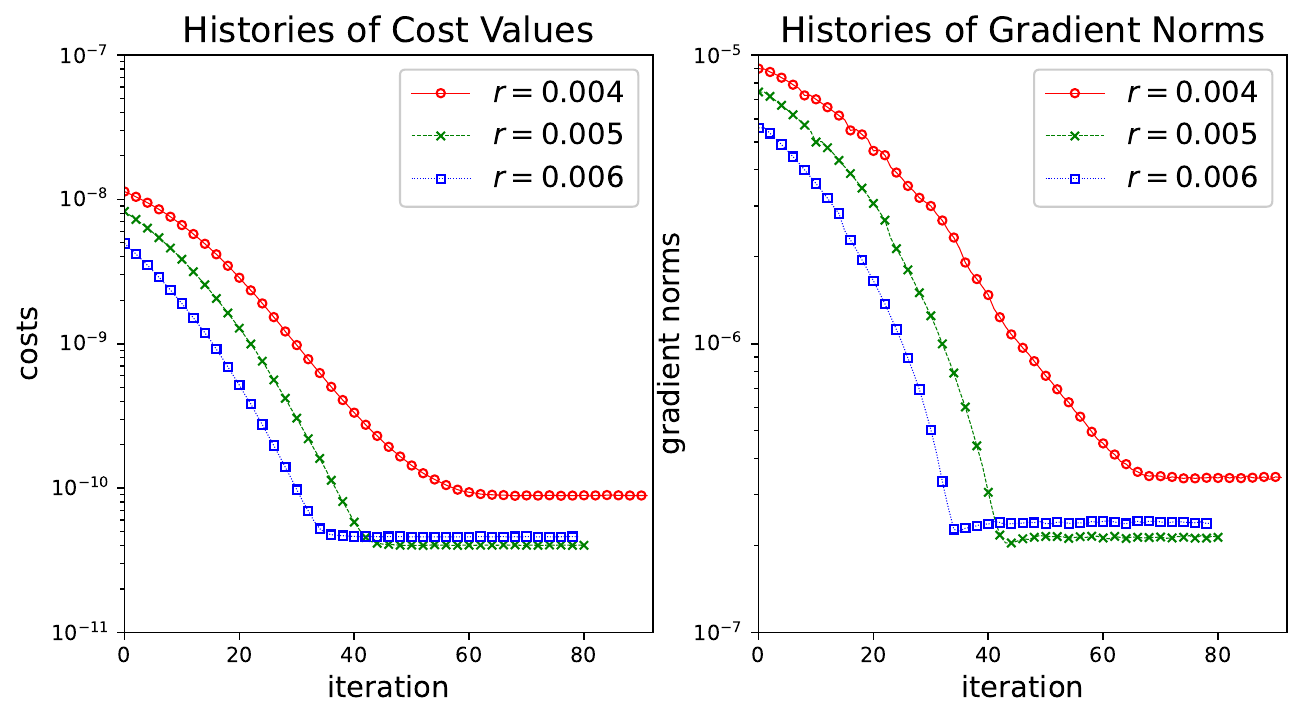}}\\[0.5em] 
\resizebox{0.85\textwidth}{!}{\includegraphics{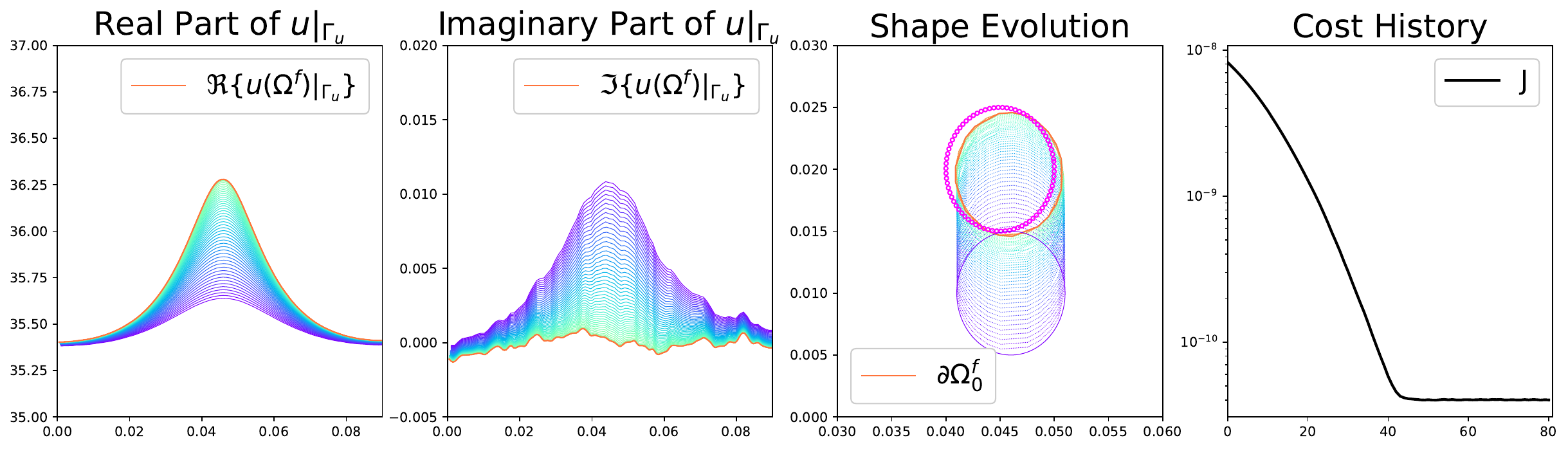}}
\resizebox{0.3\textwidth}{!}{\includegraphics{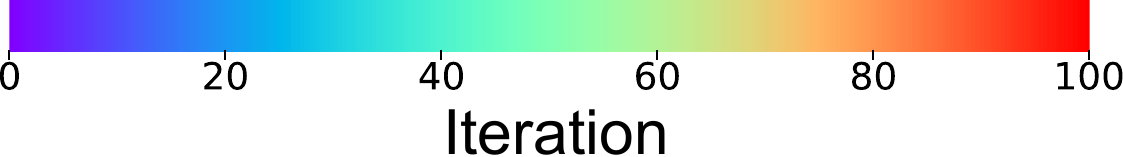}}
\caption{\harbrecht{(Top left) Temperature profile on the skin surface with noise level ${\harbrecht{\delta}} = 0.01$ and the initial setup. (Top right) Cost and gradient-norm histories for different values of the initial radius $r_{0}$. (Bottom, from left to right) Evolution of the real part of $u$; evolution of the imaginary part of $u$ restricted to $\Gtop$; evolution of the free boundary; and history of the cost value.}}
\label{fig:example_1_results}
\end{figure}
%
%
%
%
\subsection{\harbrecht{Identification of a deeper circular tumor}}\label{subsec:test_2}
We next consider identifying a smaller, deeper circular tumor centered at $(0.045, 0.015)$ with a radius of $0.003$ (m).
With ${\harbrecht{\delta}} = 1\%$ noise in the measured data, we examine the sensitivity of the method to the initial guess.

Figures~\ref{fig:example_2_results_plot_1} and \ref{fig:example_2_results_plot_2} show tumor shape reconstruction results using different initial guesses for size and location. 
The second set in Figure~\ref{fig:example_2_results_plot_2} disregards the peak measured temperature profile, typically used to guide the initial guess location. 
These figures highlight the sensitivity of the reconstruction to the initial guess. 
Notably, a smaller cost does not always mean a more accurate reconstruction for smaller or deeper tumors, even when initialized near the exact location (see Figure~\ref{fig:example_2_results_plot_1}). 
However, based on the initial cost value, $r_0 = 0.003$ or $0.0025$ consistently yield the lowest cost, while $r_0 = 0.002$ gives the highest, suggesting the former provide better tumor size approximations. 
This aligns with expectations, as the ill-posedness of the inverse problem increases when the target shape is farther from the measurement region.

\begin{figure}[htp!]
\centering 
\resizebox{0.25\textwidth}{!}{\includegraphics{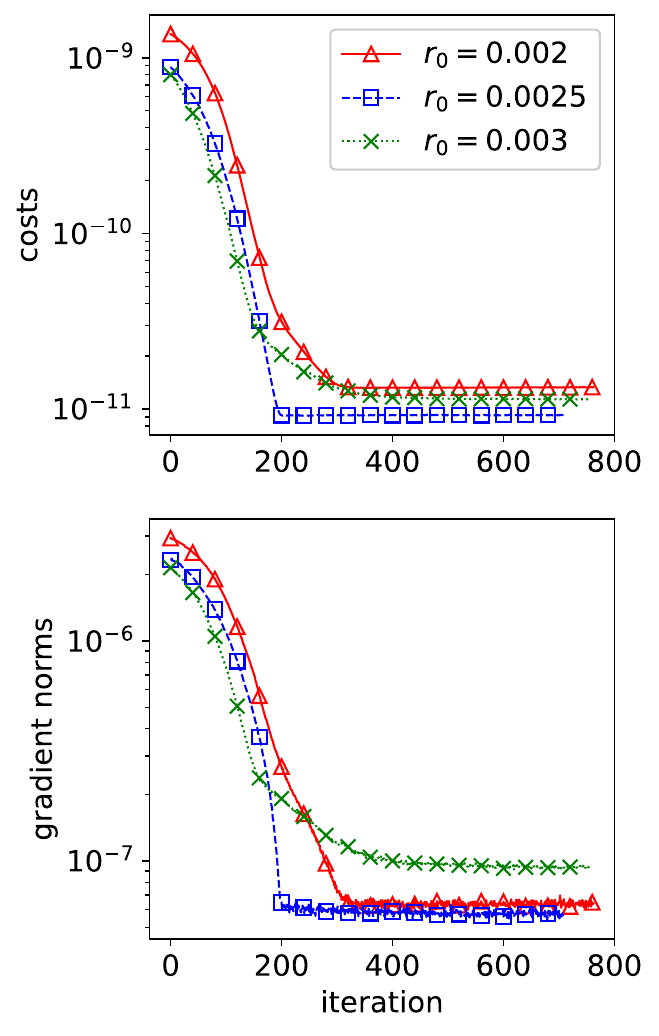}}
\resizebox{0.4\textwidth}{!}{\includegraphics{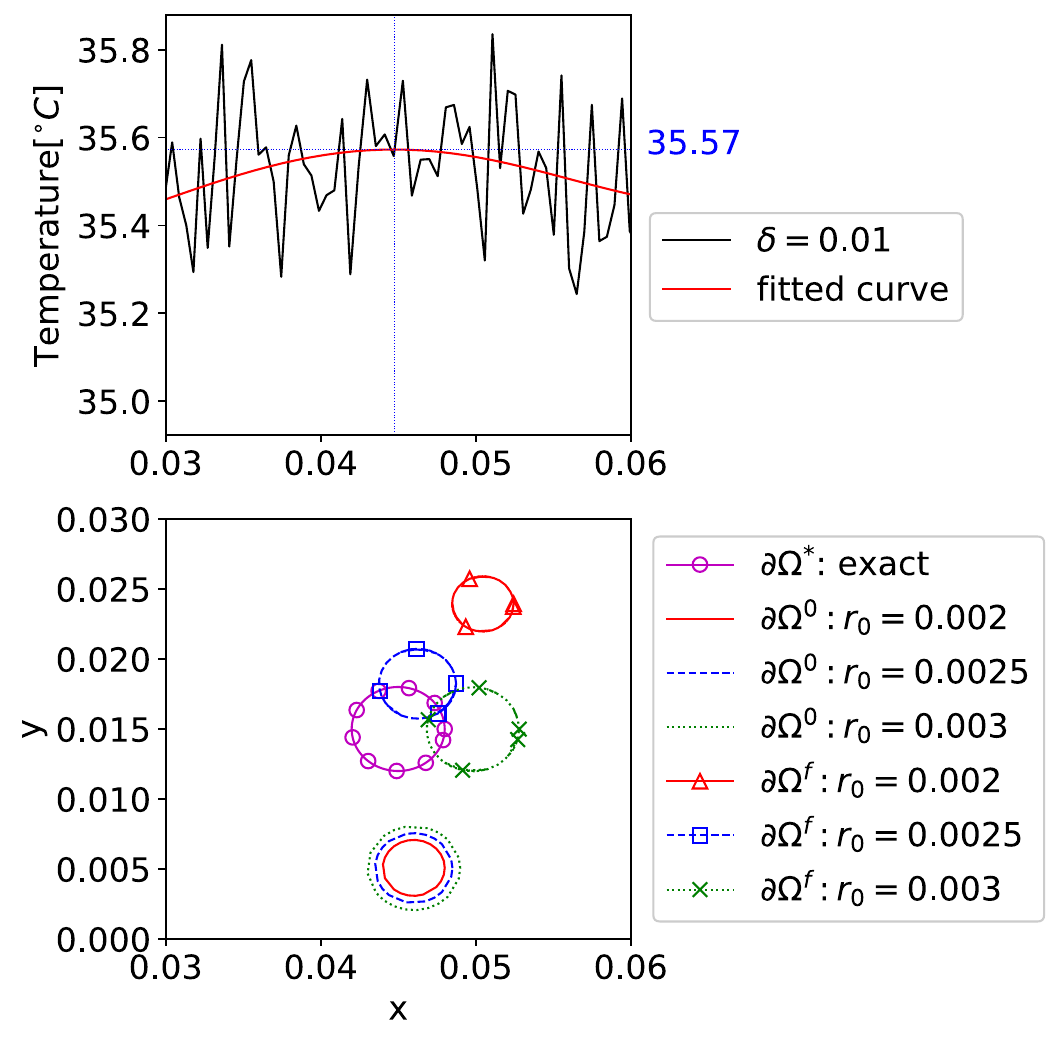}}
\caption{\harbrecht{(Left column) Cost and gradient norm histories for different sizes and locations of the initial guess; (right column) comparison of the exact, initial, and final free boundaries.}}
\label{fig:example_2_results_plot_1}
\end{figure}
\begin{figure}[htp!]
\centering 
\resizebox{0.2275\textwidth}{!}{\includegraphics{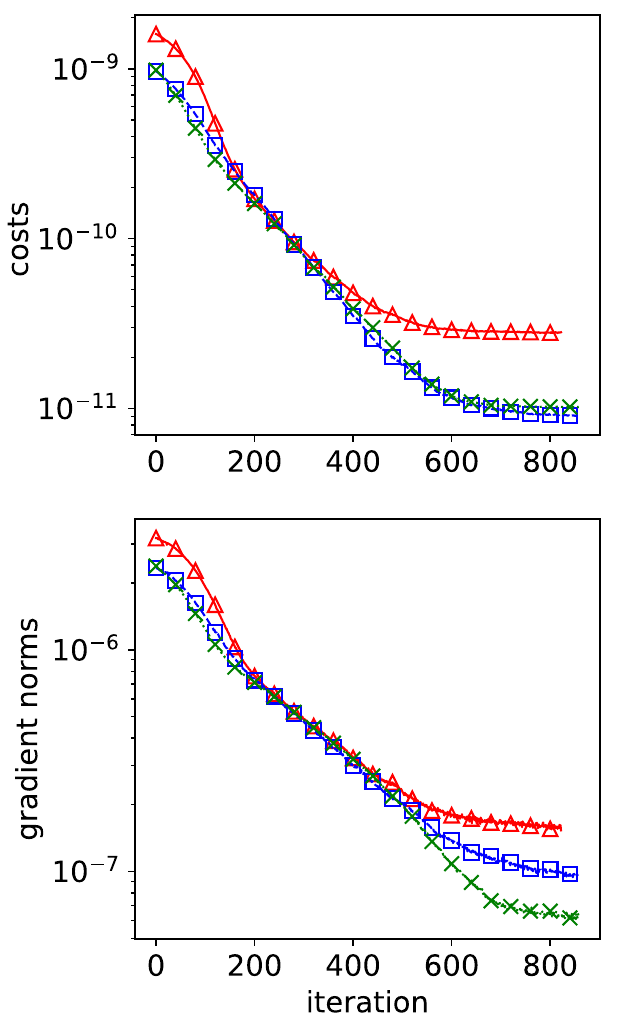}}
\resizebox{0.25\textwidth}{!}{\includegraphics{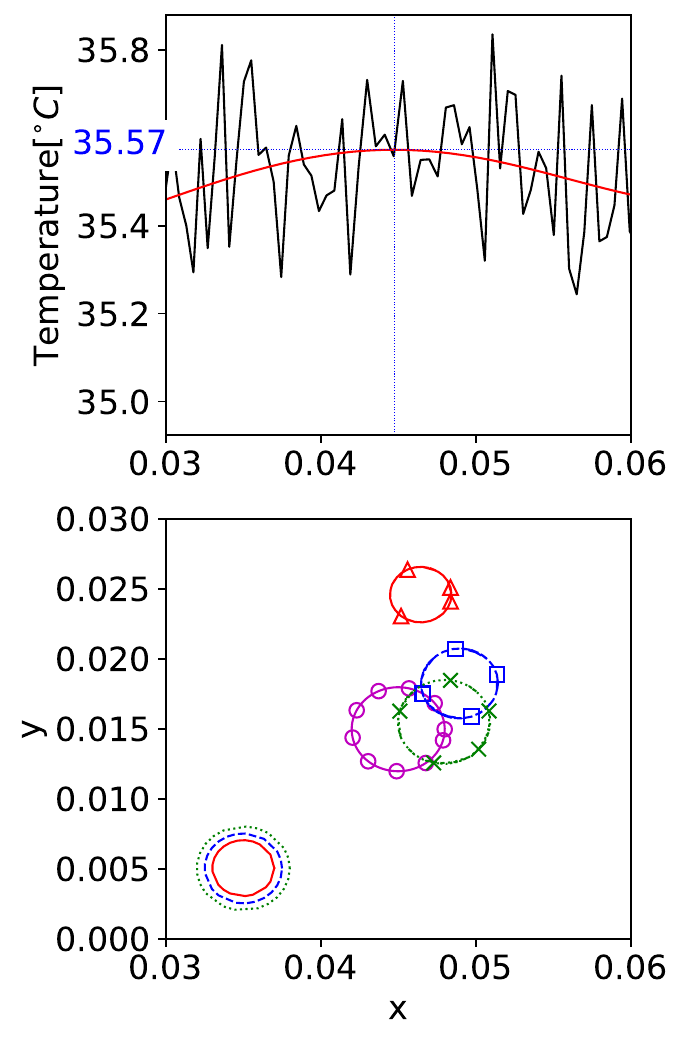}}
\hfill
\resizebox{0.2275\textwidth}{!}{\includegraphics{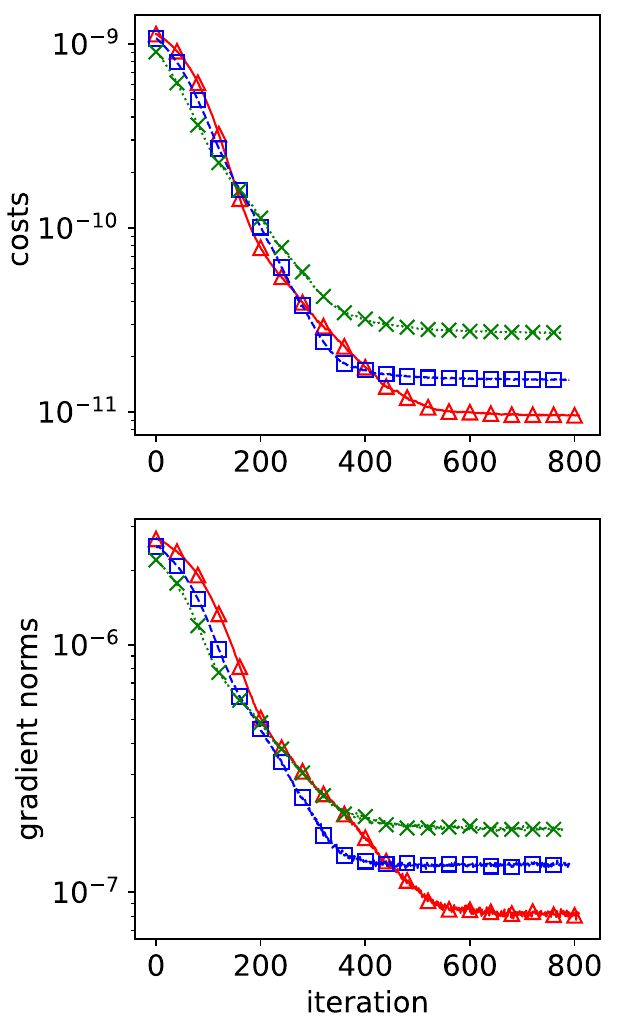}}
\resizebox{0.25\textwidth}{!}{\includegraphics{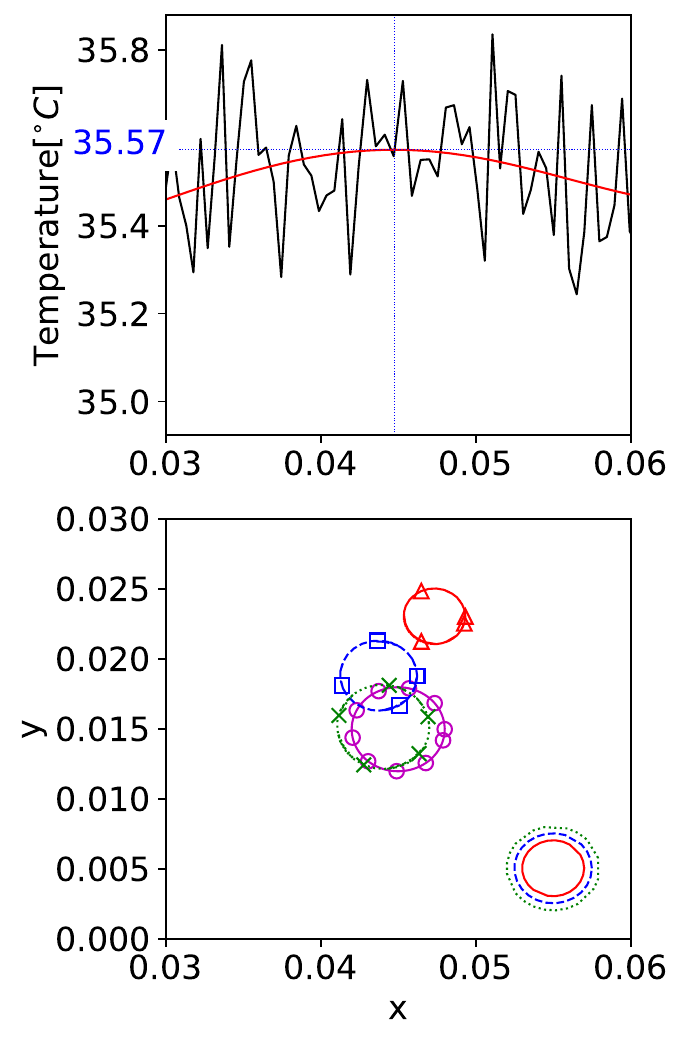}}

\caption{\harbrecht{Effect of the initial guess on the reconstruction. Refer to Figure~\ref{fig:example_2_results_plot_1} for the legend.}}
\label{fig:example_2_results_plot_2}
\end{figure}

To address the issue, we apply a penalization via a weighted volume functional and the balancing principle (see Subsection~\ref{subsec:balancing_principle}) to reduce errors and prevent overshooting. 
Figure~\ref{fig:example_3} shows improved shape approximation and reduced reconstruction errors for $r_0 = 0.0025$ and $r_0 = 0.003$, along with the histories of cost, gradient norms, and $\rho$ values. 
Final cost values alone can be misleading for selecting the best approximation, but as seen in Figure~\ref{fig:example_3_temperature_profile}, $r_0 = 0.003$ yields the best result, with the temperature profile closely matching the noisy data. 
This confirms that incorporating the measured profile ensures accurate reconstruction, even with noise.
%
%
%
%
\begin{figure}[htp!]
\centering  
\resizebox{0.9\textwidth}{!}{\includegraphics{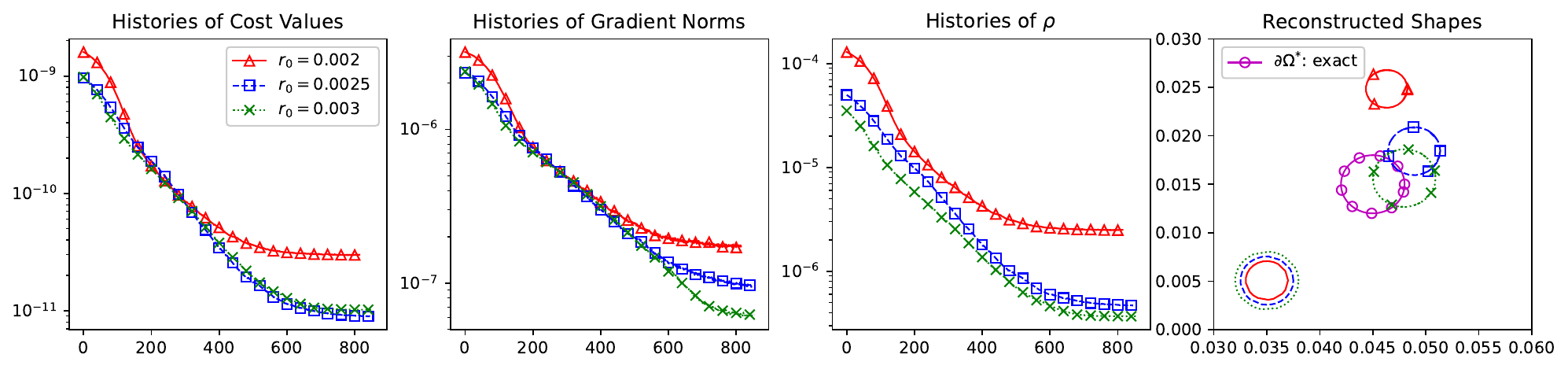}}\\
\resizebox{0.9\textwidth}{!}{\includegraphics{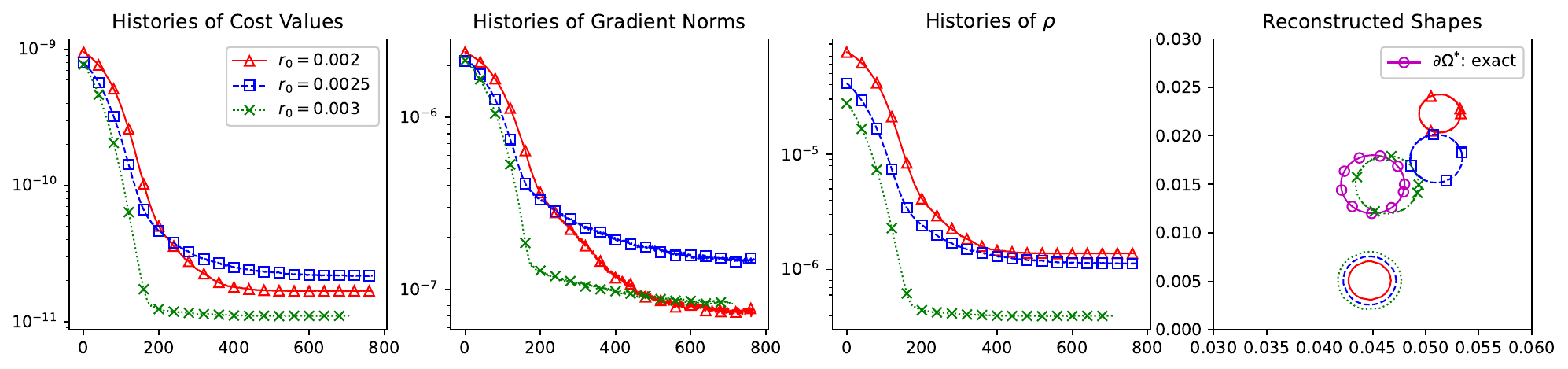}}\\
\resizebox{0.9\textwidth}{!}{\includegraphics{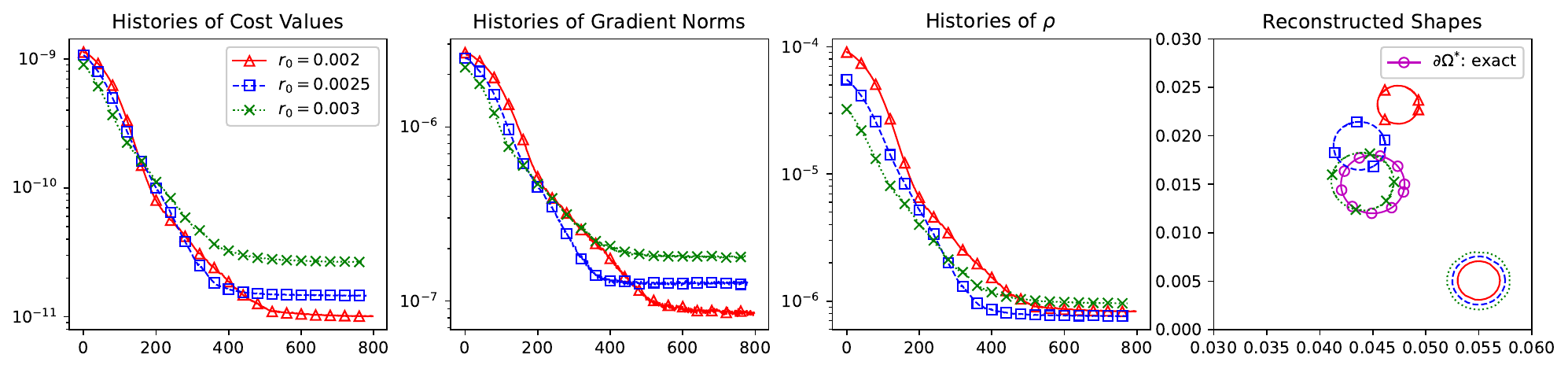}}
\caption{Reconstruction results with volume penalization coupled with the balancing principle. Refer to Figure~\ref{fig:example_2_results_plot_1} for the legend.}
\label{fig:example_3}
\end{figure}
\begin{figure}[htp!]
\centering    
\resizebox{0.25\textwidth}{!}{\includegraphics{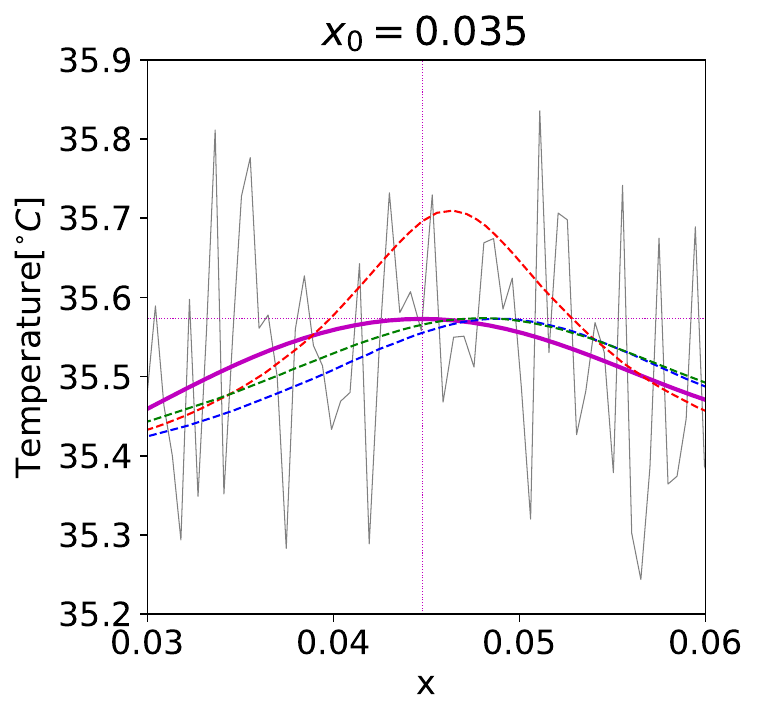}}
\resizebox{0.25\textwidth}{!}{\includegraphics{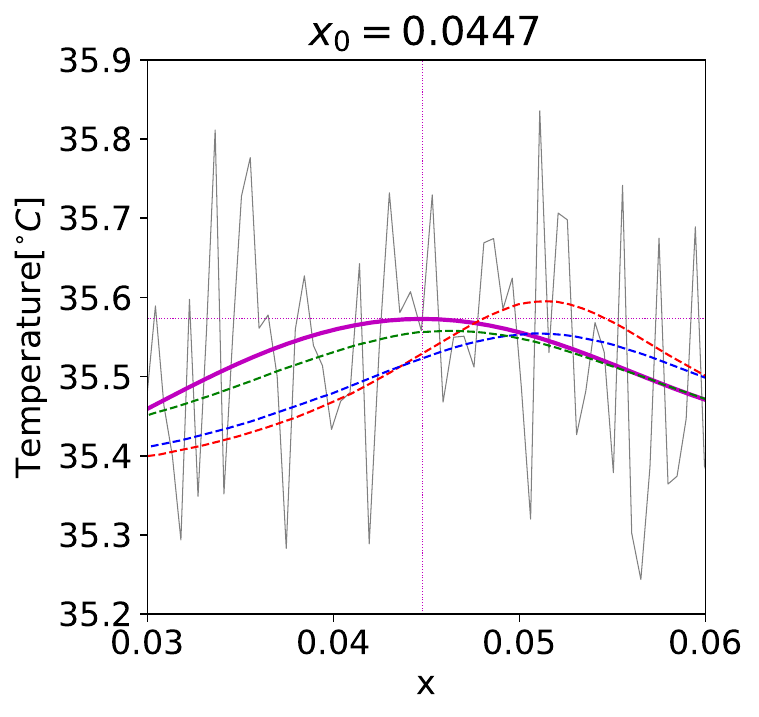}}
\resizebox{0.25\textwidth}{!}{\includegraphics{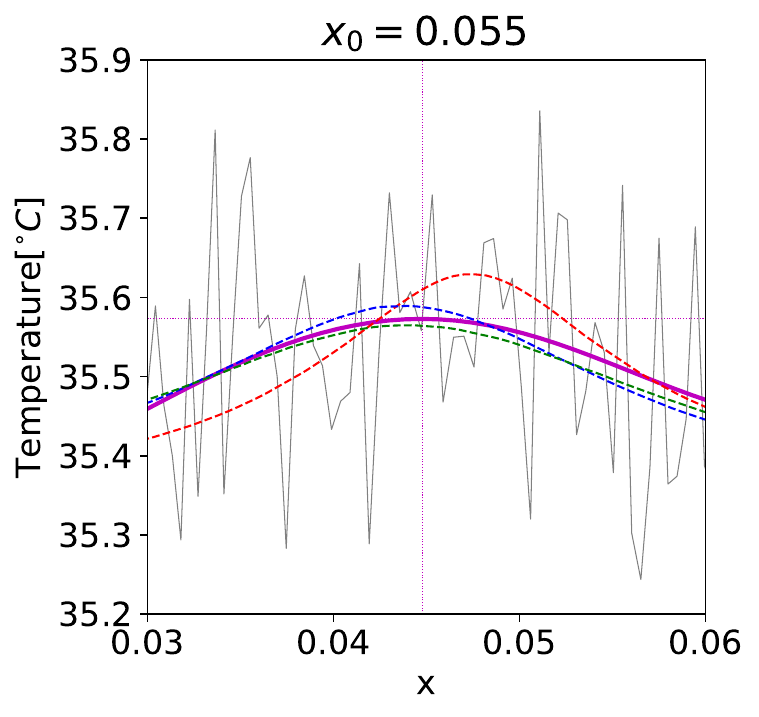}}\\[0.5em]
\resizebox{0.85\textwidth}{!}{\includegraphics{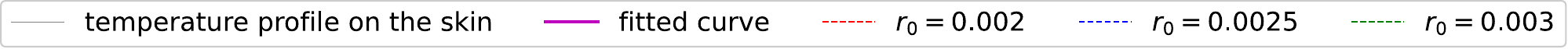}}
\caption{Comparison of the measured noisy skin temperature profile, the fitted curve, and the temperature profiles for the computed shapes in Figure~\ref{fig:example_3}.}
\label{fig:example_3_temperature_profile}
\end{figure}

\harbrecht{Based on our observations from numerous experiments (not presented here), the balancing principle tends to yield good results when the tumor radii are known in advance, although such prior knowledge is rarely available in practice.
Nevertheless, our algorithm can still achieve reasonable reconstruction by setting $c_b$ close to one. 
In the next subsection, we present results using $c_b = 1$ and $\rho \in [1 \times 10^{-5}, 5 \times 10^{-5}]$.}

\subsection{\harbrecht{Identification of non-trivial shapes}}\label{subsec:test_2_additional}
We test the proposed method on non-elliptical tumors, beginning with an initial guess that is smaller and non-circular compared to the actual tumor. 
The aim is to demonstrate accurate estimation of both tumor size and location, regardless of the size and position of the initial guess.

Figure~\ref{fig:example3b_non_trivial_shapes} presents reconstruction results with $2\%$ noise added to skin temperature measurements for four tumor shapes: one convex and three non-convex. 
The top row shows the measured data, fitted curve, and reconstructed skin temperature, while the bottom row compares the exact, initial, and reconstructed tumor shapes.

The results indicate that reconstructing the concavity of the tumor remains challenging.
This is primarily due to the severity of the ill-posedness of the problem.
Nonetheless, the method provides a reliable estimate of tumor location and size, even in the presence of noise (see third column in Figure~\ref{fig:example3b_non_trivial_shapes_summary}).

Further insights into the identification of the two cases shown in Figure~\ref{fig:example3b_non_trivial_shapes} are presented in Figure~\ref{fig:example3b_non_trivial_shapes_summary}, which displays the real and imaginary parts of the solutions, the shape evolution, and the modified-cost histories.
In fact, for these examples, the original algorithm was modified to use the value of the combined cost $J + J_{LS}$ (see the last plot in Figure~\ref{fig:example3b_non_trivial_shapes_summary}) as the termination criterion; that is, the algorithm stops when the decrease in the modified cost is sufficiently small.

\begin{figure}[htp!]
\centering  
\resizebox{0.7\textwidth}{!}{\includegraphics{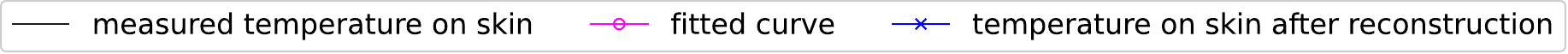}}\\
\resizebox{0.24\textwidth}{!}{\includegraphics{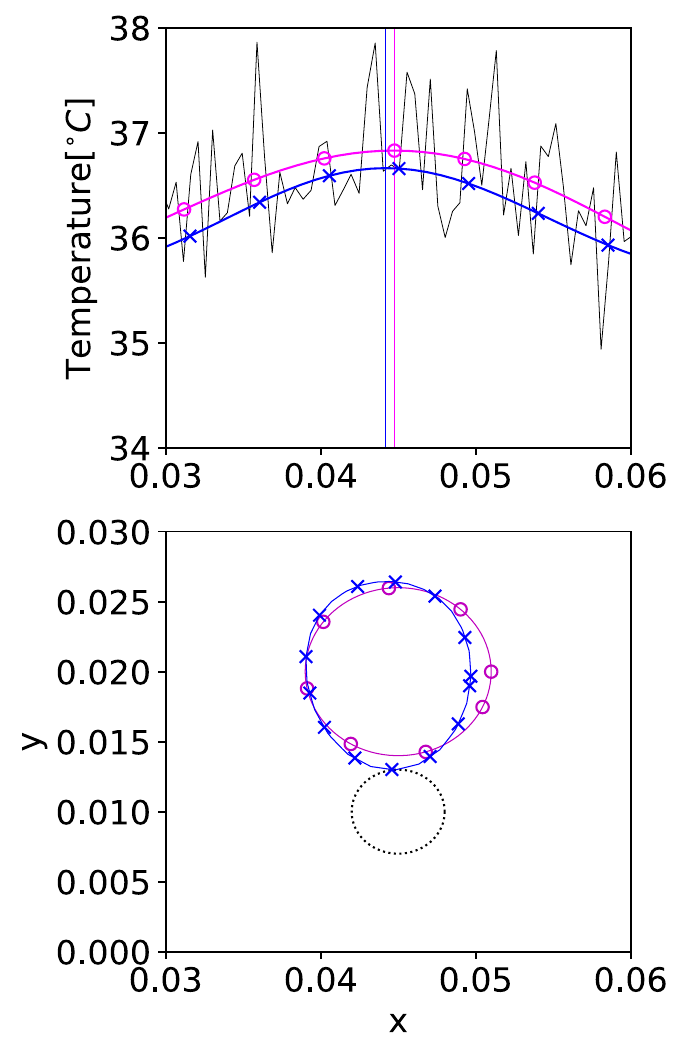}}
\resizebox{0.24\textwidth}{!}{\includegraphics{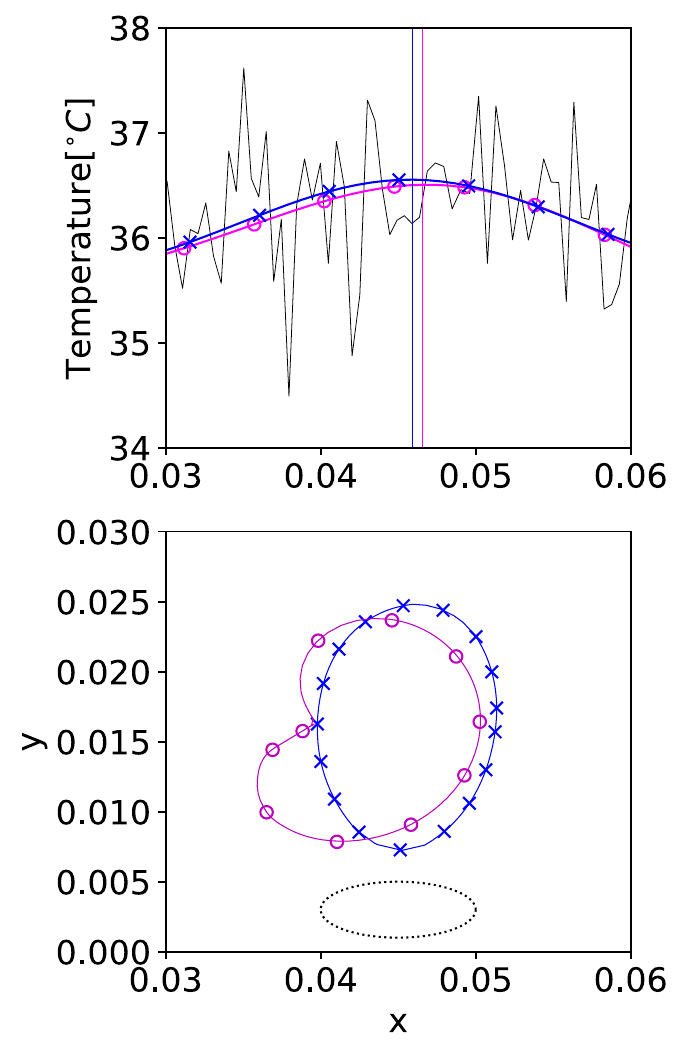}}
\resizebox{0.24\textwidth}{!}{\includegraphics{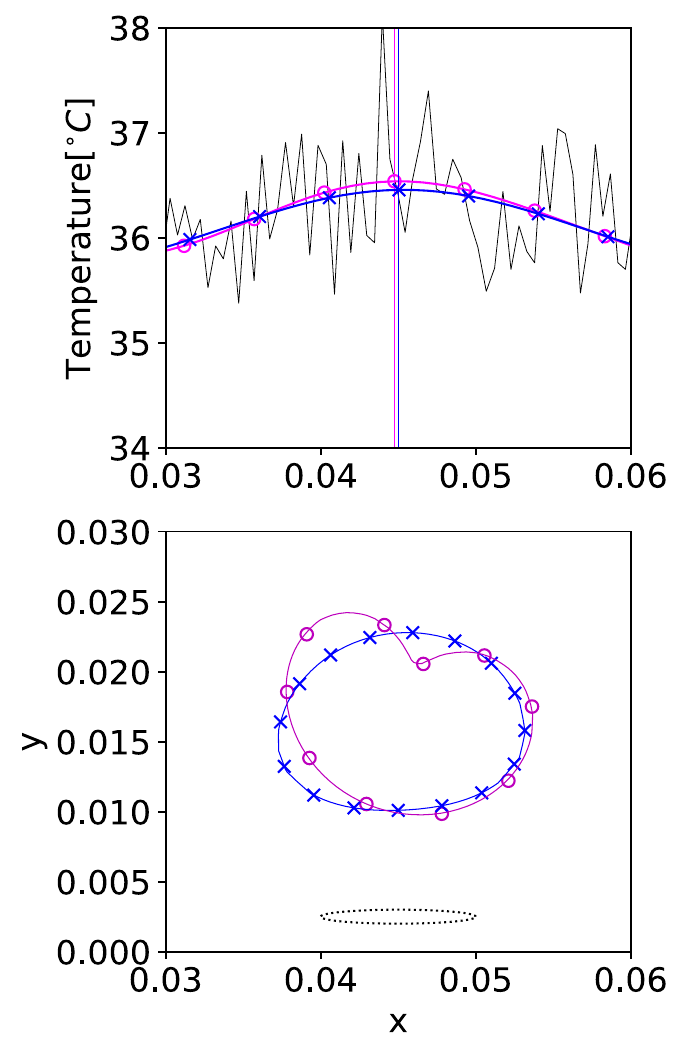}}
\resizebox{0.24\textwidth}{!}{\includegraphics{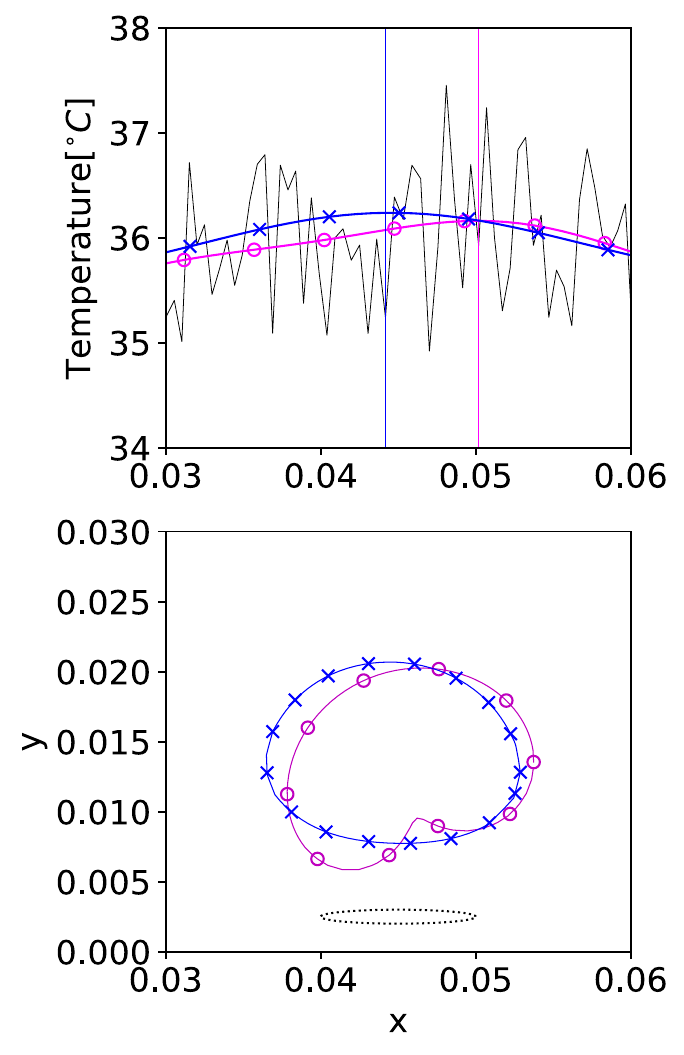}}\\ 
\resizebox{0.6\textwidth}{!}{\includegraphics{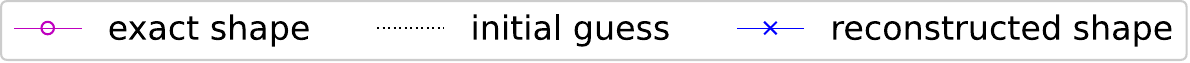}}
\caption{Reconstruction results using volume penalization with fixed $\rho$ and $c_{b} = 1$.}
\label{fig:example3b_non_trivial_shapes}
\end{figure}
\begin{figure}[htp!]
\centering  
\resizebox{0.85\textwidth}{!}{\includegraphics{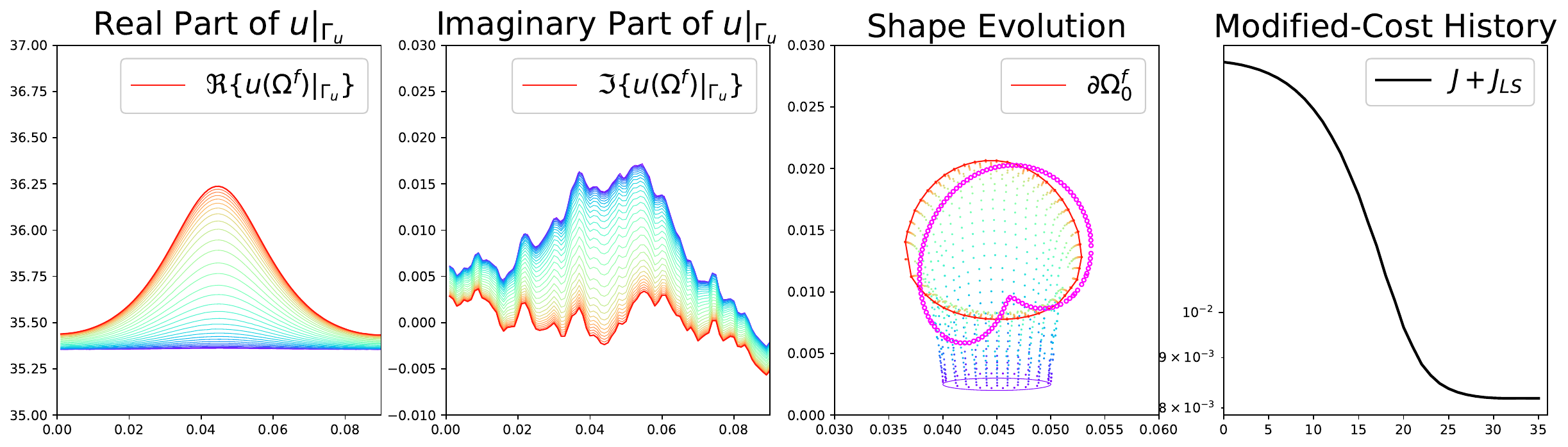}}\\ 
\resizebox{0.85\textwidth}{!}{\includegraphics{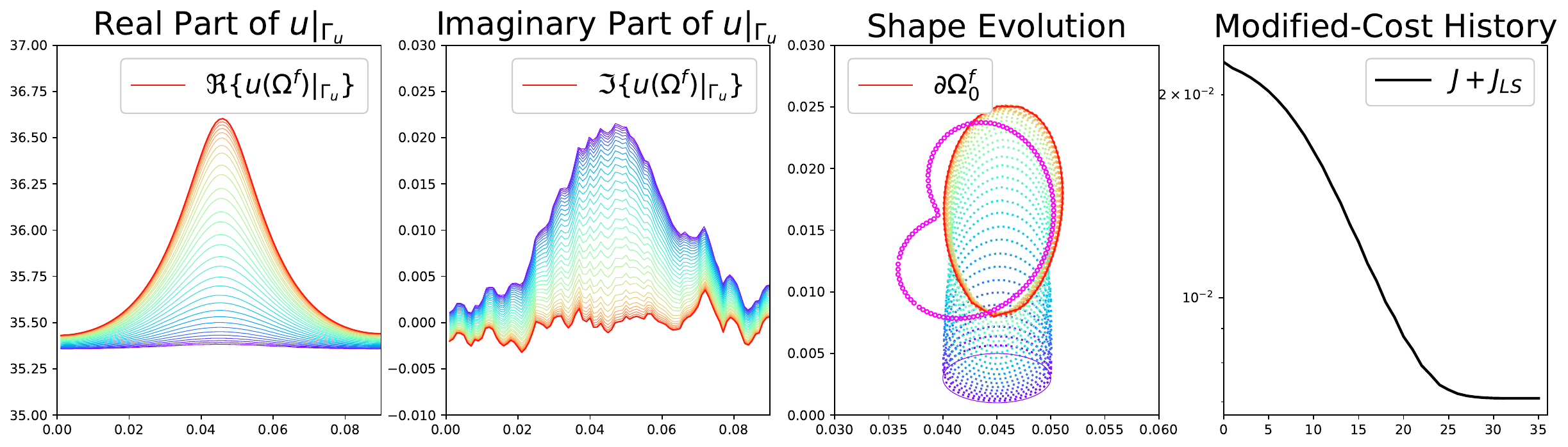}}
\caption{Real and imaginary parts of solutions, shape evolution, and modified-cost histories.}
\label{fig:example3b_non_trivial_shapes_summary}
\end{figure}
%
%
%
\subsection{Numerical example in 3D}\label{subsec:test_3d}
We also test our numerical method in the case of three spatial dimensions. 
We consider the real shape of the biological tissue of a woman's breast (see Figure~\ref{fig:3d_illustration} for an illustration), which is approximated by half of an ellipsoid with semi-axes of dimensions $R_{1} = 0.07$, $R_{2} = 0.06$, and $R_{3} = 0.07$ (m), while the shape of the tumor is approximated by an ellipsoid with semi-axes of dimensions $r_{1}^{\ast} = 0.015$ (m), $r_{2}^{\ast} = 0.015$ (m), and $r_{3}^{\ast} = 0.0125$ (m), and center position $(x_{1}^{\ast}, x_{2}^{\ast}, x_{3}^{\ast}) = (0.0175, 0.0125, 0.035)$ (m). 
The choice of numerical values for the coefficients is as given in subsection~\ref{subsec:forward_problem}, and the computational setup is essentially the same as in the 2D case.

Figure~\ref{fig:3d_setup} shows the tumor's size, location, and temperature distribution on the skin surface and inside the domain. 
The plots show an expected increase in skin temperature around the tumor, as expected. 
\alert{Figure~\ref{fig:3d_temperature_distribution} shows the exact and noisy ($1\%$ noise) temperature distributions on the skin surface.
The tumor position can be inferred from both cases, but its size is difficult to estimate due to the shape of the breast.
Even with noise, the temperature data still provides a good indication of the tumor location, especially when viewed in the $xy$-plane.}

Using this information, we initialized our approximation with an elliptical tumor shape having radii $(r_{1}^{0}, r_{2}^{0}, r_{3}^{0}) = (\varsigma r_{1}^{\ast}, \varsigma r_{2}^{\ast}, \varsigma r_{3}^{\ast})$, where $\varsigma \in \{0.8, 0.9, 1\}$, as shown in Figure~\ref{fig:3d_results} (light colors). 
The exact shape is shown in yellow, and the identified shapes are in magenta.
The left panel presents results without volume penalization and the balancing principle \eqref{eq:balancing_principle}, while the right panel includes this regularization strategy designed to select the best tumor shape approximation based on the cost history, as demonstrated previously.
Note that the computed temperature profile on the skin (see last row of Figure~\ref{fig:3d_results}) does not provide sufficient information to judge reconstruction accuracy or to determine which initial guess---small or large---yields better results, in contrast to the 2D case (Figure~\ref{fig:example_3_temperature_profile}).
Nonetheless, Figure~\ref{fig:3d_histories_of_values} shows that the lowest initial cost value always corresponds to $\varsigma = 1$, regardless of regularization. The lines labeled $\partial\varOmega_{0}^{(1)}$, $\partial\varOmega_{0}^{(2)}$, and $\partial\varOmega_{0}^{(3)}$ correspond to $\varsigma = 0.8$, $0.9$, and $1$, respectively.
Moreover, in the bottom row where regularization is applied, the final cost value is smallest for initializations with $\varsigma = 0.9$ or $\varsigma = 1$, with $\varsigma = 1$ achieving the overall lowest cost at convergence.
Figure~\ref{fig:3d_mesh_profiles} illustrates the mesh profiles for both the exact and recovered tumor shapes when $\varsigma = 1$.
Consistent with observations from the 2D case, the measured temperature profiles on the skin, both before and after the approximation, allow for accurate tumor reconstruction even in the presence of measurement noise. 
This effectiveness extends to the 3D case, confirming the reliability of our strategy for selecting the best approximation.

\begin{figure}[htp!]
\centering   
\tdplotsetmaincoords{60}{110} 

\pgfmathsetmacro{\radius}{1}

\begin{tikzpicture}[scale=1.75,tdplot_main_coords]

\draw[thin,->] (0,0,0) -- (1,0,0) node[anchor=north]{$x_{1}$};
\draw[thin,->] (0,0,0) -- (0,1,0) node[anchor=north west]{$x_{2}$};
\draw[thin,->] (0,0,0) -- (0,0,1.1) node[anchor=south]{$x_{3}$};

\shade[ball color=blue!10!white,opacity=0.2] 
    (1cm,0) arc (0:-180:1cm and 5mm) 
    arc (180:0:1cm and 1cm);

\draw[dotted] (0,\radius,0) arc (90:270:\radius);
\draw[thin] (0,-\radius,0) arc (90:270:-\radius);

\shade[ball color=gray!10!white,opacity=1] 
    (7mm,5mm) arc (0:-360:3mm and 2mm) 
    arc (360:0:2mm and 1mm);
\shade[ball color=gray!10!white,opacity=0.8] 
    (7mm,5mm) arc (0:-360:3mm and 2mm) 
    arc (360:0:2mm and 1mm);
\draw[thin, color=red] (4mm,5mm) ellipse (3mm and 2mm);

\node[above, color=red] at (17mm,3.5mm) {$\partial \varOmega_0$};
\draw[thin, ->] (7mm,5mm) to (15mm,5mm); 

\node[above] at (-12mm,8mm) {$\varGamma_{u}$};
\draw[thin, ->] (-7mm,7mm) to (-10.5mm,8.5mm); 

\node[above] at (-12mm,-6mm) {$\varGamma_{b}$};
\draw[thin, ->] (-5mm,0mm) to (-11mm,-3.5mm); 

\node at (-3mm,7mm) {$\varOmega_{1}$};  
\node at (2.75mm,5.25mm) {$\varOmega_{0}$};  

\begin{scope}[shift={(2.5cm,5mm)}]
    \shade[ball color=gray!10!white,opacity=0.8] 
        (0,0) ellipse (6mm and 3mm);

    \draw[thick, color=red] (0,0) ellipse (6mm and 3mm);
    \draw[dotted] (0,0) ellipse (6mm and 1.5mm);
    \draw[dotted] (1.5mm,0) arc (0:360:1.5mm and -3mm);

    \draw[->] (0,0) -- (6mm,0) node[below right]{$r_1$};
    \draw[->] (0,0) -- (0,3mm) node[above]{$r_3$};
    \draw[->] (0,0) -- (-2.5mm,-2.75mm) node[below]{$r_2$};

    \node[below] at (3mm,-6mm) {$(x_{1}^{\ast}, x_{2}^{\ast}, x_{3}^{\ast})$};
    \draw[thin,<-] (3mm,-6mm) to (0,0);
\end{scope}

\draw[thin,->] (4mm,5mm) to[out=30,in=150] (2.5cm,5mm);
\draw[dashed,->] (0,0,0) -- (1,0,0) node[anchor=north]{$x_{1}$};

\end{tikzpicture}
\caption{Breast tissue with tumor}
\label{fig:3d_illustration}
\end{figure}
%
%
%
%
\begin{figure}[htp!]
\centering   
\resizebox{0.24\textwidth}{!}{\includegraphics{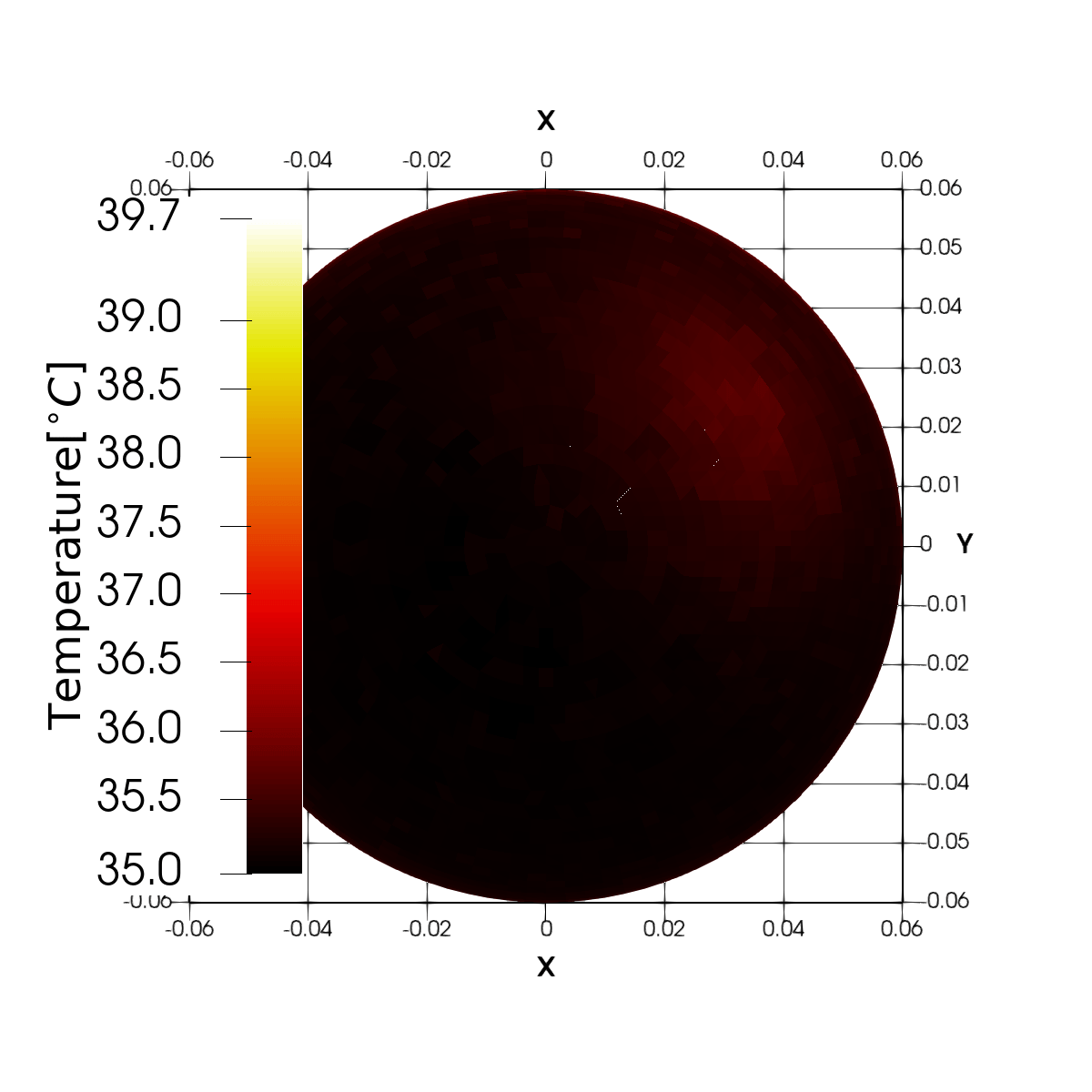}}
\resizebox{0.24\textwidth}{!}{\includegraphics{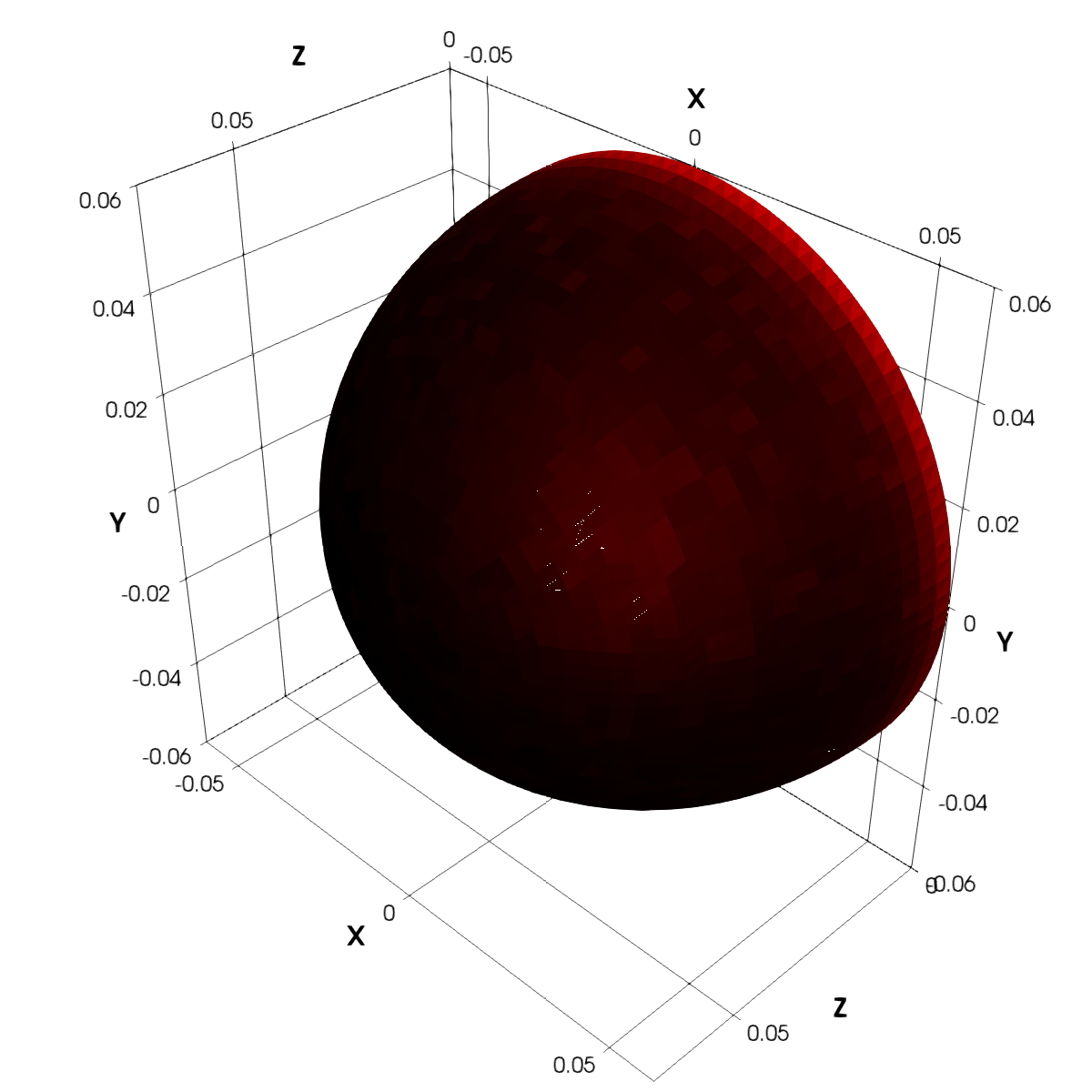}}
\resizebox{0.24\textwidth}{!}{\includegraphics{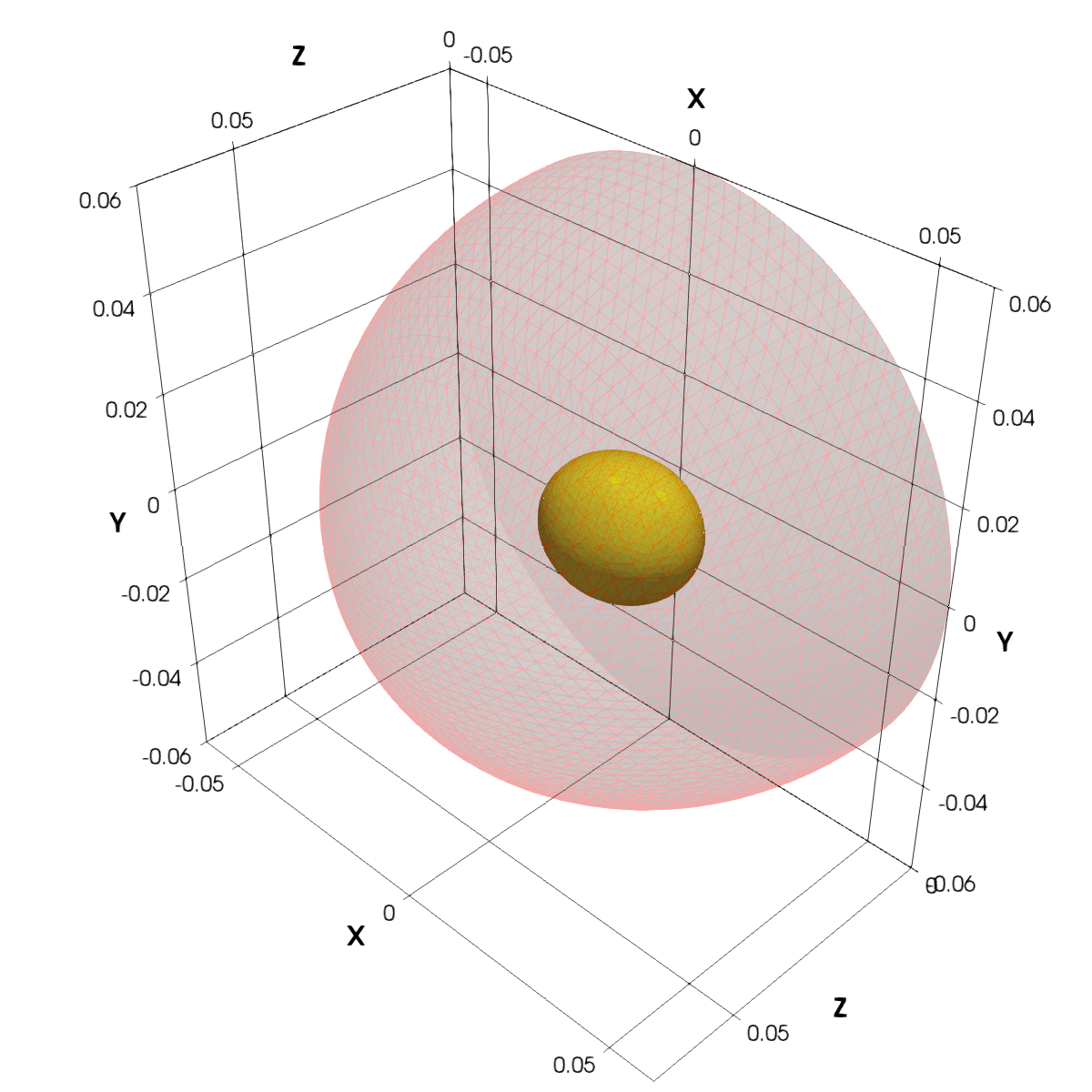}}
\resizebox{0.24\textwidth}{!}{\includegraphics{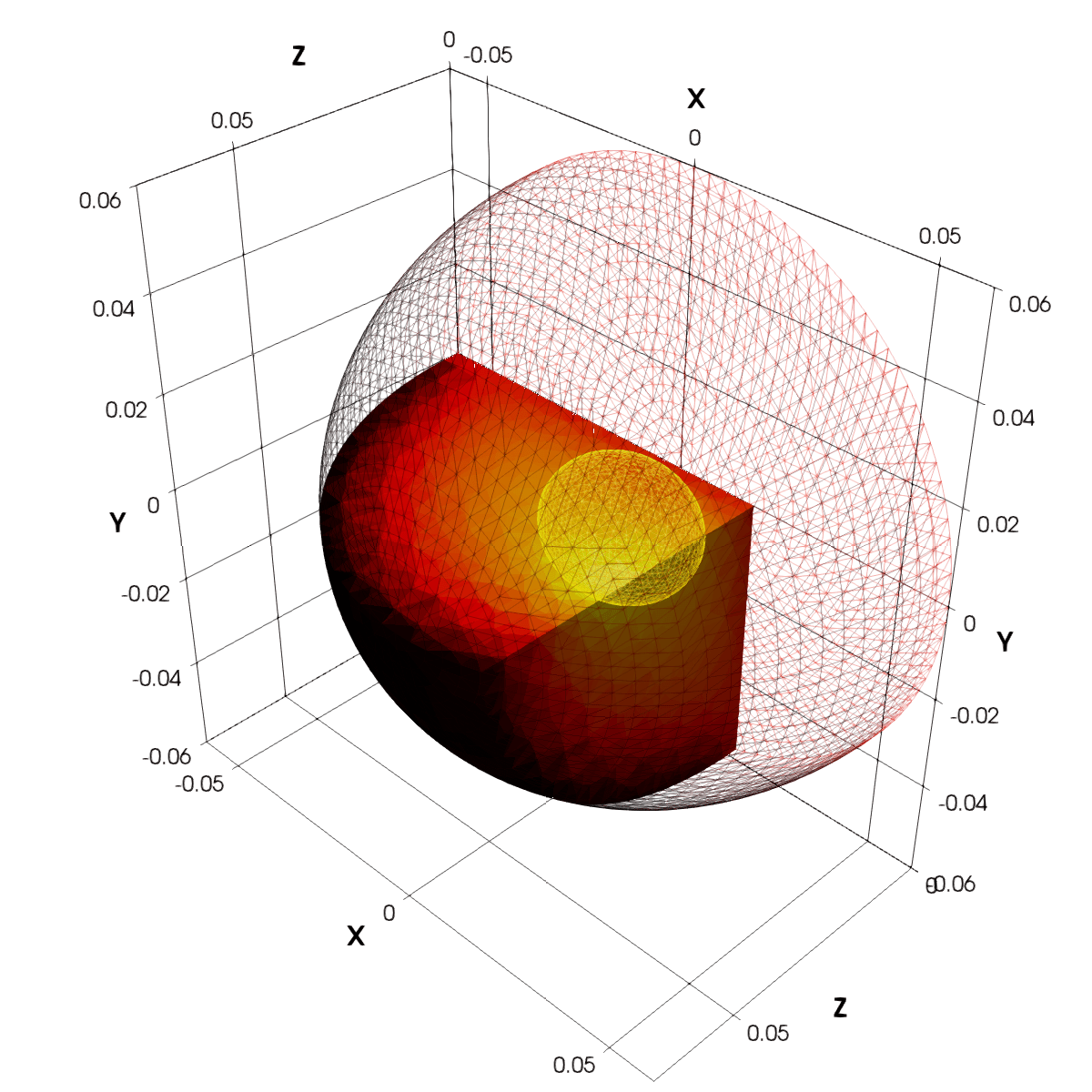}} 
\caption{Exact tumor location and temperature distribution}
\label{fig:3d_setup}
\end{figure}
\begin{figure}[htp!]
\centering   
\resizebox{0.25\textwidth}{!}{\includegraphics{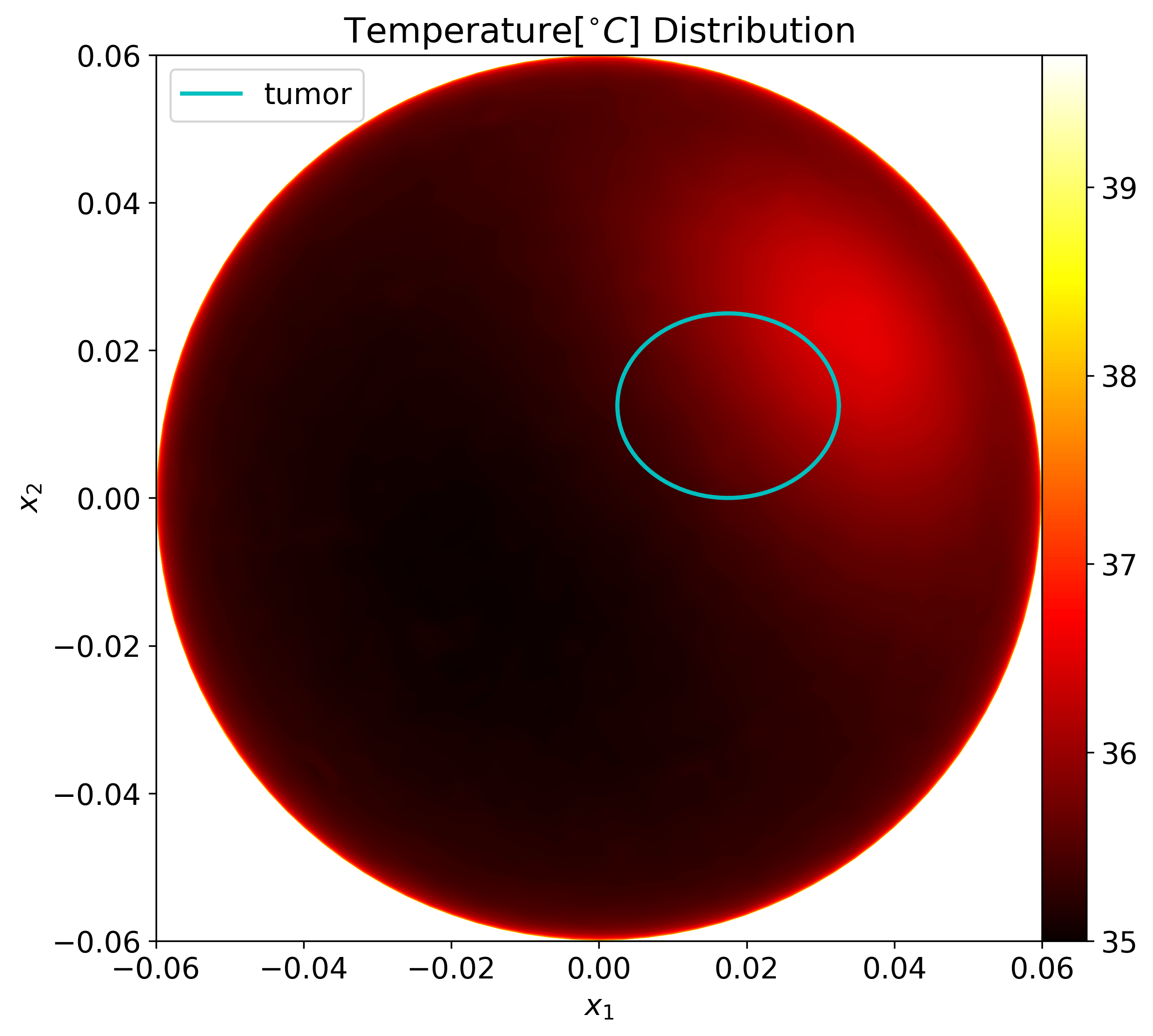}} 
\resizebox{0.25\textwidth}{!}{\includegraphics{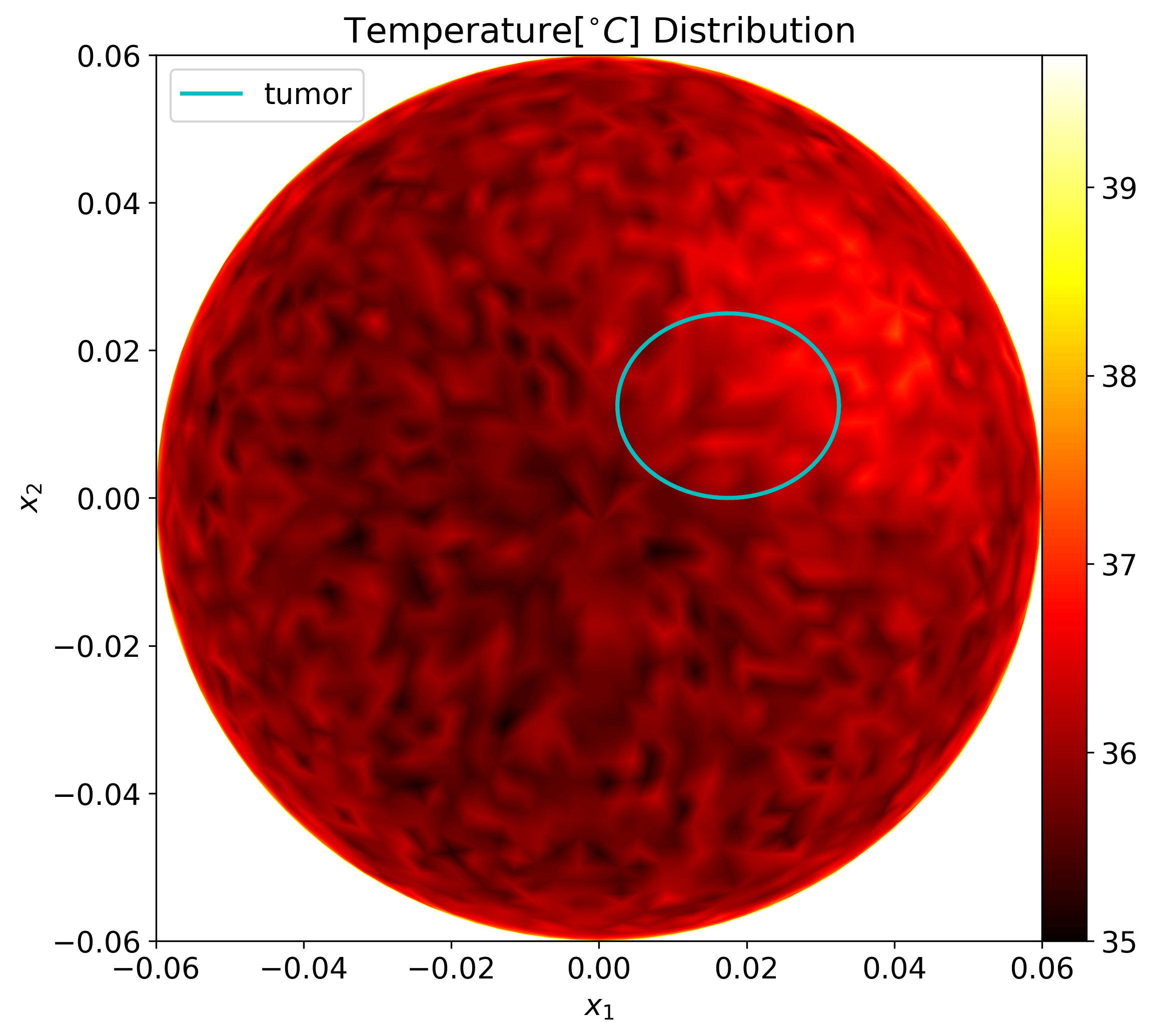}} 
\caption{Exact and noisy temperature distributions on the skin}
\label{fig:3d_temperature_distribution}
\end{figure}
\begin{figure}[htp!]
\centering   
\resizebox{0.15\textwidth}{!}{\includegraphics{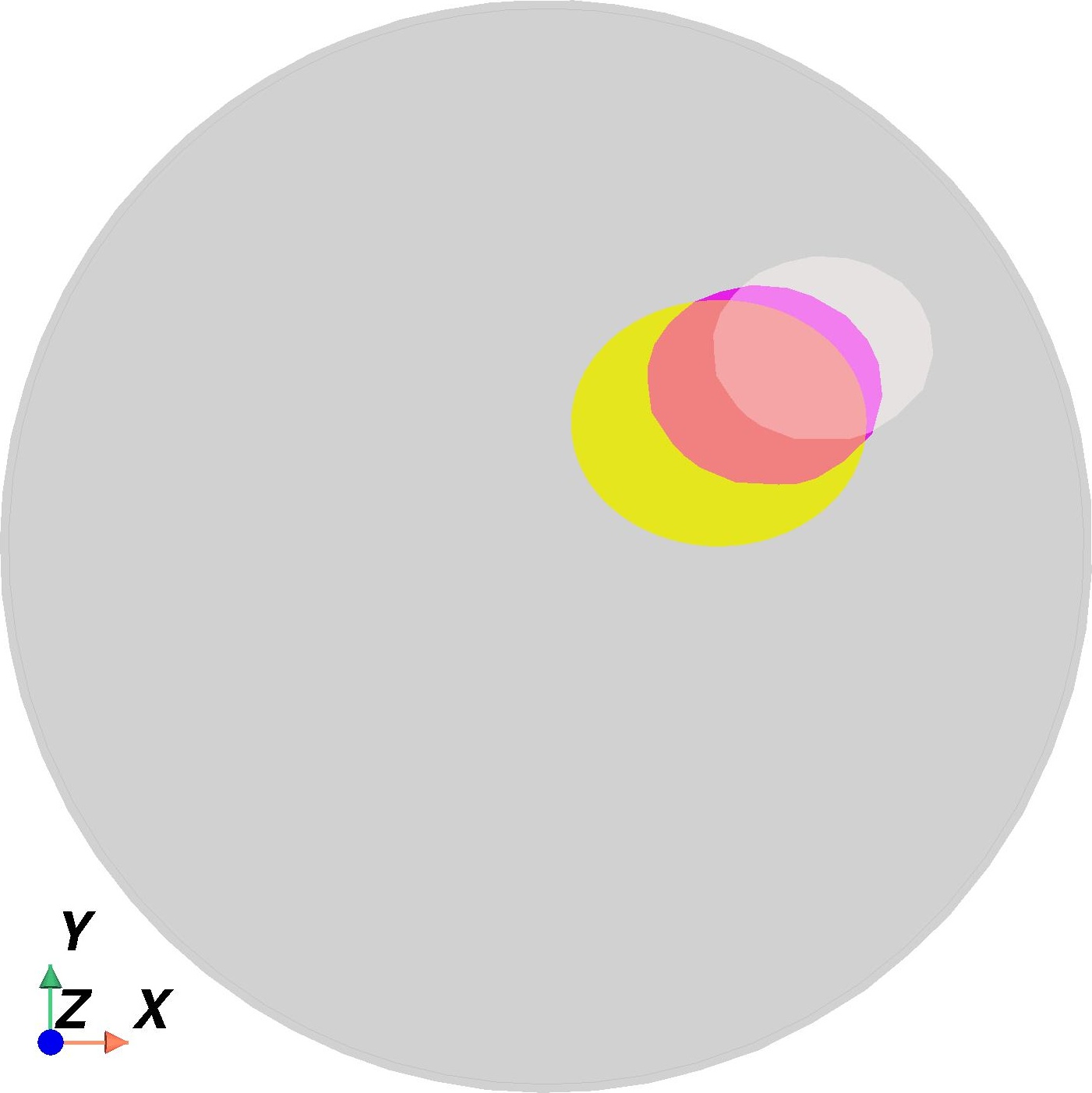}}  
\resizebox{0.15\textwidth}{!}{\includegraphics{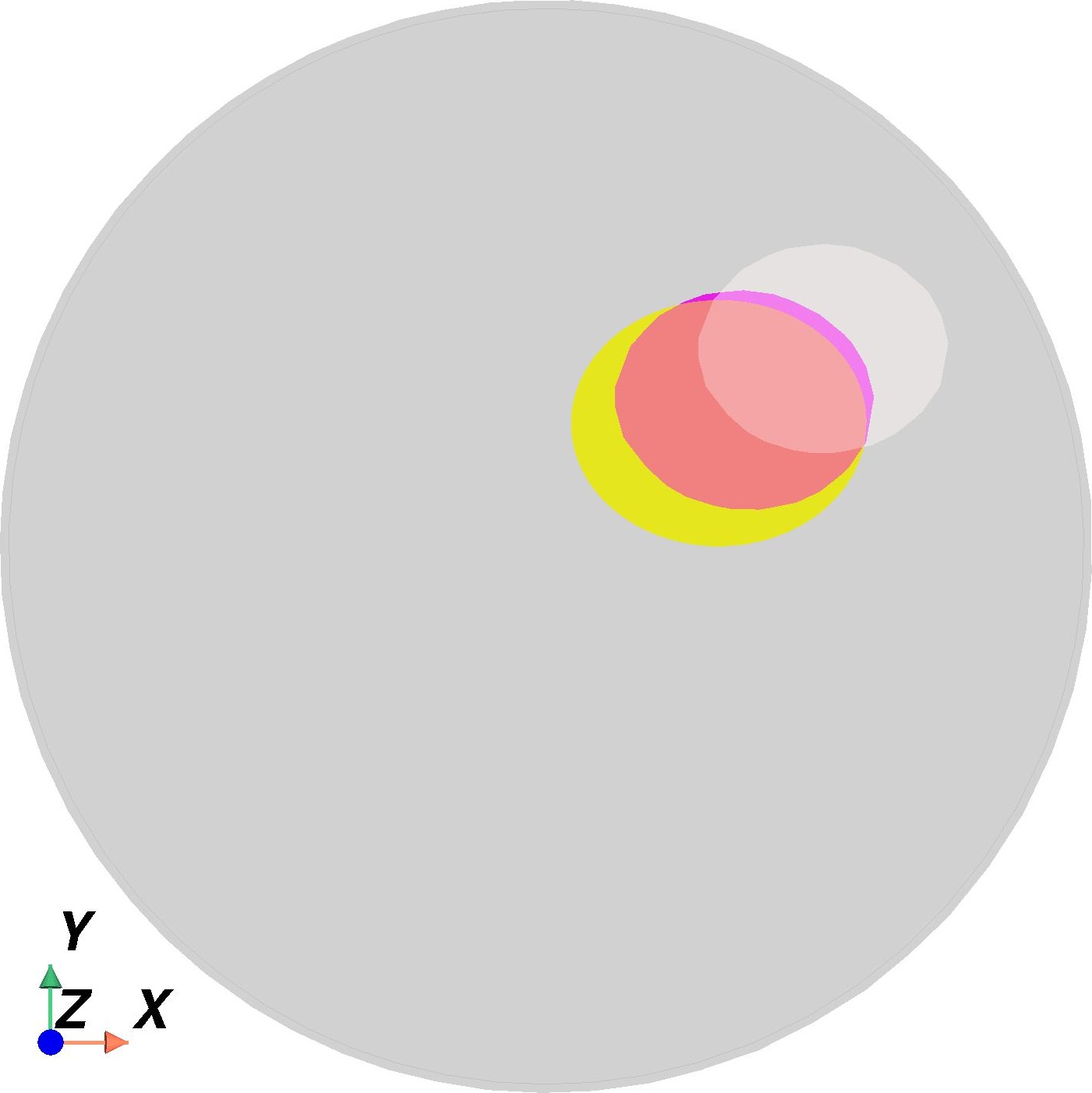}}   
\resizebox{0.15\textwidth}{!}{\includegraphics{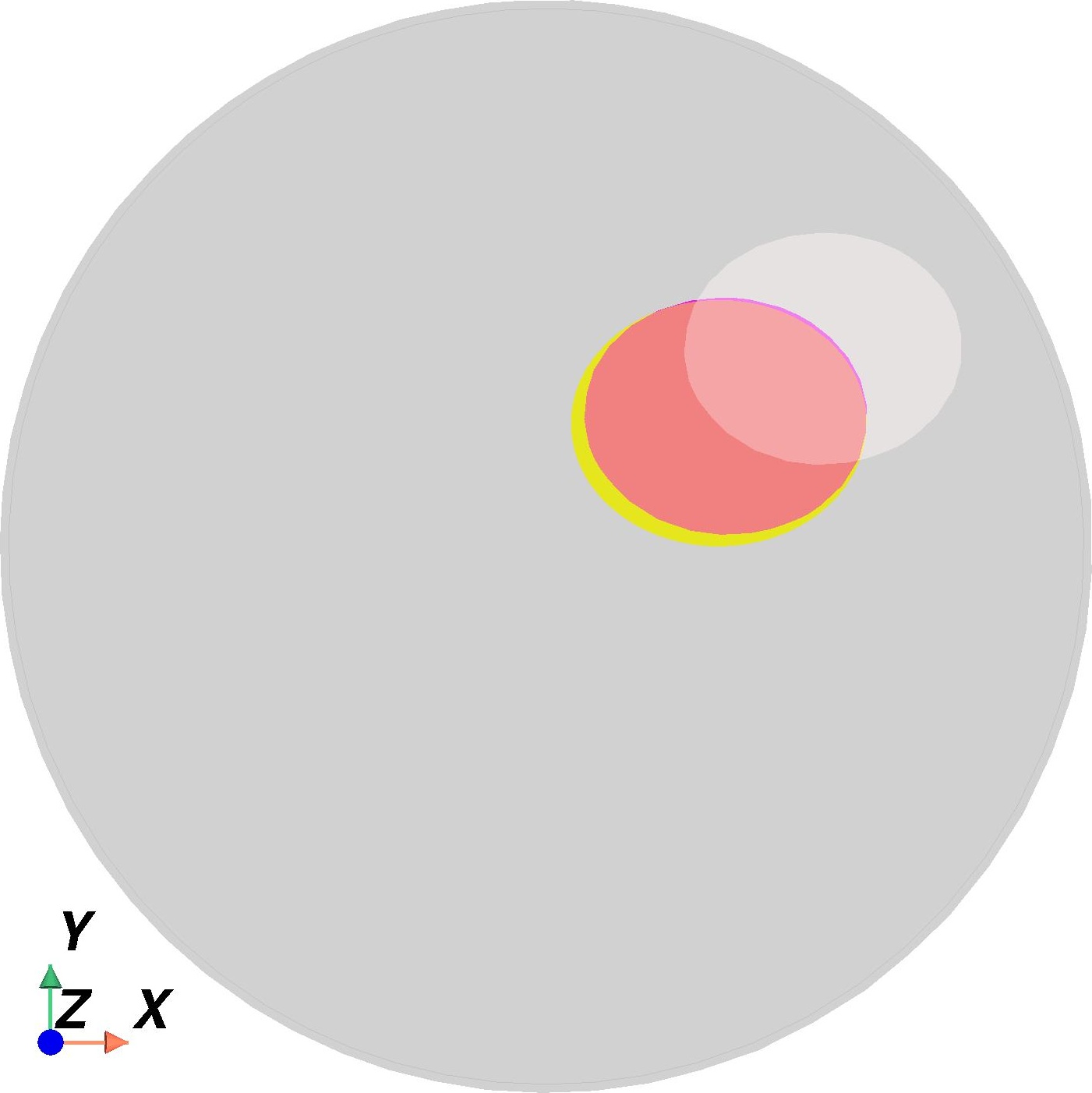}} \quad
\resizebox{0.15\textwidth}{!}{\includegraphics{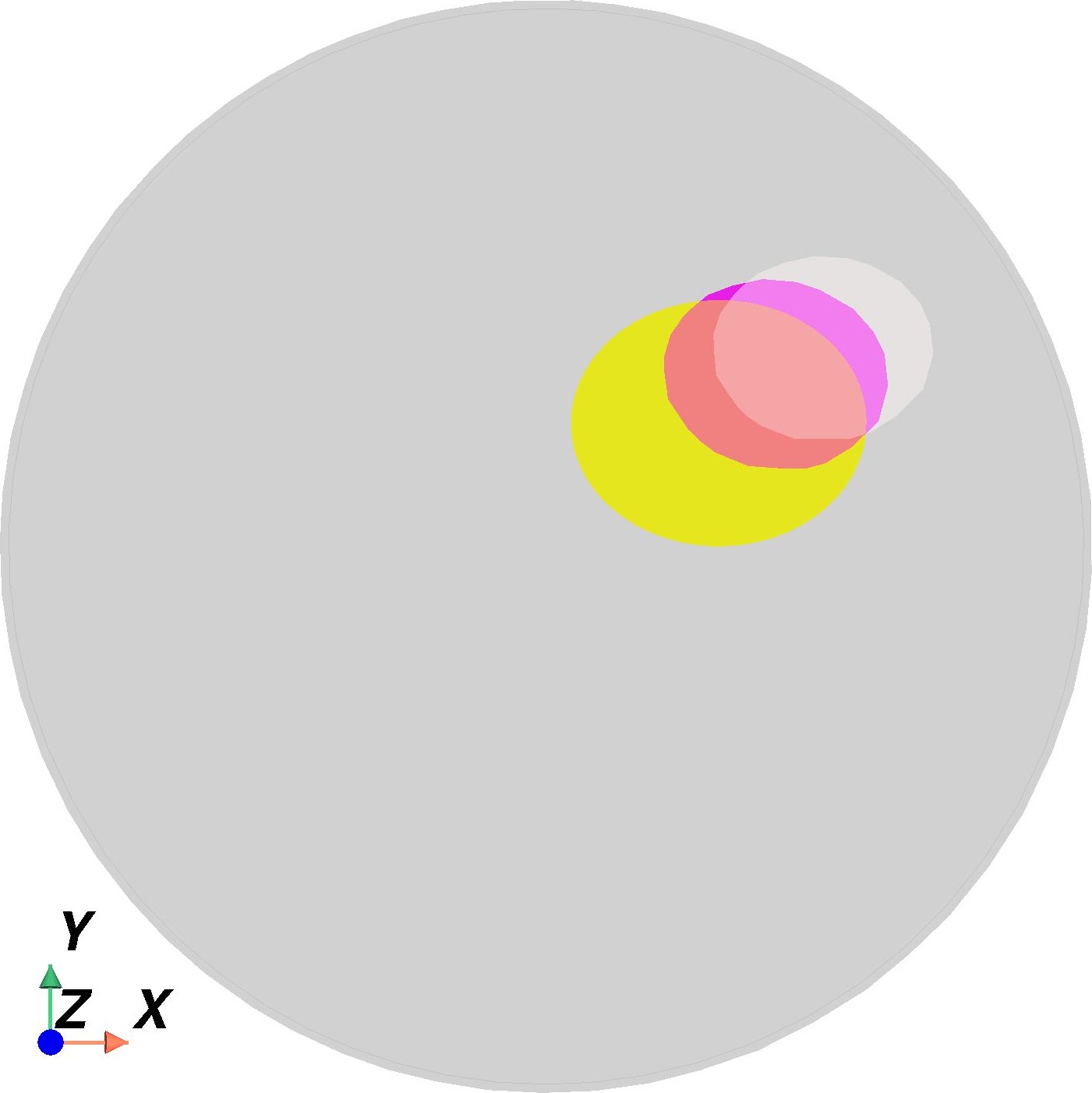}}  
\resizebox{0.15\textwidth}{!}{\includegraphics{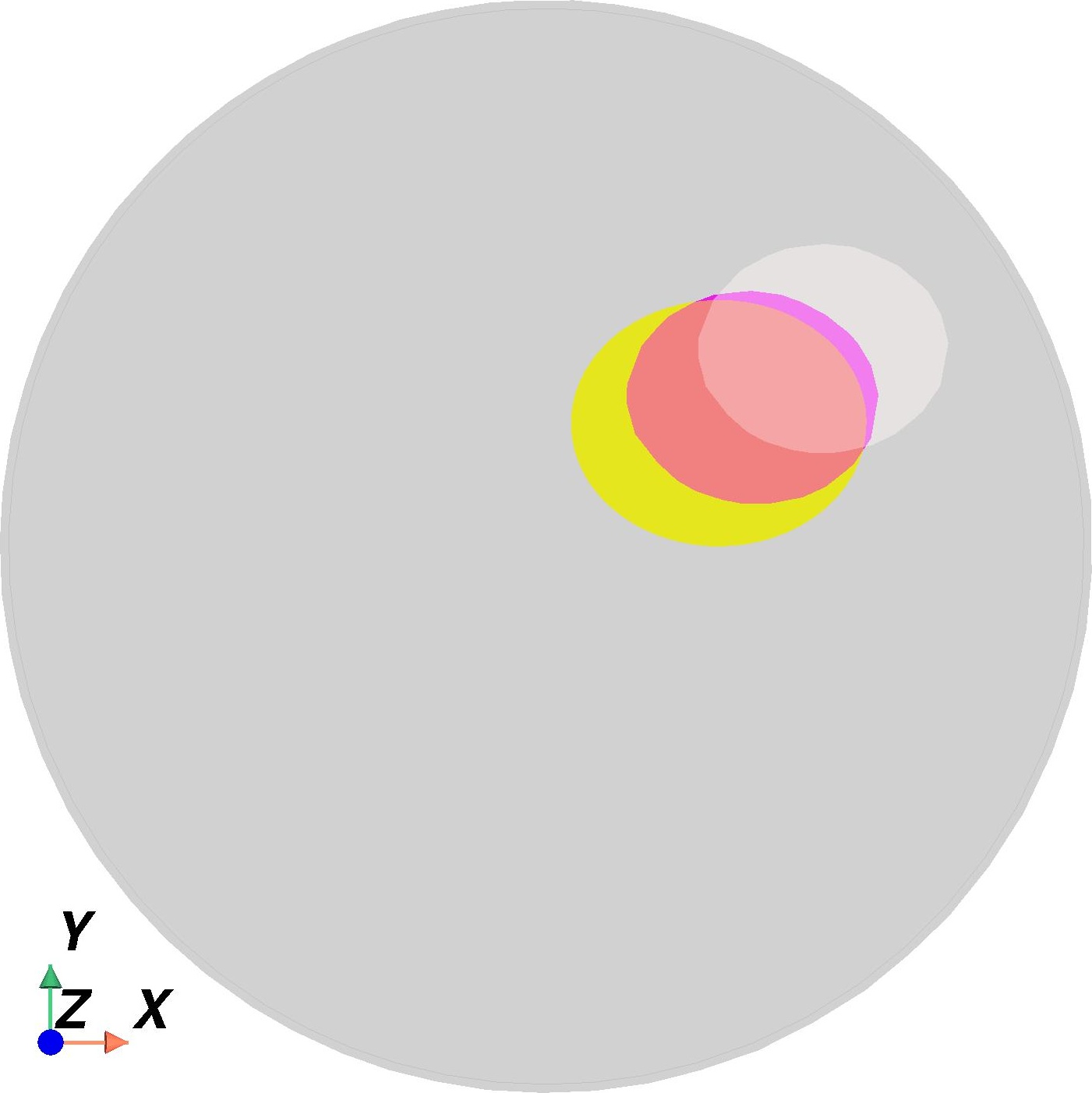}}    
\resizebox{0.15\textwidth}{!}{\includegraphics{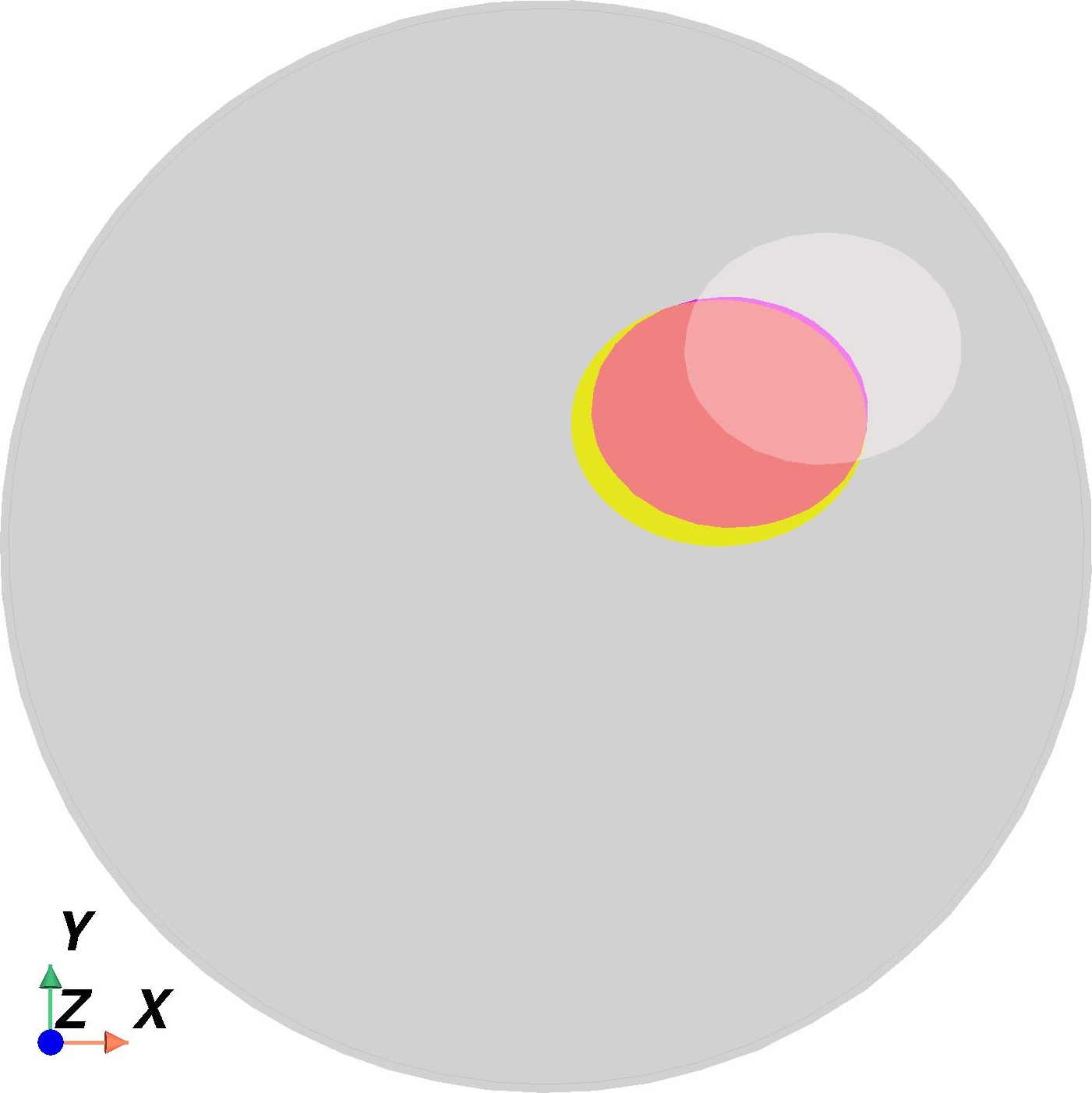}} \\[0.5em]
\resizebox{0.15\textwidth}{!}{\includegraphics{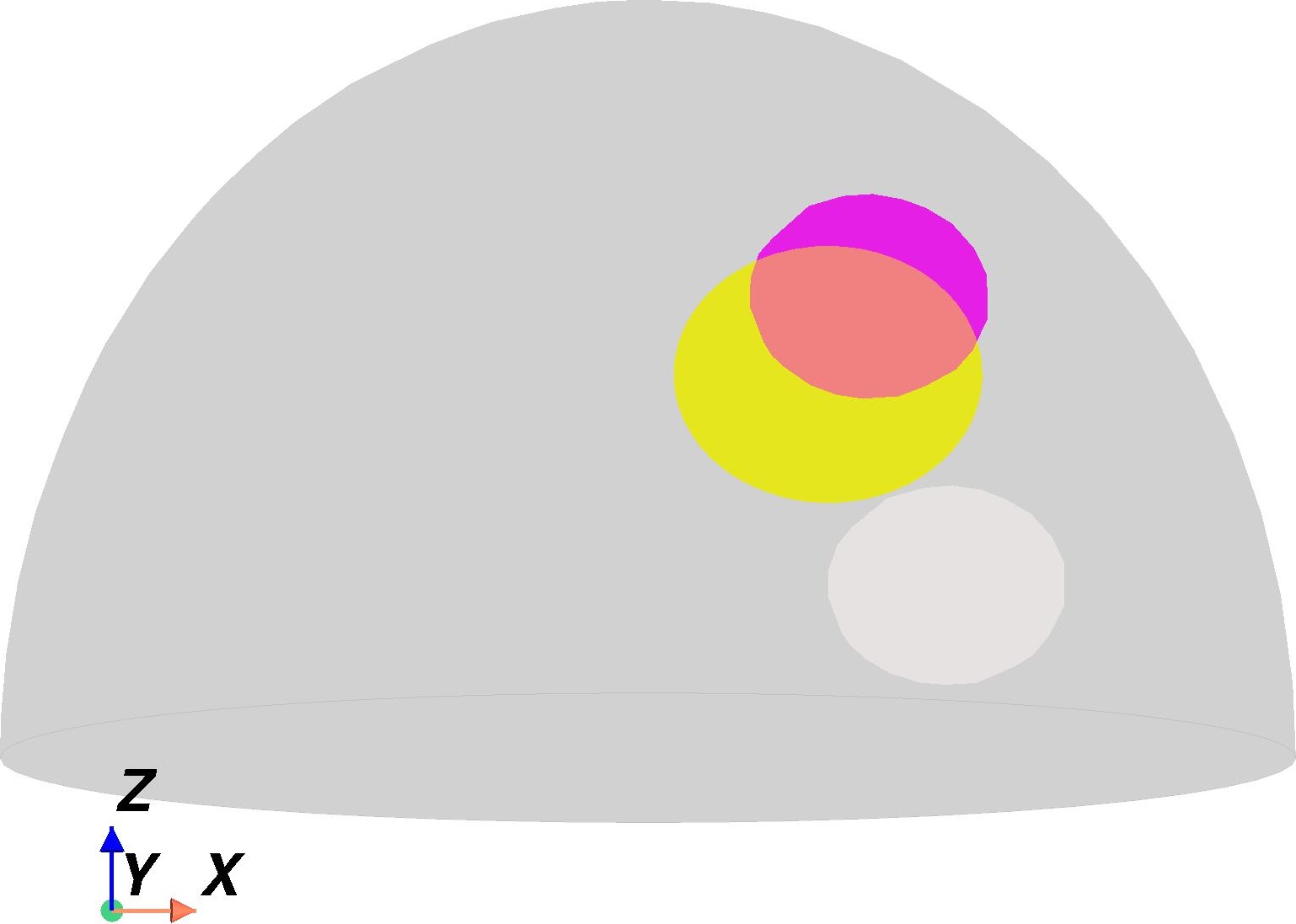}}  
\resizebox{0.15\textwidth}{!}{\includegraphics{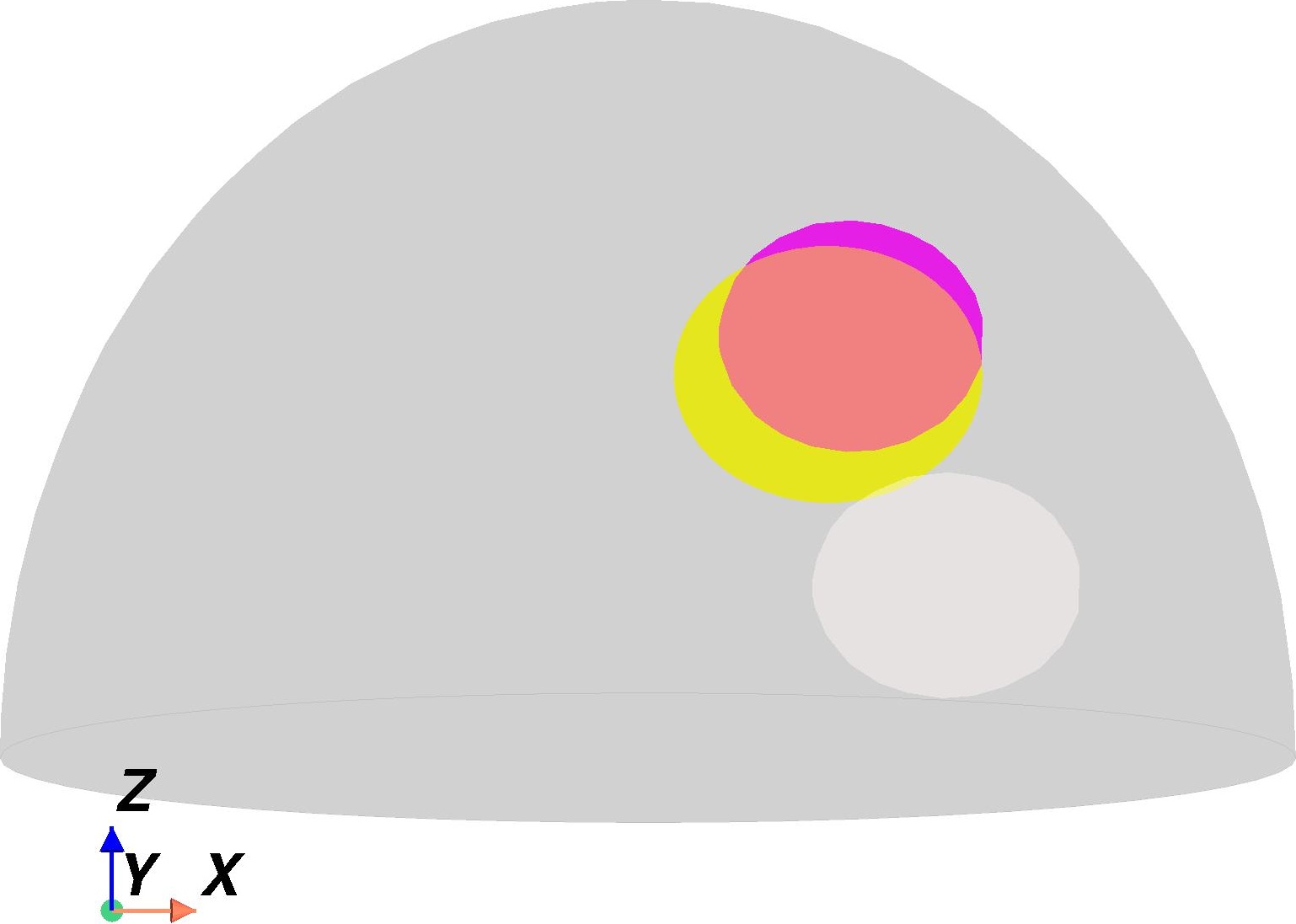}} 
\resizebox{0.15\textwidth}{!}{\includegraphics{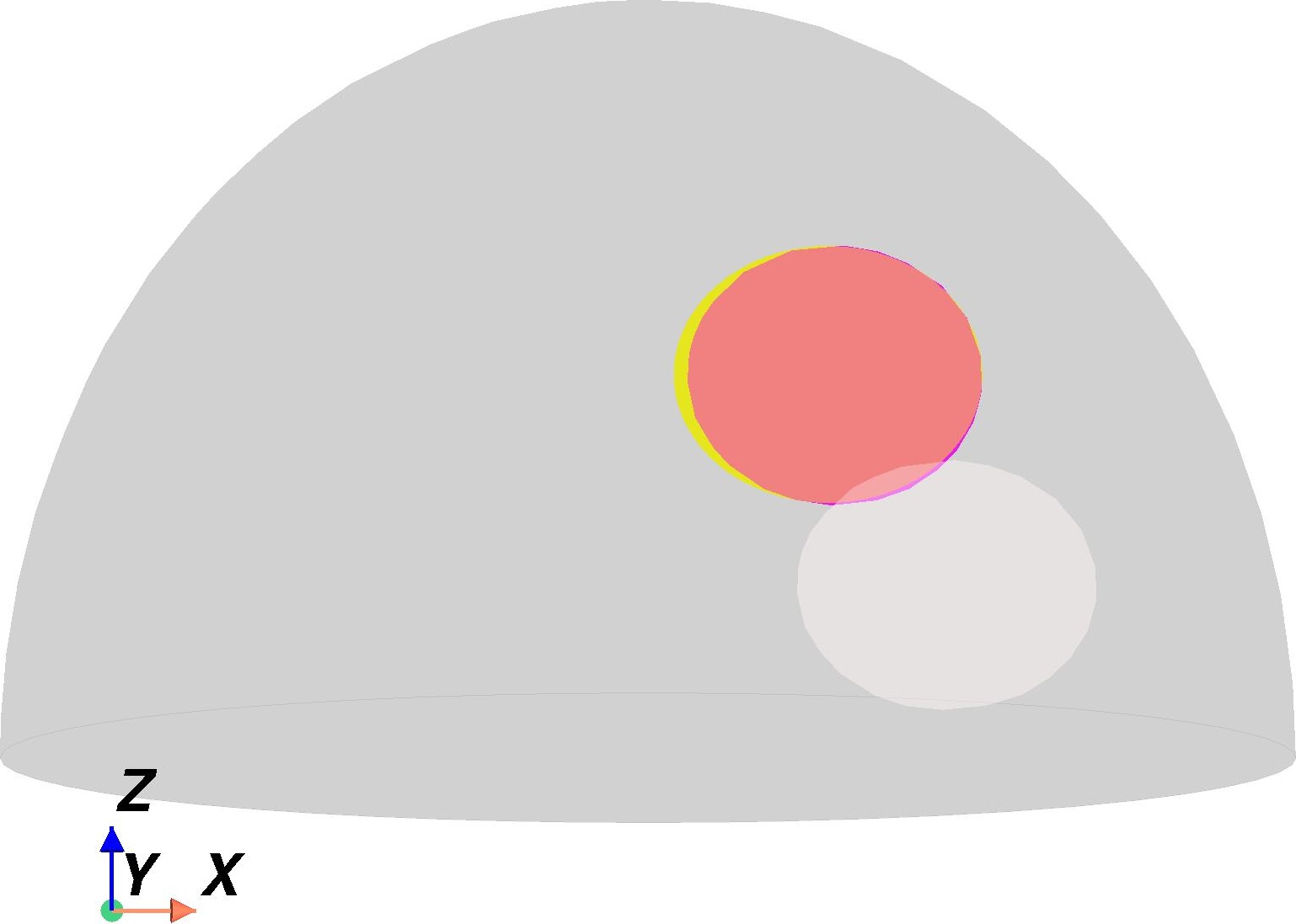}} \quad
\resizebox{0.15\textwidth}{!}{\includegraphics{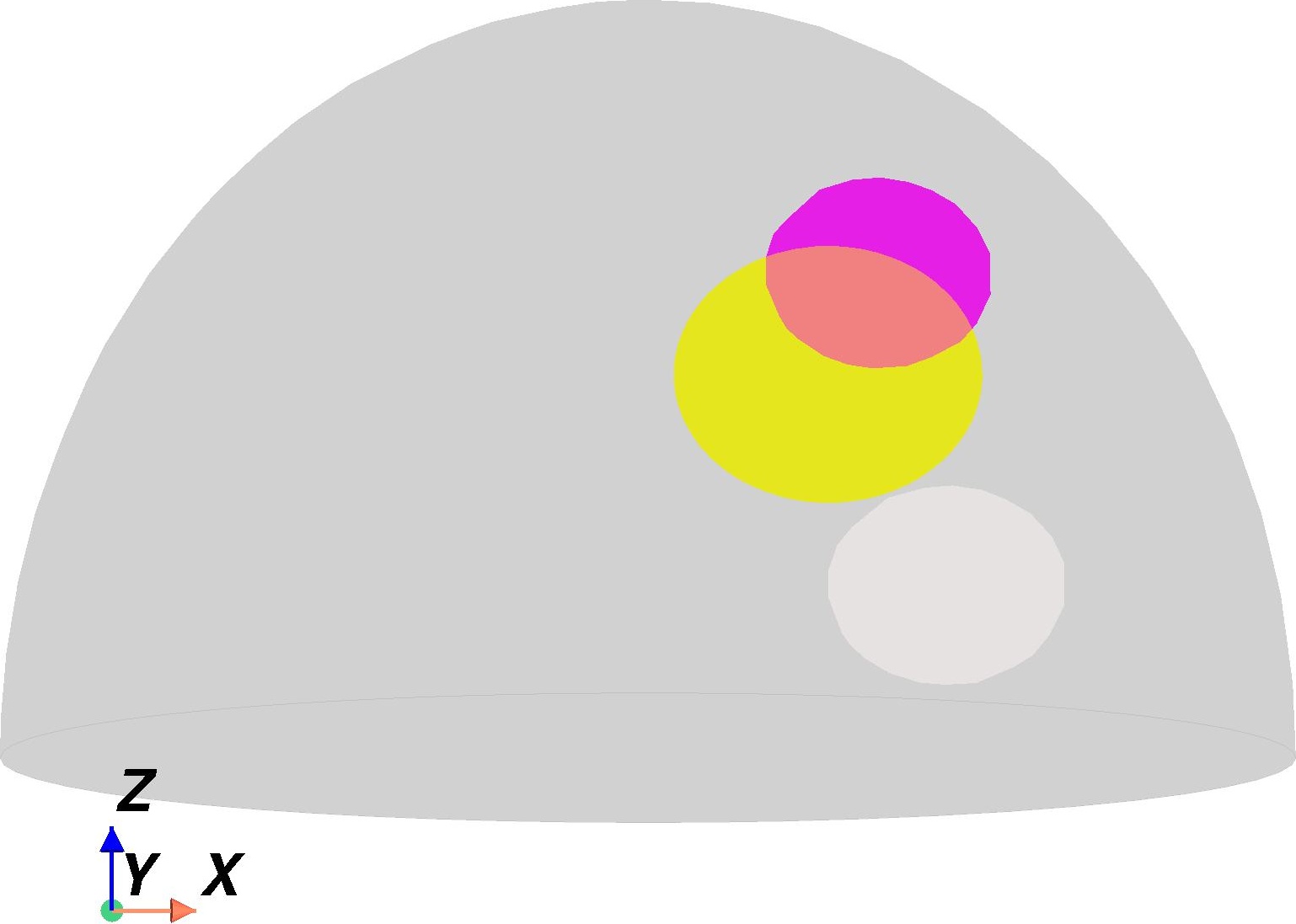}}  
\resizebox{0.15\textwidth}{!}{\includegraphics{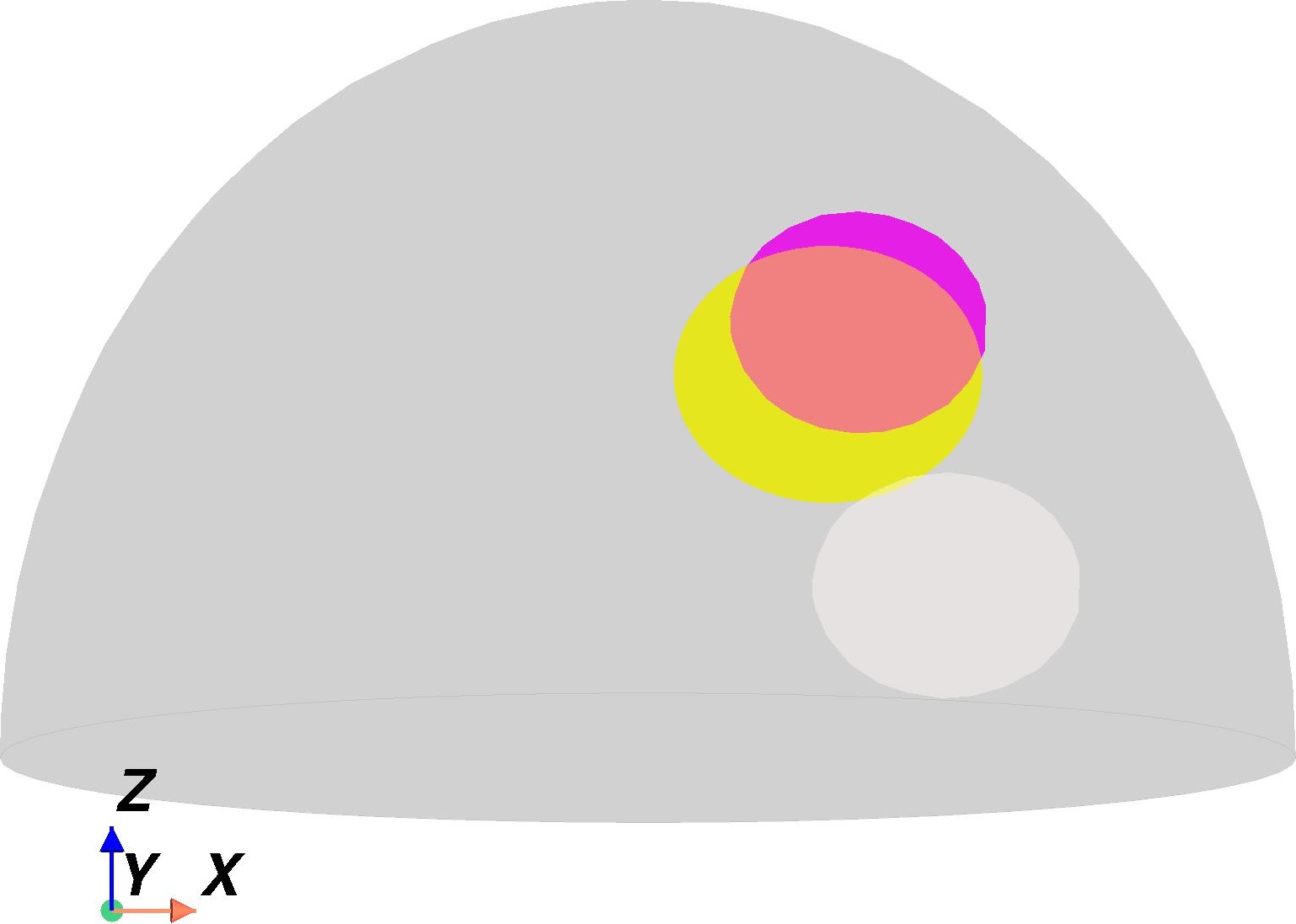}} 
\resizebox{0.15\textwidth}{!}{\includegraphics{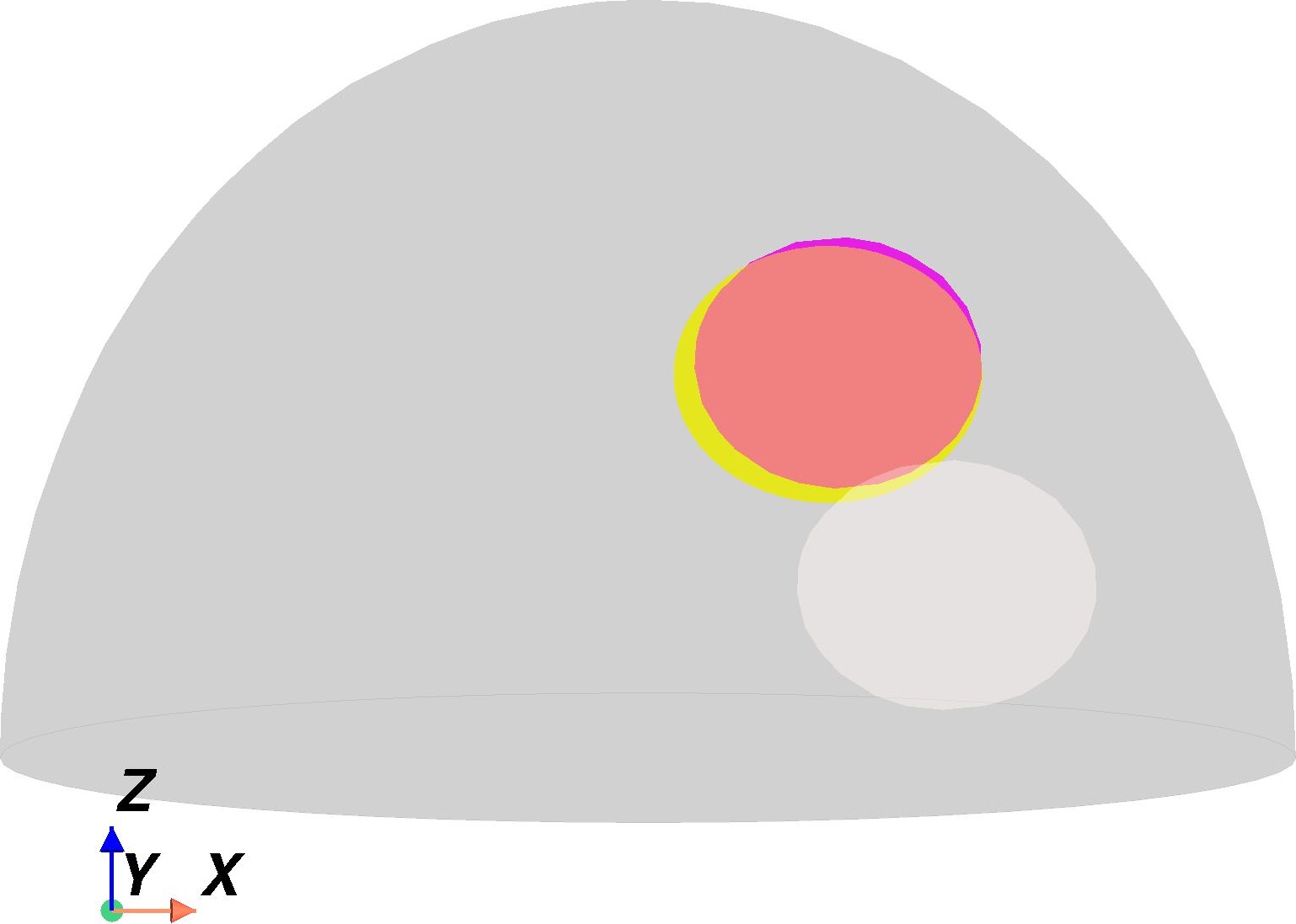}}  
\resizebox{0.15\textwidth}{!}{\includegraphics{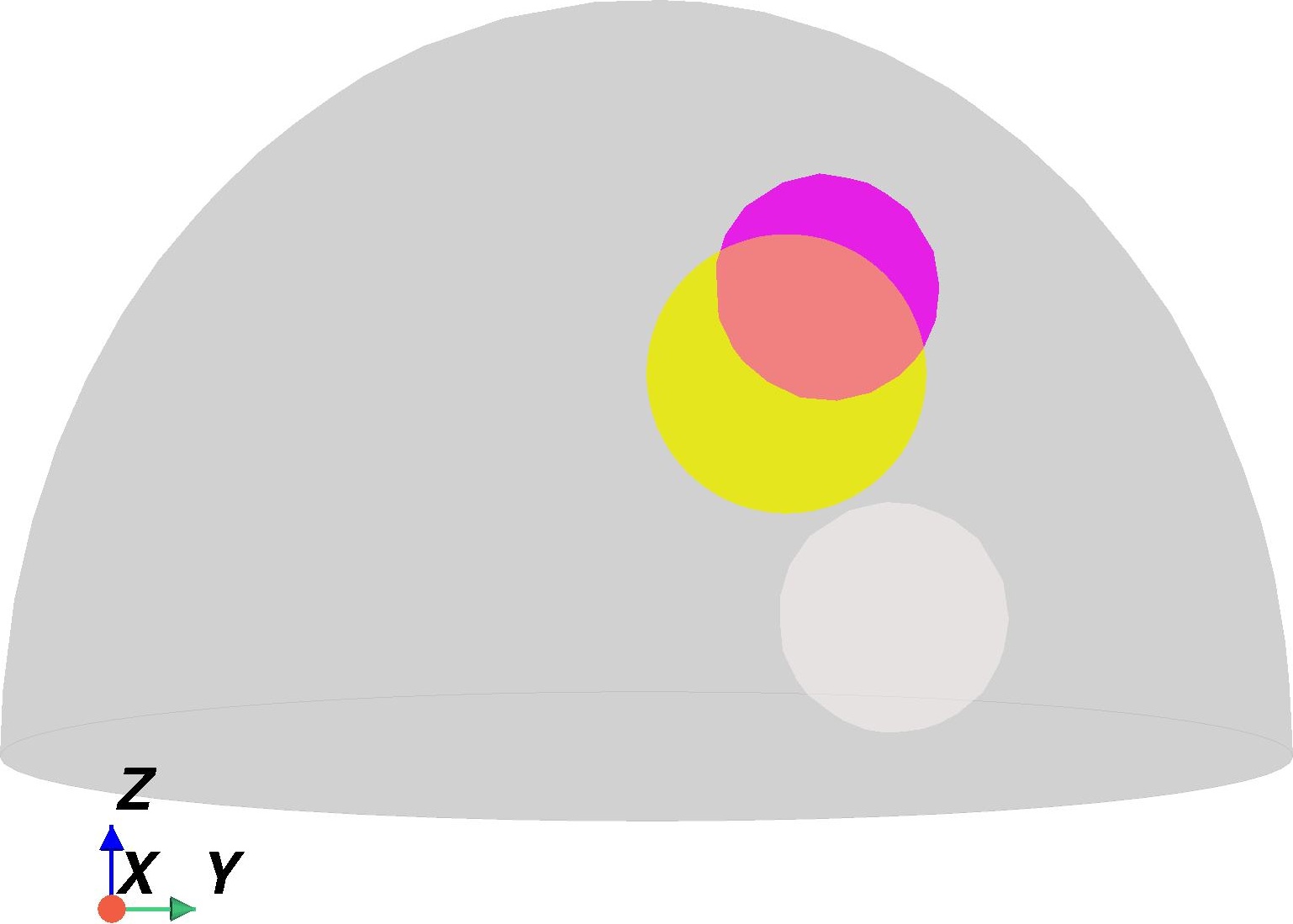}} 
\resizebox{0.15\textwidth}{!}{\includegraphics{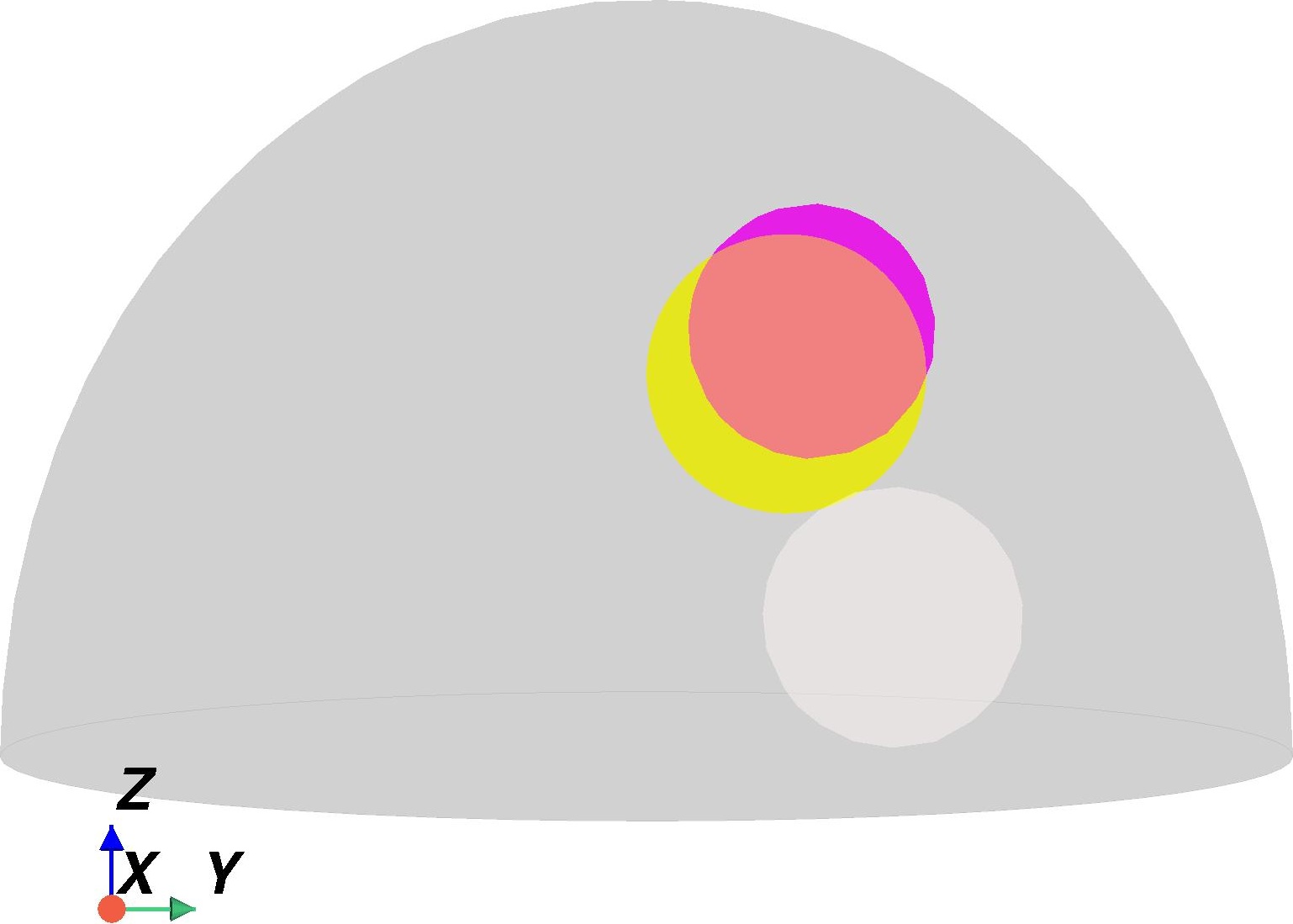}} 
\resizebox{0.15\textwidth}{!}{\includegraphics{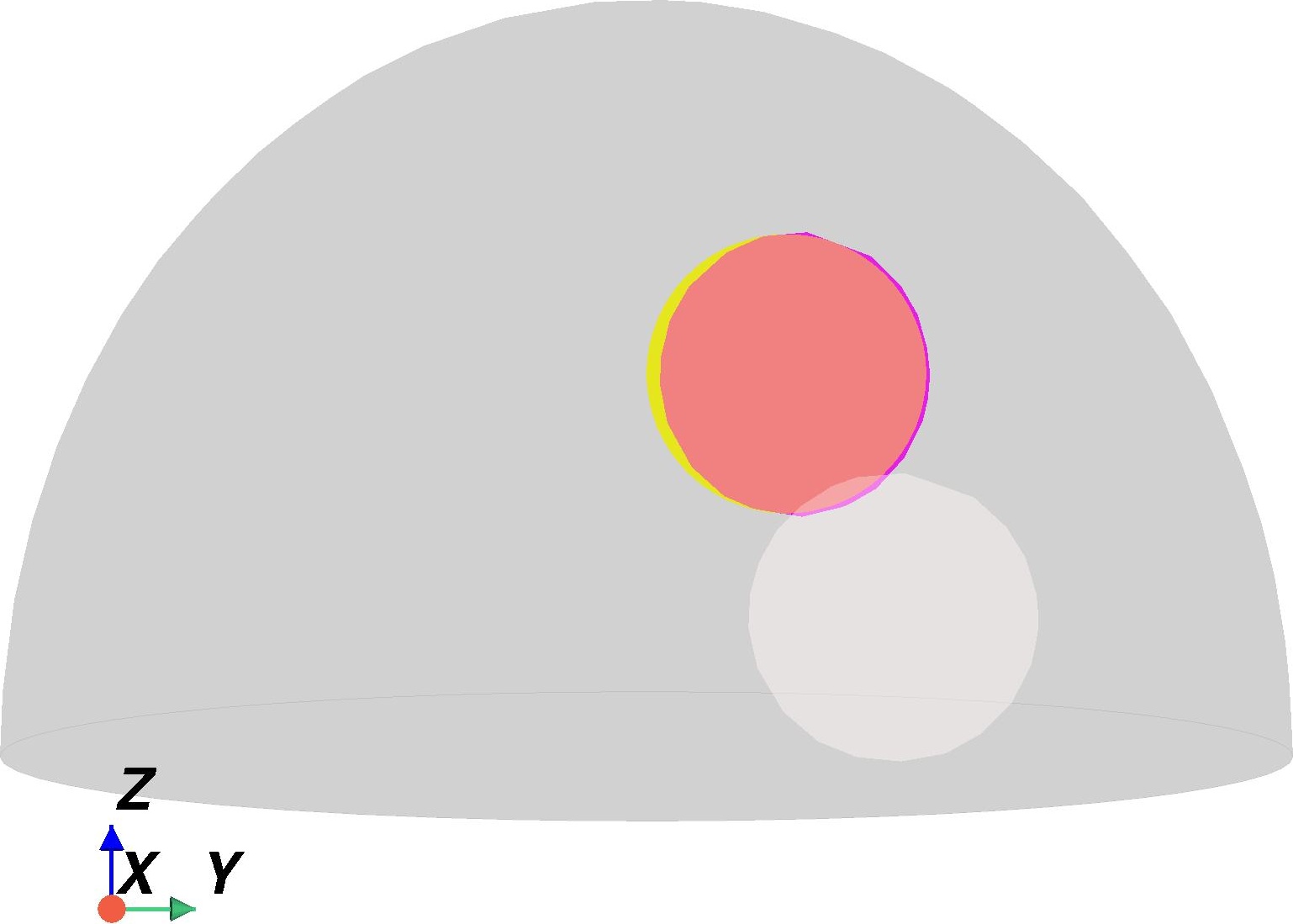}} \quad
\resizebox{0.15\textwidth}{!}{\includegraphics{example4_n0.01b1mod5rs0.8YZ.jpg}} 
\resizebox{0.15\textwidth}{!}{\includegraphics{example4_n0.01b1mod5rs0.9YZ.jpg}} 
\resizebox{0.15\textwidth}{!}{\includegraphics{example4_n0.01b1mod5rs1YZ.jpg}} \\[0.5em]
\resizebox{0.16\textwidth}{!}{\includegraphics{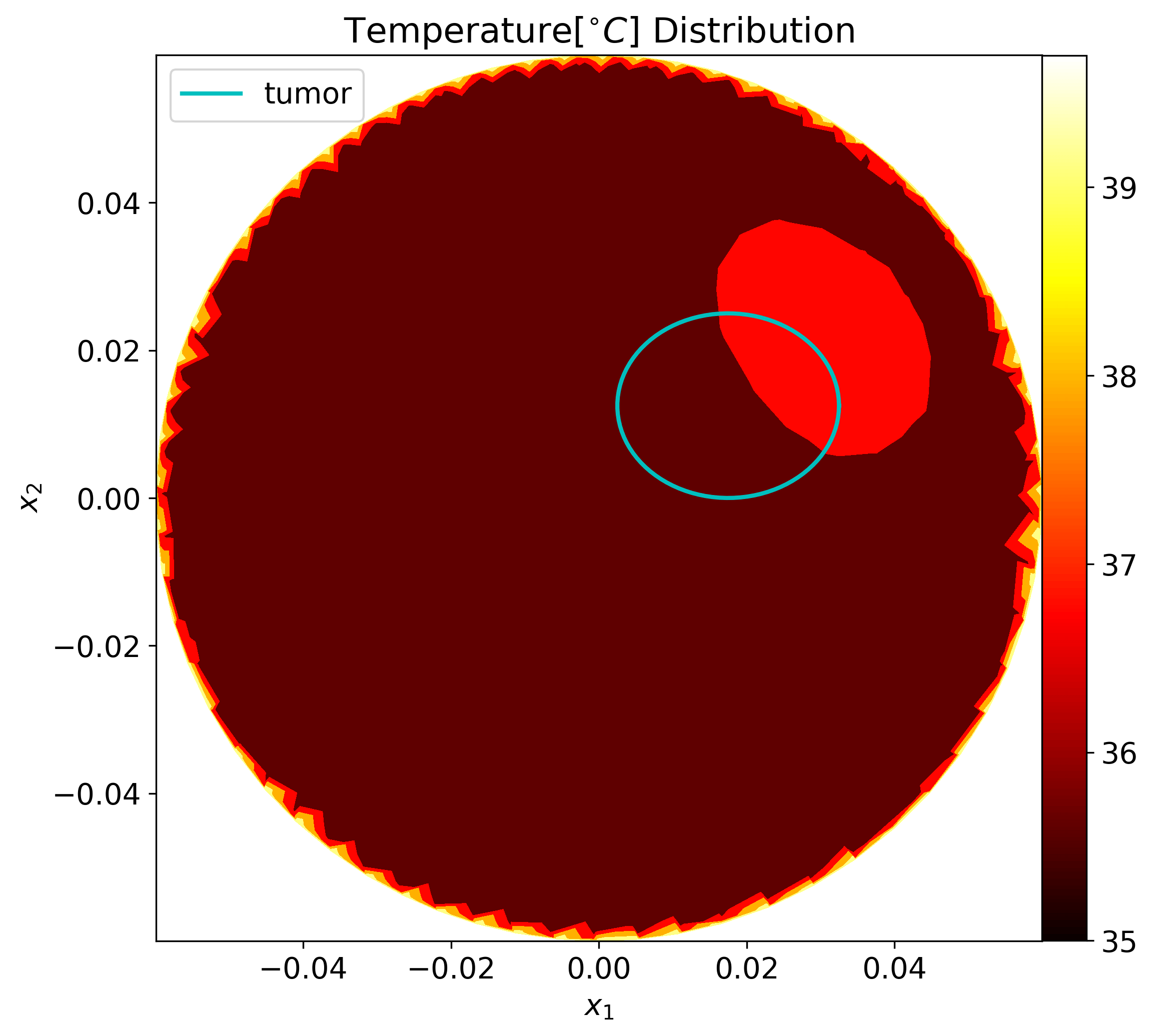}} 
\resizebox{0.16\textwidth}{!}{\includegraphics{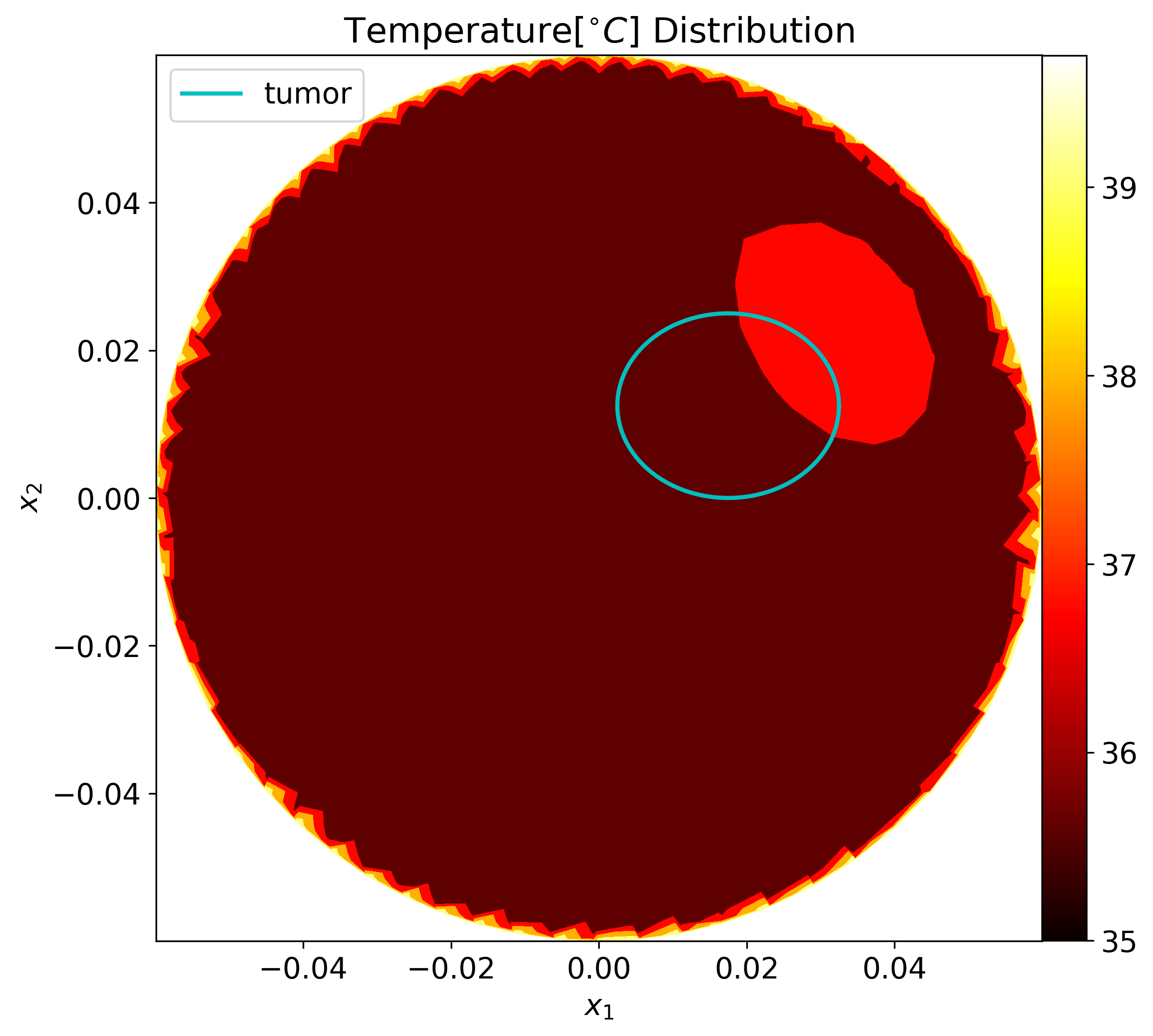}} 
\resizebox{0.16\textwidth}{!}{\includegraphics{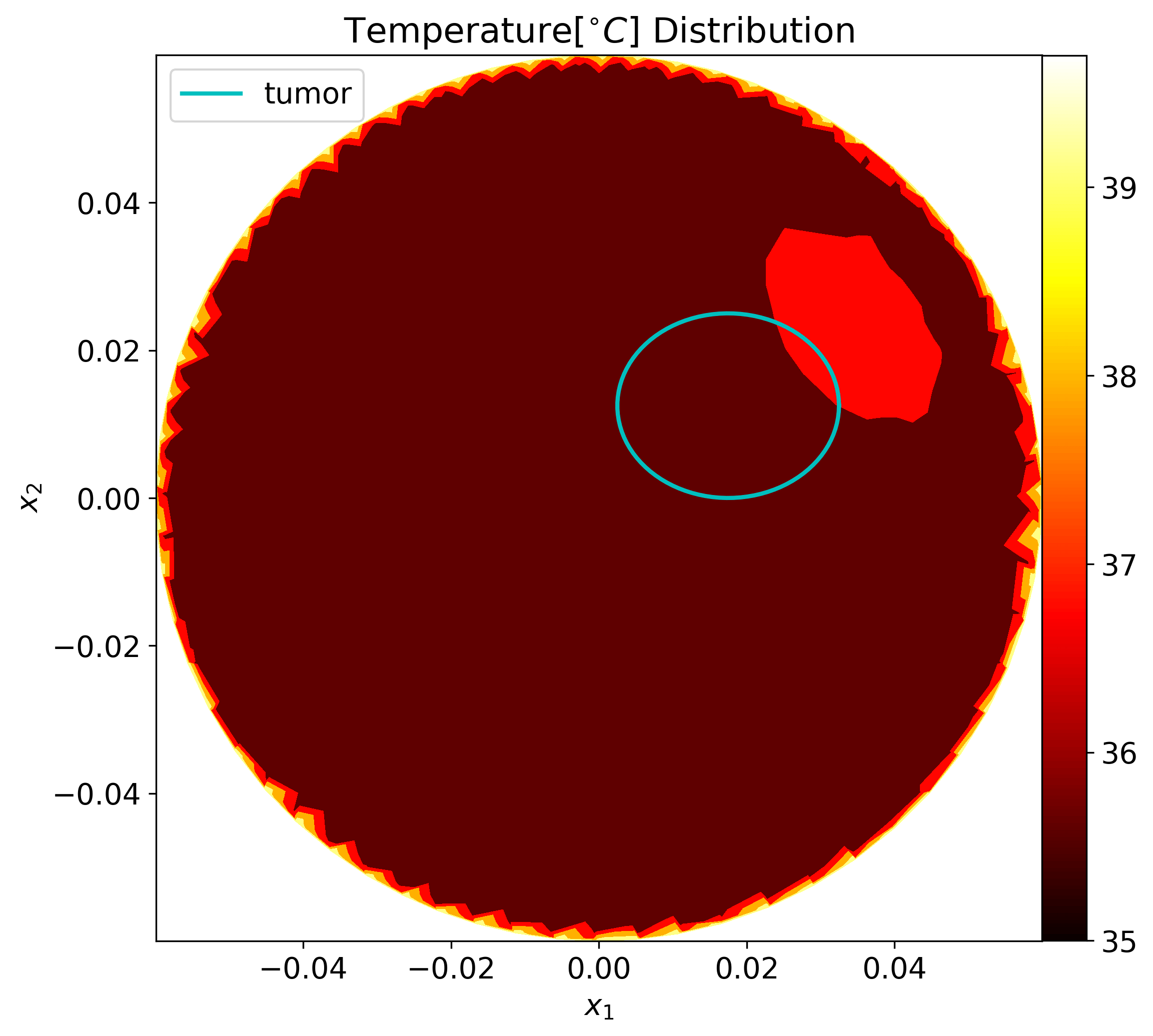}} \hfill
\resizebox{0.16\textwidth}{!}{\includegraphics{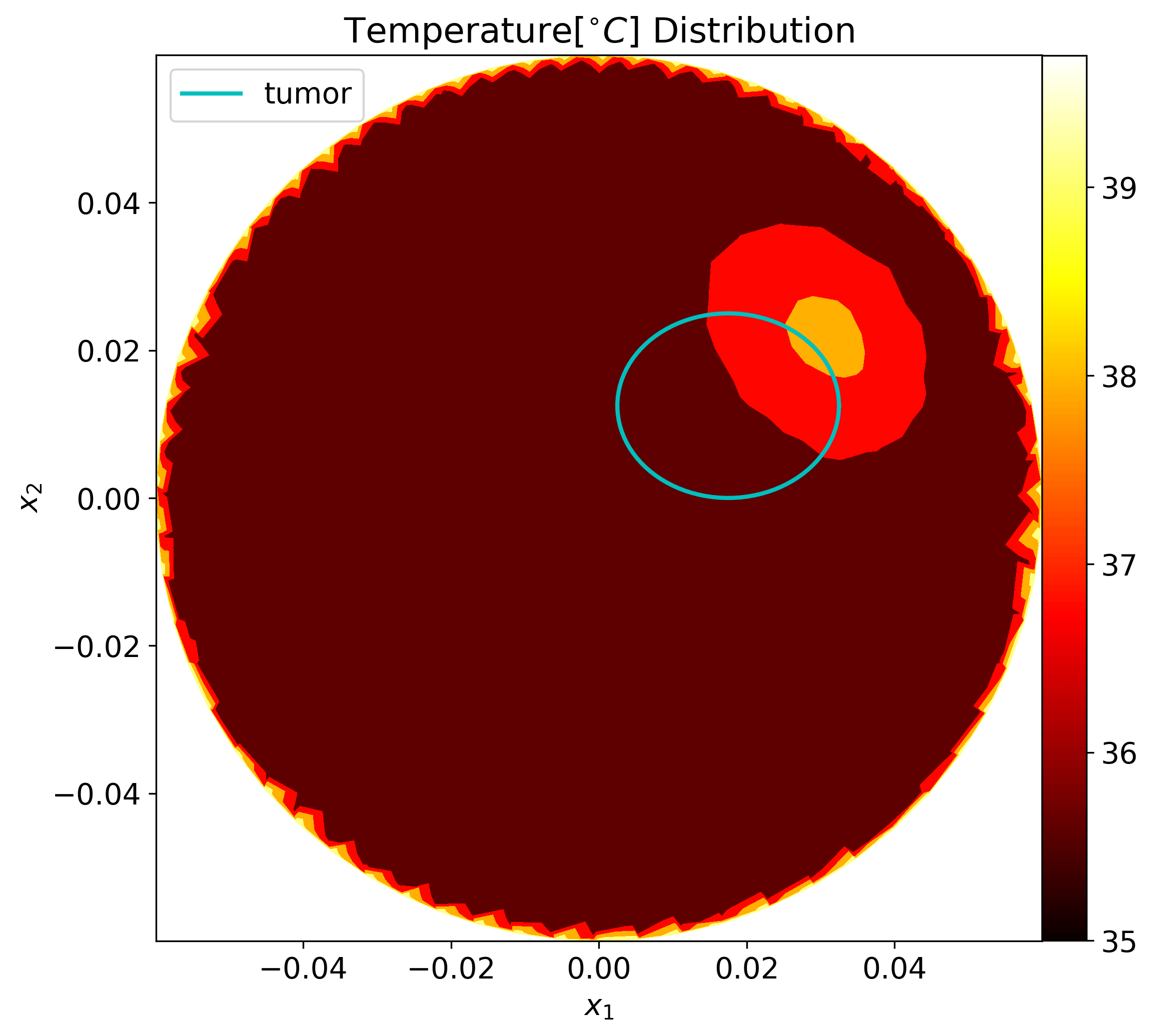}} 
\resizebox{0.16\textwidth}{!}{\includegraphics{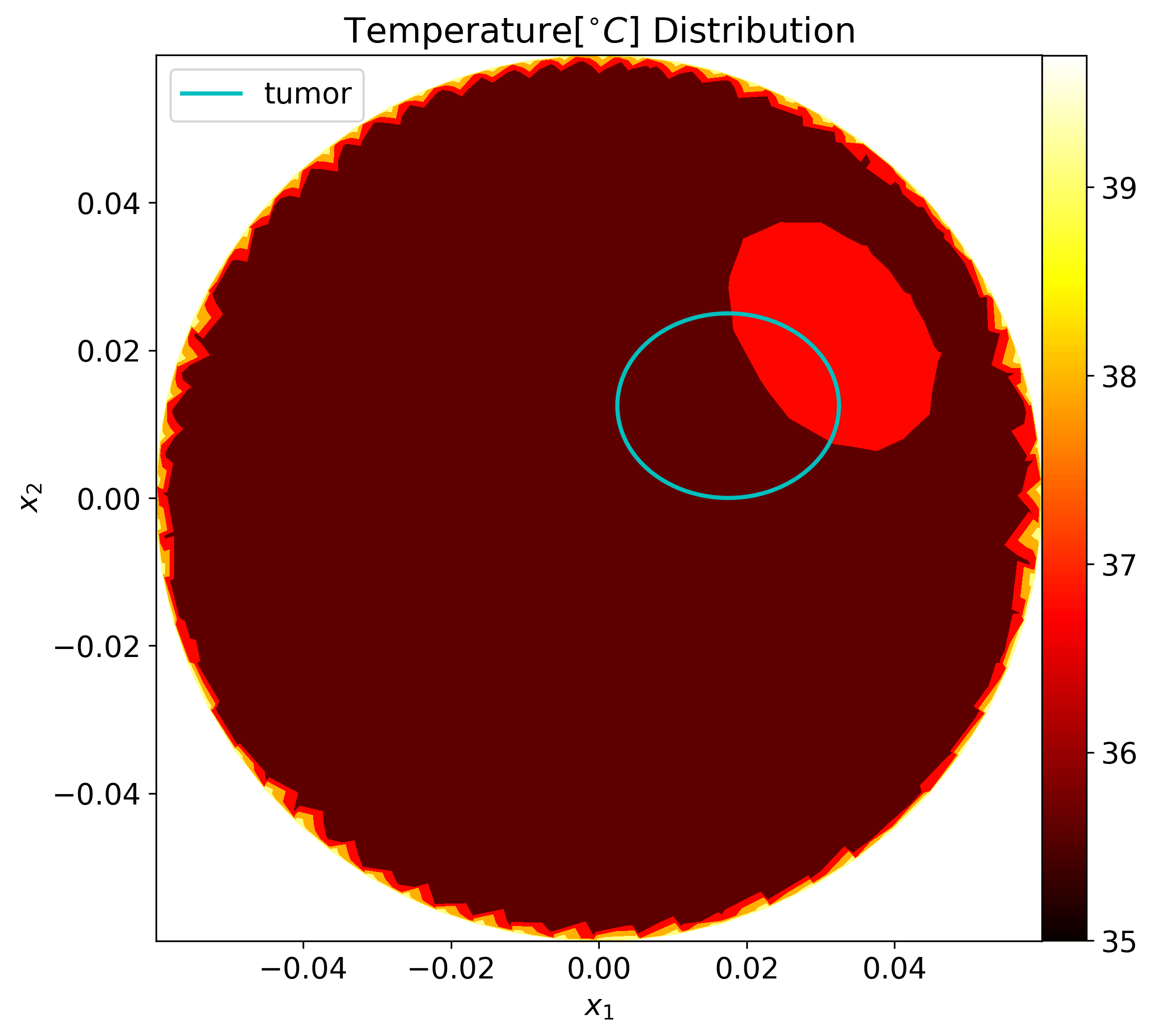}} 
\resizebox{0.16\textwidth}{!}{\includegraphics{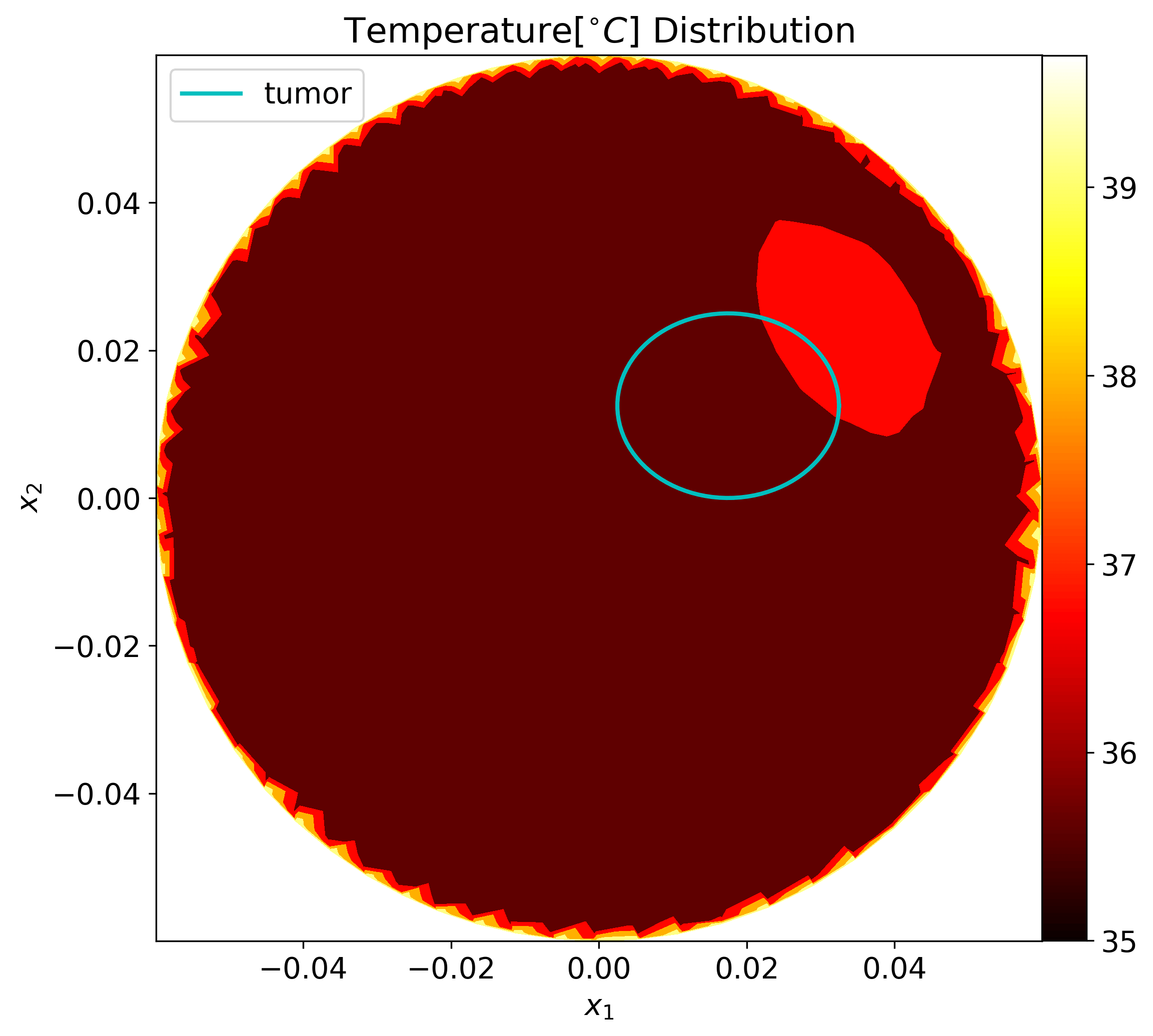}} 
\caption{Recovered shapes (magenta) from different initial guesses (light color) versus the exact tumor (yellow) from various views. The bottom row shows the temperature distributions on the skin surface of recovered shapes with three contour levels.  Left panel: without volume penalization and balancing principle \eqref{eq:balancing_principle}, right panel: with regularization. Columns (left to right) correspond to $\partial\varOmega_{0}^{(1)}$, $\partial\varOmega_{0}^{(2)}$, and $\partial\varOmega_{0}^{(3)}$.}
\label{fig:3d_results}
\end{figure}
\begin{figure}[htp!]
\centering   
\resizebox{0.4\textwidth}{!}{\includegraphics{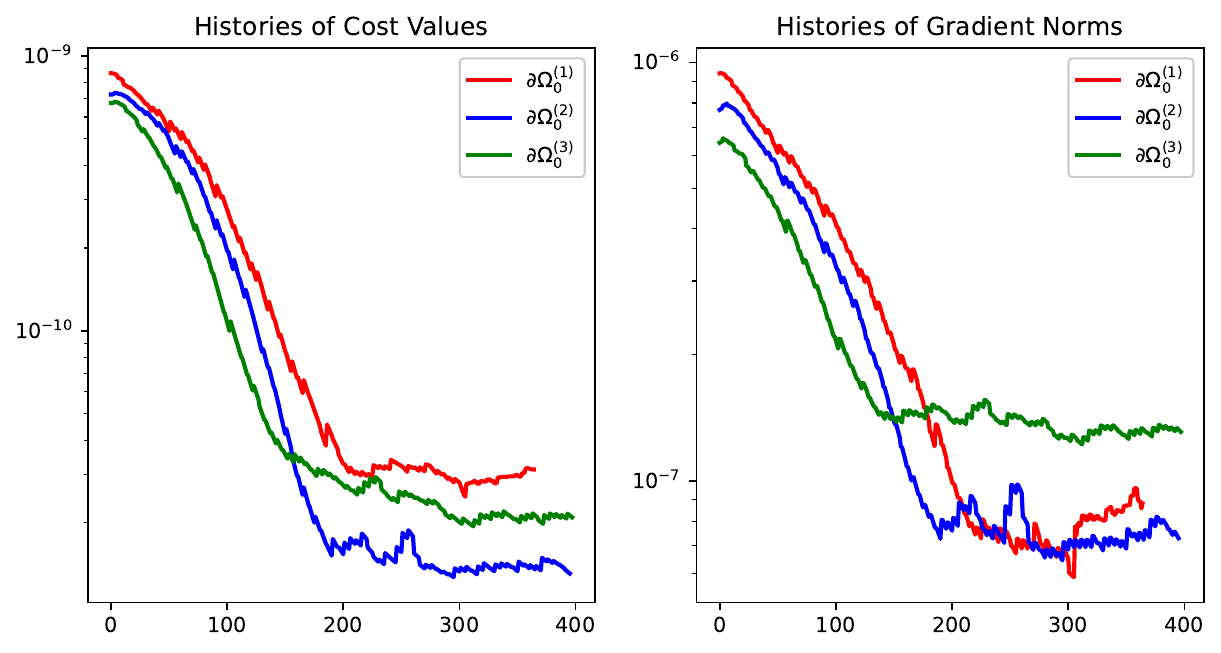}}  \hfill
\resizebox{0.575\textwidth}{!}{\includegraphics{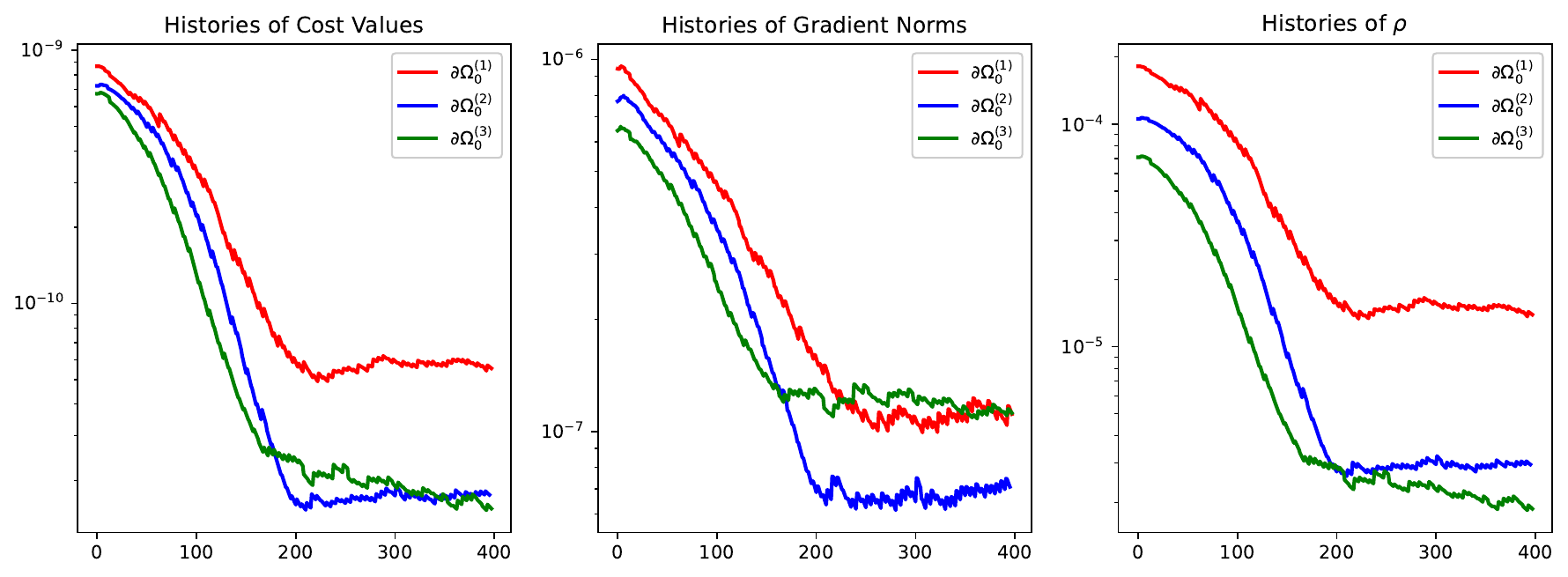}} 
\caption{Histories of cost values and gradient norms. Left: without volume penalization and balancing principle \eqref{eq:balancing_principle}; Right: with regularization, including histories of the parameter $\rho$ in \eqref{eq:balancing_principle}.}
\label{fig:3d_histories_of_values}
\end{figure}
\begin{figure}[htp!]
\centering    
\resizebox{0.2\textwidth}{!}{\includegraphics{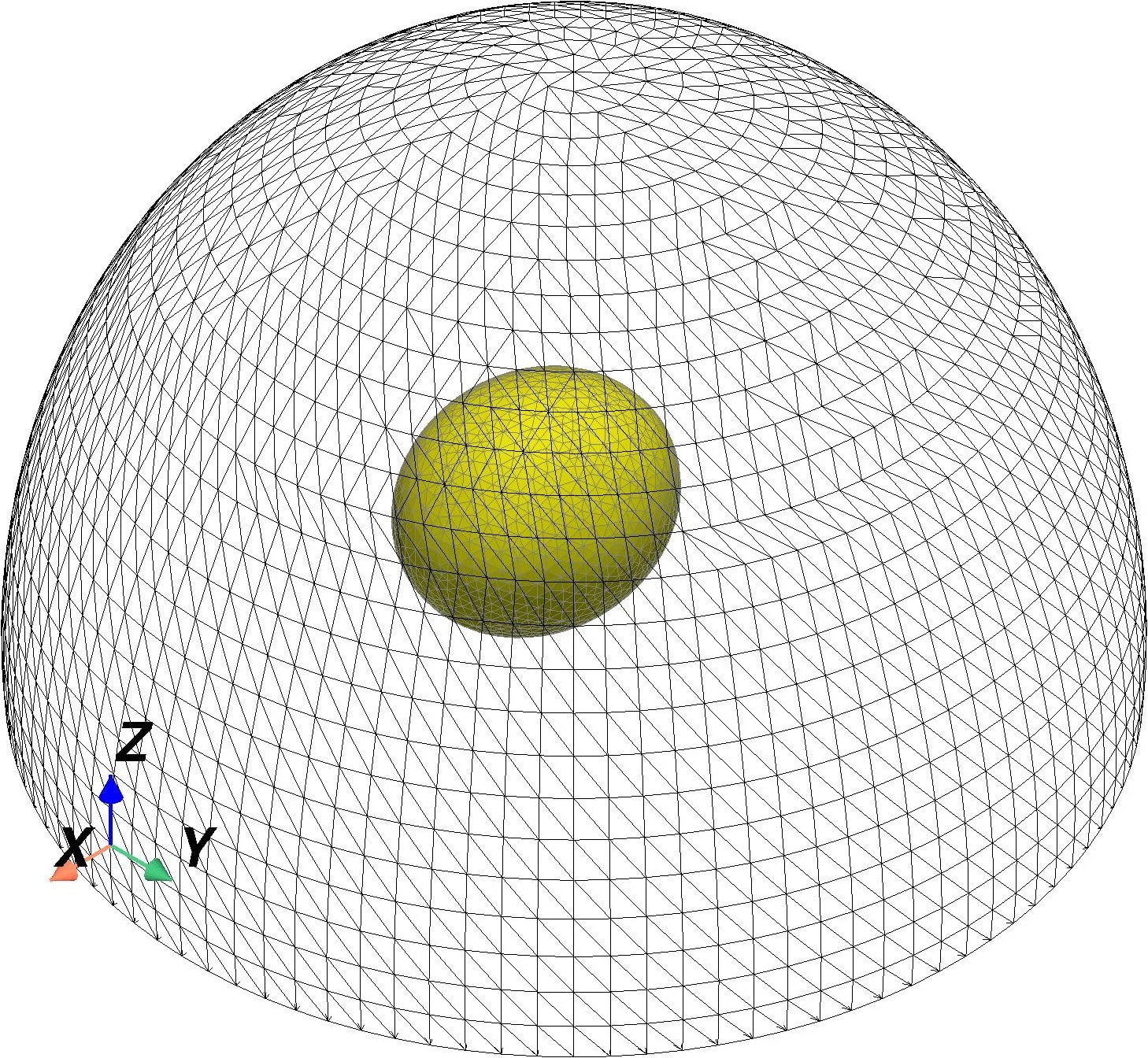}} \quad 
\resizebox{0.2\textwidth}{!}{\includegraphics{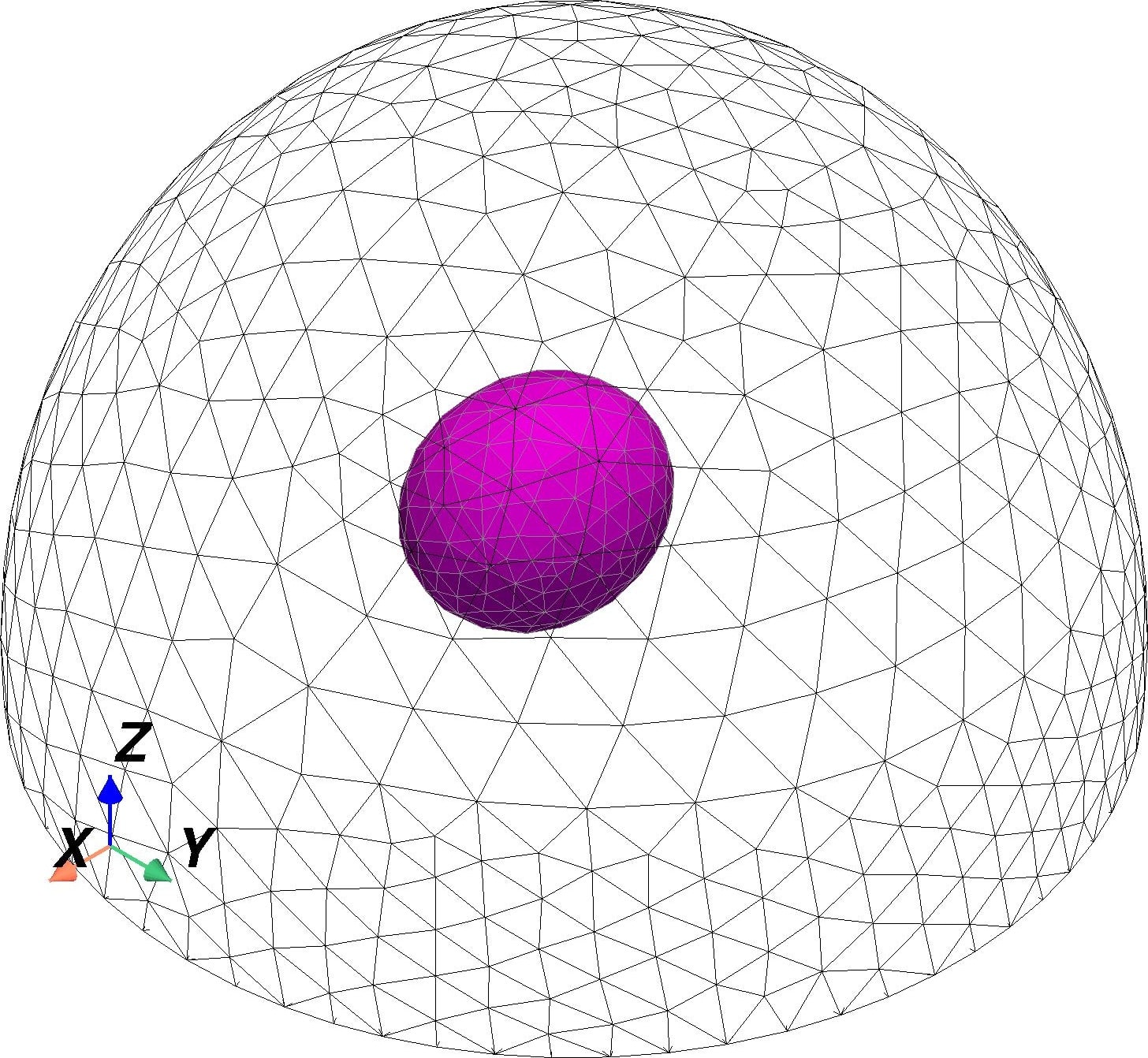}} 
\caption{The exact (yellow) and computed (magenta) shapes for $\varsigma = 1$.}
\label{fig:3d_mesh_profiles}
\end{figure}
%
%
%
\subsection{\harbrecht{Identification of two tumors}}\label{subsec:test_3d_multiple_tumors}
We repeat the previous example, but this time consider two embedded tumors. 
The tumors are ellipsoidal in shape. 
The first tumor has semi-axes $r_{1}^{\ast} = 0.0125$ (m), $r_{2}^{\ast} = 0.015$ (m), and $r_{3}^{\ast} = 0.0125$ (m), with its center located at $(x_{1}^{\ast}, x_{2}^{\ast}, x_{3}^{\ast}) = (0.0175, 0.0125, 0.035)$ (m). 
The second tumor has semi-axes $\bar{r}_{1}^{\ast} = 0.0175$ (m), $\bar{r}_{2}^{\ast} = 0.0125$ (m), and $\bar{r}_{3}^{\ast} = 0.015$ (m), with its center at $(\bar{x}_{1}^{\ast}, \bar{x}_{2}^{\ast}, \bar{x}_{3}^{\ast}) = (-0.0175, -0.015, 0.035)$ (m).

This test serves two main purposes. 
First, we aim to show that the algorithm performs effectively regardless of the size and shape of the initial guess.
Second, we demonstrate its ability to identify the tumor size and location accurately even under high noise levels. 
For this test, we set $c_{b} = 0.975$ in \eqref{eq:bilinear_form_for_regularization} and fix $\rho = 8.5 \times 10^{-5}$.

Figure~\ref{fig:skin_temperature_two_tumors} shows the skin temperature distributions, viewed from above, at different noise levels ($\delta = 0\%, 1\%, 2\%, 3\%$). 
From these measurements, one may infer the presence of more than one tumor. 
Although the observed temperature distribution does not definitively confirm this, we initialize the algorithm with two tumors positioned along the central $z$-axis.

It is worth noting that approaches such as the level-set method \cite{OsherSethian1998} can be used to identify multiple tumors. 
However, our current scheme is based on a Lagrangian framework, in contrast to the Eulerian nature of level-set methods. 
Due to this limitation, we directly initialize the algorithm with two embedded tumors.

Figure~\ref{fig:3d_results_comparisons_multiple_tumors} presents the identification results. 
The plots show that the method provides a good approximation of the true tumor geometry, with the estimated tumors (magenta ellipsoids) closely matching the actual tumors (yellow ellipsoids), even under a noise level of $\delta = 3\%$. 
The initial guess is shown in light gray.

In Figure~\ref{fig:skin_temperature_two_tumors_real_and_imaginary_parts}, we present both the exact and noisy skin temperature distributions at the initial and final states, under noise levels of $1\%$, $2\%$, and $3\%$. 
Starting from an initial distribution far from the observed data, the method converges to a final state that closely reproduces the measured temperature pattern, thereby confirming the reliability and effectiveness of the proposed approach for geometric inverse problems.

\sloppy Finally, Figure~\ref{fig:3d_results_multiple_tumors} compares the evolving shape of the free domain in magenta at iterations $k = 50, 100, 150, 200, 250,$ and $276$ (final) starting from the initial guess in light color against the exact tumor shape in yellow. 
The results shown from multiple angles illustrate the reconstruction accuracy when the measurement data contains one percent noise.
The plots demonstrate that the method reliably detects both the location and size of tumors, even when the initial guess is far from accurate and the data is noisy.
\begin{figure}[htp!]
\centering   
\resizebox{0.235\textwidth}{!}{\includegraphics{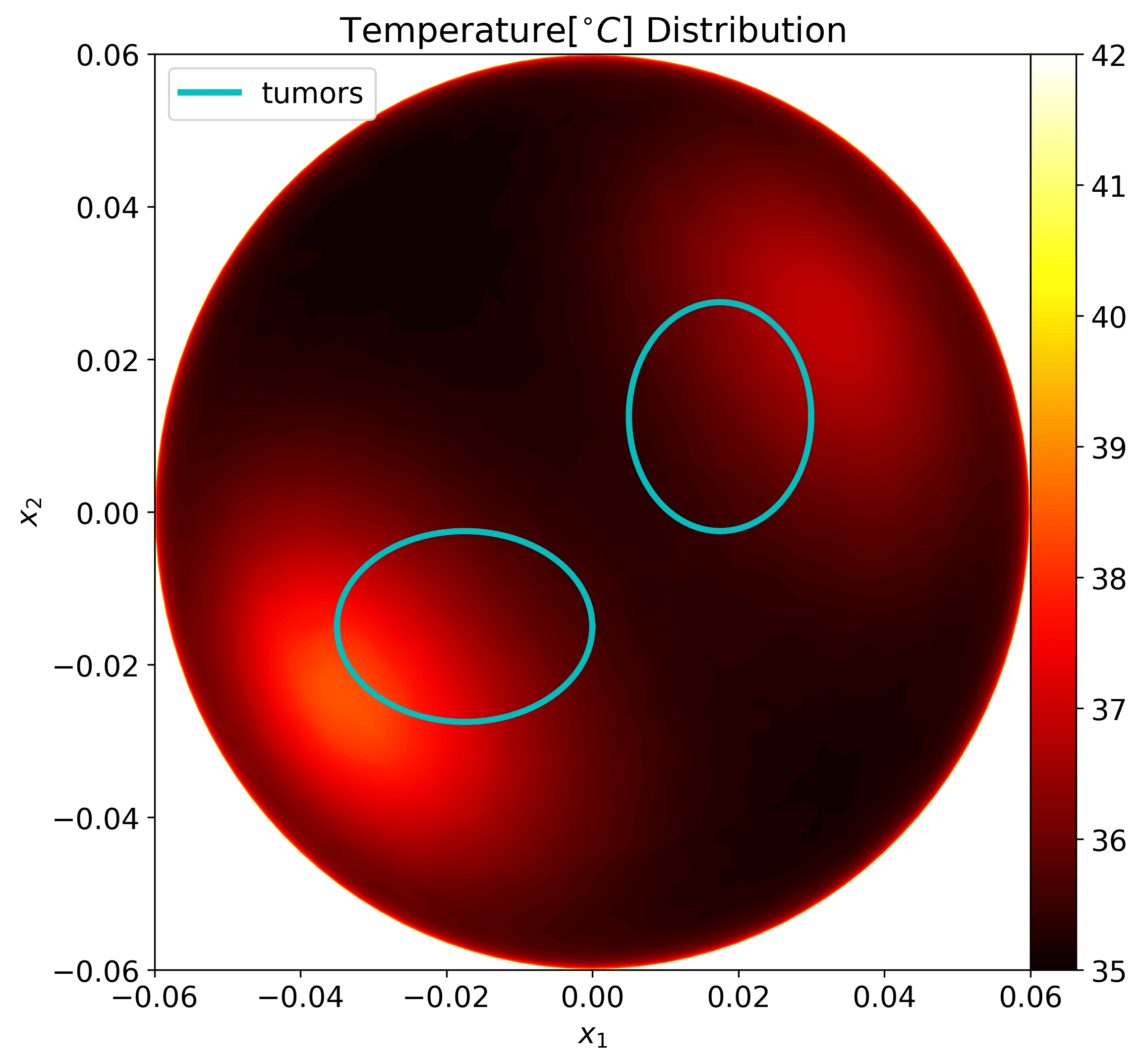}} 
\resizebox{0.235\textwidth}{!}{\includegraphics{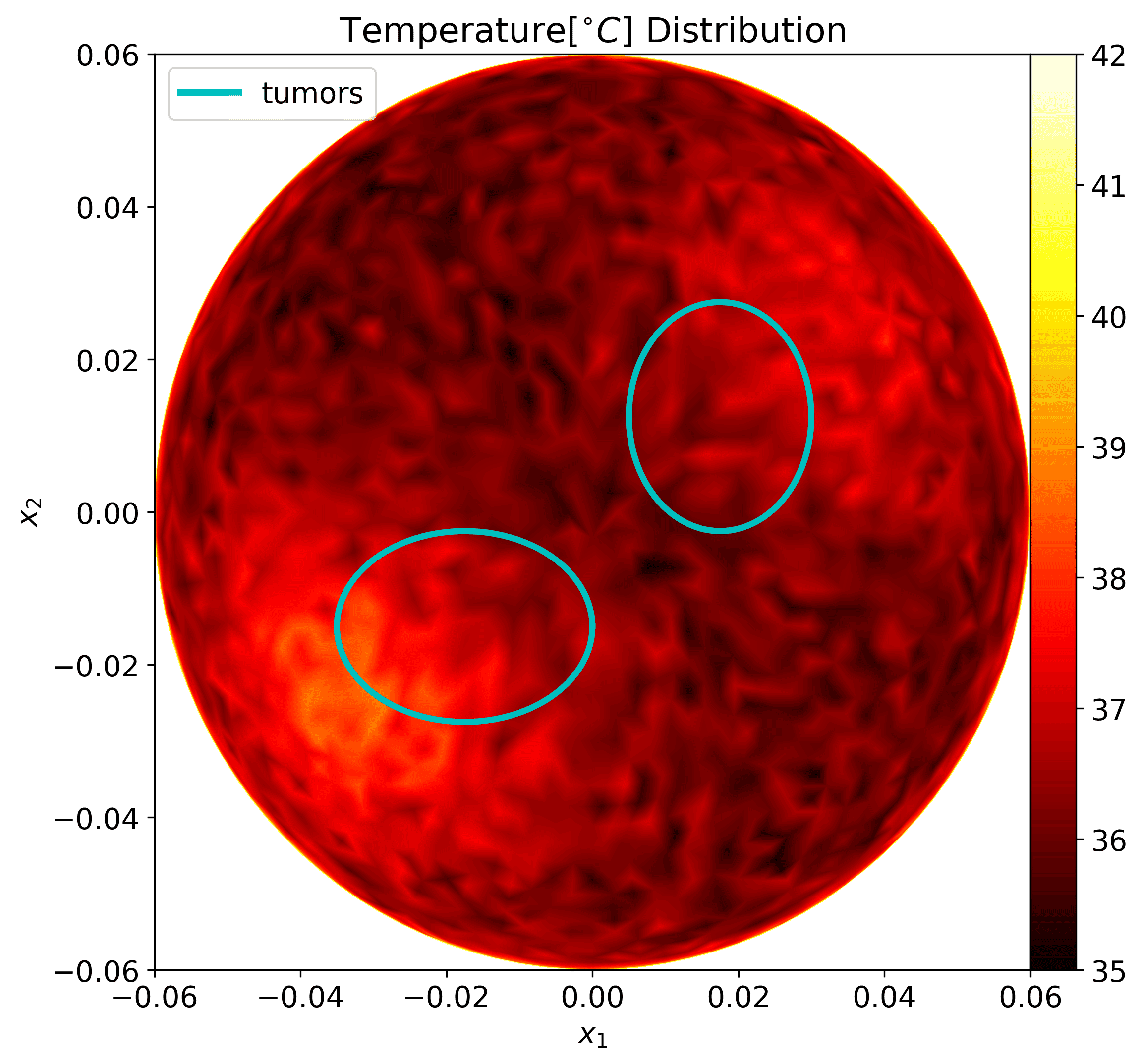}} 
\resizebox{0.235\textwidth}{!}{\includegraphics{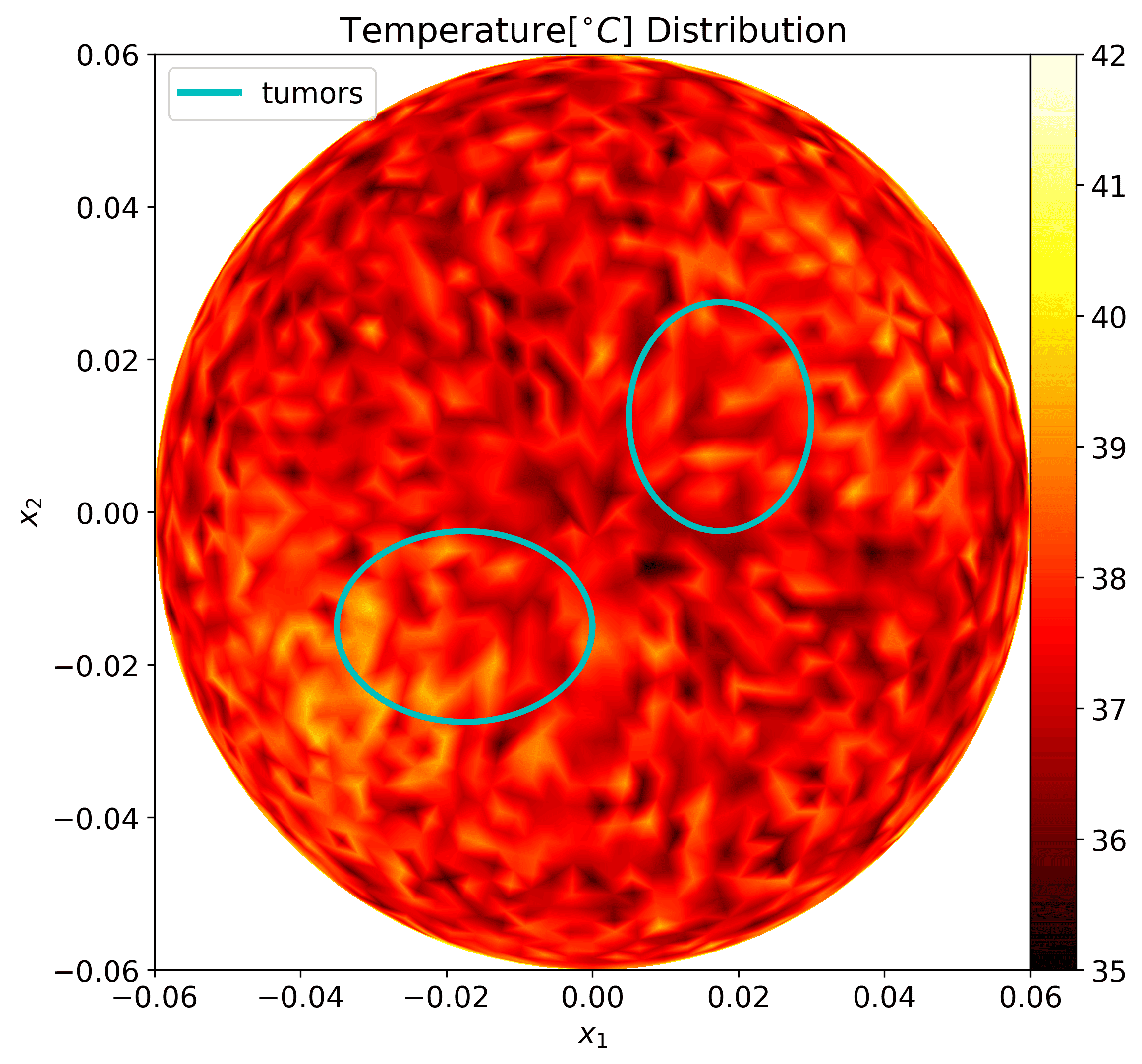}} 
\resizebox{0.235\textwidth}{!}{\includegraphics{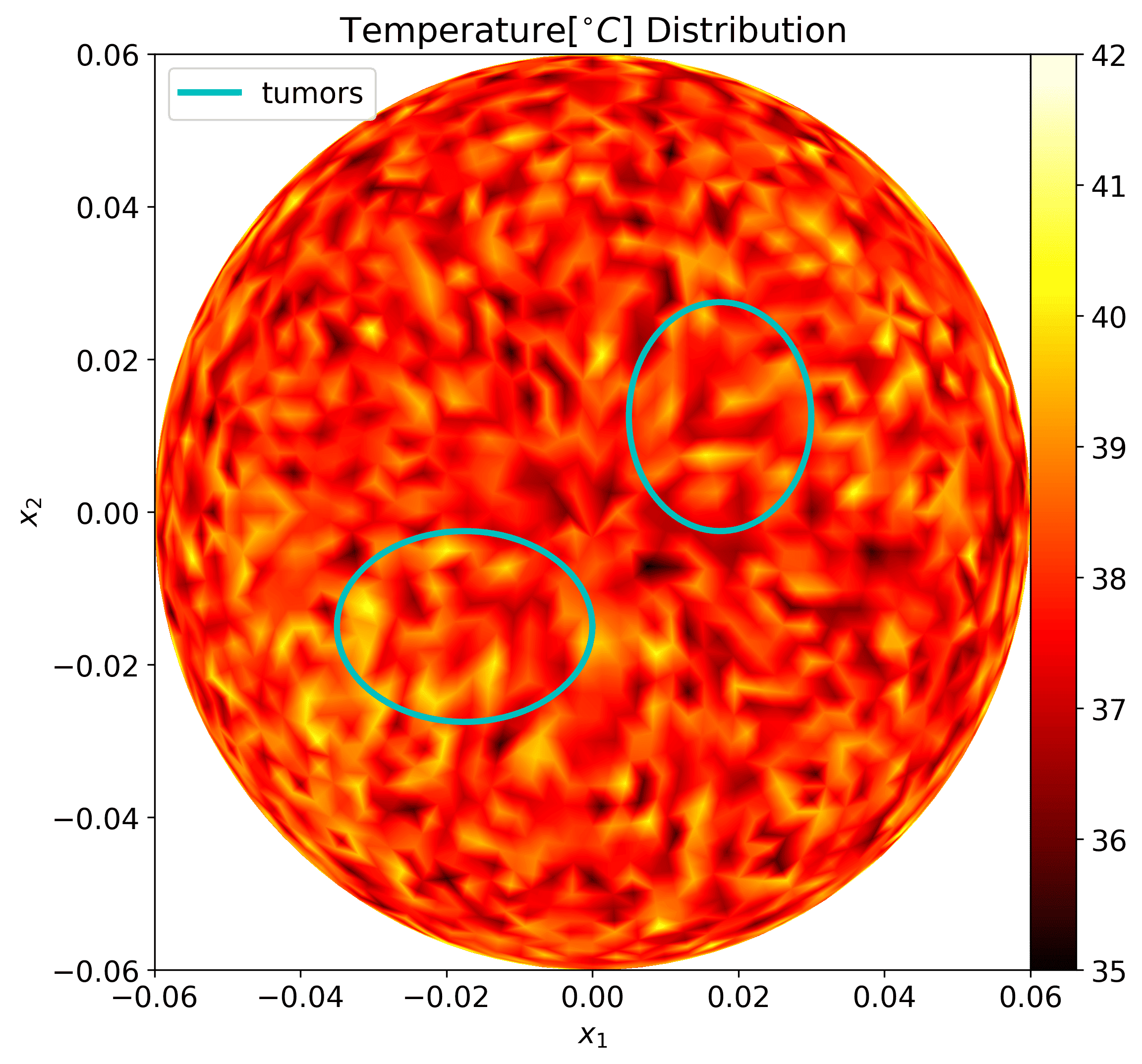}} 
\caption{Exact and noisy temperature distributions on the skin. From left to right: exact ($0\%$ noise), and noisy data with noise levels of $1\%$, $2\%$, and $3\%$, respectively.}
\label{fig:skin_temperature_two_tumors}
\end{figure}
\begin{figure}[htp!]
\centering    
\resizebox{0.18\textwidth}{!}{\includegraphics{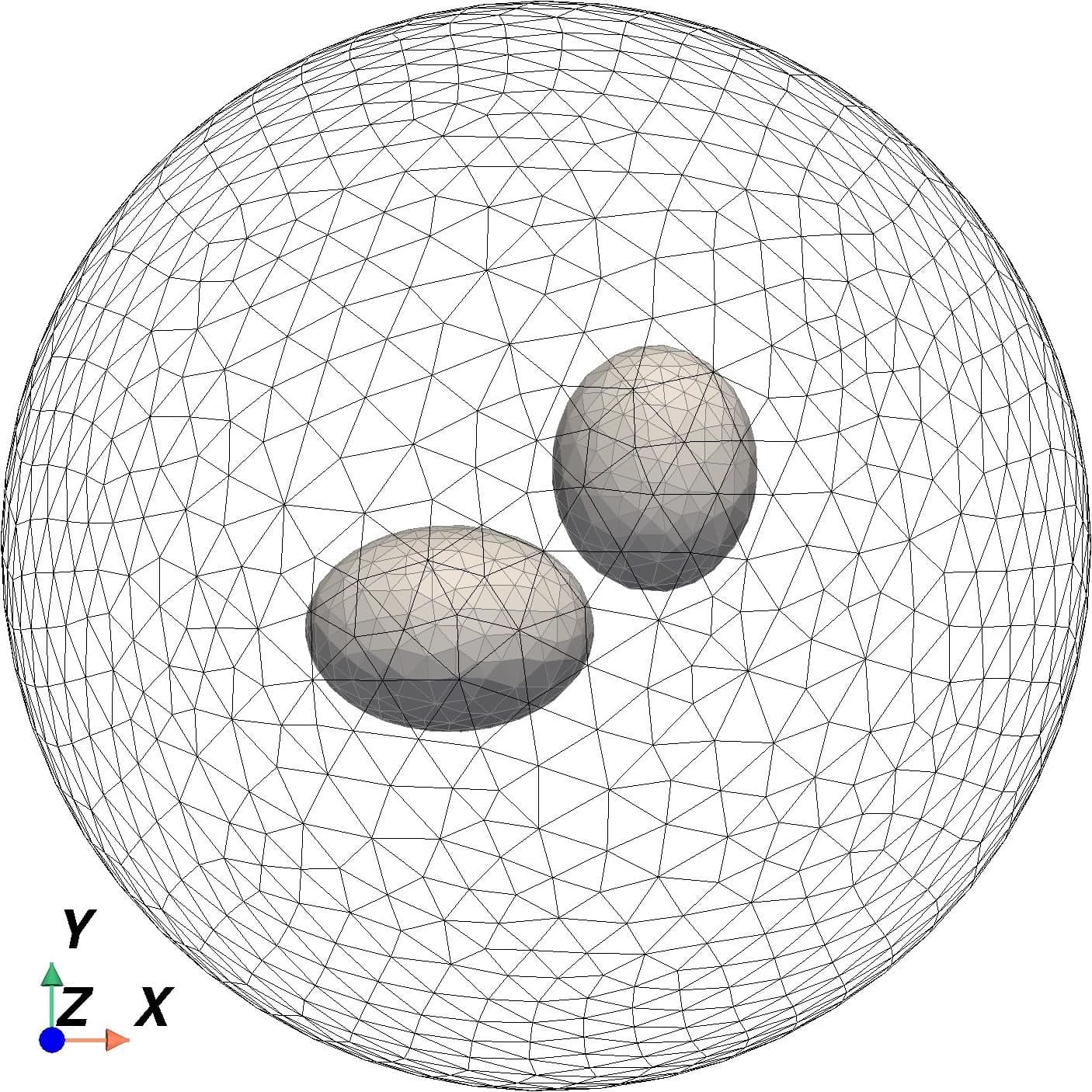}} 
\resizebox{0.18\textwidth}{!}{\includegraphics{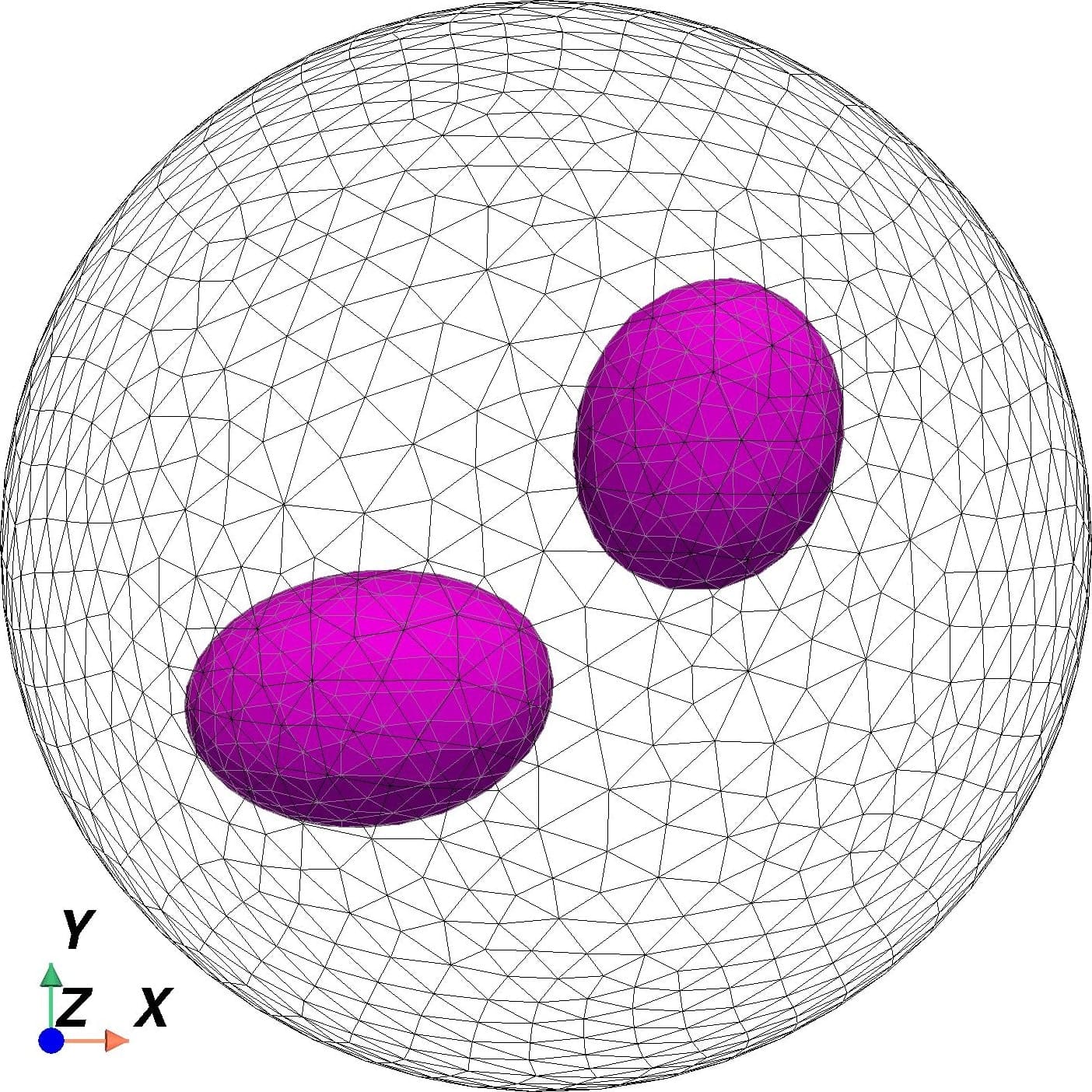}} 
\resizebox{0.18\textwidth}{!}{\includegraphics{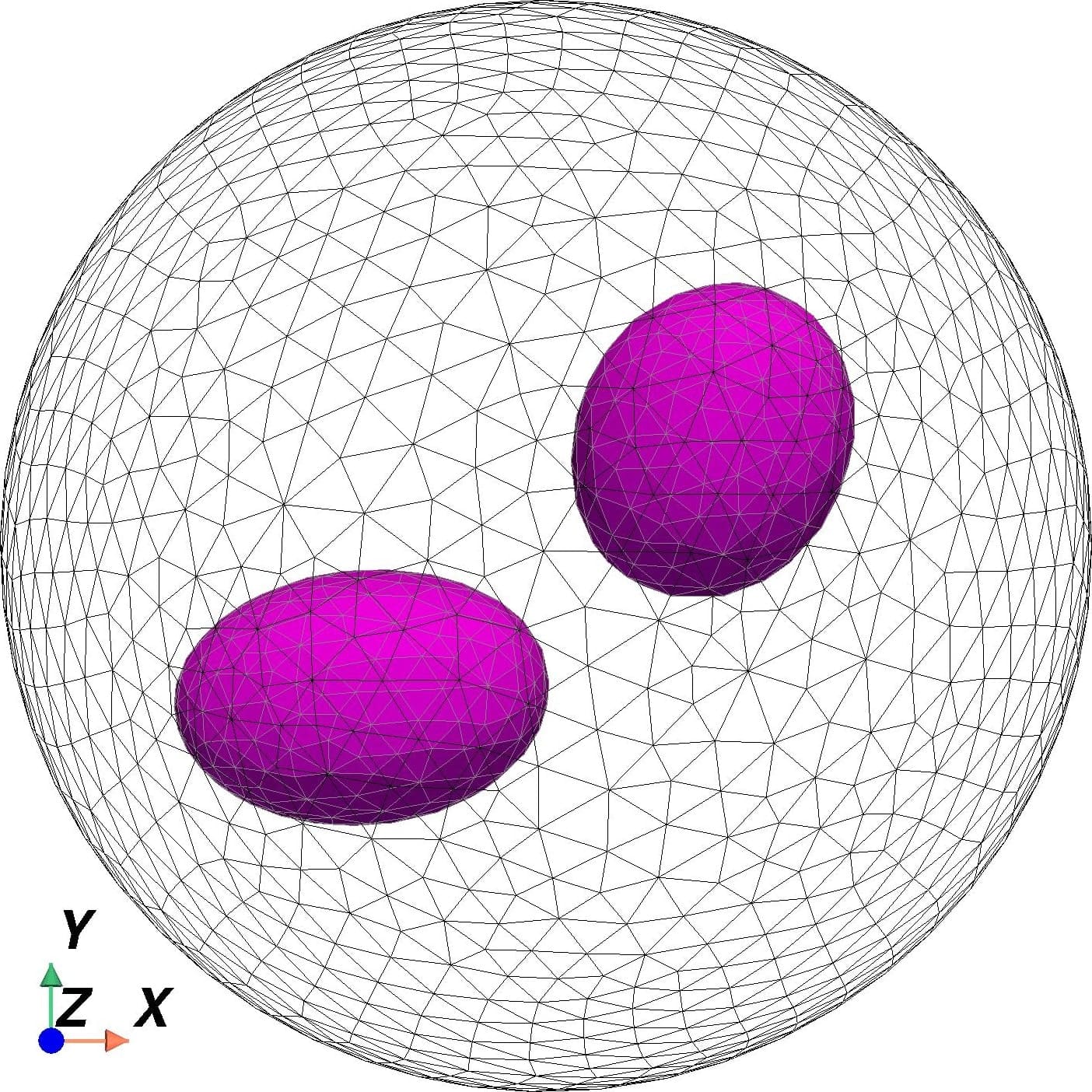}} 
\resizebox{0.18\textwidth}{!}{\includegraphics{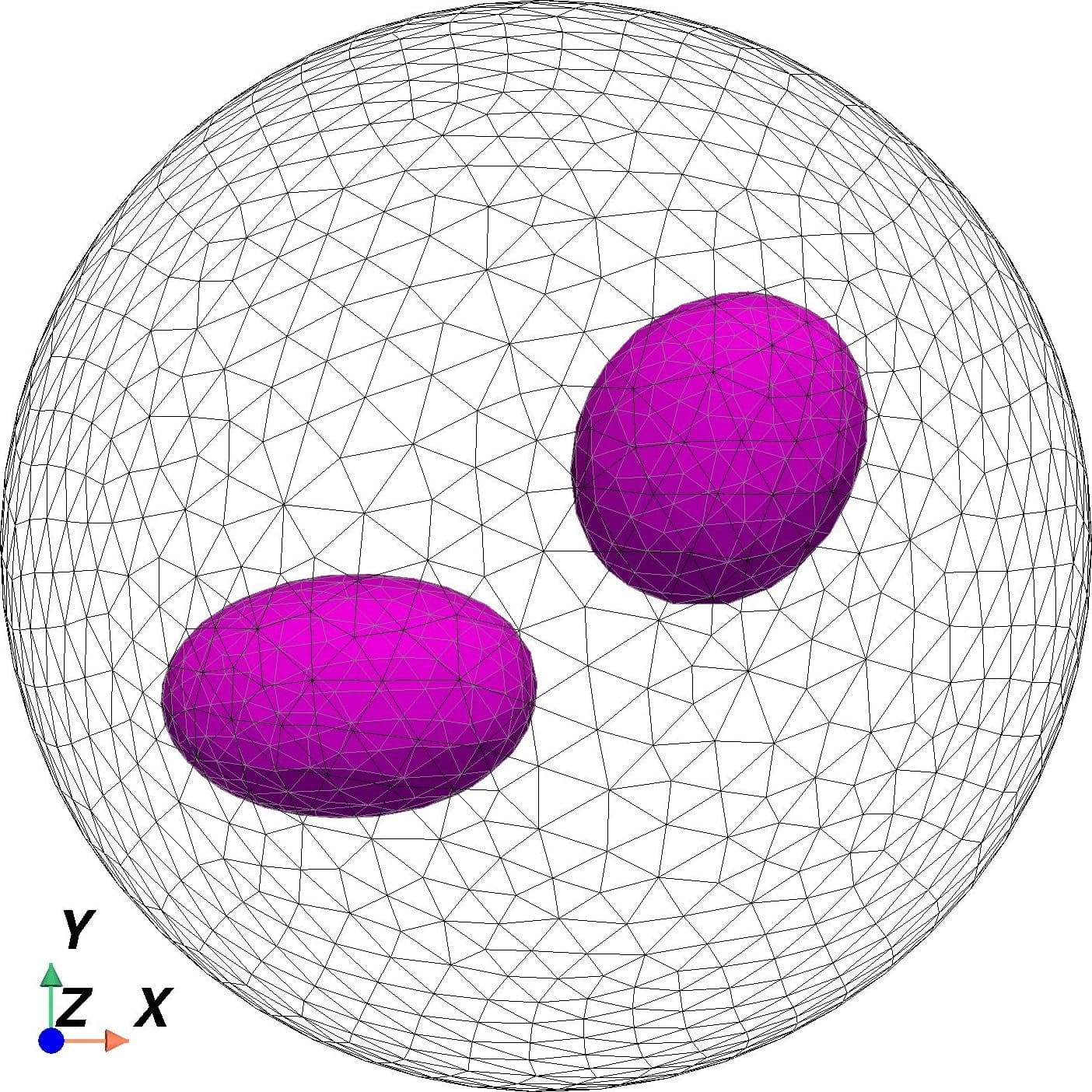}} 
\resizebox{0.18\textwidth}{!}{\includegraphics{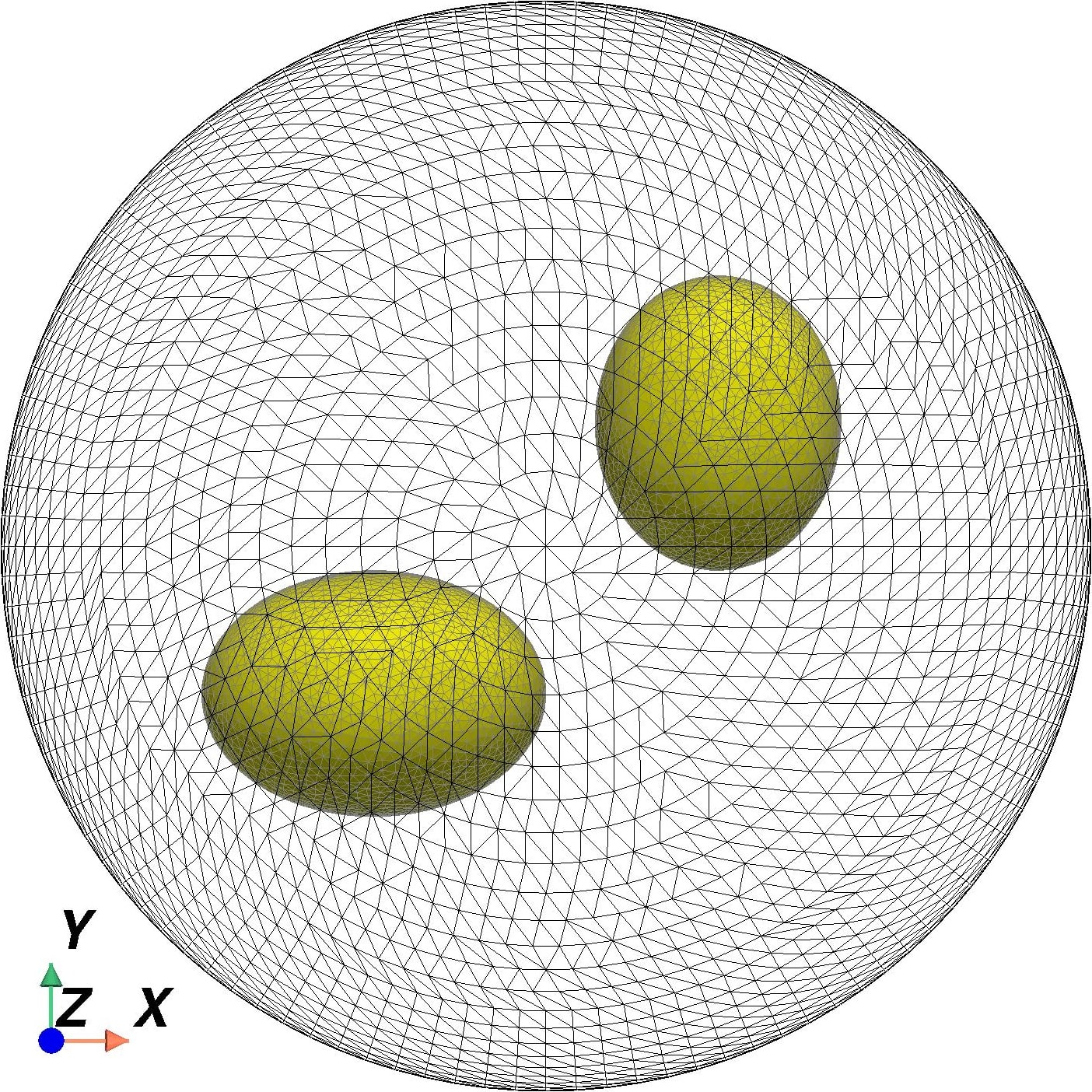}} \\[0.5em]
\resizebox{0.18\textwidth}{!}{\includegraphics{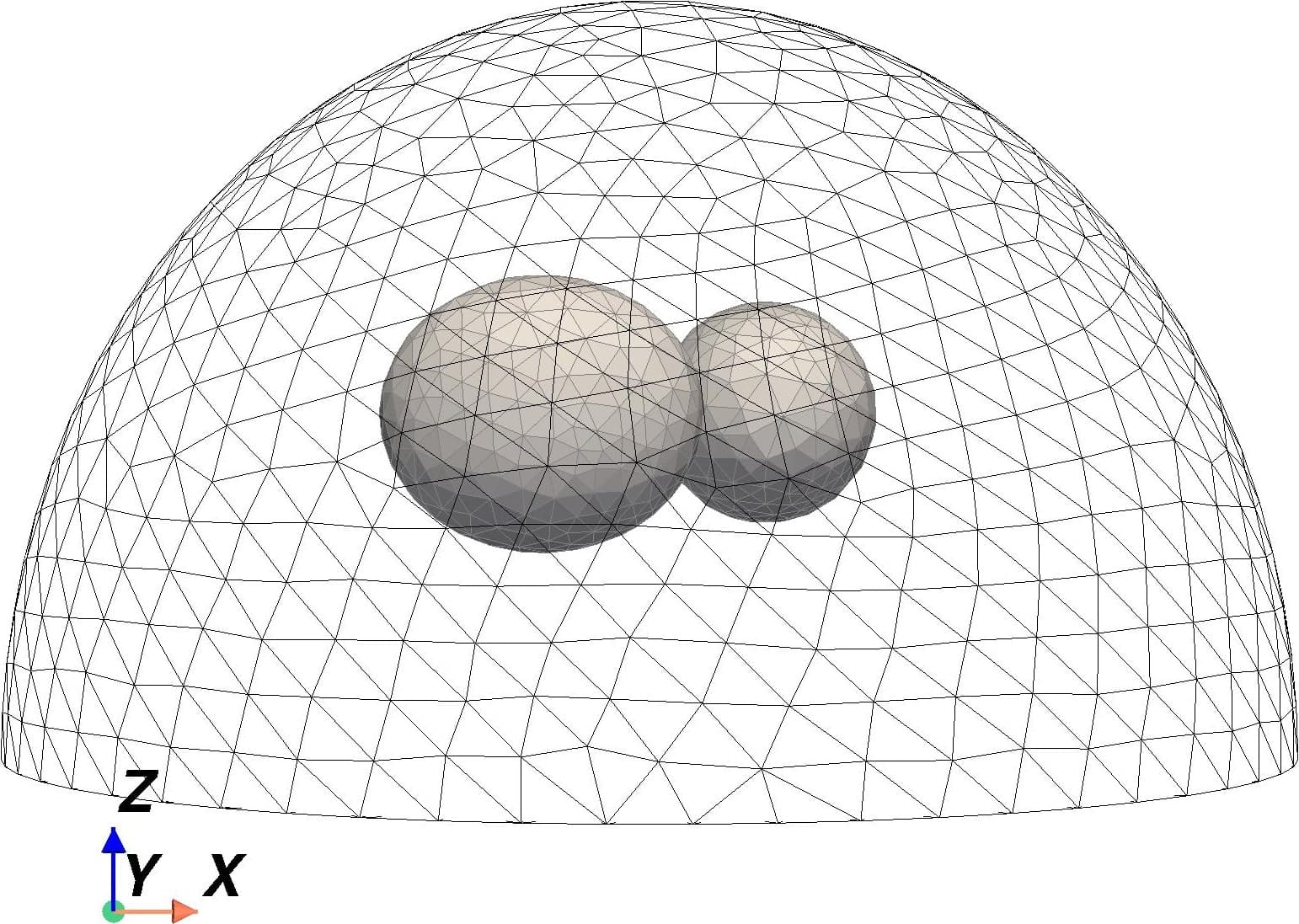}} 
\resizebox{0.18\textwidth}{!}{\includegraphics{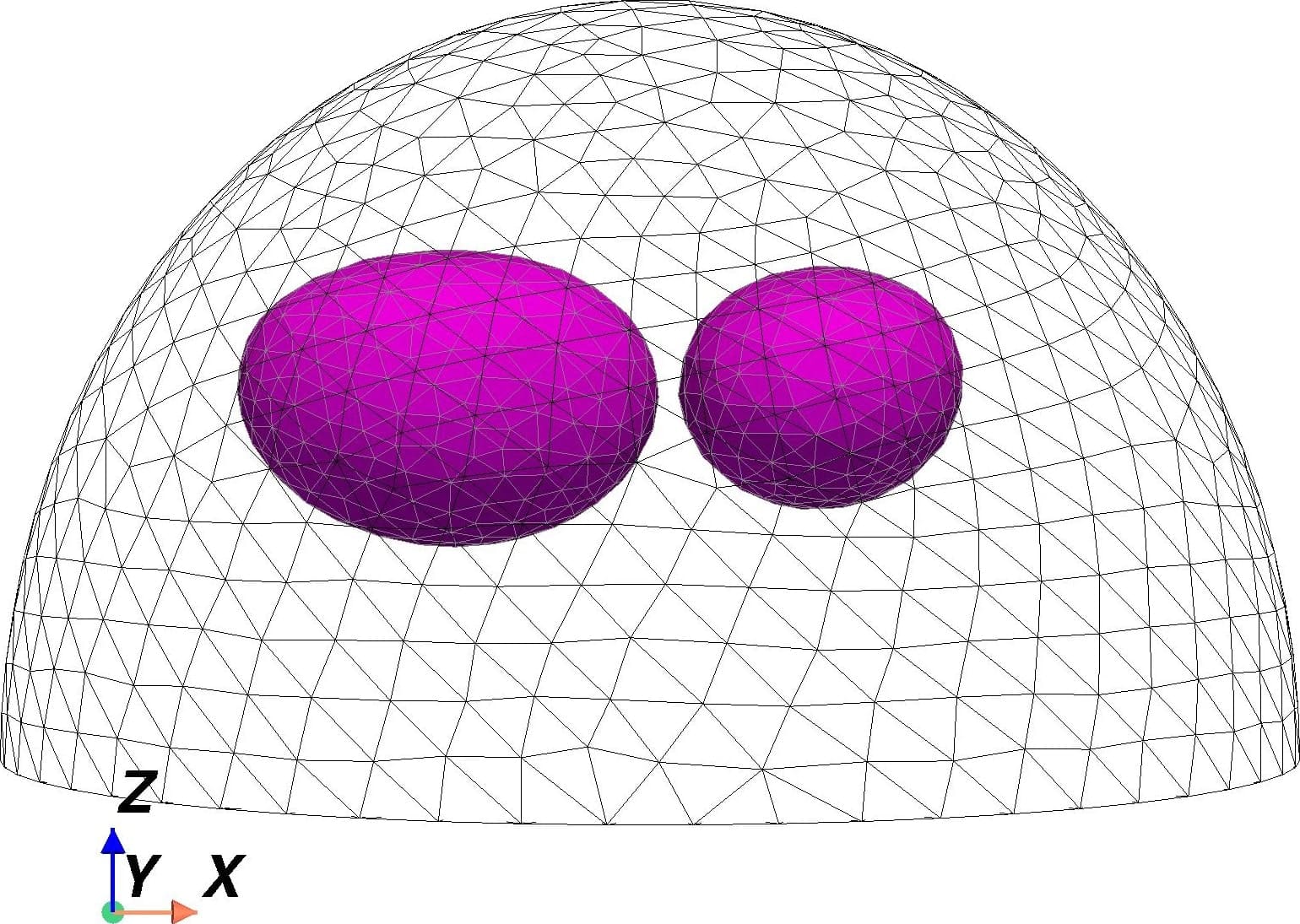}} 
\resizebox{0.18\textwidth}{!}{\includegraphics{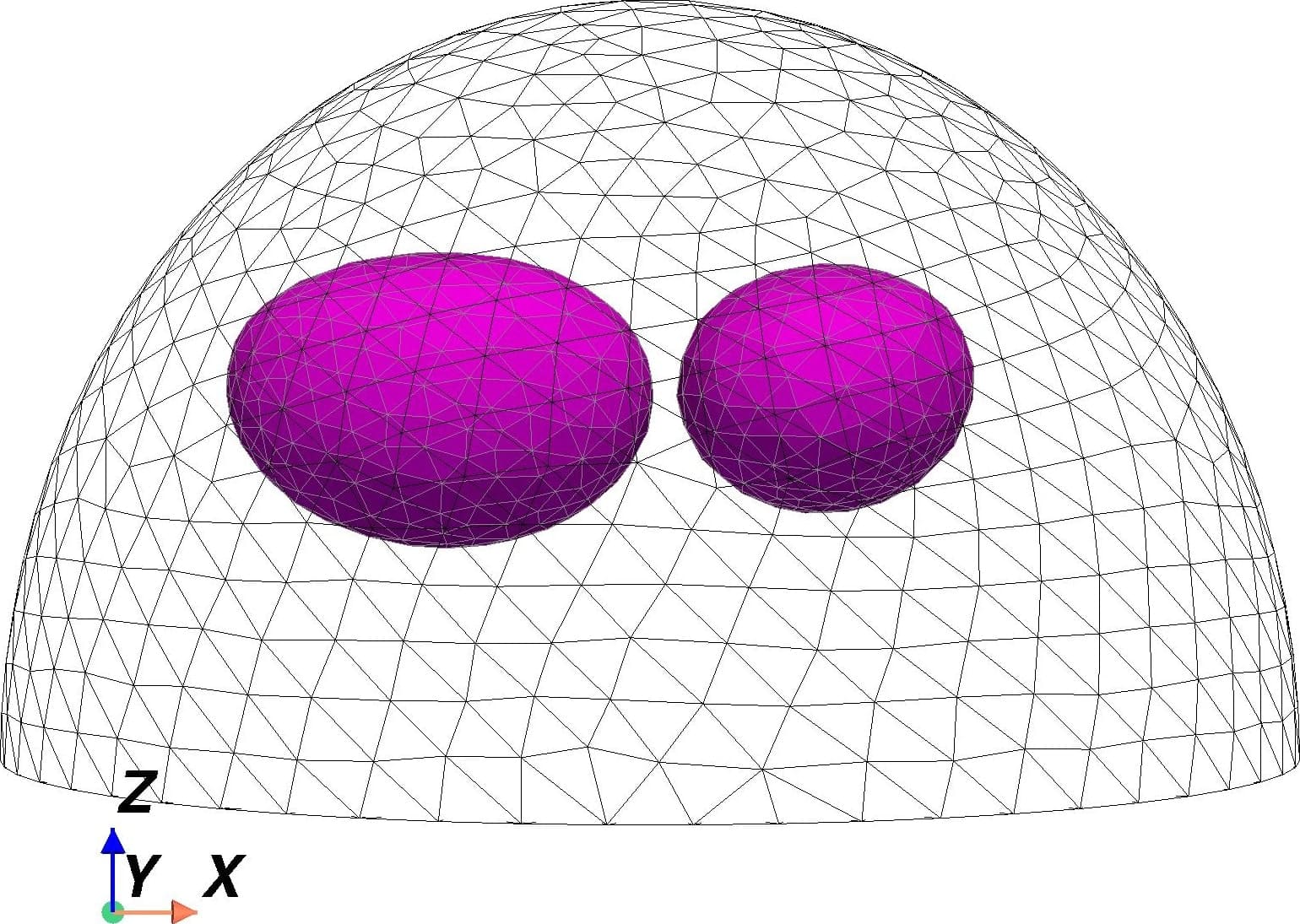}} 
\resizebox{0.18\textwidth}{!}{\includegraphics{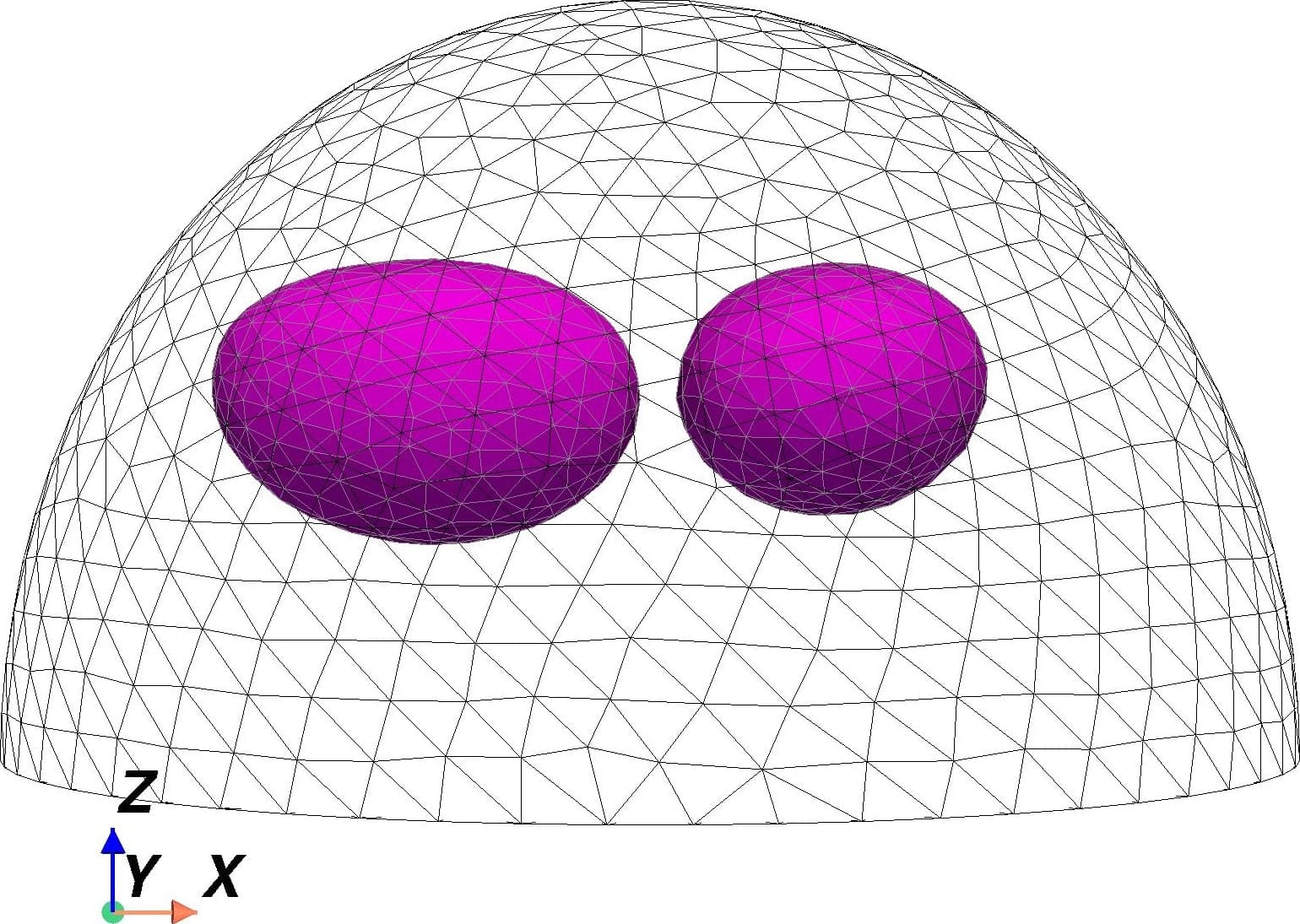}} 
\resizebox{0.18\textwidth}{!}{\includegraphics{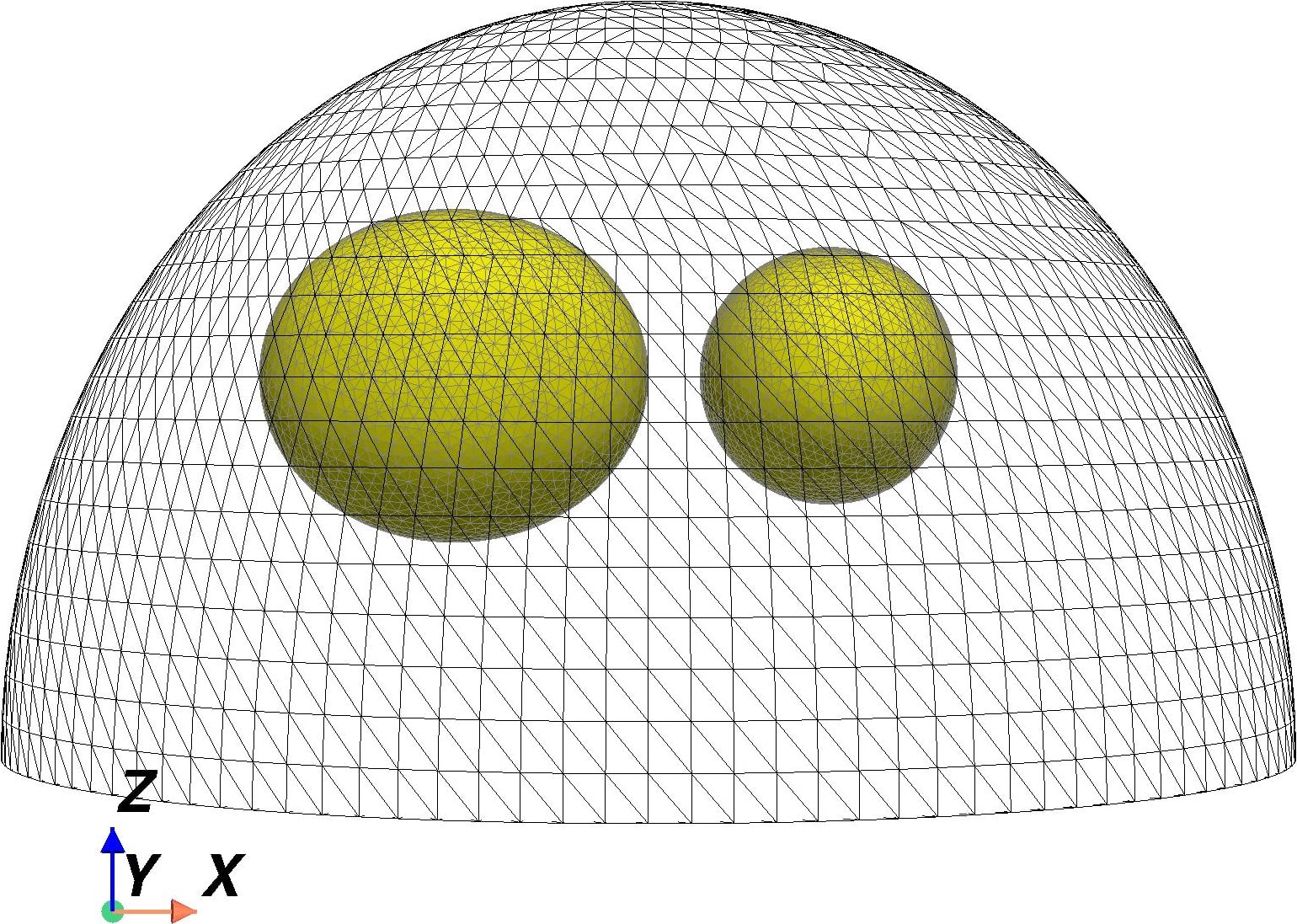}} \\[0.5em]
\resizebox{0.18\textwidth}{!}{\includegraphics{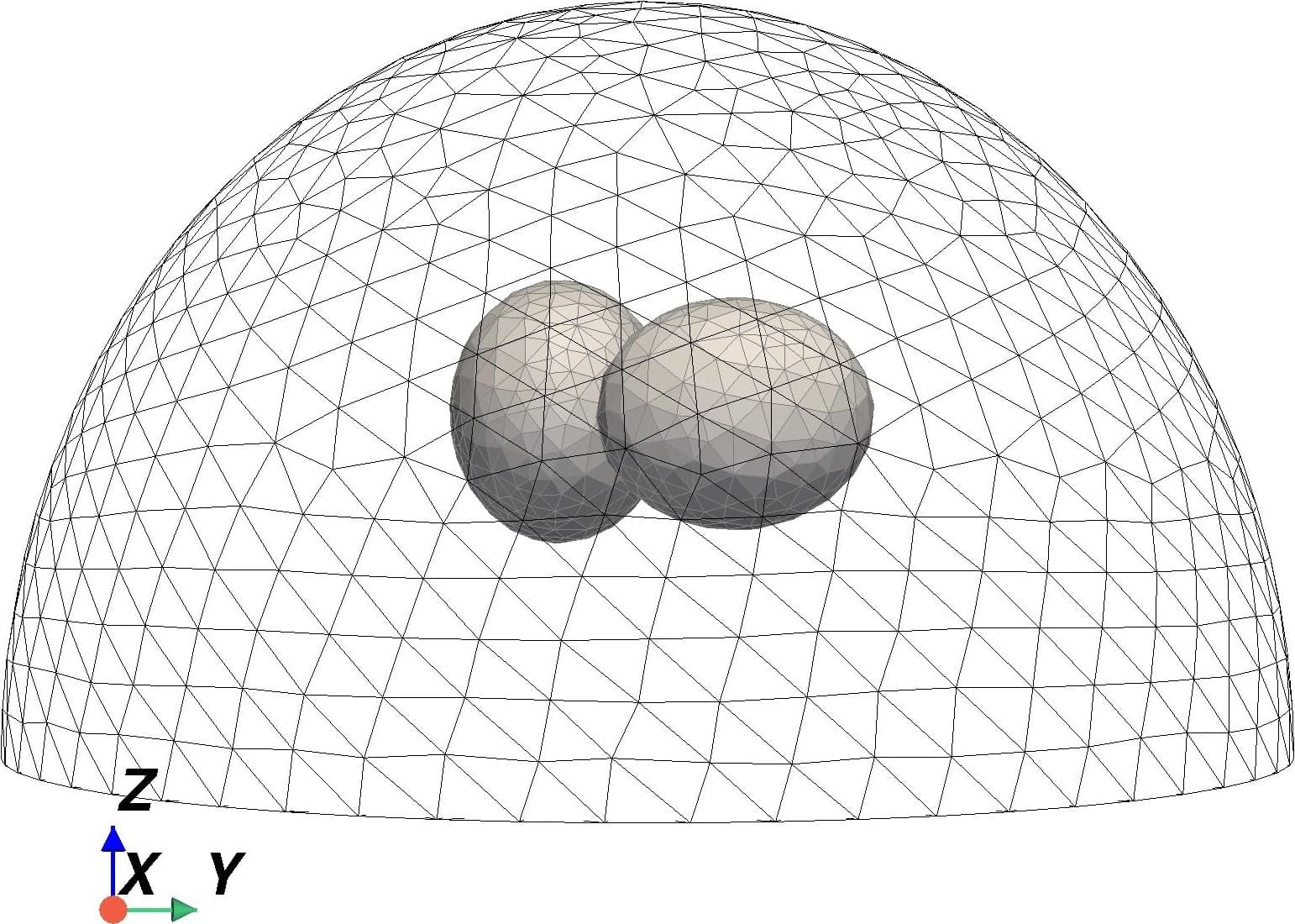}} 
\resizebox{0.18\textwidth}{!}{\includegraphics{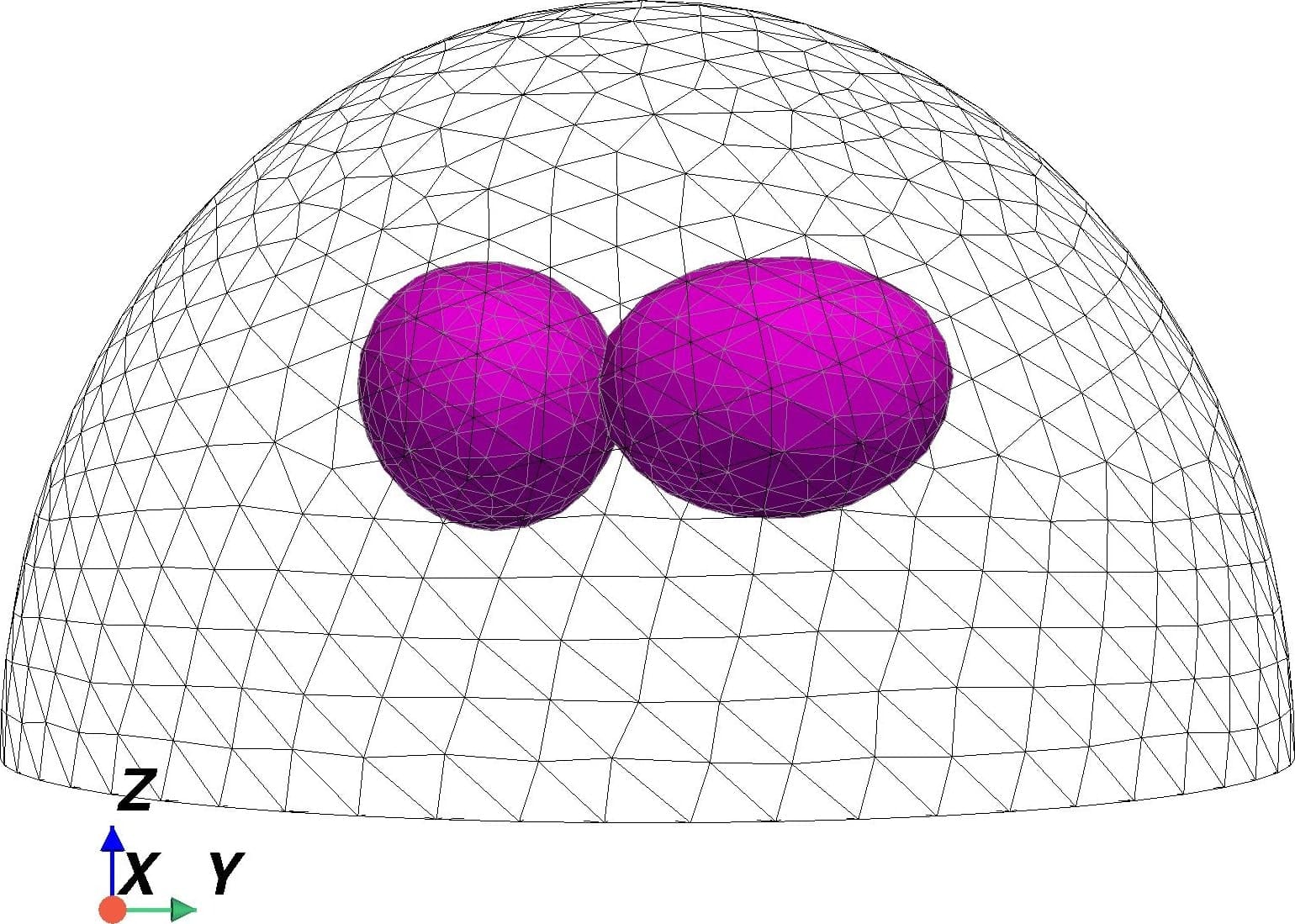}} 
\resizebox{0.18\textwidth}{!}{\includegraphics{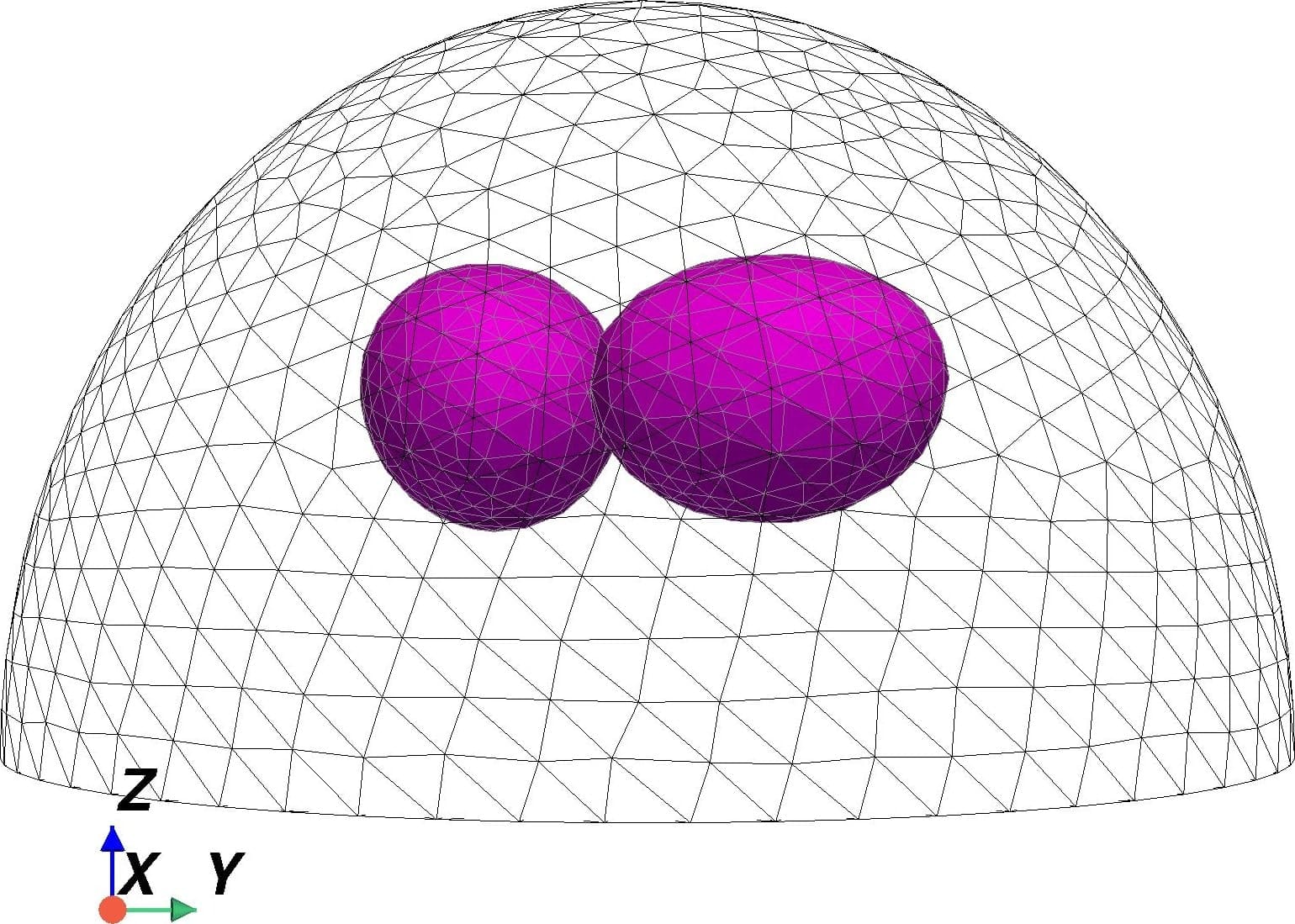}} 
\resizebox{0.18\textwidth}{!}{\includegraphics{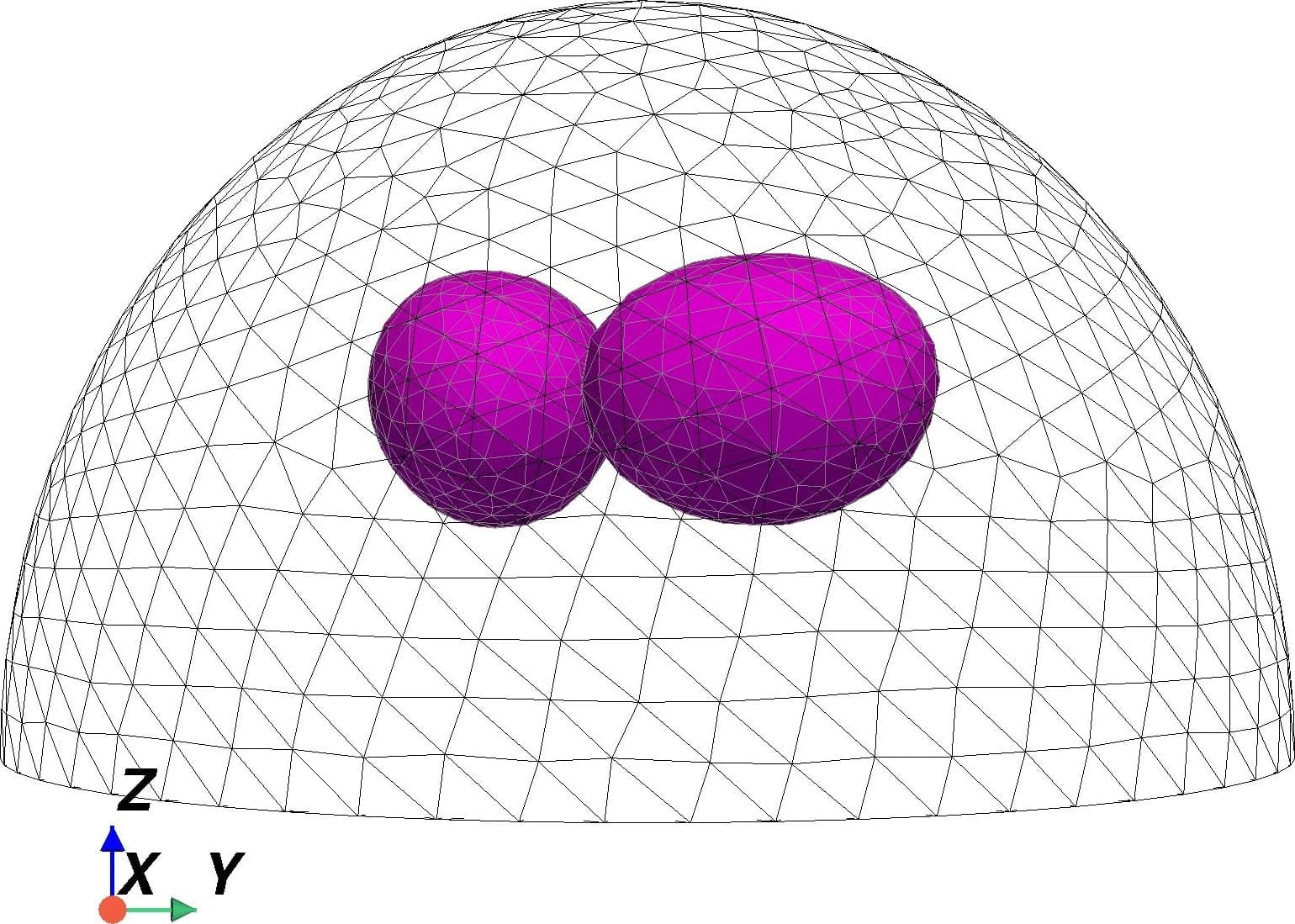}} 
\resizebox{0.18\textwidth}{!}{\includegraphics{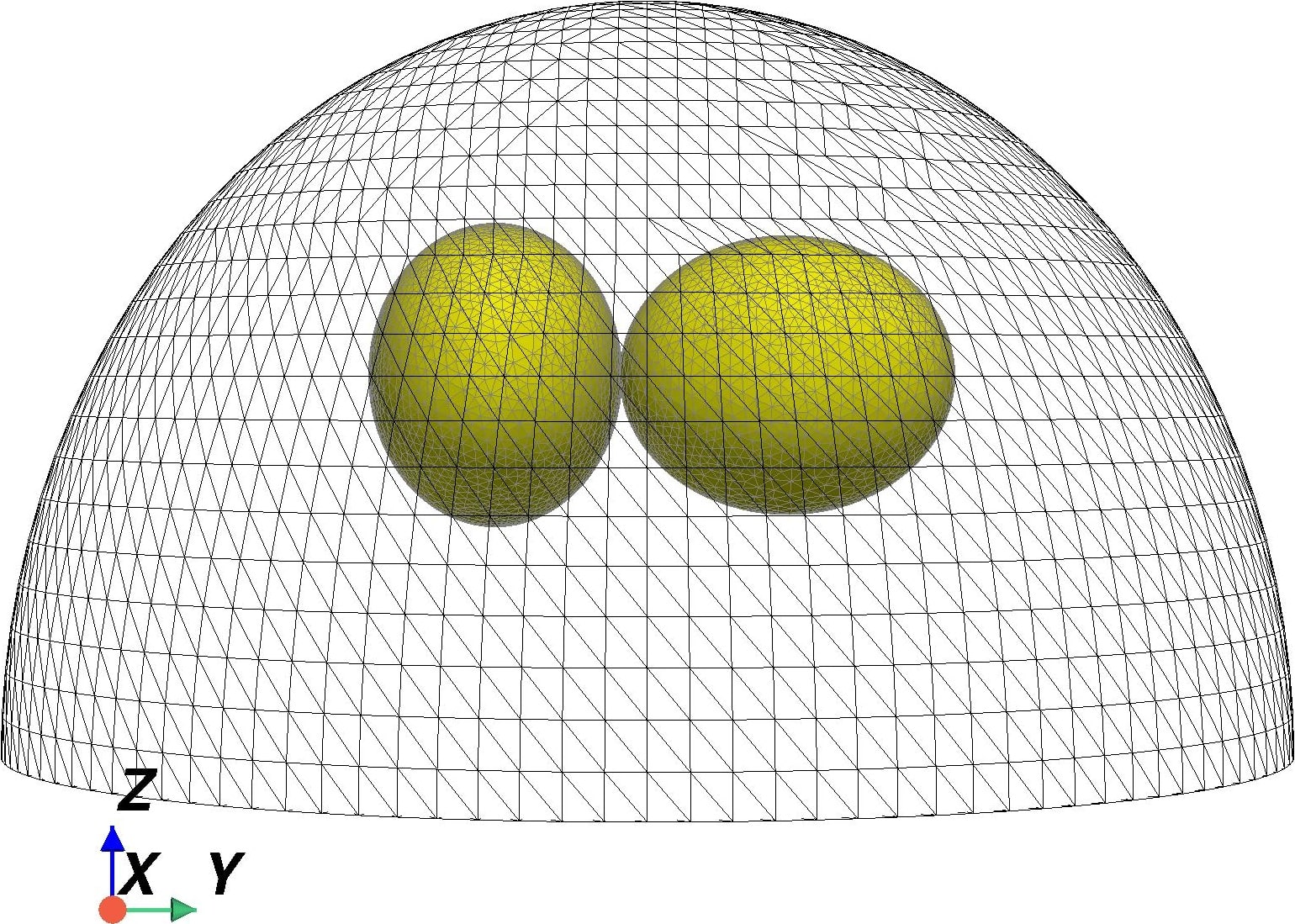}} \\[0.5em]
\caption{Recovered shape (magenta) from an initial guess (light color) compared with the exact tumor (yellow), shown from various views under $1\%$, $2\%$, and $3\%$ (second column to third column) noisy data.}
\label{fig:3d_results_comparisons_multiple_tumors}
\end{figure}
\begin{figure}[htp!]
\centering   
\resizebox{0.235\textwidth}{!}{\includegraphics{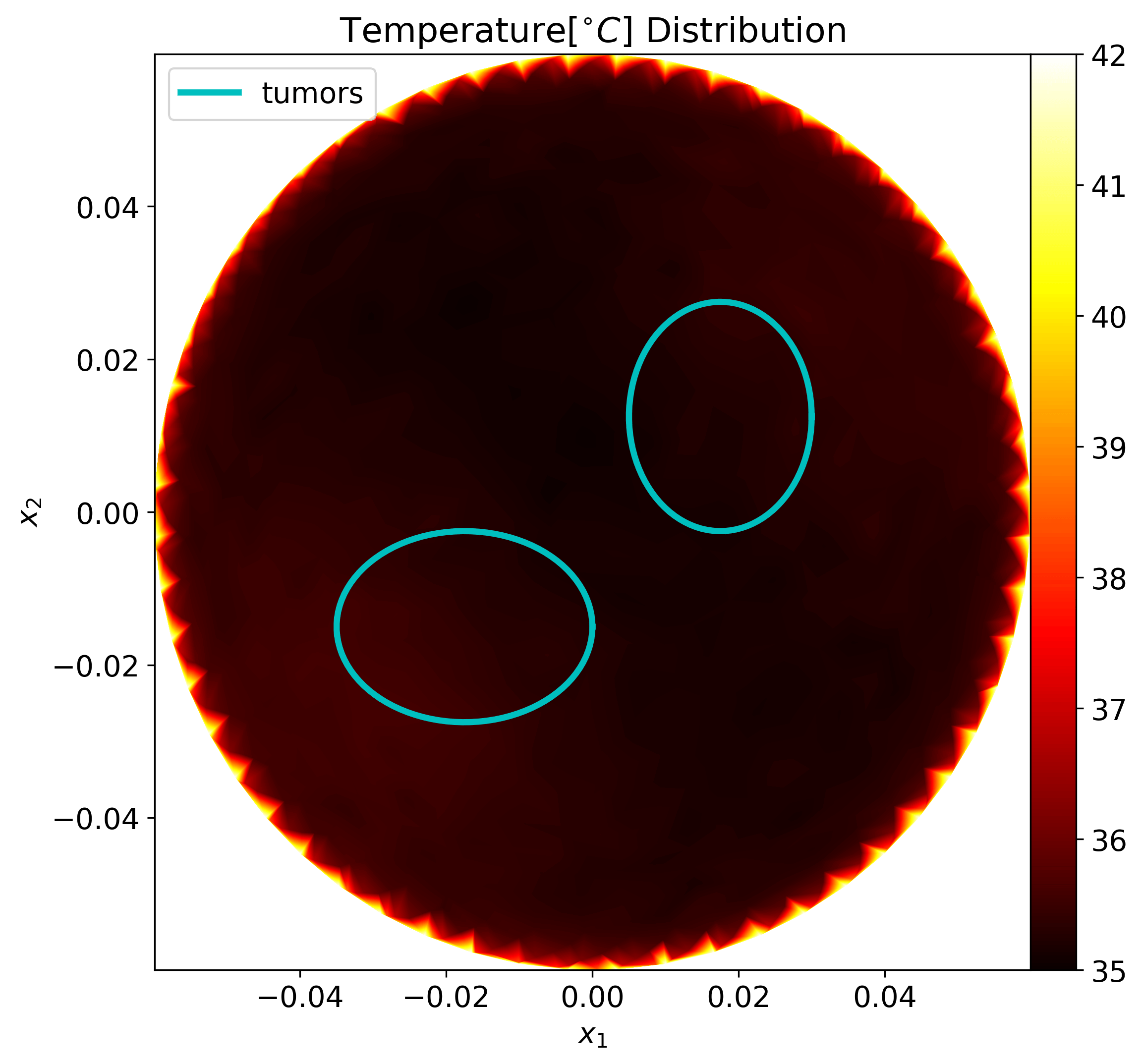}} 
\resizebox{0.235\textwidth}{!}{\includegraphics{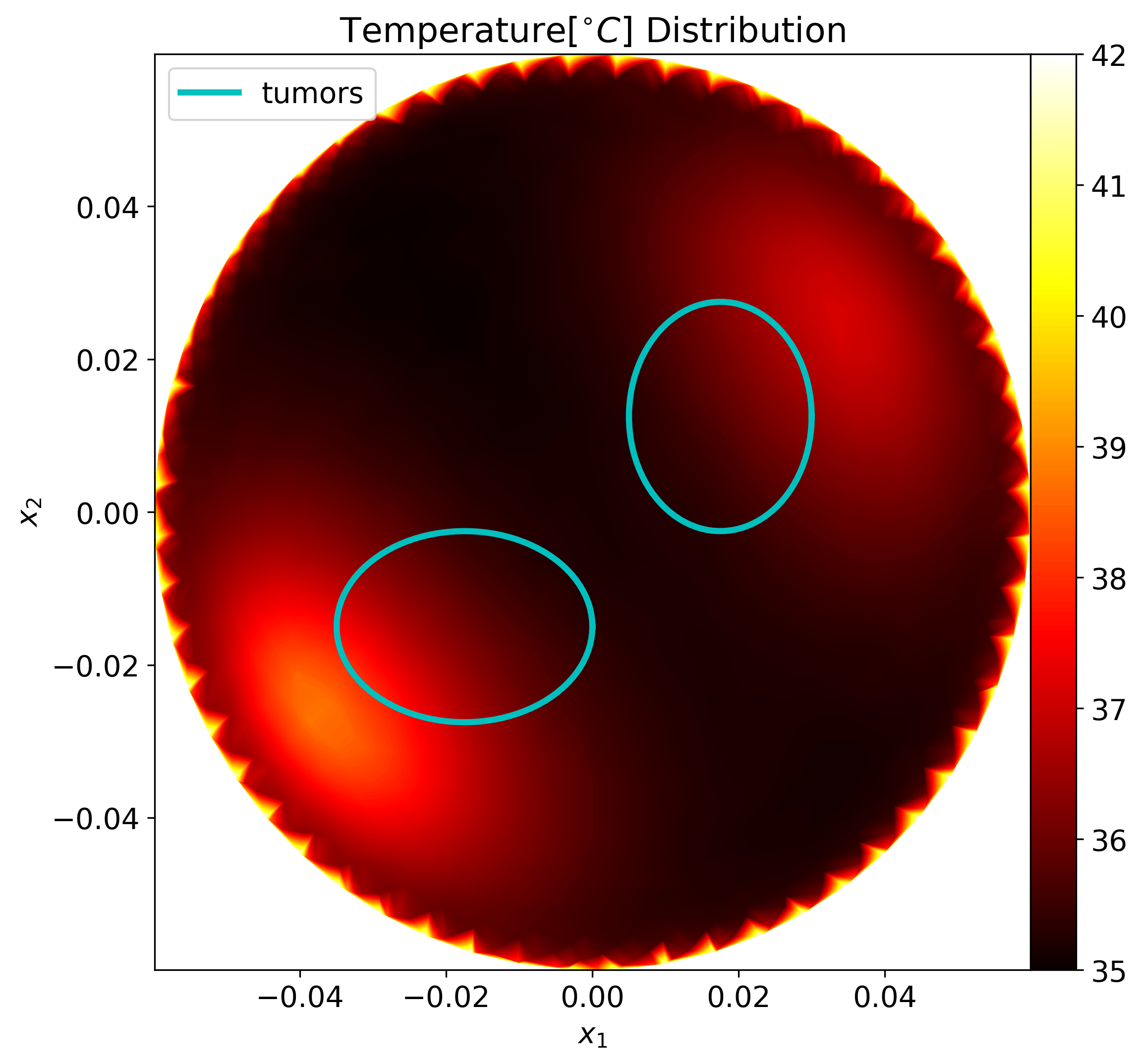}} 
\resizebox{0.235\textwidth}{!}{\includegraphics{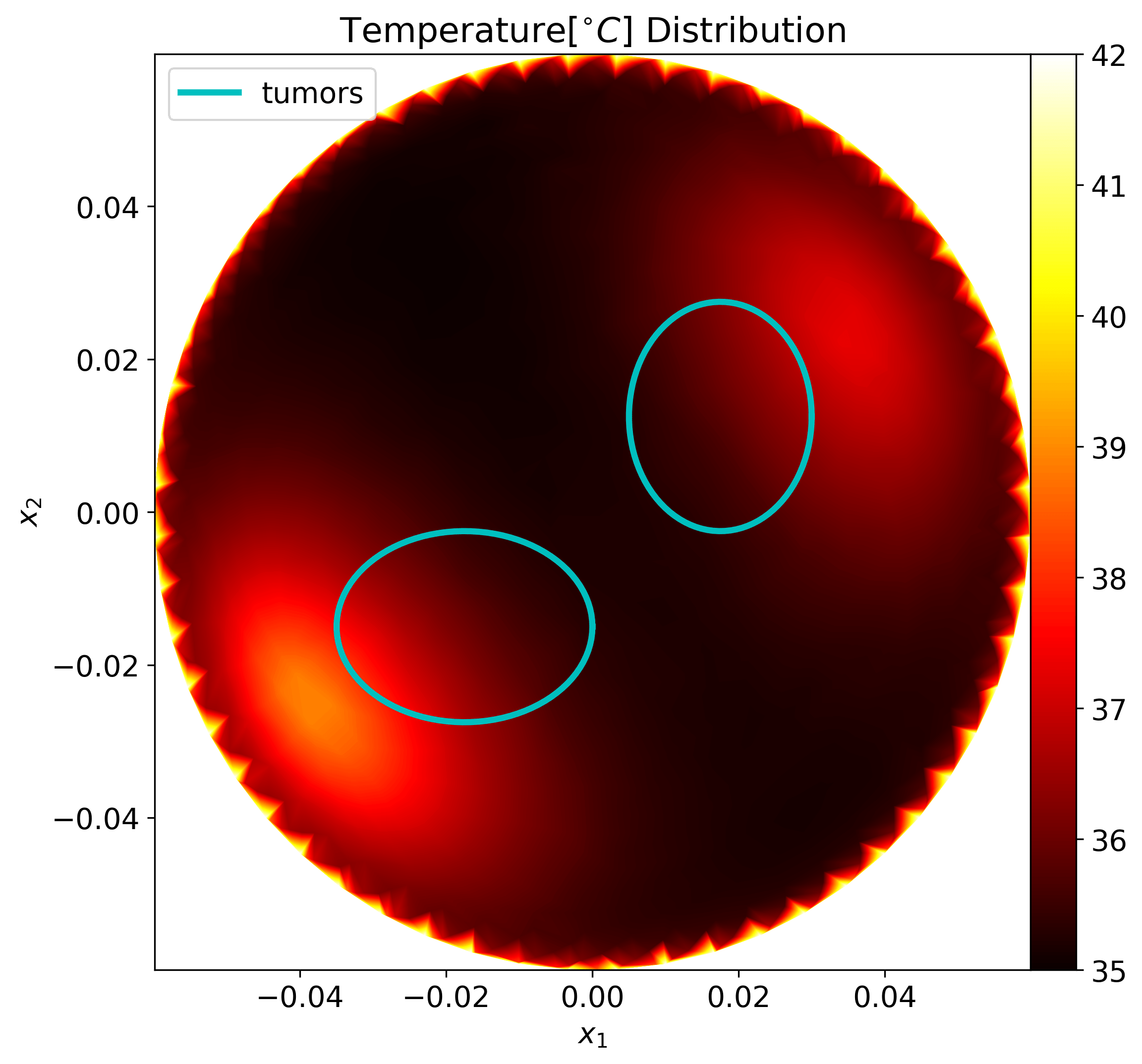}} 
\resizebox{0.235\textwidth}{!}{\includegraphics{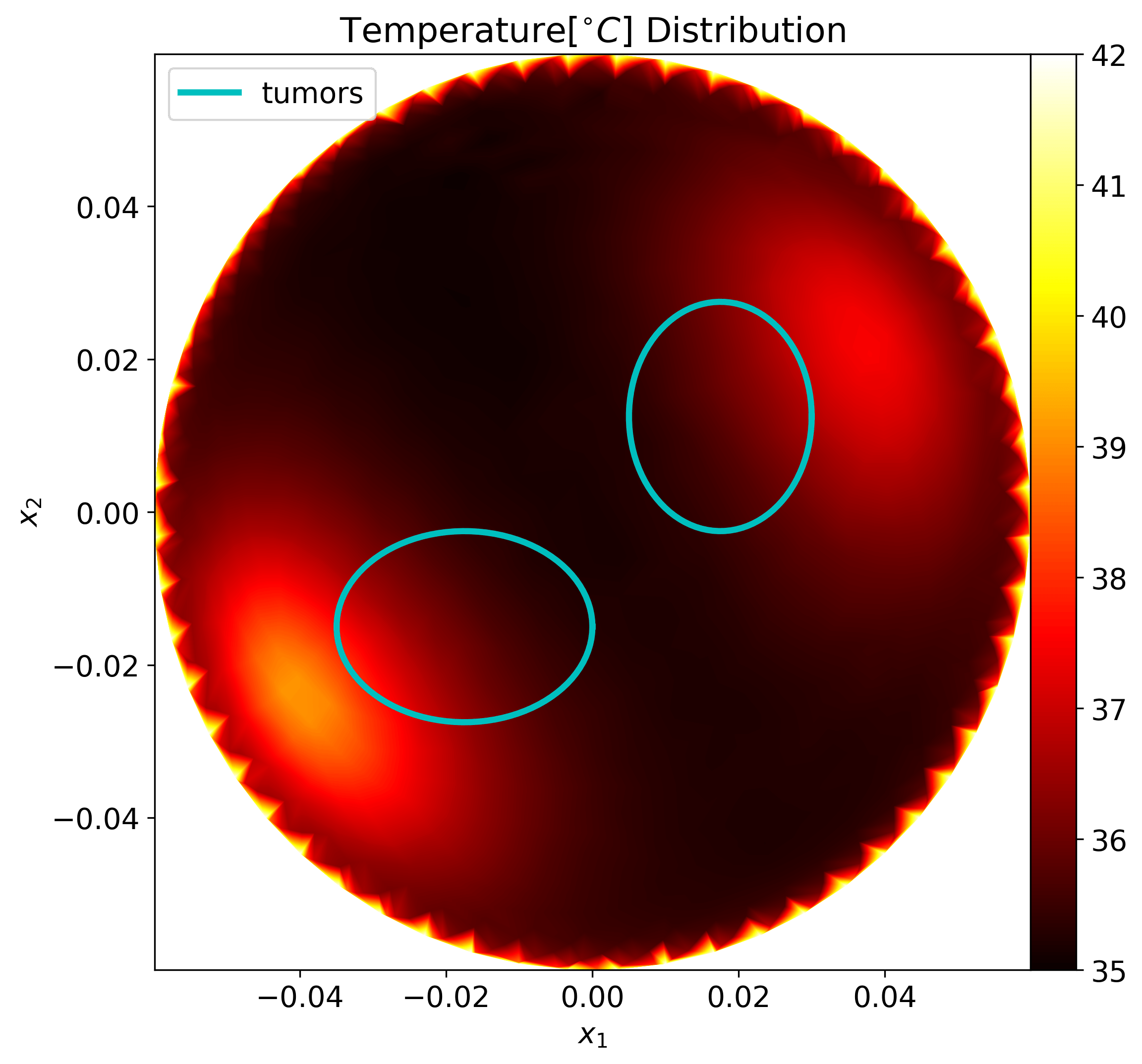}} 
\caption{Exact and noisy temperature distributions on the skin surface. From left to right: the initial noisy measurement, followed by the reconstructed temperature distributions at the final shape, using measured temperature data with added noise levels of $1\%$, $2\%$, and $3\%$, respectively.}
\label{fig:skin_temperature_two_tumors_real_and_imaginary_parts}
\end{figure}

\begin{figure}[htp!]
\centering   
\resizebox{0.15\textwidth}{!}{\includegraphics{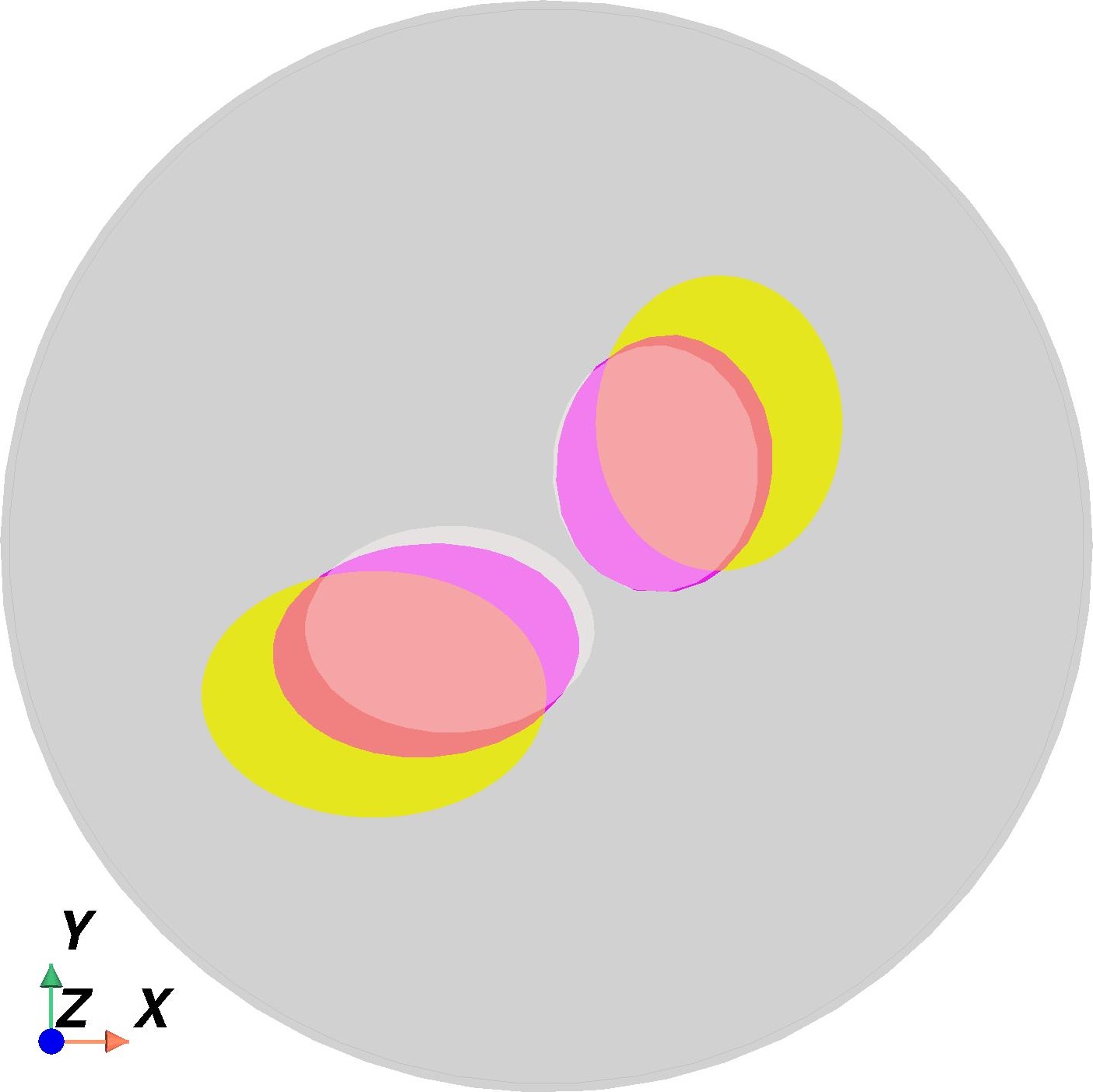}}  
\resizebox{0.15\textwidth}{!}{\includegraphics{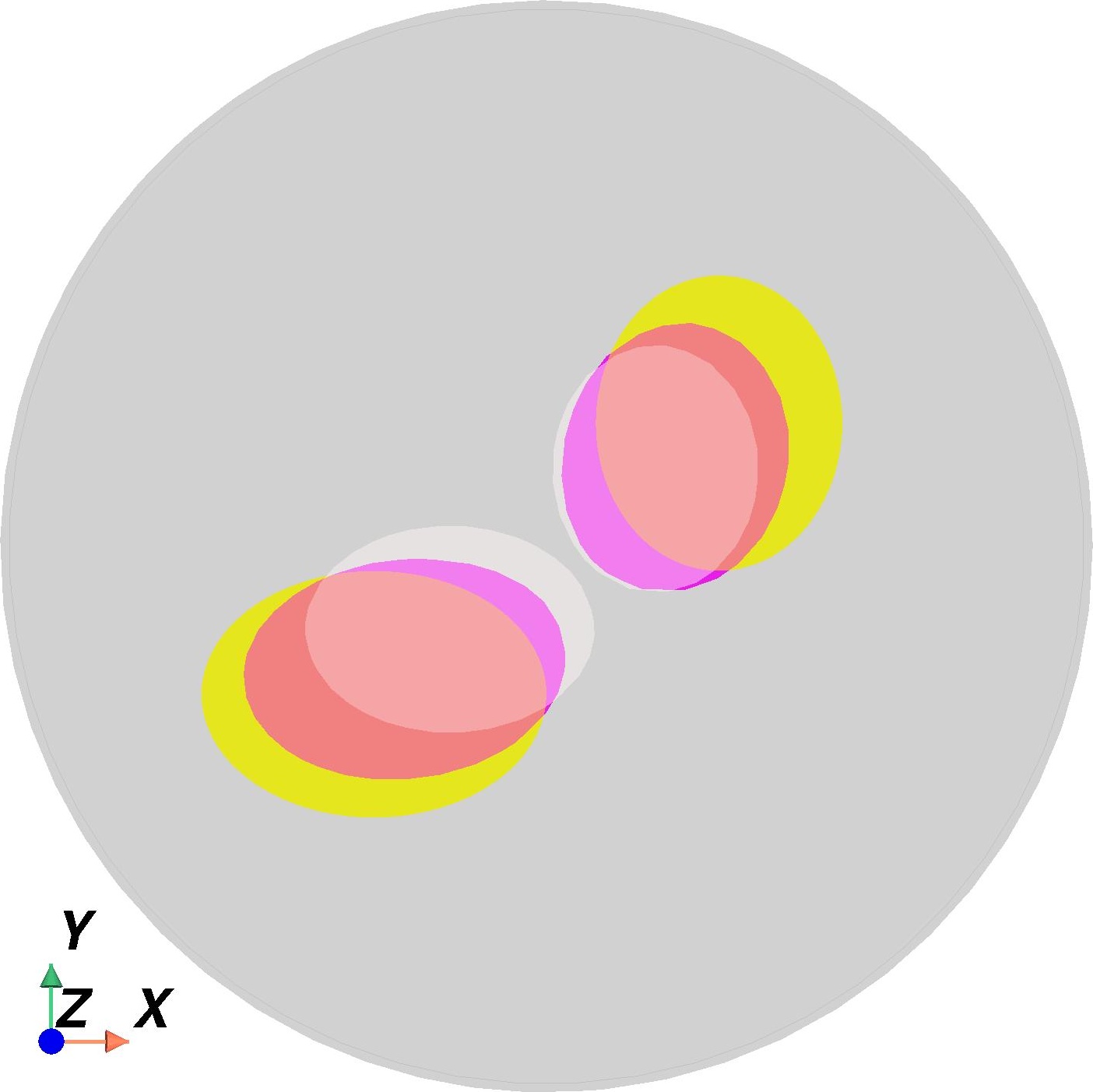}}  
\resizebox{0.15\textwidth}{!}{\includegraphics{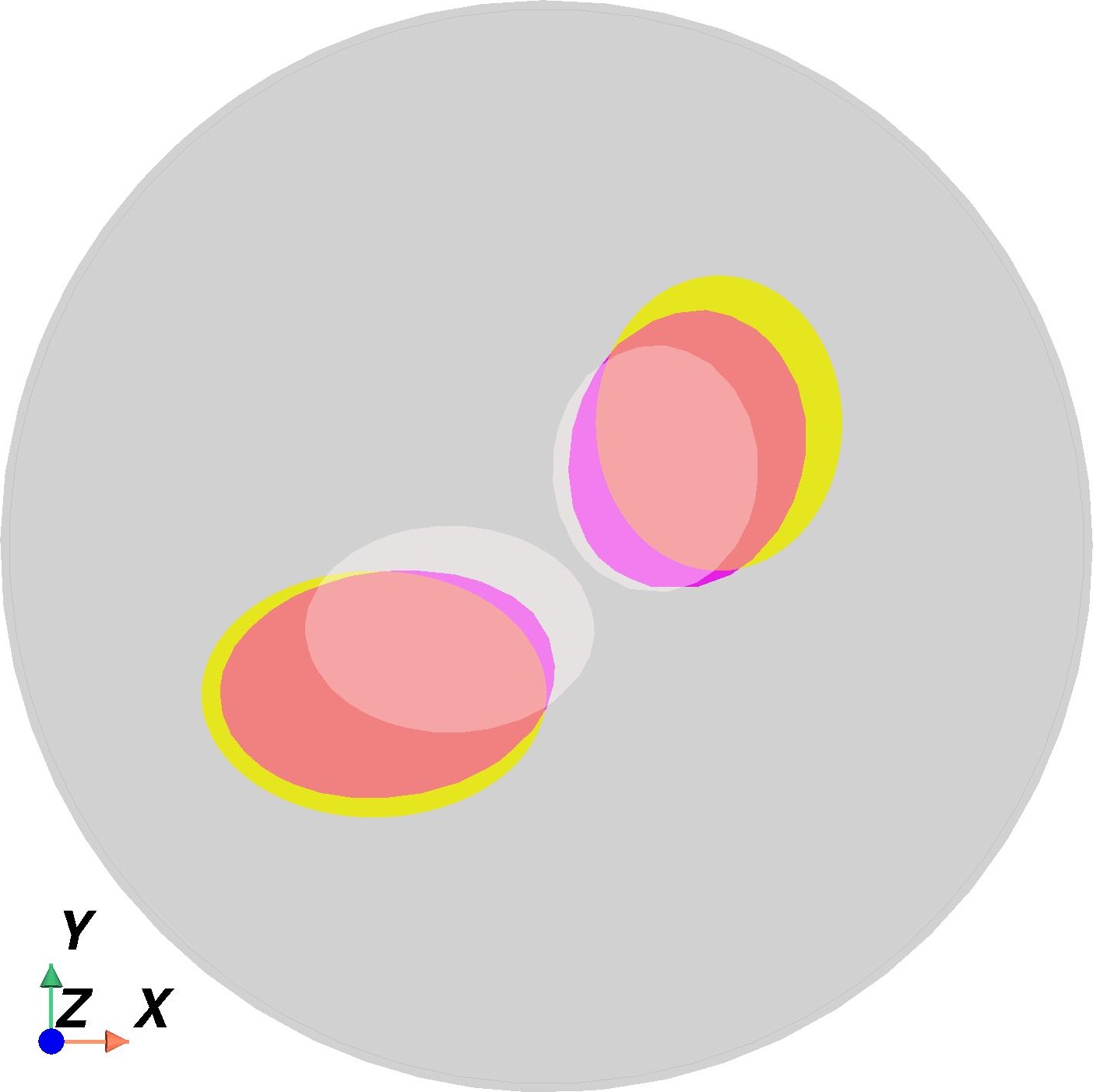}}  
\resizebox{0.15\textwidth}{!}{\includegraphics{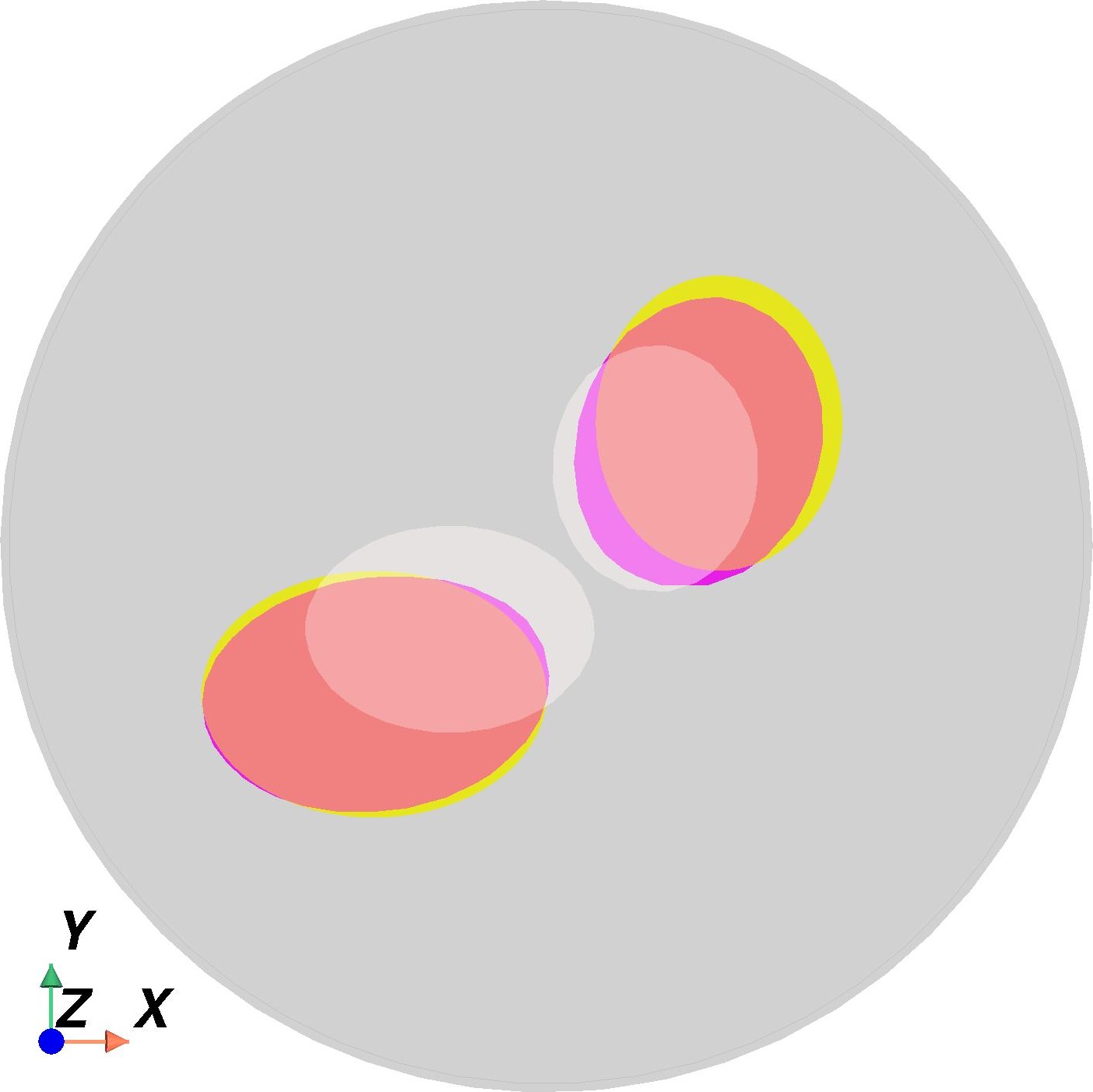}}  
\resizebox{0.15\textwidth}{!}{\includegraphics{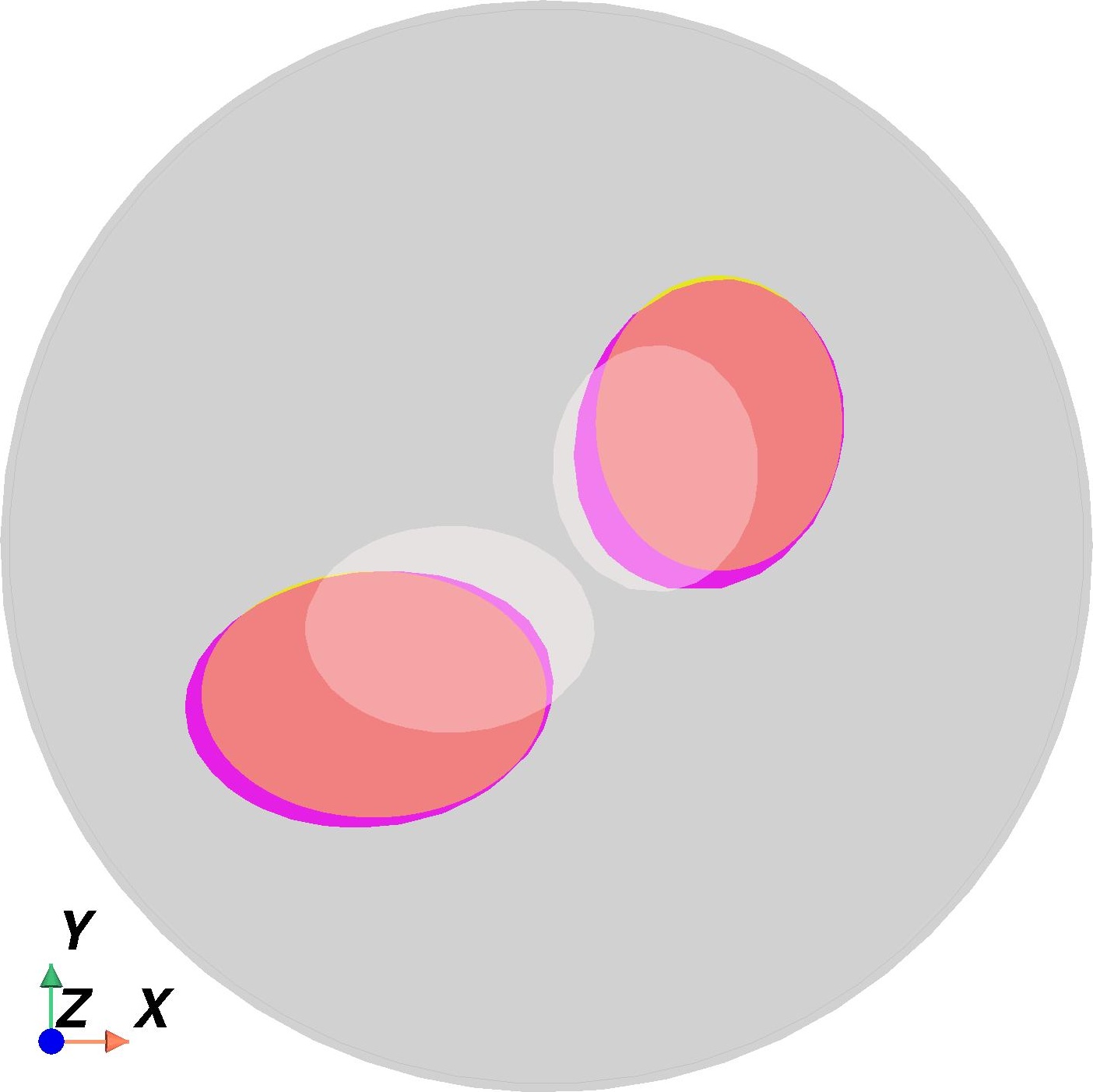}}  
\resizebox{0.15\textwidth}{!}{\includegraphics{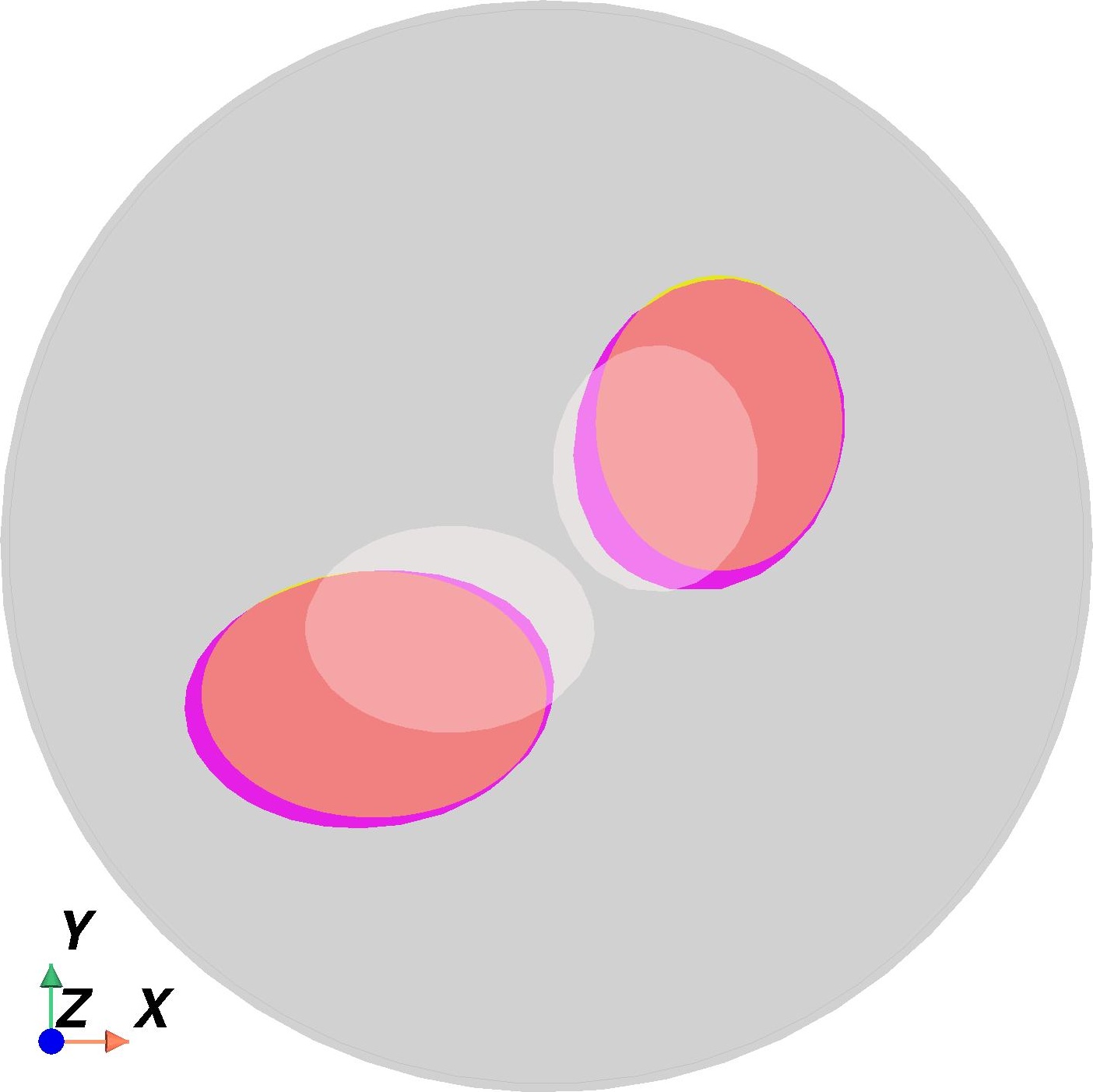}}\\[0.5em]
\resizebox{0.15\textwidth}{!}{\includegraphics{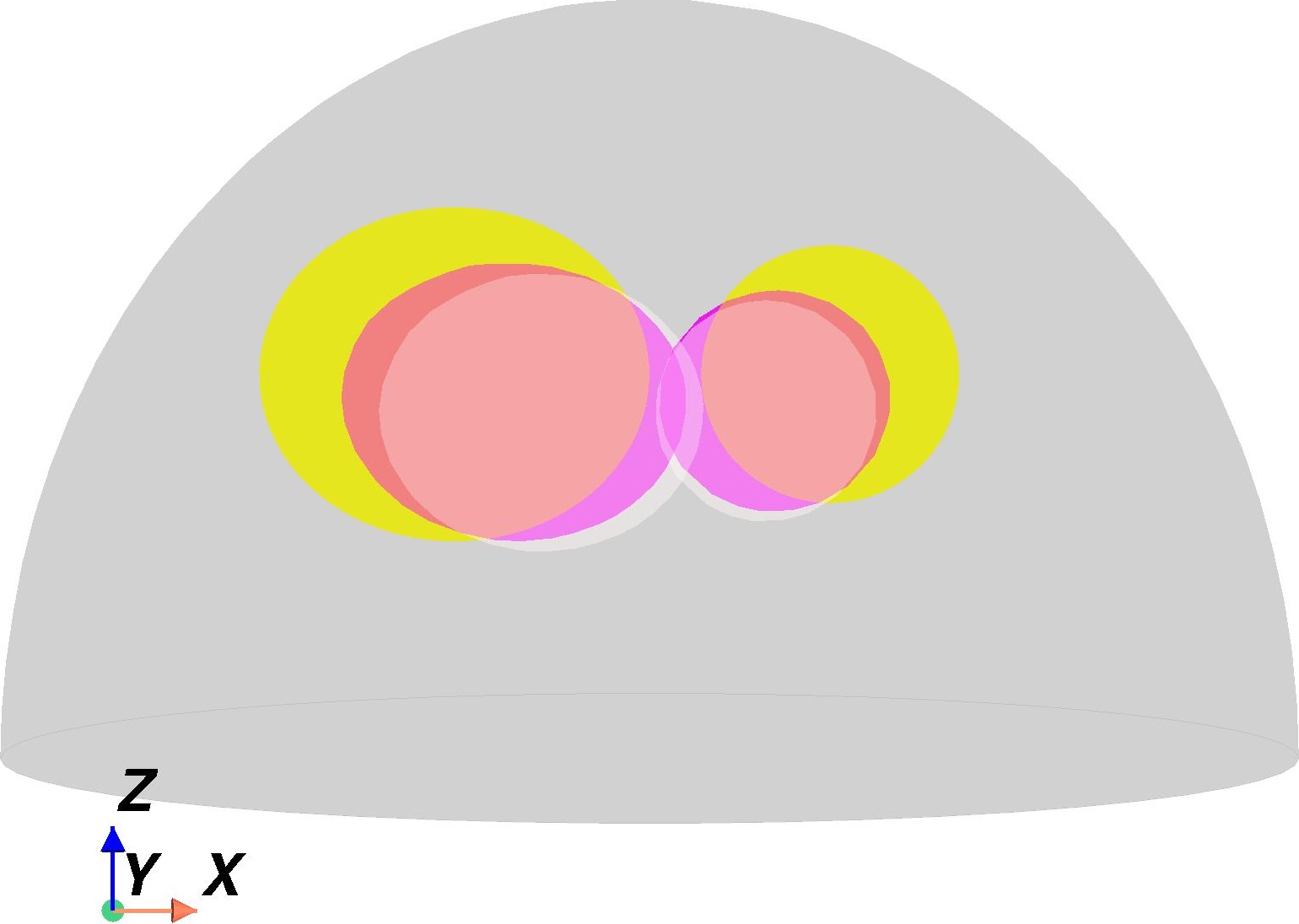}}  
\resizebox{0.15\textwidth}{!}{\includegraphics{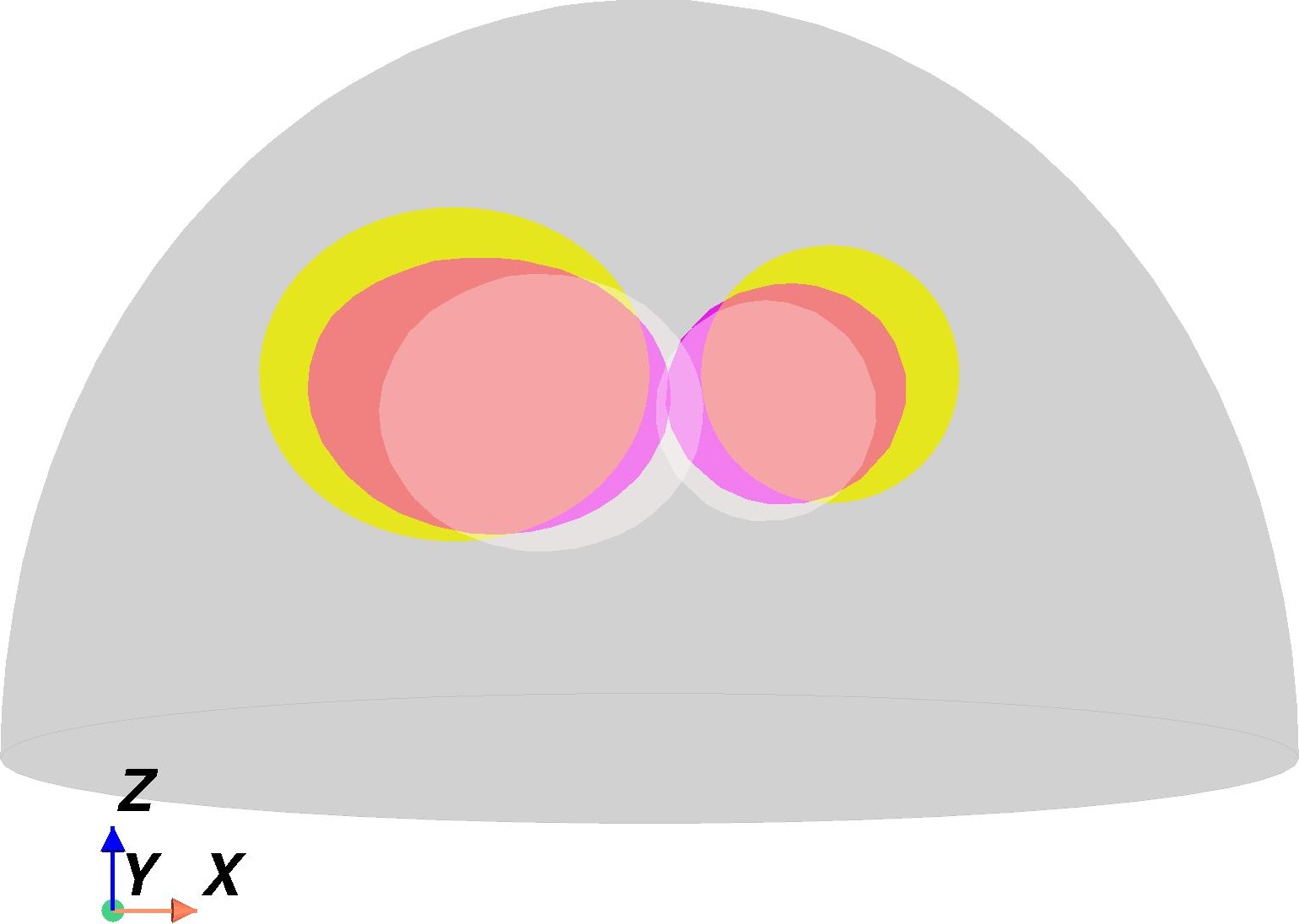}}  
\resizebox{0.15\textwidth}{!}{\includegraphics{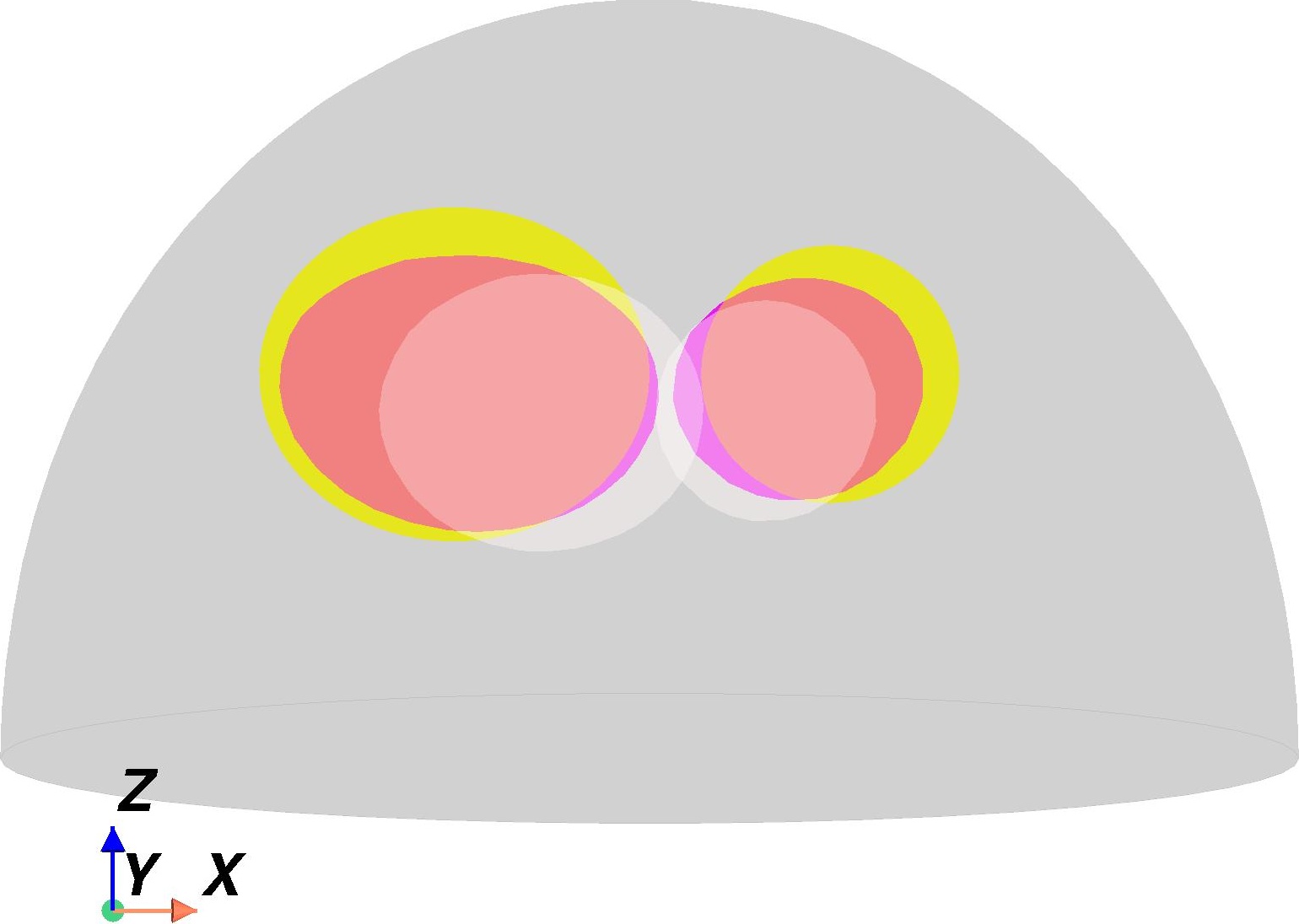}}  
\resizebox{0.15\textwidth}{!}{\includegraphics{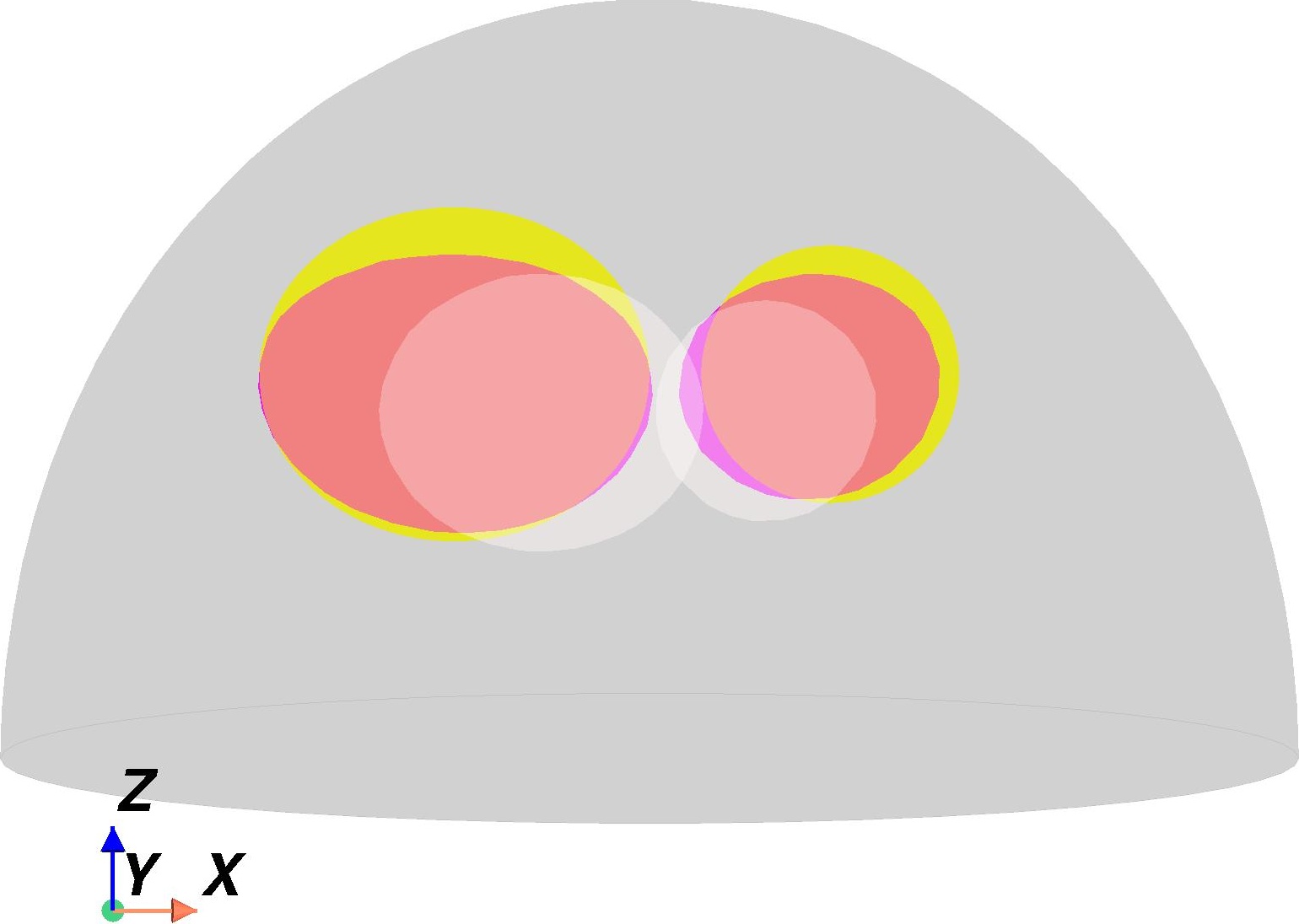}}  
\resizebox{0.15\textwidth}{!}{\includegraphics{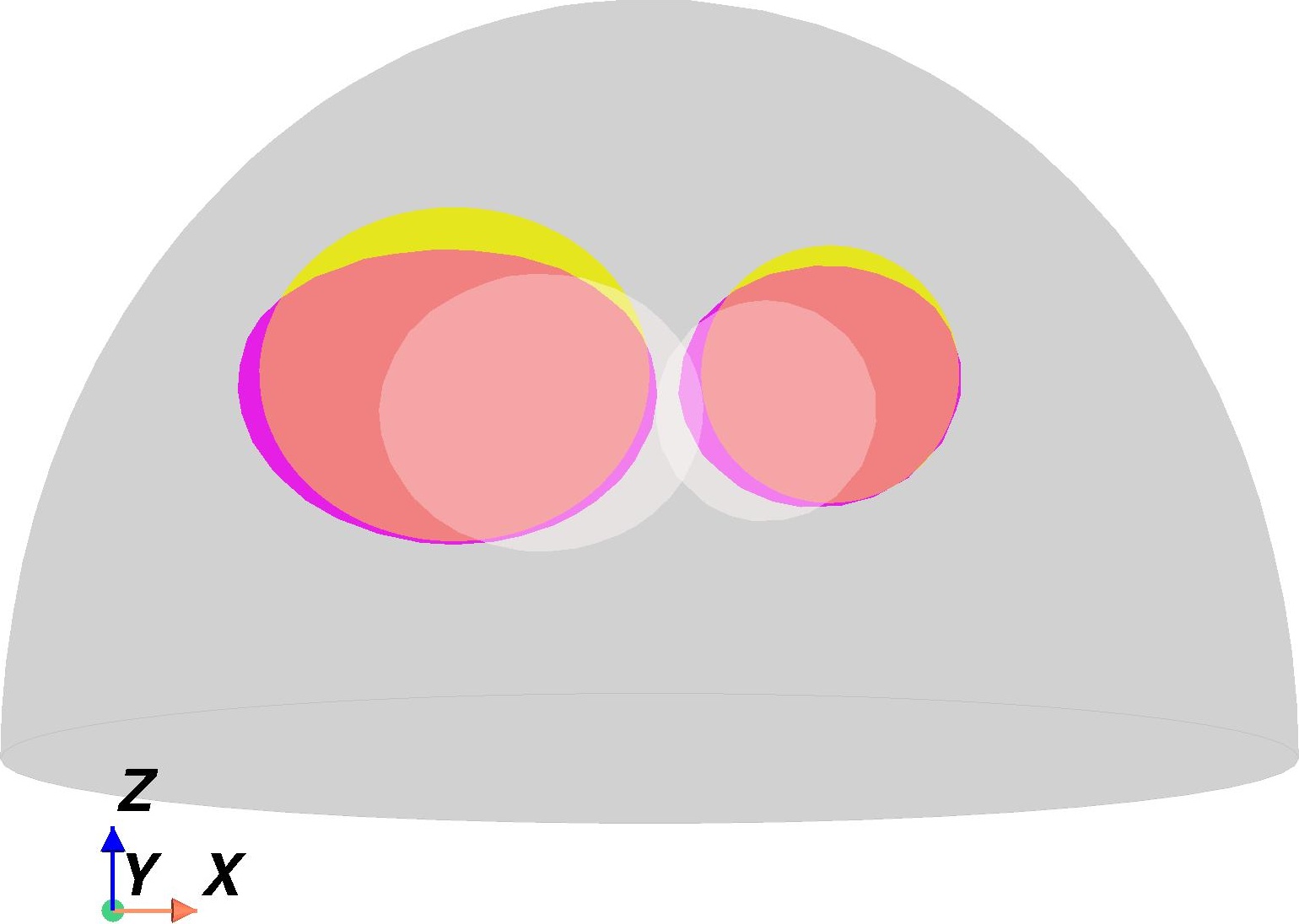}}  
\resizebox{0.15\textwidth}{!}{\includegraphics{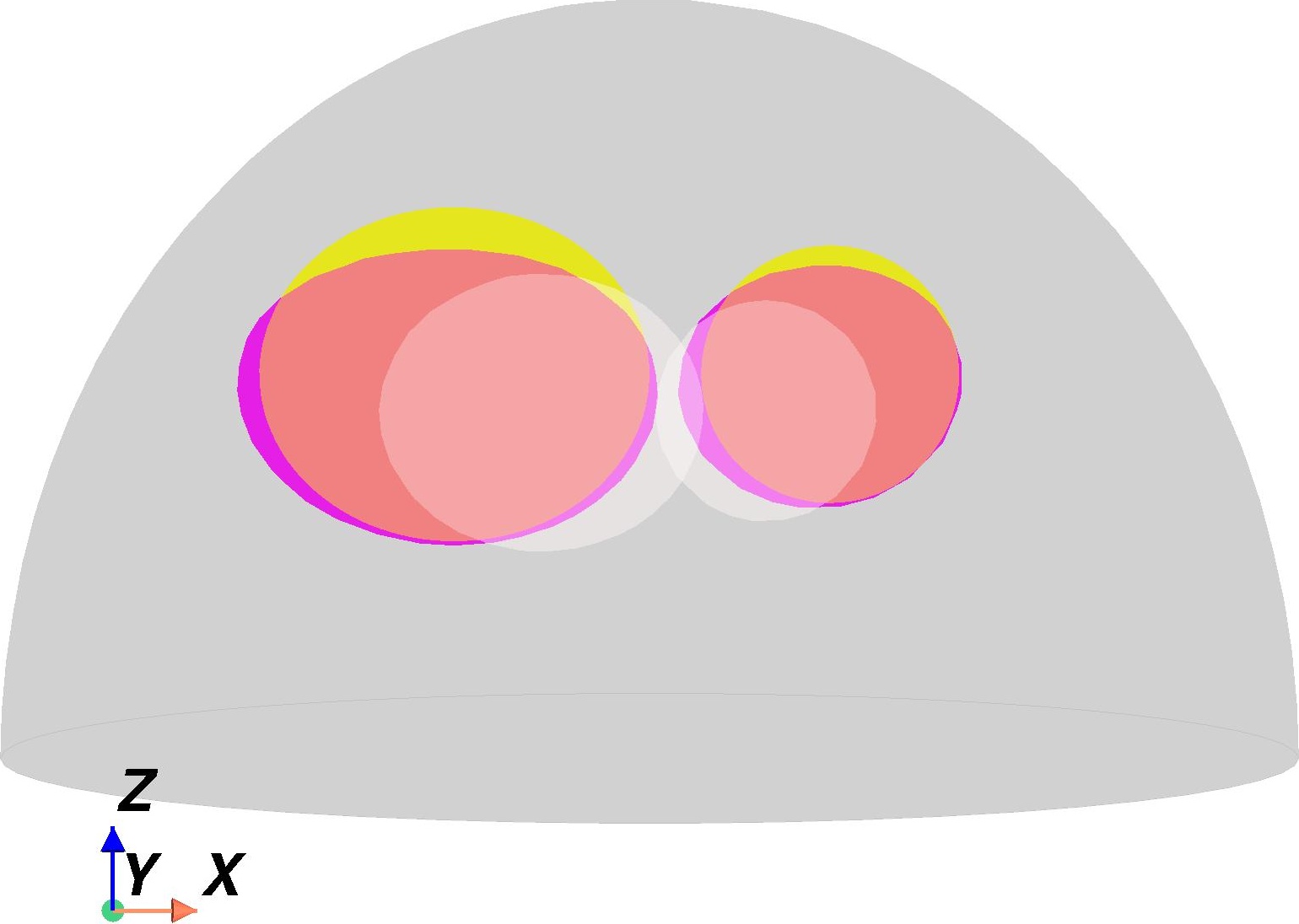}}\\[0.5em]
\resizebox{0.15\textwidth}{!}{\includegraphics{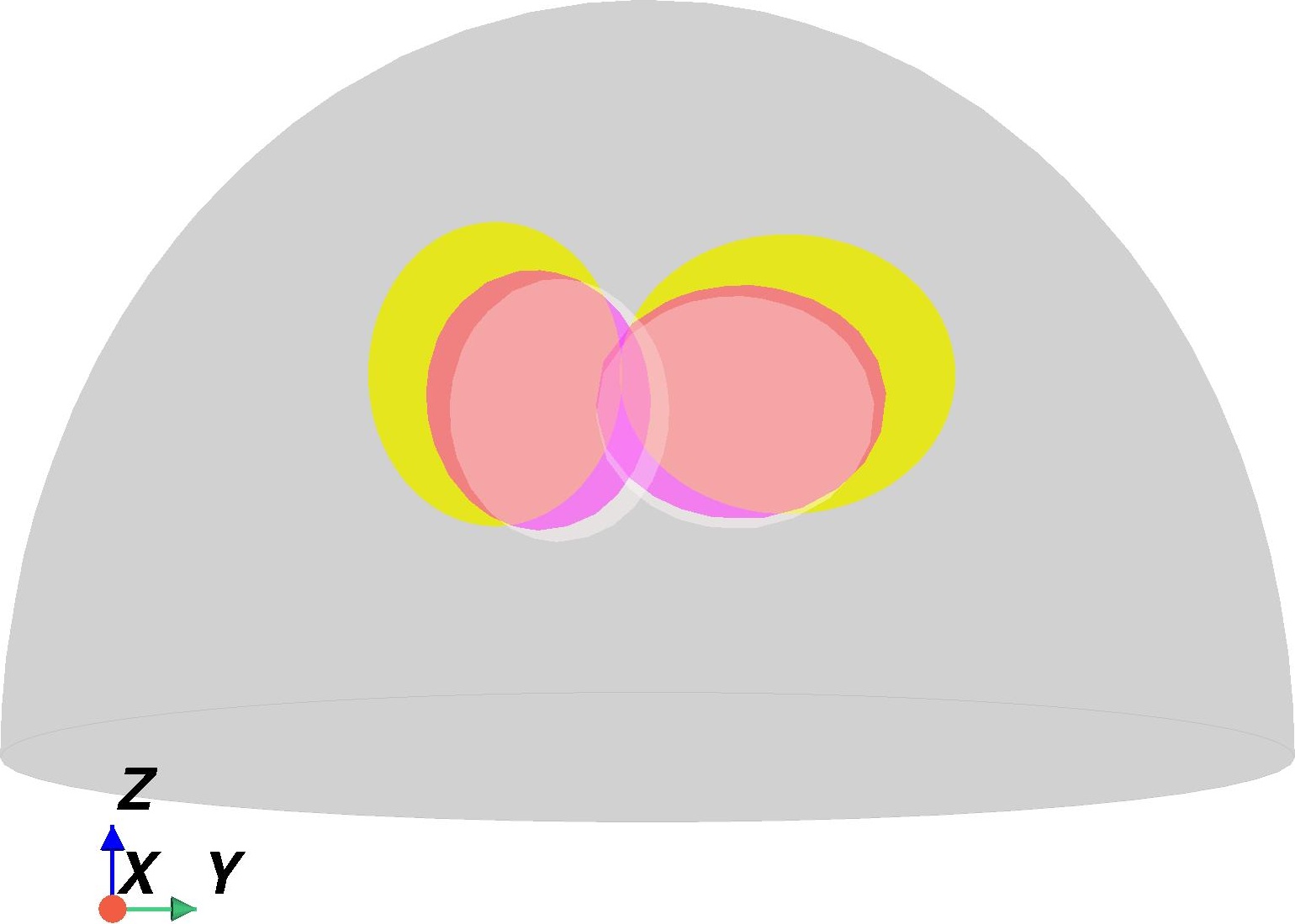}}  
\resizebox{0.15\textwidth}{!}{\includegraphics{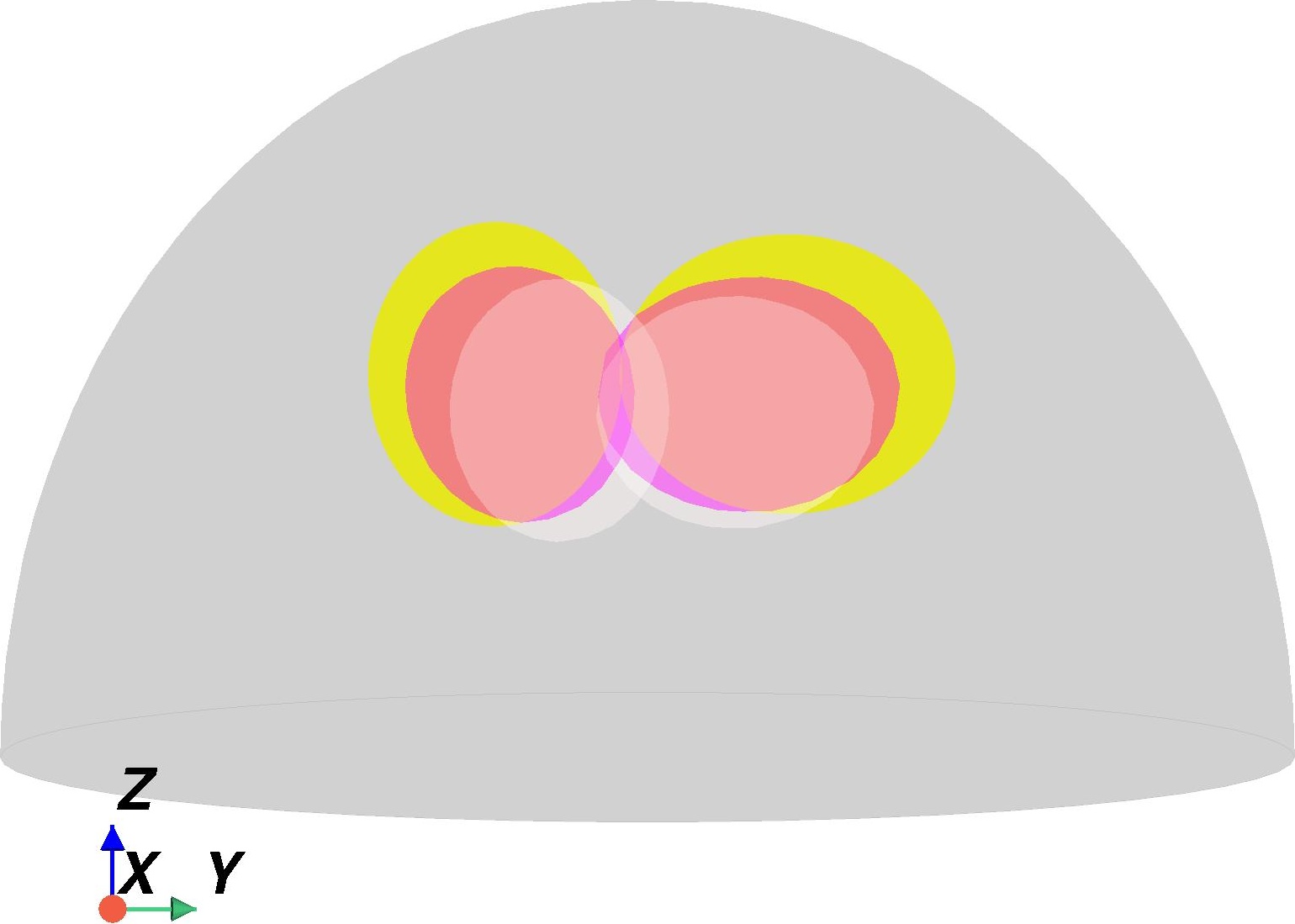}}  
\resizebox{0.15\textwidth}{!}{\includegraphics{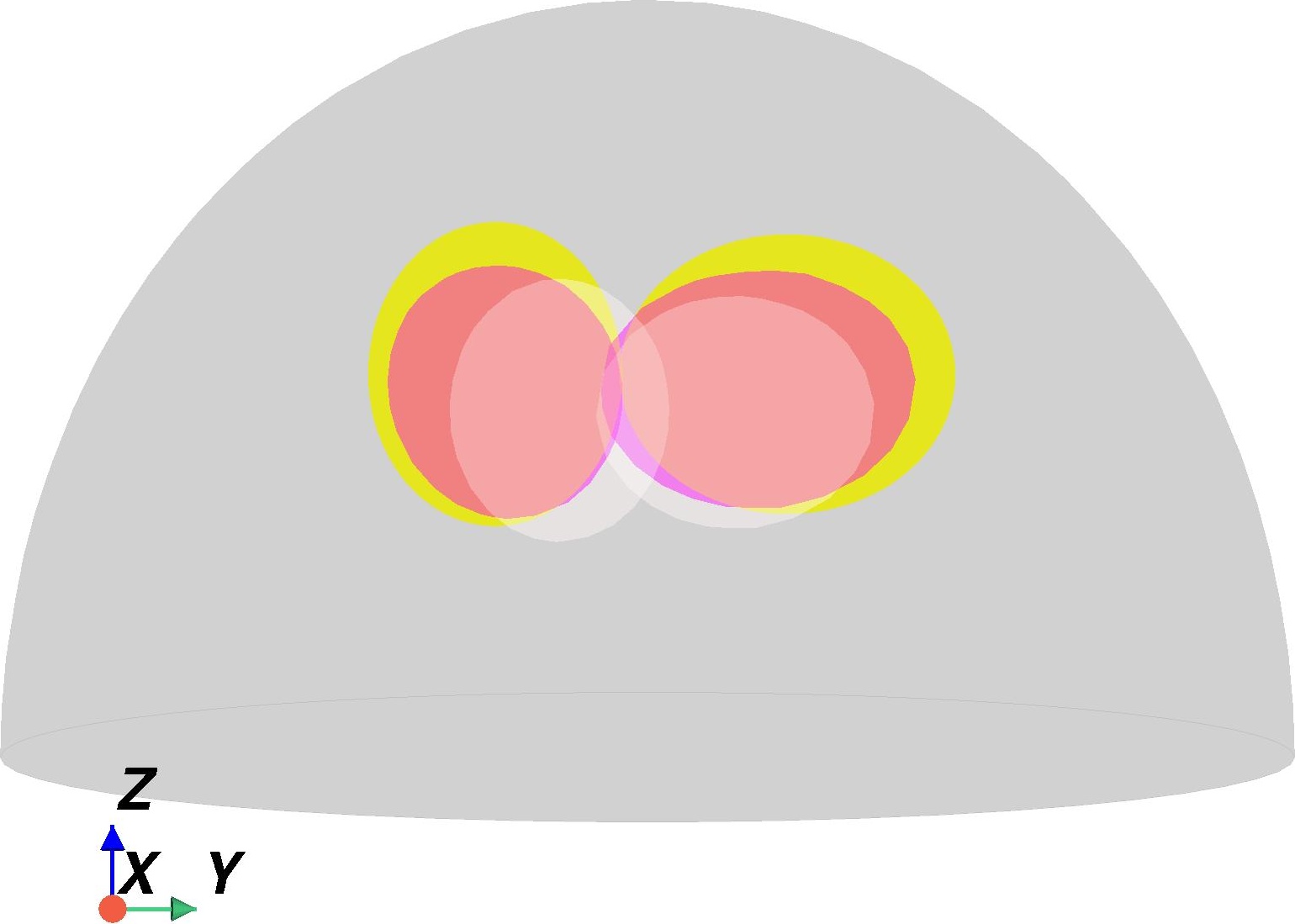}}  
\resizebox{0.15\textwidth}{!}{\includegraphics{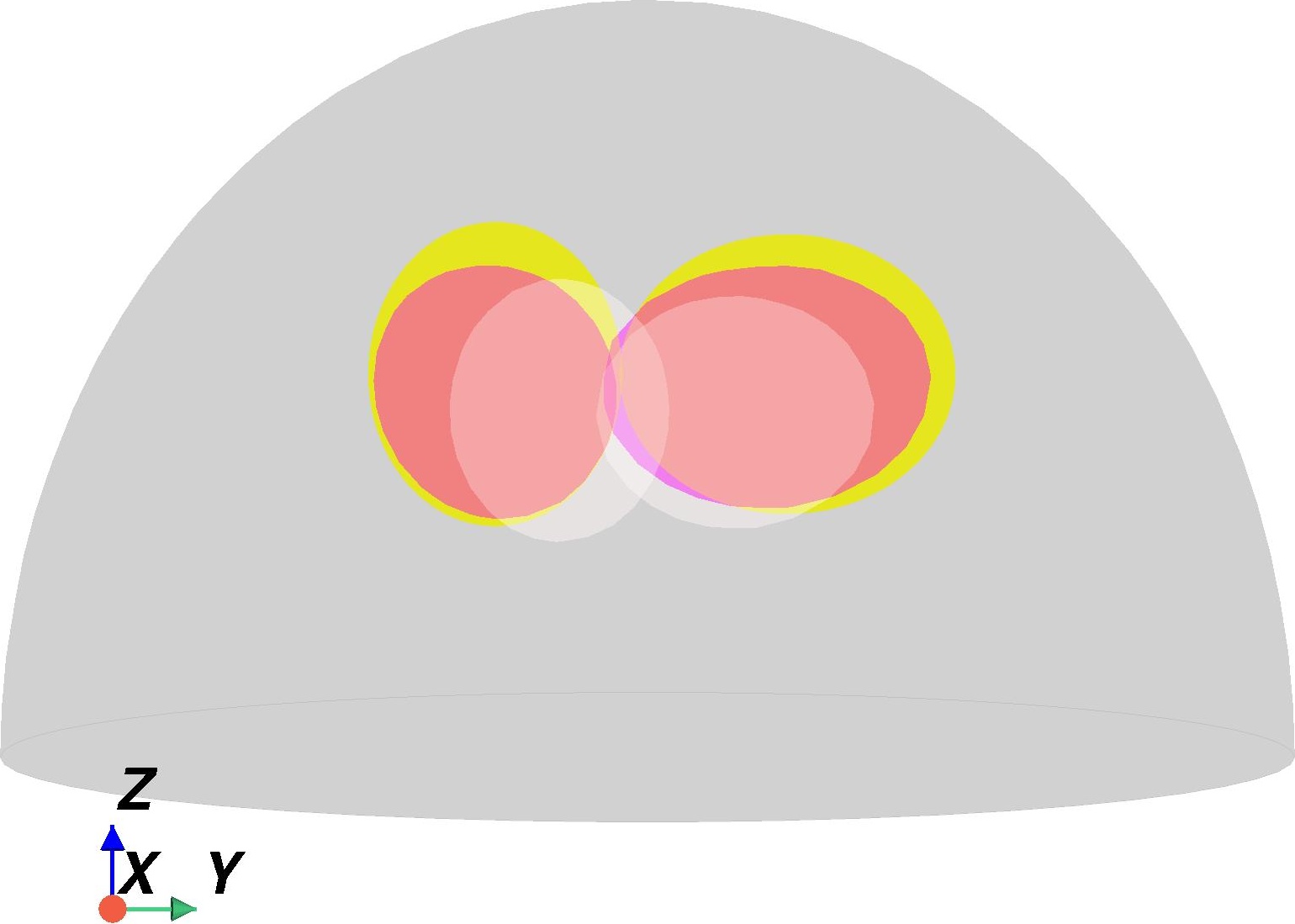}}  
\resizebox{0.15\textwidth}{!}{\includegraphics{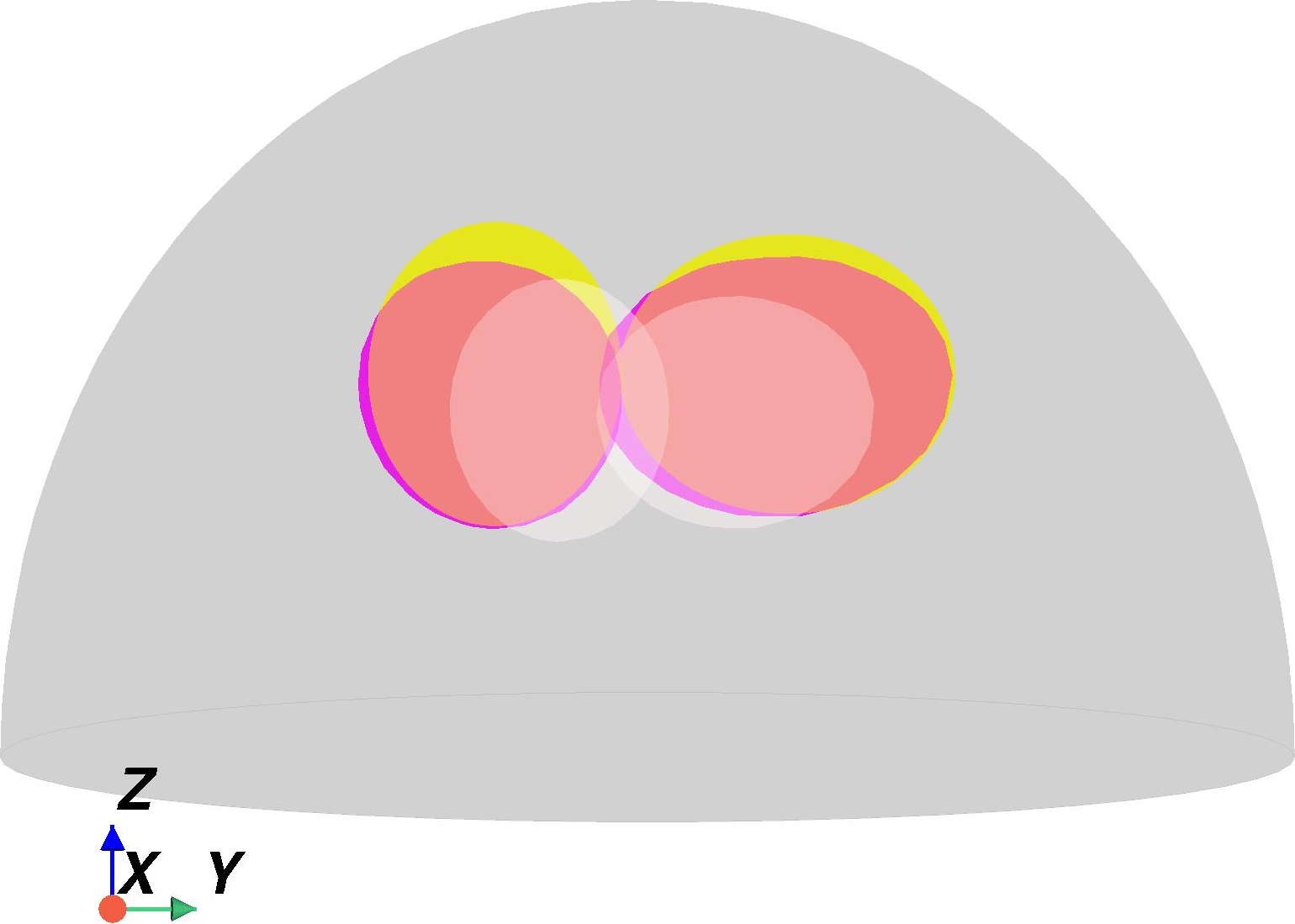}}  
\resizebox{0.15\textwidth}{!}{\includegraphics{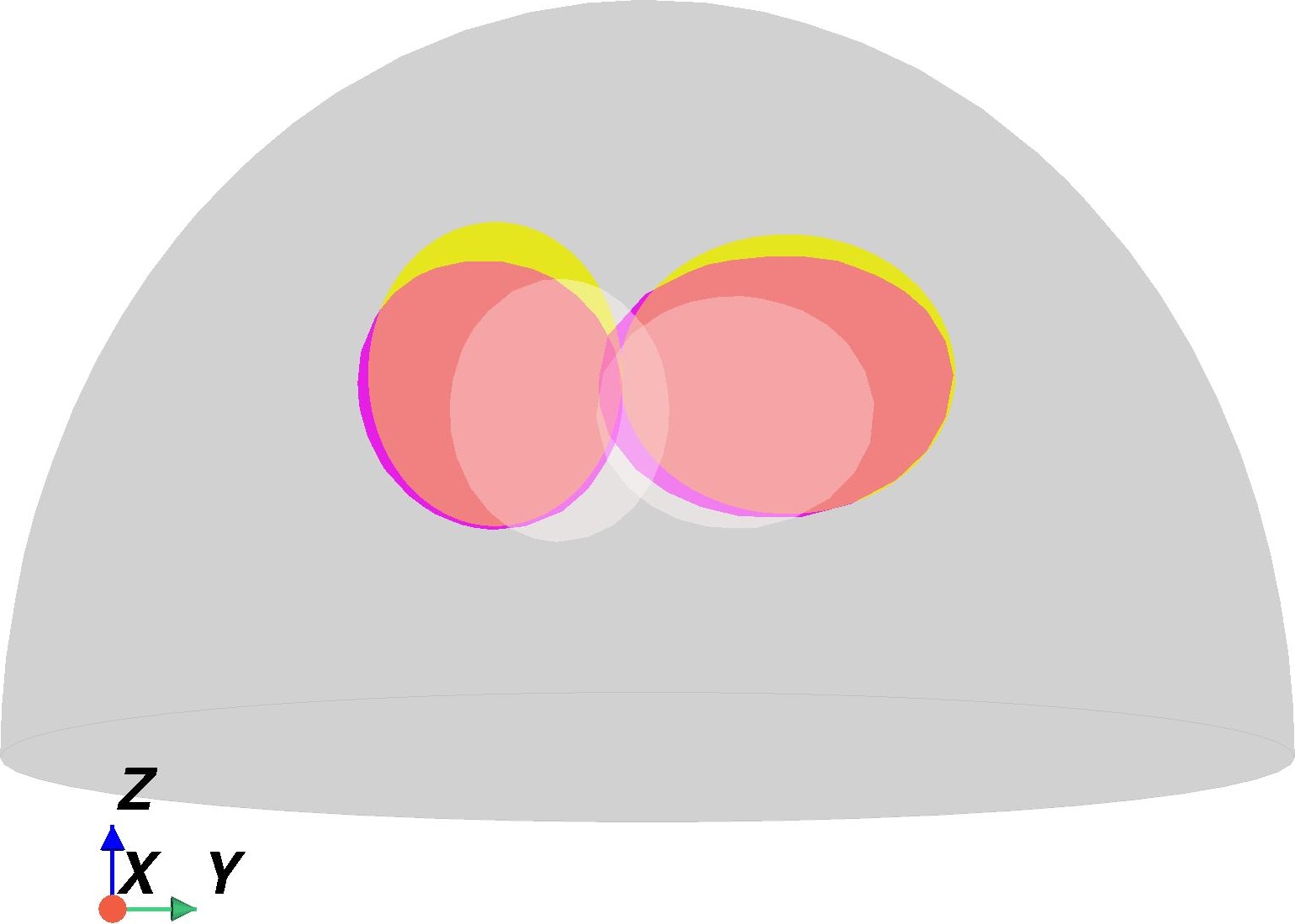}}
\caption{Evolution of the free domain (magenta) at iteration $k = 50, 100, 150, 200, 250, 276$ (final iterate) from the initial guess (light color) versus the exact tumor (yellow) from various views.}
\label{fig:3d_results_multiple_tumors}
\end{figure}
\paruch{
We end the numerical tests with remarks underlining the key findings and advantages of our method over previous work.
\begin{remark}
The numerical results presented here indicate that combining the balancing principle \eqref{eq:balancing_principle} with a small $c_{b} \in (0,1)$ in \eqref{eq:bilinear_form_for_regularization} is effective when the radii are precisely known before localization. 
Although this strongly enforces the volume constraint, it may cause overshooting, rendering the reconstruction sensitive to the initial guess and reliant on an accurate estimate of the inclusion size.
This limitation reduces the practical applicability of the method, as precise knowledge of the inclusion size is required for accurate reconstruction. In contrast, previous methods \cite{ParuchMajchrzak2007, AgnelliPadraTurner2011} assume a known shape but do not depend on the initial size. 
Our method requires a reasonably accurate initial location to apply the balancing principle effectively for detecting small, deep tumors. Nevertheless, it generally remains robust even when the initial size deviates significantly from the true value, as demonstrated in subsections~\ref{subsec:test_2_additional} and \ref{subsec:test_3d_multiple_tumors}.

Additionally, we have considered non-elliptical shapes, which may highlight a limitation of the approach used in \cite{AgnelliPadraTurner2011}, as their method controls only the radii and center.

Following the framework of \cite{PadraSalva2013}, we also address reconstructions involving multiple inclusions with varying shapes. Moreover, we extend the analysis to account for noise in the data, which was not considered in their study.
\end{remark}
}
\section{Conclusion}\label{sec:conclusion}  
This study addresses the tumor localization problem using skin temperature profiles in combination with a non-conventional shape optimization approach.  
The numerical method is based on the $H^1$ Riesz representative of the distributed shape gradient, rather than the conventional boundary-type shape gradient. 
The surface temperature distribution provides an initial estimate of the tumor position. By comparing final cost values across different initial radii, the method can reasonably approximate small, deep tumors even in the presence of noise.
\harbrecht{The use of weighted penalization of the volume integral, combined with the balancing principle, enhances tumor recovery. 
It was also observed that the method remains effective when $c_b$ is chosen close to one and the volume penalization parameter $\rho$ is fixed to a small value. 
In this setting, the balancing principle is not employed, meaning the volume constraint is not strongly enforced.
In conclusion, the proposed Lagrangian--type shape optimization method performs well regardless of tumor size and location and effectively identifies multiple tumors under noisy conditions.}

As a future direction, we aim to extend the Pennes bioheat equation by relaxing its assumptions of constant arterial blood temperature and uniform perfusion, which may not always hold in practice. 
Current research focuses on incorporating spatially varying perfusion, nonlinear tissue properties, and time-dependent effects to enhance the model’s accuracy and applicability. 
In addition, the simultaneous recovery of tumor region parameters and geometry using skin surface temperature data is being actively investigated by the author.

\harbrecht{Incorporating an adaptive strategy based on \textit{a posteriori} error estimation is also planned to improve reconstruction accuracy. 
Finally, future work includes developing a new cost functional that combines the strengths of the conventional and CCBM functionals. 
Although beyond the scope of the present study, this direction---supported by \cite{AgnelliPadraTurner2011,ParuchMajchrzak2007,PadraSalva2013} and our current results---offers promising potential for improvement.} 
The results will be reported in a separate venue.

\medskip

\section*{Acknowledgements} This work is partially supported by the JSPS Postdoctoral Fellowships for Research in Japan and partially by the JSPS Grant-in-Aid for Early-Career Scientists under Japan Grant Number JP23K13012 and the JST CREST Grant Number JPMJCR2014.
The author wishes to thank the referees for their helpful comments, suggestions, and remarks, which have greatly improved the quality of the manuscript, including the addition of subsections~\ref{subsec:test_2_additional} and \ref{subsec:test_3d_multiple_tumors}, as well as Appendix~\ref{appx:a_posteriori_error_estimations}.
He also thanks one of the referees for bringing to his attention the references \cite{Paruch2020,MajchrzakParuch2011,NgJamil2014,VaronOrlandeElicabe2015}, and the other referee for pointing out \cite{PadraSalva2013}.

\renewcommand{\appendix}{\par
  \setcounter{section}{0}
  \setcounter{subsection}{0}
  \gdef\thesection{\Alph{section}}
}
\renewcommand{\thesubsection}{\Alph{subsection}}
\renewcommand{\theequation}{\Alph{subsection}.\arabic{equation}}
\renewcommand\thelemma{\thesection.\arabic{lemma}}   
\renewcommand{\thesubsection}{\Alph{subsection}}
\appendix
\section{Appendices}\label{appx} 
\subsection{Some properties of the transformation $T_{t}$}\label{appx:properties_of_POI}
The following regularities hold (see, e.g., \cite{IKP2006, IKP2008} or \cite[Lem. 3.2, p.~111]{SokolowskiZolesio1992}):
\begin{equation}\label{eq:regular_maps}
\left\{
\begin{aligned}
	[t \mapsto DT_{t}] &\in {{C}}^{1}({\textsf{I}},{{C}}^{0,1}(\overline{\varOmega})^{d\times d}),
		&& [t \mapsto (DT_{t})^{-\top}] \in {{C}}^{1}({\textsf{I}},{{C}}(\overline{\varOmega})^{d\times d}),\\	
	[t \mapsto \dett] &\in {{C}}^{1}({\textsf{I}},{{C}}(\overline{\varOmega})), 
		&&\ [t \mapsto \dett := \det \, DT_{t}] \in {{C}}^{1}({\textsf{I}},{{C}}^{0,1}(\overline{\varOmega})),\\
	[t \mapsto \At] &\in {{C}}^{1}({\textsf{I}},{{C}}(\overline{\varOmega})^{d \times d}),
		&&[t \mapsto \At := \dett\Mt^{\top}\Mt] \in {{C}}({\textsf{I}},{{C}}(\overline{\varOmega})^{d \times d}),\\
	[t \mapsto \bt := \dett \abs{\Mt \nn}] &\in {{C}}^{1}({\textsf{I}},{{C}}(\partial{\omega})), && \Mt:=(DT_{t})^{-\top}.
\end{aligned}
\right.
\end{equation} 
The derivatives of the maps $[t \mapsto \dett]$, $[t \mapsto \At]$, and $[t \mapsto \bt]$ are given by:
\begin{equation}\label{eq:derivatives_of_maps}
\begin{aligned}
\frac{d}{dt} \dett \big|_{t=0} = \dive \VV, \qquad
\frac{d}{dt} \At \big|_{t=0} &= (\dive \VV)\idmat - D\VV - (D\VV)^\top =: A, \\
\frac{d}{dt} \bt \big|_{t=0} &= \dive_{\tau} \VV = \dive \VV \big|_{\tau} - (D\VV\nn) \cdot \nn.
\end{aligned}
\end{equation}
By choosing a smaller interval $\textsf{I}$ if necessary, we assume in this paper that
\begin{equation}\label{eq:bounds_At_and_{b}t}
 \abs{\xi}^2 \leqslant \sigma \At \xi \cdot \xi \leqslant \abs{\xi}^2, \quad \forall \xi \in \mathbb{R}^{d}.
\end{equation}

\subsection{Correa-Seeger Theorem}
Let ${\varepsilon} > 0$ be a fixed real number and consider a functional
\[
	\Lag:[0,{\varepsilon}]\times X \times Y \to \mathbb{R},
\]
for some topological spaces $X$ and $Y$. 
For each $t \in [0, {\varepsilon}]$, we define
\[
	M(t) := \min_{x\in X} \sup_{y\in Y} \Lag(t,x,y)
	\quad\text{and}\quad
	m(t) = \sup_{y\in Y} \min_{x\in X} \Lag(t,x,y),
\]
and the associated sets
\[
	X(t) := \left\{ \hat{x} \in X \mid \sup_{y\in Y} \Lag(t,\hat{x}, y) = M(t)\right\}
	\quad\text{and}\quad
	Y(t) := \left\{ \hat{y} \in Y \mid \min_{x\in X} \Lag(t,x,\hat{y}) = m(t)\right\}.
\]
We introduce the \textit{set of saddle points} $S(t)=\{(\hat{x},\hat{y})\in X\times Y : M(t) = \Lag(t, \hat{x}, \hat{y}) = m(t)\}$, which may be empty.
In general, we always have the inequality $m(t) \leqslant M(t)$, and when $m(t) = M(t)$, the set $S(t)$ is exactly $X(t) \times Y(t)$.

We quote below an improved version of Correa-Seeger theorem stated in \cite[pp.~556--559]{DelfourZolesio2011}.
\begin{theorem}[{{Correa and Seeger, \cite{CorreaSeeger1985}}}]
	\label{thm:Correa_Seeger}
	Let the sets $X$ and $Y$, the real number ${\varepsilon} > 0$, and the functional $\Lag: [0,{\varepsilon}] \times X \times Y \to \mathbb{R}$ be given.
	Assume that the following assumptions hold:
	\begin{description}
		\item[{(H1)}] for $0 \leqslant t \leqslant {\varepsilon}$, the set $S(t)$ is non-empty;
		
		\item[{(H2)}] \sloppy the partial derivative $\partial_{t}\Lag(t,x,y)$ exists everywhere in $[0,{\varepsilon}]$, for all $(x,y) \in \left[ \bigcup_{t \in [0,\varepsilon]} X(t)  \times Y(0)  \right]  \bigcup  \left[X(0) \times \bigcup_{t \in [0,\varepsilon]} Y(t) \right]$;
		
		\item[{(H3)}] there exists a topology $\mathcal{T}_X$ on $X$ such that for any sequence $\{t_{n} : 0 < t_{n} \leqslant {\varepsilon}\}, t_{n} \to t_0 = 0$,
						there exist an $x^{0} \in X(0)$ and a subsequence $\{t_{n_k} \}$ of $\{t_{n}\}$, 
						and for each $k \geqslant 1$, there exists $x_{n_k} \in X(t_{n_k})$ such that
						{{(i)}} $x_{n_k} \to x^{0}$ in the $\mathcal{T}_X$-topology, and
						{{(ii)}} for all $y$ in $Y(0)$, $\liminf_{t\searrow 0,\ k \to \infty} \partial_{t}\Lag(t, x_{n_k} , y) \geqslant \partial_{t}\Lag(0, x^{0}, y)$;
		
		\item[{(H4)}] there exists a topology $\mathcal{T}_Y$ on $Y$ such that for any sequence $\{t_{n} : 0 < t_{n} \leqslant {\varepsilon}\}, t_{n} \to t_0 = 0$,
						there exist $y^{0} \in Y(0)$ and a subsequence $\{t_{n_k} \}$ of $\{t_{n}\}$, 
						and for each $k \geqslant 1$, there exists $y_{n_k} \in Y(t_{n_k})$ such that
						{{(i)}} $y_{n_k} \to y^{0}$ in the $\mathcal{T}_Y$-topology, and
						{{(ii)}} for all $x$ in $X(0)$, $\limsup_{t\searrow 0,\ k \to \infty} \partial_{t}\Lag(t, x, y_{n_k}) \leqslant \partial_{t}\Lag(0, x, y^{0})$;
	\end{description}
	Then, there exists $(x^{0}, y^{0}) \in X(0) \times Y(0)$ such that
	\[
		{d}M(0) = \min_{x\in X(0)} \sup_{y\in Y(0)} \partial_{t}\Lag(0, x, y) 
			= \partial_{t}\Lag(0, x^{0}, y^{0})
			= \sup_{y\in Y(0)} \min_{x\in X(0)} \partial_{t}\Lag(0,x,y).
	\]
	Thus, $(x^{0}, y^{0})$ is a saddle point of $\partial_{t}\Lag(0, x, y)$ on $X(0) \times Y(0)$.
\end{theorem}
%
%
%
\subsection{Material derivative of the state}\label{subsec:material_derivative_of_the_state}
This appendix describes the structure of the \textit{material} derivative of the state, defined as follows (see, e.g., \cite[Eq.~(3.38), p.~111]{SokolowskiZolesio1992}):
\begin{equation}\label{eq:definition_of_the_material_derivative} 
	\dot{u} = \dot{u}(\varOmega)[\VV] = \lim_{t \searrow 0} \frac{u(\varOmega_{t}) \circ T_t - u(\varOmega)}{t} 
\end{equation}
provided that the limit $\dot{u}$ exists in $H^{1}(\varOmega)$, where $(u(\varOmega_{t}) \circ T_t)(x) = u(\varOmega_{t})(T_t(x))$, $x \in \varOmega$.


We have the following assumptions: 
\begin{assumption}\label{eq:strong_assumptions}
We assume $\alpha, T_{a}, T_{b} \in \mathbb{R}_{+}$, $\varOmega \in \Upsilon^{1}$, $\VV \in \sfTheta^{1}$, $\sigma_0, \sigma_1 \in C^1(\overline{\varOmega})$ with $\sigma_0, \sigma_1 > 0$, $\coeffk_0, \coeffk_1 \in C^{1}(\overline{\varOmega})$ with $\coeffk_0, \coeffk_1 > 0$, $Q_0, Q_1 \in C^{1}(\overline{\varOmega})$, and $h \in H^{1/2}(\Gtop)$.
\end{assumption}
\begin{theorem}\label{thm:material_derivative_of_the_state}
Under Assumption~\ref{eq:strong_assumptions}, the state $u = u(\varOmega) \in \Vomega$ has a material derivative $\dot{u} \in \Vomega$ satisfying
\begin{equation}\label{eq:material_derivative}
\begin{aligned} 
a(\dot{u},v)
    &= -\sum_{j=0}^{1}\int_{\varOmega_{j}}{ \left[ 
    		\nabla{{\sigma}_{j}} \cdot \VV (\nabla{u} \cdot \nabla{\conj{v}}) 
    			+ \nabla{\coeffk_{j}} \cdot \VV {u} {\conj{v}}
			- \nabla{{Q}_{j}}\cdot\VV {\conj{v}}
    	\right] }{\, dx}\\
   &\qquad -\sum_{j=0}^{1}\int_{\varOmega_{j}}{ \left[
     		{\sigma_{j}} A \nabla{u} \cdot \nabla{\conj{v}}
		+ {\dive}{\VV}{\coeffk_{j}} {u} {\conj{v}}
		+ {\dive}{\VV}{{Q}_{j}}{\conj{v}}
    	\right] }{\, dx}, \qquad \forall v \in {\Vomega},
\end{aligned}
\end{equation}
where the sesquilinear form $a(\cdot,\cdot)$ is given by \eqref{eq:forms} and $A =  (\dive \VV)\idmat - D\VV - (D\VV)^\top $.
\end{theorem}
\begin{proof}
In the proof, we do not split the integrals over the sub-domains $\healthy$ and $\tumor$ for notational convenience. 
Let the assumptions of the assertion be satisfied.
Let us consider $u_{t} \in {\Vomegat} := H^{1}(\varOmega_{t})$, the solution of the perturbed problem for a given variation {$\VV \in {{\sfTheta}^{1}}$} is given by the solution of 
\begin{equation}\label{eq:perturbed_problem}
a_{t}(u_{t},v_{t}) = l_{t}(v_{t}), \quad \forall v_{t} \in {\Vomegat}. 
\end{equation}
where
\[ 
\begin{aligned}
	a_{t}(u_{t}, v) &:= \intOt{ \left( \sigma_{t} \nabla{u}_{t} \cdot \nabla \conj{v} + \coeffk_{t}u_{t}\conj{v} \right) } + \intGtop{ ( \alpha + i )  u_{t}\conj{v} },\\
	l_{t}(v) &:= \intOt{Q_{t}\conj{v}} + \intGtopt{ (\alpha T_{a} + i h) \conj{v} }.
\end{aligned} 
\]
Here, ${\sigma}_{t} = \sigma_{1,t} \chi_{\varOmega \setminus \bartumort} + \sigma_{0,t} \chi_{\tumort}$, ${\coeffk}_{t} = \coeffk_{1,t} \chi_{\varOmega \setminus \bartumort} + \coeffk_{0,t} \chi_{\tumort}$, and ${Q}= {Q}_{1,t} \chi_{\varOmega \setminus \bartumort} + {Q}_{0,t} \chi_{\tumort}$, where $\chi_{(\cdot)}$ denotes the characterisitc function.  
The change of variables (cf. \cite[pp.~482--484]{DelfourZolesio2011}) allows us to rewrite \eqref{eq:perturbed_problem} as follows:
\[
	a^{t}(\ut,v) = l^{t}(v), \quad \forall v \in {\Vomega},
\]
where
\[
\left\{
\begin{aligned}
    a^{t}(\ut,v) &= \intO{( {{\sigma}_{t}} \At \nabla{\ut} \cdot \nabla{\conj{v}} 
    			+ \dett {\kt} {\ut}{\conj{v}})} 
    			+ \intGtop{ \bt ( \alpha + i ) {\ut}{\conj{v}} }, \quad \text{for } \ut,v \in {\Vomega},\\
    l^{t}(v) &=  \intO{\dett \Qt {\conj{v}}} + \intGtop{ \bt (\alpha T_{a} + i h) {\conj{v}} }, \quad \text{for } v \in {\Vomega},
\end{aligned}
\right.
\]
where $\varphi^{t} = \varphi_{t} \circ T_{t} : \varOmega \to \mathbb{R}$.
We underline here that on $\partial\varOmega = \Gtop \cup \Gwall \cup \Gbottom$, $\bt = \dett \abs{({D}T_t)^{-\top} \nn} = 1$ because $\VV \big|_{\partial\varOmega} = {0}$.

Using the properties of $T_t$ from \eqref{eq:regular_maps} and the bounds from \eqref{eq:bounds_At_and_{b}t}, it can be shown that $\wt = \ut - u \in \Vomega$ is the unique solution to the variational equation $a^{\top}(\ut, v) - a(u, v) = l^{\top}(v) - l(v)$ for all $v \in \Vomega$. 
This equation can also be written as:
\begin{equation}\label{eq:difference_equation}
\tilde{a}(\wt, v) = \tilde{l}(v), \quad \forall v \in {\Vomega},
\end{equation}
where
\begin{equation}\label{eq:difference_equation_for_material_derivative}
\left\{
\begin{aligned}
\tilde{a}(\wt, v) &= \intO{( {{\sigma}_{t}} \nabla{\wt} \cdot \nabla{\conj{v}} + {\kt} {\wt}{\conj{v}})} + \intGtop{ ( \alpha + i ) {\wt}{\conj{v}} }, \quad  \wt, v \in {\Vomega},\\
\tilde{l}(v) &= -\intO{({\sigma}_{t} - {\sigma}) \At \nabla{\ut} \cdot \nabla{\conj{v}}}
-\intO{ {\sigma} (\At - \idmat) \nabla{\ut} \cdot \nabla{\conj{v}}}\\
&\qquad -\intO{ \dett (\kt - k) {\ut} {\conj{v}}}
-\intO{ (\dett - 1){\coeffk}{\ut} {\conj{v}}}
\\
&\qquad +\intO{ \dett (\Qt - Q) {\conj{v}}}
+\intO{ (\dett - 1) Q {\conj{v}}}, \quad \ut, v \in {\Vomega}.\\
\end{aligned}
\right.
\end{equation}
For all $t \in \textsf{I}$, the well-posedness of \eqref{eq:difference_equation} follows from the Lax-Milgram theorem \cite[p.~376]{DautrayLionsv21998}, using standard arguments along with the uniform boundedness of $\{\At\}$ and $\{\dett\}$ on ${\varOmega}$, and Assumption~\ref{eq:strong_assumptions}.
Thus, $\cnorm{\wt}_{V} \lesssim \cnorm{u}_{V}$, for $t \in \textsf{I}$; that is $\{\wt \mid t \in \textsf{I}\}$ is bounded in $V$.  

Let us define $\zt = \frac{1}{t} \wt$ for $t \in (0,t_{0})$ which also belongs to $V$. 
Then, we have
\begin{align*}
\tilde{a}(\zt, v) 
	&= -\intO{\left(\dfrac{{\sigma}_{t} - {\sigma}}{t} \right) \At \nabla{\ut} \cdot \nabla{\conj{v}}}
	-\intO{ {\sigma} \left(\dfrac{\At - \idmat}{t} \right) \nabla{\ut} \cdot \nabla{\conj{v}}}\\
	&\qquad -\intO{ \dett \left(\dfrac{\kt - {\coeffk}}{t} \right) {\ut} {\conj{v}} }
	-\intO{ \left(\dfrac{\dett - 1}{t} \right){\coeffk}{\ut} {\conj{v}} } 
\\
	&\qquad +\intO{ \dett \left(\dfrac{\Qt - Q}{t} \right) {\conj{v}}}
	+\intO{ \left(\dfrac{\dett - 1}{t} \right) {Q} {\conj{v}}}
	= \frac{1}{t} \tilde{l}(v) =: l_{t}(v), \quad (v \in {\Vomega}).
\end{align*}
By choosing $v = \zt$ above, we deduce that ${\zt}$ is bounded in ${\Vomega}$.
Thus, there exists a (sub)sequence ${t_n} \to 0$ and $z \in {\Vomega}$ such that $z^{t_n} \rightharpoonup z$ in ${\Vomega}$.
Since $\nabla{u}^{t_{n}} \to \nabla{u}$ in $L^{2}(\varOmega)^{d}$ and $I_{t_{n}} \rightarrow 1$ and $A_{t_{n}} \rightarrow {\idmat}$ uniformly on $\varOmega$ as $n \rightarrow \infty$, we conclude, by \eqref{eq:derivatives_of_maps} and \cite[Cor.~3.1]{IKP2006}, that
\[
\begin{aligned}
a_{0}(z,v) &:= \intO{( {\sigma}\nabla{z} \cdot \nabla{\conj{v}} + {{\coeffk} {z} {\conj{v}})} } + \intGtop{ (\alpha + i){z}{\conj{v}} }\\ 
&= -\intO{ \left( \nabla{{\sigma}} \cdot \VV (\nabla{u} \cdot \nabla{\conj{v}}) + {\sigma} A \nabla{u} \cdot \nabla{\conj{v}} \right) }
- \intO{ \left( \nabla{\coeffk} \cdot \VV {u} {\conj{v}} + {\dive}{\VV}{\coeffk}{u} {\conj{v}} \right)}\\
&\quad + \intO{ \left( \nabla{{Q}}\cdot\VV {\conj{v}} + {\dive}{\VV} {Q} {\conj{v}} \right) }
=: {l}_{1}(v) + {l}_{2}(v) + {l}_{3}(v) =: l_{0}(v), \quad (\forall v \in {\Vomega}).
\end{aligned}
\]
Because the solution of the above equation is unique, $\{z^{t_n}\}$ converges weakly to $z$ in $\Vomega$ for any sequence $\{t_n\} \to 0$. 
The strong convergence follows from  $a_0(z, z) = \lim_{t_n \searrow 0} \tilde{a}(z^{t_n}, z^{t_n}) = \lim_{t_n \searrow 0} l_{t_n}(z^{t_n}) = l_0(z)$,
combined with the weak convergence.  
Thus, the unique material derivative $z = \dot{u} \in \Vomega$ of $u \in \Vomega$ is characterized by \eqref{eq:material_derivative}.
\end{proof}
\subsection{\harbrecht{A posteriori error estimations}}\label{appx:a_posteriori_error_estimations}
In discrete shape variation, mesh deformation can distort or nearly degenerate triangles. 
This requires remeshing at every iteration to keep the mesh regular. 
To improve temperature accuracy along the boundary $\Gtop$, mesh refinement is also needed after several iterations. 
Identifying the regions for refinement is therefore essential.
Following \cite{PadraSalva2013}, we introduce error estimators for the temperature, $\cnorm{u - \uh}_{H^{1}(\varOmega)} $, and the adjoint variable, $\cnorm{p - \ph}_{H^{1}(\varOmega)}$, which provide a bound for the objective function.

Throughout this section, we assume that $T_{b} \in \mathbb{R}_{+}$ and restrict the problem to two dimensions for simplicity.
In this case, the solution of Problem~\ref{prob:CCBM_weak_form} is obtained in the space $\Womega$. 
We then introduce the finite element space
\[
	\spaceWh = \{\wh \in C(\overline{\varOmega}) \mid \ \wh|_{\Gbottom} = T_{b}, \ \wh|_K \in P_{1}(K), \ \forall K \in \Th \}.
\]
Accordingly, Problem~\ref{prob:CCBM_weak_form_fem} is reformulated as follows:
\begin{equation}\label{eq:weak_form_exact_ccbm}
	\text{Find $\uh \in \spaceWh$ such that $a(\uh, \psi) = l(\psi)$, for all $\psi \in \spaceVh$}.
\end{equation}
Hereinafter, we assume that the triangulation $\Th$ of $\varOmega$ is conforming and exact.
We let $N$ be the number of nodes, and let $B$ be the set of indices of the nodes located on the Dirichlet boundary $\Gbottom$.
We let $\{\eta_{j}\}_{j=1}^{N}$ be the basis of $\spacePone$ and $\{\zeta_{j}\}_{j \in B}$ be the nodes in $\Gbottom$.
Then, the finite element solutions of the state and adjoint problems can be written as $\uh = \sum_{j \notin B} u_{j} \eta_{j} + \sum_{j \in B} T_{b} \eta_{j} \in \spacePone$ and $\ph = \sum_{j \notin B} p_{j} \eta_{j} \in \spacePone$, respectively, for some unknown constants $\{u_{j}\}_{j=1}^{N}$ and $\{p_{j}\}_{j=1}^{N}$.

For convenience, we introduce the notations $\erru = u - \uh \in \Vomega$ and $\errp = p - \ph \in \Vomega$, where the latter represents the error in the finite element solution of the state.
\subsubsection{Estimator of the objective function} 
\begin{theorem}\label{thm:estimator_of_objective_function}
	Let $u$ and $p$ be the weak solutions of the state and adjoint problems, and let $\uh$ and $\ph$ be their respective finite element solutions.
	Then there exists a constant $c>0$ such that
	\[
		\abs{J(\tumor) - \Jh(\tumorh)}
			\leqslant c\left( \cnorm{\erru}_{H^{1}(\varOmega)}^{2} + \cnorm{\errp}_{H^{1}(\varOmega)}^{2} \right)
				+ \text{h.o.t.},
	\]
	where h.o.t. represents higher-order terms.
\end{theorem}
\begin{proof}
	It can be verified that
	\[
		J(\tumor) - \Jh(\tumorh)
		= \frac{1}{2} \intO{ \abs{\imu - \imuh}^{2} } + \intO{ \imuh (\imu - \imuh) }.
	\]
	Obviously, the first integral is bounded  by $\cnorm{\erru}_{H^{1}(\varOmega)}^{2}$.
	For the second integral, we have the following estimations:
	\[
		\bigabs{\intO{ \imuh (\imu - \imuh) }}
		\leqslant \bigabs{\intO{ \imu {\erru} }}
		= \bigabs{ \overline{\tilde{a}(p, \erru)} }
		= \bigabs{ a(\erru, \errp) },
	\]
	where the last equation follows from Galerkin orthogonality.
	From \eqref{eq:bounds_for_sigma_and_k}, we get
	\begin{align*}
	\bigabs{ a(\erru, \errp) } 
		&\leqslant \max\{\maxsigma,\maxcoeffk\}
		\bigabs{ 
			\intO{ \left( \nabla{{\erru}} \cdot \nabla{ \conj{{\errp}} }  
				+  {{\erru}}\conj{{\errp}} \right) }
		} 
		+ \max\{\alpha, 1\} \bigabs{ \intGtop{ {{\erru}} \conj{ {\errp} } }  }.
	\end{align*}
	To get a bound in terms of $\cnorm{\erru}_{H^{1}(\varOmega)}^{2}$ and $\cnorm{\errp}_{H^{1}(\varOmega)}^{2}$, we apply the Cauchy--Schwarz inequality and then Young's inequality to the first summand above.
	
	Meanwhile, the higher-order terms can be estimated from the second summand using the following bound:
	\begin{align*}
		\bigabs{ \intGtop{ {{\erru}} \conj{ {\errp} } }  }
		\leqslant c h \cnorm{\erru}_{H^{1}(\varOmega)} \cnorm{\errp}_{H^{1}(\varOmega)}.
	\end{align*}
	The conclusion follows by applying Young's inequality and then combining the resulting bounds.
\end{proof}
By the previous theorem, the objective error is estimated via the temperature and adjoint temperature errors, computed in the next two subsections.
\subsubsection{Temperature estimator} 
Let $\RK$ be the volume residual inside each triangle $K \in \Th$, and let $\Jgamma$ be the residual on each side $\gamma$ of $K$ (i.e., the edge of $K$ previously denoted by $E_{i}$).
For each triangle $K \in \Th$, we define the \textit{local} indicator $\etaK$ and the \textit{global} indicator $\eta$ as follows:
\[
	\etaK = \left( \hk^{2} \intK{ \abs{\RK}^{2} }  + \sum_{\gamma \subset \partial{K}} \meshgamma \intgamma{ \abs{\Jgamma}^{2} } \right)^{1/2}
	\qquad 
		\text{and}
	\qquad
	\eta = \left( \sum_{K \in \Th} \etaK^{2} \right)^{1/2},
\]
where $\meshgamma = \abs{\gamma}$, $\RK = Q + \operatorname{div}(\sigma \nabla{\uh}) - \coeffk{\uh} = Q - \coeffk{\uh}$,
\[
\Jgamma = 
\begin{cases}
        -\sigma_{1} \ddn{u_{1,\meshh}} - \alpha(u_{1,\meshh} - T_{a}) - i (u_{1,\meshh} - h) & \text{if } \gamma \subset \Gtop, \\[0.5em]
        -\sigma_{1} \ddn{u_{1,\meshh}} & \text{if } \gamma \subset \Gwall, \\[0.5em]
	-\dfrac{1}{2} \left( \sigma_{1} \ddn{u_{1,\meshh}} -\sigma_{0} \ddn{u_{0,\meshh}} \right) & \text{if } \gamma \subset \mathcal{E}_{I}, \\[0.5em]
        0 & \text{if } \gamma \subset \Gbottom,
\end{cases}
\]
and $\mathcal{E}_{I}$ is the set of interior sides.
\begin{lemma}\label{eq:error_estimation_state}
	The error $\erru = u - \uh \in \Vomega$ for the finite element solution of the state satisfies the variational equation:
	\[
		a(\erru, v) = \sum_{K \in \Th} \intK{ \RK (v-\vh)} 
			+ \sum_{\gamma \subset \partial{K}} \intgamma{ \Jgamma (v-\vh)},
		\quad \forall v \in \Vomega, \ \forall \vh \in \spaceVh.
	\]
\end{lemma}
\begin{proof}
Given that $a(\erru,v) = - a(\uh,v-\vh) + l(v-\vh)$, for all $v \in \Vomega$, for all $\vh \in \spaceVh$, the desired equation follows immediately by applying integration by parts on each triangle $K \in \Th$ and using the fact that $\uh \in \spaceWh$.
\end{proof}
\begin{lemma}
	There exists a constant $c>0$ such that $\cnorm{\erru}_{H^{1}(\varOmega)} \leqslant c \eta$.
\end{lemma}
\begin{proof}
	Let $\erru^{I} \in \spaceVh$ denote the Cl\'{e}ment interpolation, as defined in \cite{Clement1975} (see also \cite{Carstensen2006}).
	The following error estimate, which follows directly from the results in that reference, holds for this interpolation:
	\[
		\Psi(\erru) := \left[ 
			\sum_{K \in \Th} \left( \cnorm{\erru - \erru^{I}}_{L^{2}(K)}^{2} \hk^{-2} 
			+ \sum_{\gamma \subset \partial{K}} \cnorm{ \erru - \erru^{I} }_{L^{2}(\gamma)}^{2} \meshgamma^{-1} \right)
		\right]^{1/2} 
		\leqslant c_{2} \cnorm{\nabla{\erru}}_{L^{2}(\varOmega)},
	\]
	for some constant $c_{2} > 0$.
	Letting $v = \erru$ and $\vh = \erru^{I}$ in Lemma~\ref{eq:error_estimation_state}, and using the Cauchy--Schwarz inequality, the above estimate, and the definition of the local indicator $\etaK$, we get
	\begin{align*}
		c_{1}\cnorm{\erru}_{H^{1}(\varOmega)}^{2} 
		\leqslant \abs{a(\erru,\erru)}
		&\leqslant \Psi(\erru) \left( \sum_{K \in \Th} \etaK^{2}	\right)^{1/2} 
		\leqslant c_{2} \cnorm{\erru}_{H^{1}(\varOmega)}
		\left( 
			\sum_{K \in \Th}  \etaK^{2}
		\right)^{1/2}.
	\end{align*}
	Using the definition of the global indicator $\eta$, we obtain the desired estimate.
\end{proof}
\begin{theorem}\label{eq:estimate_per_triangle}
	Let $\mathcal{N}_{K}$ be the union of all triangles that share a side with $K$.
	Then, there exists a constant $c > 0$ such that, for all $K \in \Th$, $\etaK \leqslant c \cnorm{\erru}_{H^{1}(\mathcal{N}_{K})}$.
\end{theorem}
\begin{proof}
We begin with an estimate over the interior of each triangle $K \in \Th$.
Let $\bK = 27 \lambda_1 \lambda_2 \lambda_3$ be the \emph{cubic bubble} of the triangle $K$.
Replacing $v = \RK \bK$ and $\vh = 0$ in Lemma~\ref{eq:error_estimation_state}, and observing that $\bK$ is zero on every side of $K$, we get
\[
	\intK{ {\RK^{2} \bK} }
	= \intK{ \left( \sigma \nabla{\erru} \cdot \nabla{\RK\bK} + \coeffk{\RK\bK} \right) } + \int_{\Gtop \cap K} (\alpha + i){e}{\RK\bK} \, ds.
\]
In finite dimensional setting, the following inequality holds:
\[
	\cnorm{ \varphi }_{L^{2}(K)}^{2} \leqslant c \intK{ { \varphi^{2} \bK} }
	\qquad \text{and} \qquad
	\cnorm{\nabla{ \varphi }}_{L^{2}(K)} \leqslant \dfrac{c}{\hK} \cnorm{ \varphi }_{L^{2}(K)},
\]
for some constant $c>0$.
Using the Cauchy--Schwarz inequality, the previous inequalities, and the fact that the cubic bubble function is at most one, we obtain the following, assuming that $\cnorm{ \RK }_{L^{2}(K)} \neq 0$:
\[
	\cnorm{ \RK }_{L^{2}(K)}^{2} \leqslant \dfrac{c}{\hK} \cnorm{ \erru }_{H^{1}(K)} \cnorm{ \RK }_{L^{2}(K)}
	\quad \Longleftrightarrow \quad
	\left( \hK^{2} \intK{ {\RK^{2} } } \right)^{1/2} \leqslant \hK \cnorm{ \erru }_{H^{1}(K)}.
\]
Now, the estimation over the edges of the triangle $K$, along with those of the neighboring triangles that share a side with $K$, is carried out in a similar fashion and follows the same argument as in the latter part of the proof of Theorem~4 in \cite{PadraSalva2013}; thus, we omit the details.  
Combining the estimate with the one previously obtained leads to the conclusion of the theorem.
\end{proof}
\subsubsection{Adjoint estimator} 
We can also derive the corresponding estimates for the adjoint error $\errp = p - \ph \in \Vomega$.  
We only state the results and omit the proofs, as they essentially follow the same arguments as in the previous subsection, using the same ideas from \cite{PadraSalva2013}.

For each triangle $K \in \Th$, we define the \textit{local} indicator $\muK$ and the \textit{global} indicator $\mu$ as follows:
\[
	\muK = \left( \hk^{2} \intK{ \abs{\RKa}^{2} }  + \sum_{\gamma \subset \partial{K}} \meshgamma \intgamma{ \abs{\Jgammaa}^{2} } \right)^{1/2}
	\qquad 
		\text{and}
	\qquad
	\mu = \left( \sum_{K \in \Th} \muK^{2} \right)^{1/2},
\]
where $\meshgamma = \abs{\gamma}$, $\RKa = \imuh + \operatorname{div}(\sigma \nabla{\ph}) - \coeffk{\ph} = \uh - \coeffk{\ph} $,
\[
\Jgammaa = 
\begin{cases}
        -\sigma_{1} \ddn{p_{1,\meshh}} - \alpha p_{1,\meshh} + i p_{1,\meshh} & \text{if } \gamma \subset \Gtop, \\[0.5em]
        -\sigma_{1} \ddn{p_{1,\meshh}} & \text{if } \gamma \subset \Gwall, \\[0.5em]
	-\dfrac{1}{2} \left( \sigma_{1} \ddn{p_{1,\meshh}} -\sigma_{0} \ddn{p_{0,\meshh}} \right) & \text{if } \gamma \subset \mathcal{E}_{I}, \\[0.5em]
        0 & \text{if } \gamma \subset \Gbottom.
\end{cases}
\]
\begin{lemma}\label{eq:error_estimation_adjoint}
	The error $\errp = u - \uh \in \Vomega$ for the finite element solution of the state satisfies the variational equation:
	\[
		a(\errp, v) = \sum_{K \in \Th} \intK{ \RKa (v-\vh)} 
			+ \sum_{\gamma \subset \partial{K}} \intgamma{ \Jgammaa (v-\vh)},
		\quad \forall v \in \Vomega, \ \forall \vh \in \spaceVh.
	\]
\end{lemma}
\begin{lemma}
	There exists a constant $c>0$ such that $\cnorm{\errp}_{H^{1}(\varOmega)} \leqslant c \mu$.
\end{lemma}
\begin{theorem}
	There exists a constant $c > 0$ such that, for all $K \in \Th$, $\muK \leqslant c \cnorm{\errp}_{H^{1}(\mathcal{N}_{K})}$, where $\mathcal{N}_{K}$ is defined as Theorem~\ref{eq:estimate_per_triangle}.
\end{theorem}
\begin{remark}[Criterion for adaptive strategy]\label{rem:criterion_for_adaptive_strategy}
We define the \textit{local objective function indicator} $\xi_{K}$ for each triangle $K \in \Th$ as $\xi_{K} = \left( \frac{\kappa}{2} \etaK^2 + \frac{1}{2\kappa} \muK^2 \right)^{1/2}$, 
for some constant $\kappa > 0$.
By choosing $\kappa = \sqrt{\max_j\{\eta_j\}/\max_j\{\eta_j^a\}}$, one is ensured that $\max\left\{\frac{\kappa}{2} \muK^2\right\} = \max\left\{\frac{1}{2\kappa} \etaK^2\right\}$, which implies that triangles with large $\muK$ or large $\etaK$ can be refined in a similar manner, depending on the adaptive refinement strategy employed. 
\end{remark}
We then define the \textit{global objective function estimator} $\xi$ by $\xi = \left( \sum_{K \in \Th} \xi_{K}^{2} \right)^{1/2}$.

Using the estimates for the state and adjoint variables, we obtain the following consequence of Theorem~\ref{thm:estimator_of_objective_function}.
\begin{corollary}
	There exists a constant $c>0$ such that $\abs{J(\tumor) - \Jh(\tumorh)} \leqslant c \xi + \text{h.o.t.}$, where h.o.t. represents higher-order terms.
\end{corollary}
%
\bibliographystyle{elsarticle-num} 
\bibliography{main}   

\end{document}